\documentclass[11pt,twoside,a4paper]{report}

\usepackage[DM,newLogo,final]{uaThesis}

\def\ThesisYear{2012}

\usepackage[english]{babel}
\usepackage{hyperref}
\usepackage{amsmath}
\usepackage{amssymb}
\usepackage{amsthm}
\usepackage{graphicx}
\usepackage{newlfont}
\usepackage{setspace}
\usepackage{subfig}
\usepackage{color}
\usepackage{mathtools}
\usepackage{relsize}

\usepackage{url}
\usepackage{fancybox}
\usepackage[Lenny]{fncychap}
\usepackage{makeidx}
\usepackage{multirow}

\newtheorem{theorem}{Theorem}[section]
\newtheorem{proposition}{Proposition}
\newtheorem{definition}{Definition}
\newtheorem{example}{Example}
\newtheorem{remark}{Remark}

\makeindex

\def\I{\mathtt{i}}

\def\EXP#1.{e^{2\pi\I #1}}

\def\topfigrule{\kern 7.8pt \hrule width\textwidth\kern -8.2pt\relax}
\def\dblfigrule{\kern 7.8pt \hrule width\textwidth\kern -8.2pt\relax}
\def\botfigrule{\kern -7.8pt \hrule width\textwidth\kern 8.2pt\relax}

\begin{document}


\TitlePage
\HEADER{\BAR\FIG{\includegraphics[height=60mm]{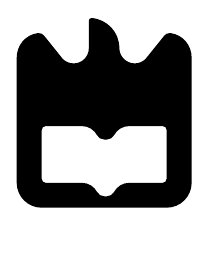}}}{\ThesisYear}
\TITLE{Helena Sofia \newline Ferreira Rodrigues}{Optimal Control
and Numerical Optimization Applied to Epidemiological Models\\[1cm]
Controlo \'{O}timo e Otimiza\c{c}\~{a}o Num\'{e}rica \newline Aplicados a Modelos Epidemiol\'{o}gicos}
\EndTitlePage


\TitlePage
\HEADER{}{\ThesisYear}
\TITLE{Helena Sofia \newline Ferreira Rodrigues}{Optimal Control
and Numerical Optimization Applied to Epidemiological Models\\[1cm]
Controlo \'{O}timo e Otimiza\c{c}\~{a}o Num\'{e}rica \newline Aplicados a Modelos Epidemiol\'{o}gicos}
\vspace*{15mm}
\TEXT{}{Tese de Doutoramento apresentada \`a Universidade de Aveiro para cumprimento dos requisitos
necess\'arios \`a obten\c c\~ao do grau de Doutor em Matem\'{a}tica,
Programa Doutoral em Matem\'{a}tica e Aplica\c{c}\~{o}es, PDMA 2008--2012,
da Universidade de Aveiro e Universidade do Minho realizada sob a orienta\c c\~ao
cient\'\i fica do Prof. Doutor Delfim Fernando Marado Torres, Professor Associado
com Agrega\c c\~ao do Departamento de Matem\'{a}tica da Universidade de Aveiro
e da Prof. Doutora Maria Teresa Torres Monteiro, Professora Auxiliar
do Departamento de Produ\c c\~ao e Sistemas da Universidade do Minho.}
\vspace*{5mm}
\TEXT{}{Ph.D. thesis submitted to the University of Aveiro in fulfilment
of the requirements for the degree of Doctor of Philosophy in Mathematics,
Doctoral Programme in Mathematics and Applications 2008--2012,
of the University of Aveiro and University of Minho, under the supervision
of Professor Delfim Fernando Marado Torres, Associate Professor with Habilitation
and tenure of the Department of Mathematics of University of Aveiro
and Professor Maria Teresa Torres Monteiro, Assistant Professor
of the Department of Production and Systems in University of Minho.}
\EndTitlePage


\titlepage\ \endtitlepage


\TitlePage
\vspace*{55mm}
\TEXT{\textbf{o j\'uri~/~the jury\newline}}{}
\TEXT{presidente~/~president}{\textbf{Prof. Doutor Fernando Manuel dos Santos Ramos}\newline
{\small Professor Catedr\'atico da Universidade de Aveiro}}
\vspace*{5mm}
\TEXT{vogais~/~examiners committee}{\textbf{Prof. Doutor Pedro Nuno Ferreira Pinto de Oliveira}\newline
{\small Professor Associado com Agrega\c{c}\~{a}o do Instituto de Ci\^{e}ncias Biom\'{e}dicas\\ Abel Salazar da Universidade do Porto}}
\vspace*{5mm}
\TEXT{}{\textbf{Prof. Doutor Delfim Fernando Marado Torres}\newline
{\small Professor Associado com Agrega\c{c}\~{a}o da Universidade de Aveiro (Orientador)}}
\vspace*{5mm}
\TEXT{}{\textbf{Prof. Doutora Senhorinha F\'{a}tima Capela Fortunas Teixeira}\newline
{\small Professora Associada da Escola de Engenharia da Universidade do Minho}}
\vspace*{5mm}
\TEXT{}{\textbf{Prof. Doutor Gast\~{a}o Silves Ferreira Frederico}\newline
{\small Professor Auxiliar da Universidade de Cabo Verde}}
\vspace*{5mm}
\TEXT{}{\textbf{Prof. Doutora Maria Teresa Torres Monteiro}\newline
{\small Professora Auxiliar da Escola de Engenharia da Universidade do Minho\\ (Coorientadora)}}
\vspace*{5mm}
\TEXT{}{\textbf{Prof. Doutor Jo\~{a}o Pedro Antunes Ferreira da Cruz}\newline
{\small Professor Auxiliar da Universidade de Aveiro}}
\EndTitlePage


\titlepage\ \endtitlepage


\TitlePage
\vspace*{55mm}
\TEXT{\textbf{agradecimentos~/\newline acknowledgements}}{It is a pleasure
to thank the many people who made this thesis possible.}
\TEXT{}{First and foremost my gratitude goes to my advisors Prof. Delfim Torres and Prof. Teresa Monteiro.
They were a continuous source of inspiration and encouragement, always willing to listen
to my ideas and questions and they always provided me excellent advice. Their friendship was constant over the work.}
\TEXT{}{I owe particular thanks to Prof. Ismael Vaz who has been generous
with his informatics knowledge and to Prof. Matthias Gerdts for his support in OC-ODE software.}
\TEXT{}{Finally, I must thank my friends and family for being so patient and supportive. Special thanks to my mother,
to whom I dedicate this thesis, for her endless energy and optimism. Vitor, thank you for giving
me the strength to embark on this journey.}
\vspace*{80mm}
\TEXT{}{I wish to express special gratitude for the financial support given by the Portuguese Foundation
for Science and Technology (FCT), not only for the Ph.D. grant SFRH/BD/33384/2008,
but also for the support provided to R\&D center ALGORITMI (University of Minho)
and to the Center for Research and Development in Mathematics and Applications (CIDMA - University of Aveiro)
allowing for them to invest in my formation and participation in international conferences.\newline\newline
\includegraphics[scale=0.4]{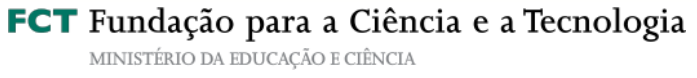}\newline
\includegraphics[scale=0.25]{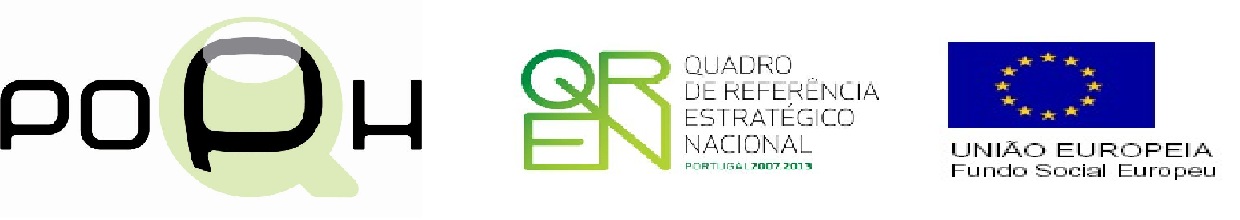}}
\EndTitlePage


\titlepage\ \endtitlepage


\TitlePage
\vspace*{55mm}
\TEXT{\textbf{Resumo}}{A rela\c{c}\~{a}o entre a epidemiologia, a modela\c{c}\~{a}o matem\'{a}tica
e as ferramentas computacionais permite construir e testar teorias sobre o
desenvolvimento e combate de uma doen\c{c}a.}
\TEXT{}{Esta tese tem como motiva\c{c}\~{a}o o estudo de modelos epidemiol\'{o}gicos aplicados
a doen\c{c}as infeciosas numa perspetiva de Controlo \'{O}timo, dando parti\-cular relev\^{a}ncia ao Dengue.
Sendo uma doen\c{c}a tropical e subtropical transmitida por mosquitos,
afecta cerca de 100 milh\~{o}es de pessoas por ano, e \'{e} considerada
pela Organiza\c{c}\~{a}o Mundial de Sa\'{u}de como uma grande preocupa\c{c}\~{a}o para a sa\'{u}de p\'{u}blica.}
\TEXT{}{Os modelos matem\'{a}ticos desenvolvidos e testados neste trabalho, baseiam-se em equa\c{c}\~{o}es
diferenciais ordin\'{a}rias que descrevem a din\^{a}mica subjacente \`{a} doen\c{c}a nomeadamente
a intera\c{c}\~{a}o entre humanos e mosquitos. \'{E} feito um estudo anal\'{\i}tico dos mesmos relativamente
aos pontos de equil\'{\i}brio, sua estabilidade e n\'{u}mero b\'{a}sico de reprodu\c{c}\~{a}o.}
\TEXT{}{A  propaga\c{c}\~{a}o do Dengue pode ser atenuada atrav\'{e}s de medidas de controlo do vetor transmissor,
tais como o uso de inseticidas espec\'{\i}ficos e campanhas educacionais.
Como o desenvolvimento de uma potencial vacina tem sido uma aposta mundial recente,
s\~{a}o propostos modelos baseados na simula\c{c}\~{a}o de um hipot\'{e}tico processo de vacina\c{c}\~{a}o numa popula\c{c}\~{a}o.}
\TEXT{}{Tendo por base a teoria de Controlo \'{O}timo, s\~{a}o analisadas as estrat\'{e}gias \'{o}timas
para o uso destes controlos e respetivas repercuss\~{o}es na redu\c{c}\~{a}o/erradica\c{c}\~{a}o
da doen\c{c}a aquando de um surto na popula\c{c}\~{a}o, considerando uma abordagem bioecon\'{o}mica.}
\TEXT{}{Os problemas formulados s\~{a}o resolvidos numericamente usando m\'{e}todos diretos e indiretos.
Os primeiros discretizam o problema reformulando-o num problema de optimiza\c{c}\~{a}o n\~{a}o linear.
Os m\'{e}todos indiretos usam o Princ\'{\i}pio do M\'{a}ximo de Pontryagin como condi\c{c}\~{a}o necess\'{a}ria
para encontrar a curva \'{o}tima para o respetivo controlo. Nestas duas estrat\'{e}gias
utilizam-se v\'{a}rios pacotes de software num\'{e}rico.}
\TEXT{}{Ao longo deste trabalho, houve sempre um compromisso entre o realismo
dos modelos epidemiol\'{o}gicos e a sua tratabilidade em termos matem\'{a}ticos.}
\TEXT{}{}
\TEXT{}{}
\vspace*{5mm}
\TEXT{\textbf{Palavras Chave}}{Controlo \'{O}timo}
\TEXT{}{Otimiza\c{c}\~{a}o N\~{a}o Linear}
\TEXT{}{Modelos Epidemiol\'{o}gicos}
\TEXT{}{Dengue}
\TEXT{}{}
\TEXT{}{}
\TEXT{\textbf{2010 Mathematics\newline Subject Classification}}{34H05, 92B05, 92D30, 93C15, 93C95}
\EndTitlePage


\titlepage\ \endtitlepage


\TitlePage
\vspace*{55mm}
\TEXT{\textbf{Abstract}}{The relationship between epidemiology,
mathematical modeling and computational tools allows to build
and test theories on the development and battling of a disease.}
\TEXT{}{This thesis is motivated by the study of epidemiological
models applied to infectious diseases in an Optimal Control perspective,
giving particular relevance to Dengue. It is a subtropical and tropical
disease transmitted by mosquitoes, that affects about 100 million people
per year and is considered by the World Health Organization
as a major concern for public health.}
\TEXT{}{The mathematical models developed and tested in this work,
are based on ordinary differential equations that describe the dynamics
underlying the disease, including the interaction between humans and mosquitoes.
An analytical study is made related to equilibrium points,
their stability and basic reproduction number.}
\TEXT{}{The spreading of Dengue can be attenuated through measures to control
the transmission vector, such as the use of specific insecticides and educational campaigns.
Since the development of a potential vaccine has been a recent global bet,
models based on the simulation of a hypothetical vaccination process in a population are proposed.}
\TEXT{}{Based on the Optimal Control theory, we have analyzed the optimal strategies
for using these controls and respective impact on the reduction / eradication of the disease
during an outbreak in the population considering a bioeconomic approach.}
\TEXT{}{The formulated problems are numerically solved using direct and indirect methods.
The first discretize the problem turning it into a nonlinear optimization problem.
Indirect methods use the Pontryagin Maximum Principle as a necessary condition
to find the optimal curve for the respective control. In these two strategies
several numerical software packages are used.}
\TEXT{}{Throughout this study, there was a compromise between
the realism of epidemiological models and their mathematical tractability.}
\TEXT{}{}
\TEXT{}{}
\vspace*{5mm}
\TEXT{\textbf{Keywords}}{Optimal control}
\TEXT{}{Nonlinear Optimization}
\TEXT{}{Epidemiological models}
\TEXT{}{Dengue}
\TEXT{}{}
\TEXT{\textbf{2010 Mathematics\newline Subject Classification}}{34H05, 92B05, 92D30, 93C15, 93C95}
\EndTitlePage


\titlepage\ \endtitlepage


\sffamily
\pagenumbering{roman}
\tableofcontents
\clearpage{\thispagestyle{empty}\cleardoublepage}


\listoffigures
\clearpage{\thispagestyle{empty}\cleardoublepage}

\listoftables
\clearpage{\thispagestyle{empty}\cleardoublepage}


\chapter*{Acronyms}

\onehalfspacing
\small{
\begin{tabular}{ll} \hline
& \\
BDF & : Backward differentiation formulae \\
BRDFE & : Biologically Realistic Disease Free equilibrium\\
BVP & : Boundary Value Problem\\
DAE & : Differential Algebraic Equation\\
DF & : Dengue Fever \\
DHF & : Dengue Hemorrhagic Fever \\
DFE & : Disease-Free Equilibrium\\
EE & : Endemic Equilibrium \\
IP & : Interior Point \\
IVP & : Initial Value Problem \\
OC & : Optimal Control \\
ODE & : Ordinary Differential Equation \\
MSEIR & : Maternal immunity-Susceptible-Exposed-Infected-Recovered\\
NLP & : Nonlinear Problem\\
PMP & : Pontryagin's Maximum Principle \\
$\mathcal{R}_0$ & : Basic Reproduction Number \\
SEIR & : Susceptible-Exposed-Infected-Recovered\\
SEIR+ASEI & : Susceptible-Exposed-Infected-Recovered + Aquatic phase-Susceptible-Exposed-Infected\\
SIS & : Susceptible-Infected-Susceptible\\
SIR & : Susceptible-Infected-Recovered\\
SIR+ASI & : Susceptible-Infected-Recovered + Aquatic phase-Susceptible-Infected\\
SVIR & : Susceptible-Vaccinated-Infected-Recovered\\
SQP & : Sequential Quadratic Programming \\
WHO & : World Health Organization\\ \hline
\end{tabular}
}

\clearpage{\thispagestyle{empty}\cleardoublepage}


\phantomsection\addcontentsline{toc}{chapter}{Introduction}
\chapter*{Introduction}
\label{intro}

\begin{flushright}
\begin{minipage}[r]{9cm}
\bigskip
\small {\emph{``Mathematical biology is a fast-growing,
well-recognized, albeit not clearly defined, subject and is, to my
mind, the most exciting modern application of mathematics.''}}

\begin{flushright}
\tiny{--- J. D. Murray, Mathematical Biology, 2002}
\end{flushright}
\bigskip

\hrule
\end{minipage}
\end{flushright}

\bigskip

\bigskip

\onehalfspacing

Epidemiology has become an important issue for modern society. The
relationship between mathematics and epidemiology has been increasing.
For the mathematician, epidemiology provides new and
exciting branches, while for the epidemiologist, mathematical
modeling offers an important research tool in the study of the evolution of diseases.

In 1760, a smallpox model was proposed by Daniel Bernoulli and is
considered by many authors the first epidemiological mathematical
model. Theoretical papers by Kermack and McKendrinck, between 1927 and 1933
about infectious disease models, have had a great influence in the deve\-lopment
of mathematical epidemiology models \cite{Murray2002}. Most of the basic theory
had been developed during that time, but the theoretical progress has been steady since then
\cite{Brauer2008}. Mathematical models are being increasingly used
to elucidate the transmission of several diseases. These models,
usually based on compartment models, may be rather simple,
but studying them is crucial in gaining important knowledge of the
underlying aspects of the infectious diseases spread out \cite{Hethcote1994},
and to evaluate the potential impact of control programs
in reducing morbidity and mortality.

After the Second World War, the strategy of public health has been
focusing on the control and elimination of the organisms that
cause the diseases. The appearance of new antibiotics and vaccines
brought a positive perspective of the diseases eradication.
However, factors such as resistance to the medicine by the
microorganisms, demographic evolution, accelerated urbanization,
increased travelling and climate change, led to new diseases and
the resurgence of old ones. In 1981, the human immunodeficiency virus (HIV)
appears and since then, become as important sexually transmitted disease throughout
the world \cite{Hethcote2000}. Futhermore, malaria, tuberculosis, dengue and yellow fever
have re-emerged and, as a result of climate changes, has been spreading into new regions \cite{Hethcote2000}.

Recent years have seen an increasing trend in the representation
of mathematical models in publications in the epidemiological
literature, from specialist journals of medicine, biology and
mathematics to the highest impact generalist journals
\cite{Ferguson2006}, showing the importance of interdisciplinary.
Their role in comparing, planning,
implementing and evaluating various control programs is of major
importance for public health decision makers. This interest has
been reinforced by the recent examples of SARS - Severe Acute
Respiratory Syndrome - epidemic in 2003 and Influenza pandemic in 2009.

Although chronic diseases, such as cancer and heart diseases have been
receiving more attention in developed countries, infectious diseases
are still important and cause suffering and mortality in developing countries.
These, remain a serious medical burden all around the world with 15 million deaths
per year estimated to be directly related to infectious diseases \cite{Jones2008}.

The successful containment of the emerging diseases is not just linked to
medical infrastructure but also on the capacity to recognize its
transmission characteristics and apply optimal medical and logistic policies.
Public health often asks information such as \cite{Hethcote2000}: how many people
will be infected, how many require hospitalization, what is the maximum number
of people ill at a given time and how long will the epidemic last.
As a result, it is necessary an ever-increasing capacity for a rapid response.

Education, vaccination campaigns, preventive drugs
administration and surveillance programs, are all examples of
prevention methods that authorities must consider for disease
prevention. Whenever the disease declares itself, the emergency
interventions such as disinfectants, insecticide application, mechanical controls
and quarantine measures must be considered. Intervention strategies can be
modelled with the goal of understanding how they will influence the disease's battle.

As financial resources are limited, there is a pressing need to
optimize investments for disease prevention and fight. Traditionally,
the study of disease dynamics has been focused on identifying the mechanisms
responsible for epidemics but has taken little into account economic constraints
in analyzing control strategies. On the other hand, economic models have given insight
into optimal control under constraints imposed by limited resources, but they are
frequently ignored by the spatial and temporal dynamics of the disease. Therefore,
progress requires a combination of epidemiological and economic factors for modelling
what until here tended to remain separate. More recently, bioeconomic approaches
to disease management have been advocated, since infectious diseases can be modelled
thinking that the limited resources involved require trade-offs. Finding the optimal
strategy depends on the balance of economic and epidemiological parameters that reflect
the nature of the host-pathogen system and the efficiency of the control method.

The main goal of this thesis is to formulate epidemiological models, giving a special
importance to Dengue disease. Moreover, it is our aim to frame the disease management
question into an optimal control problem requiring the maximization/minimization
of some objective function that depends on the infected individuals (biological issues)
and control costs (economic issues), given some initial conditions. This way,
will allow us to propose practical control measures to the autho\-rities to assess
and forecast the disease burden, such as an attack rate, morbidity,
hospitalization and mortality.

\bigskip

\bigskip

The thesis is composed by two parts. The First Part, comprising chapters 1 to 3,
gives a mathematical background to support the original results presented in the Second Part,
that is composed by chapters 4 to 7. In Chapter~\ref{chp1}, the definition of Optimal Control Problem,
its possible versions and the adapted first order necessary conditions based
on the Pontryagin Maximum Principle\index{Pontryagin's Maximum Principle} are introduced.
Simple examples are chosen to exemplify the mathematical concepts.

With the increasing of variables and complexity, Optimal Control problems can no longer
be solved analytically and numerical methods are required. For this purpose, in Chapter~\ref{chp2},
direct and indirect methods are presented for their resolution. Direct methods\index{Direct method}
consist in the discretization of the Optimal Control problem, reducing it to a nonlinear constrained
optimization problem. Indirect methods\index{Indirect methods} are based on the Pontryagin Maximum Principle,
which in turn reduces the problem to a boundary value problem.
For each approach, the software packages used in this thesis are described.

In Chapter~\ref{chp3}, the basic building blocks of most epidemiological models are reviewed:
SIR (composed by Susceptible-Infected-Recovered) and SIS models  (Susceptible-Infected-Susceptible).
For these, it is possible to develop some analytical results which can be useful in the understanding
of simple epidemics. Taking this as the basis, we discuss  the dynamics of other compartmental models,
bringing more complex realities, such as those with exposed or carrier classes.

The Second Part of the thesis contains the original results and is focused on Dengue Fever disease.
Dengue is a vector borne disease, caused by a mosquito from the \emph{Aedes} family.
It is mostly found in tropical and sub-tropical climates, mostly in urban areas.
It can provoke a severe flu-like illness, and sometimes, in severe cases can be lethal.
According to the World Health Organization about 40\% of the world's population is now at risk \cite{Who}.

The main reasons for the choice of this particular disease are:

\begin{itemize}
\item the importance of this disease around the world, as well as the challenges
of its transmission features, prevention and control measures;
\item two Portuguese-speaking countries (Brazil and Cape Verde) have already
experience with Dengue, and in the last one, a first outbreak occurred during
the development of this thesis, which allowed the development of a groundbreaking work;
\item the mosquito \emph{Aedes Aegypti}, the main vector that transmits Dengue,
is already in Portugal, on Madeira island \cite{Relatorio2011}, which without carrying the disease,
is considered a potential threat to public health and has been followed by the health authorities.
\end{itemize}

In Chapter~\ref{chp4} information about the mosquito, disease symptoms, and measures
to fight Dengue are reported. An old problem related to Dengue is revisited and solved
by different approaches. Finally, a numerical study is performed to compare different
discretization schemes in order to analyze the best strategies for future implementations.

In Chapter~\ref{chp5}, a SEIR+ASEI model is studied. The basic reproduction
number\index{Basic reproduction number} and the equilibrium points\index{Equilibrium point}
are computed as well as their relationship with the local stability of the
Disease Free Equilibrium\index{Disease Free Equilibrium}.
This model implements a control measure: adulticide. A study to find the best strategies
available to apply the insecticide is made. Continuous and piecewise constant strategies
are used involving the system of  ordinary differential equations only
or resorting to the Optimal Control theory.

Chapter~\ref{chp6} is concerned with a SIR+ASI model that incorporates three controls:
adulticide, larvicide and mechanical control. A detailed discussion on the effects
of each control, individually or together, on the development of the disease is given.
An analysis of the importance of each control in the decreasing of the basic reproduction
number is given. These results are strengthened when the optimal strategy for the model is calculated.
Bioeconomic approaches, using distinct weights for the respective control costs
and treatments for infected individuals have also been provided.

In Chapter~\ref{chp7} simulations for a hypothetical vaccine for Dengue are carried out.
The features of the vaccine are unknown because the clinical trials are ongoing.
Using models with a new compartment for vaccinated individuals, perfect and imperfect
vaccines are studied aiming the analysis of the repercussions of the vaccination process
in the morbidity and/or eradication of the disease. Then, an optimal control approach is studied,
considering the vaccination not as a new compartment, but as a measure control in fighting the disease.

Finally, the main conclusions are reported and future directions of research are pointed out.

\clearpage{\thispagestyle{empty}\cleardoublepage}


\pagenumbering{arabic} \sffamily
\part{State of the art}
\clearpage{\thispagestyle{empty}\cleardoublepage}


\chapter{Optimal control}
\label{chp1}

\begin{flushright}
\begin{minipage}[r]{9cm}

\bigskip
\small {\emph{The optimal control definition and its possible formulations are introduced,
followed by some examples related to epidemiolo\-gical models. The Pontryagin Maximum Principle
is presented with the aim of finding the best control policy.}}

\bigskip

\hrule
\end{minipage}
\end{flushright}

\bigskip
\bigskip

\onehalfspacing

Optimal Control (OC) is the process of determining control
and state trajectories for a dynamic system over a period
of time in order to minimize a performance index \cite{Bryson1996}.

Historically, OC is an extension of the calculus of
variations. In the seventeenth century, the first formal
results of calculus of variations can be found.
Johann Bernoulli challenged other
famous contemporary mathematicians - such as Newton, Leibniz,
Jacob Bernoulli, L'H\^{o}pital and von Tschirnhaus - with the
Brachistochrone problem: ``if a small object moves under the
influence of gravity, which part between two fixed points enables
it to make the trip in the shortest time?''

Other specific problems were solved and a general mathematical
theory was developed by Euler and Lagrange. The most fruitful
applications of the calculus of variations have been to
theoretical physics, particularly in connection with Hamilton's
principle or the Principle of Least Action. Early applications to
economics appeared in the late 1920s and early 1930s by Ross,
Evans, Hottelling and Ramsey, with further applications published
occasionally thereafter \cite{Sussmann1997}.

The generalization of the calculus of variations to optimal
control theory was strongly motivated by military applications and has
developed rapidly since 1950. The decisive breakthrough was achieved
by the Russian mathematician Lev S. Pontryagin (1908-1988) and his
co-workers (V. G. Boltyanskii, R. V. Gamkrelidz and E. F.
Misshchenko) with the formulation and demonstration of the Pontryagin Maximum
Principle \cite{Pontryagin1962}. This principle has provided research with
suitable conditions for optimization problems with differential
equations as constraints. The Russian team generalized variational problems by
separating control and state variables and admitting control
constraints. In such problems, OC gives equivalents
results, as one would have expected. However, the two approaches differ
and the OC approach sometimes affords insight into a
problem that might be less readily apparent through the calculus
of variations. OC is also applied to problems where
the calculus of variations is not convenient, such as those
involving constraints on the derivatives of functions \cite{Leitmann1997}.

The theory of OC brought new approaches to Mathematics with Dynamic Programming.
Introduced by R. E. Bellman, Dynamic Programming makes use of the principle
of optimality and it is suitable for solving discrete problems, allowing for a significant
reduction in the computation of the optimal controls (see \cite{Kirk1998}). It is also possible
to obtain a continuous approach to the principle of optimality that leads to the solution
of a partial differential equation called the Hamilton-Jacobi-Bellman equation.
This result allowed to bring new connections between the OC problem and the Lyapunov stability theory.

Before the arrival of the computer, only fairly simple the OC problems could be solved.
The arrival of the computer age enabled the application of OC theory and some methods
to many complex problems. Selected examples are as follows:
\begin{itemize}
\item Physical systems, such as stable performance of motors and
machinery, robotics, optimal guidance of rockets \cite{Goh2008, Molavi2008};

\item Aerospace, including driven problems, orbits transfers, development of
satellite launchers and recoverable problems of atmospheric
reentry \cite{Bonnard2008, Hermant2010};

\item Economics and management, such as optimal exploitation of natural
resources, energy policies, optimal investment of production
strategies \cite{Munteanu2008, Sun2008};

\item Biology and medicine, as regulation of physiological functions, plants growth,
infectious diseases, oncology, radiotherapy \cite{Joshi2002, Joshi2006, Lenhart2007, Nanda2007}.
\end{itemize}

Today, the OC theory is extensive and with several approaches.
One can adjust controls in a system
to achieve a goal, where the underlying system can include:
ordinary differential equations, partial differential equations,
discrete equations, stochastic differential equations,
integro-difference equations, combination of discrete and
continuous systems. In this work the goal is the OC theory of
ordinary differential equations with time fixed.


\section{Optimal control problem}
\label{sec:1:1}

A typical OC problem requires a performance index or cost functional ($J[x(t),u(t)]$),
a set of state variables ($x(t) \in X$), a set of control variables
($u(t) \in U$) in a time $t$, with $t_0 \leq t \leq t_f$.
The main goal consists in finding a piecewise continuous control $u(t)$
and the associated state variable $x(t)$ to maximize a given objective functional.
The development of this chapter will be closely structured from Lenhart and Workman work \cite{Lenhart2007}.

\begin{definition}[Basic OC Problem in Lagrange formulation]
\label{OC_Lagrange}
\index{Lagrange form}

\noindent An OC problem is in the form
\begin{equation}
\label{OC_Lagrange_eq}
\begin{tabular}{ll}
$\underset{u}{max}$ & $J[x(t),u(t)]=\int_{t_0}^{t_f}f(t,x(t),u(t))dt$\\
$s.t.$ & $\dot{x}(t)=g(t,x(t),u(t))$\\
&$x(t_0)=x_0$\\
\end{tabular}
\end{equation}
\noindent $x(t_f)$ could be free, which means that the value of $x(t_f)$ is
unrestricted, or could be fixed, \emph{i.e}, $x(t_f)=x_f$.
\end{definition}

For our purposes, $f$ and $g$ will always be continuously differentiable
functions in all three arguments. We assume that the control set $U$
is a Lebesgue measurable function. Thus, as the control(s) will always
be piecewise continuous, the associated states will always
be piecewise differentiable.

We have been focused on finding the maximum of a function. We can switch back
and forth between maximization and minimization by simply negating the cost functional:
$$min\{J\}=-max\{-J\}.$$

An OC problem can be presented in different ways, but equivalent,
depending on the purpose or the software to be used.


\section{Lagrange, Mayer and Bolza formulations}
\label{sec:1:2}

There are three well known equivalent formulations to describe the
OC problem, which are Lagrange (already presented in previous section),
Mayer and Bolza forms \cite{Chachuat2007,Zabczyk2008}.

\begin{definition}[Bolza formulation]\label{OC_Bolza}\index{Bolza form}

The Bolza formulation of the OC problem can be defined as

\begin{equation}
\label{OC_Bolza_eq}
\begin{tabular}{ll}
$\underset{u}{max}$ & $J[x(t),u(t)]
=\phi\left(t_{0},x(t_{0}),t_f,x(t_f)\right)+\int_{t_0}^{t_f}f(t,x(t),u(t))dt$\\
$s.t.$ & $\dot{x}(t)=g(t,x(t),u(t))$\\
&$x(t_0)=x_0$\\
\end{tabular}
\end{equation}
\noindent where $\phi$ is a continuously differentiable function.
\end{definition}

\medskip

\begin{definition}[Mayer formulation]\label{OC_Mayer}\index{Mayer form}

The Mayer formulation of the OC problem can be defined as

\begin{equation}\label{OC_Mayer_eq}
\begin{tabular}{ll}
$\underset{u}{max}$ & $J[x(t),u(t)]=\phi\left(t_{0},x(t_{0},t_f,x(t_f))\right)$\\
$s.t.$ & $\dot{x}(t)=g(t,x(t),u(t))$\\
&$x(t_0)=x_0$\\
\end{tabular}
\end{equation}
\end{definition}

\begin{theorem}
The three formulations Lagrange (Definition~\ref{OC_Lagrange})\index{Lagrange form},
Bolza (Definition~\ref{OC_Bolza})\index{Bolza form} and Mayer
(Definition~\ref{OC_Mayer}\index{Mayer form}) are equivalent.
\end{theorem}

\begin{proof}
$(2)\Rightarrow (1)$ To get the proof, we formulate the Bolza problem
as one of Lagrange, using an extended state vector.

Let $(x(\cdot),u(\cdot))$ an admissible pair for the problem (\ref{OC_Bolza_eq})
and let $z(t)=\left(x_a(t),x(t)\right)$ with
$x_a(t)\equiv\frac{\phi\left(t_0,x(t_0),t_f,x(t_f)\right)}{t_f-t_0}$.

So, $\left(z(\cdot),u(\cdot)\right)$ is an admissible pair for the Lagrange problem

\begin{equation}
\label{cap1_lagrange_proof}
\begin{tabular}{ll}
$\underset{u}{max}$ & $\int_{t_0}^{t_f}\left[f(t,x(t),u(t))+x_a(t)\right]dt$\\
$s.t.$ & $\dot{x}(t)=g(t,x(t),u(t))$\\
& $\dot{x}_a(t)=0$\\
& $x(t_0)=x_0$\\
& $x_a(t_0)=\frac{\phi\left(t_0,x(t_0),t_f,x(t_f)\right)}{t_f-t_0}$\\
& $x_a(t_f)=\frac{\phi\left(t_0,x(t_0),t_f,x(t_f)\right)}{t_f-t_0}$\\
\end{tabular}
\end{equation}

Thus, the value of the functionals in both formulations matches.

\medskip

$(1)\Rightarrow (2)$ Conversely, each admissible pair $\left(z(\cdot),u(\cdot)\right)$
for the problem (\ref{cap1_lagrange_proof}) corresponds the pair $\left(x(\cdot),u(\cdot)\right)$,
where $x(\cdot)$ is composed by the last component of $z$, admissible for the problem
(\ref{OC_Bolza_eq}) and matching the respective values of the functionals.

\medskip

$(2)\Rightarrow (3)$ For this statement we also need to use an extended state vector.
Let be $z(t)=\left(x_b(t),x(t)\right)$, with $[t_0,t_f]$,
where $x_b(\cdot)$ is a continuous function such is
\begin{equation*}
\begin{tabular}{ll}
$\dot{x}_b(t)=f(\tau,x(\tau),u(\tau)),$ & $x_b(t_0)=0,$
\end{tabular}
\end{equation*}
for almost $t$ in $[t_0,t_f]$:
$$x_b(t)=\int_{t_0}^{t_f}f(\tau, x(\tau),u(\tau))d\tau.$$
Thus we have $(z(\cdot),u(\cdot))$ an admissible pair for the following Mayer problem:
\begin{equation}
\label{cap1_mayer_proof}
\begin{tabular}{ll}
$\underset{u}{max}$ & $\phi(t_0,x(t_0),t_f,x(t_f))+x_b(t_f)$\\
$s.t.$ & $\dot{x}(t)=g(t,x(t),u(t))$\\
& $\dot{x}_b(t)=f(t,x(t),u(t))$\\
& $x(t_0)=x_0$\\
& $x_b(t_0)=0$\\
\end{tabular}
\end{equation}
This way, the values of the functional for both formulations are the same.

\medskip

$(3)\Rightarrow (2)$ Conversely, to all admissible pair $(z(\cdot),u(\cdot))$
for the Mayer problem (\ref{cap1_mayer_proof}) corresponds to an admissible pair
$(x(\cdot),u(\cdot))$ for the Bolza problem (\ref{OC_Bolza_eq}),
where $x(\cdot)$ consists in the last component of $z(\cdot)$.
\end{proof}

For the proof of the previous theorem it was not necessary to show that
the Lagrange problem\index{Lagrange form} is equivalent to the Mayer
formulation\index{Mayer form}. However, in the second part of the thesis,
and due to computational issues, some of the OC problems (usually presented in the Lagrange form)
will be converted into the equivalent Mayer one. Hence, using a standard procedure
is possible to rewrite the cost functional (\textrm{cf.} \cite{Lewis1995}),
augmenting the state vector with an extra component.
So, the Lagrange formulation (\ref{OC_Lagrange_eq}) can be rewritten as
\begin{equation}
\label{cap1_lagrange-mayer}
\begin{tabular}{ll}
$\underset{u}{max}$ & $x_c(t_f)$\\
$s.t.$ & $\dot{x}(t)=g(t,x(t),u(t))$\\
& $\dot{x}_c(t)=f(t,x(t),u(t))$\\
& $x(t_0)=x_0$\\
& $x_c(t_0)=0$\\
\end{tabular}
\end{equation}


\section{Pontryagin's Maximum Principle}
\label{sec:1:3}
\index{Pontryagin's Maximum Principle}

The necessary first order conditions to find the optimal control were developed by
Pontryagin and his co-workers. This result is considered as one of the most important
results of Mathematics in the 20th century.

\begin{figure}[ptbh]
\center
  \includegraphics [scale=0.5]{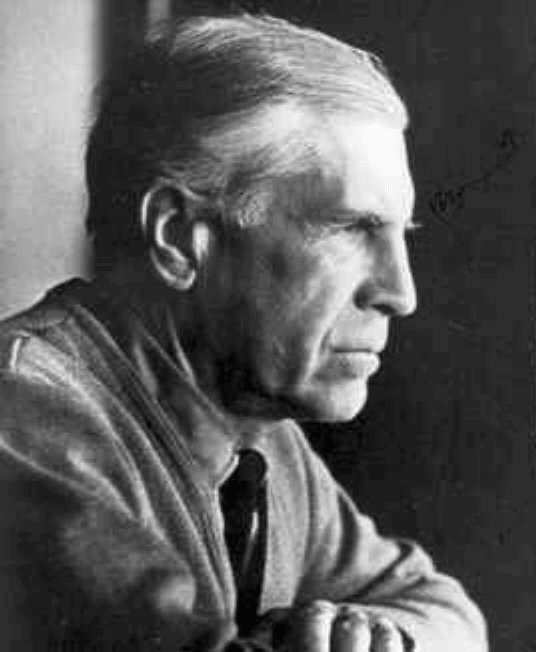}\\
   {\caption{\label{cap1_Pontryagin} Lev Semyonovich Pontryagin }}
\end{figure}

Pontryagin introduced the idea of adjoint functions to append the
differential equation to the objective functional. Adjoint
functions have a similar purpose as Lagrange multipliers in
multivariate calculus, which append constraints to the function of
several variables to be maximized or minimized.

\begin{definition}[Hamiltonian]\label{Hamiltonian}\index{Hamiltonian}

Let the previous OC problem considered in (\ref{OC_Lagrange_eq}). The function
\begin{center}
\begin{tabular}{l}
$H(t, x(t),u(t),\lambda(t))=f(t,x(t),u(t))+\lambda(t)
g(t,x(t),u(t))$
\end{tabular}
\end{center}
\noindent is called Hamiltonian function and $\lambda$ is the
adjoint variable.
\end{definition}

\bigskip

Now we are ready to announce the Pontryagin Maximum Principle (PMP).

\begin{theorem}[Pontryagin's Maximum Principle]
\label{PMP}\index{Pontryagin's Maximum Principle}
If $u^{*}(t)$ and $x^{*}(t)$ are optimal for problem (\ref{OC_Lagrange_eq}),
then there exists a piecewise differentiable adjoint variable
$\lambda(t)$ such that
\begin{center}
$H(t,x^{*}(t),u(t),\lambda(t))\leq
H(t,x^{*}(t),u^{*}(t),\lambda(t))$
\end{center}
\noindent for all controls $u$ at each time $t$, where $H$ is the
Hamiltonian previously defined and
\begin{equation*}
\begin{tabular}{l}
$\displaystyle{\lambda'(t)
=-\frac{\partial H(t,x^{*}(t),u^{*}(t),\lambda(t))}{\partial x}}$,\\
\\
$\lambda(t_f)=0$.
\end{tabular}
\end{equation*}
\end{theorem}

\begin{proof}
The proof of this theorem is quite technical and we opted to omit it.
The original Pontryagin's text \cite{Pontryagin1962}
or Clarke's book \cite{Clarke1990} are good
references to find the proof.
\end{proof}

\bigskip

\begin{remark}
The last condition, $\lambda(t_f)=0$, called transversality condition\index{Transversality condition},
is only used when the OC problem does not have terminal value in the state variable,
\emph{i.e.}, $x(t_f)$ is free.
\end{remark}

\bigskip

This principle converted the problem of finding a control which
maximizes the objective functional subject to the state ODE and
initial condition into the problem of optimizing the Hamiltonian
pointwise. As consequence, with this adjoint equation and
Hamiltonian, we have
\begin{equation}
\label{Hu}
\displaystyle{\frac{\partial H}{\partial u}}=0
\end{equation}
\noindent at $u^{*}$ for each $t$, namely, the Hamiltonian\index{Hamiltonian} has a
critical point; usually this condition is called \emph{optimality
condition}\index{Optimality condition}. Thus to find the necessary conditions, we do not need
to calculate the integral in the objective functional, but only
use the Hamiltonian.

Here is presented a simple example to illustrate this principle.

\begin{example}[from \cite{Neilan2010}]
\label{examplePMP}
\hrulefill

Consider the OC problem:
\begin{center}
\begin{tabular}{ll}
$\underset{x, u}{min}$ & $J[u]=\int_{0}^{2}\left(x+\frac{1}{2}u^2\right) dt$\\
$s.t.$ & $\dot{x}=x+u$\\
 & $x(0)=\frac{1}{2}e^2-1$\\
 & $x(1)=2$
\end{tabular}
\end{center}

The calculus of this OC problem can be done by steps.

Step 1 --- Form the Hamiltonian for the problem.

The Hamiltonian can be written as:

\begin{center}
\begin{tabular}{l}
$H(t,x,u,\lambda)=x+\frac{1}{2}u^2 + \lambda (x+u)$
\end{tabular}
\end{center}

\bigskip

Step 2 --- Write the adjoint differential equation,
the optimality condition and transversality\index{Transversality condition}
boundary condition (if necessary). Try to eliminate $u^{*}$ by using
the optimality equation $H_{u}=0$, i.e., solve for $u^{*}$ in terms of $x^{*}$
and $\lambda$.

Using the Hamiltonian to find the differential equation of the adjoint $\lambda$, we obtained

\begin{center}
\begin{tabular}{l}
$\lambda'(t)=-\displaystyle{\frac{\partial H}{\partial
x}}\Leftrightarrow \lambda'=-1-\lambda.$
\end{tabular}
\end{center}

The optimality condition is given by

\begin{center}
\begin{tabular}{l}
$\displaystyle{\frac{\partial H}{\partial u}}=0 \Leftrightarrow u +\lambda =0. $
\end{tabular}
\end{center}

In this way we obtain an expression for the OC:
$$u^{*}=-\lambda.$$
As the problem has just an initial condition for the state variable,
it is necessary to calculate the transversality condition: $$\lambda(2)=0.$$

\bigskip

Step 3 --- Solve the set of two differential equations for $x^{*}$
and $\lambda$ with the boundary conditions, replacing $u^{*}$
in the differential equations by the expression for the optimal
control from the previous step.

By the adjoint equation $\lambda'=-1-\lambda$ and the transversality condition $\lambda(2)=0$ we have
$$\lambda=e^{2-t}-1.$$
Hence, the optimality condition leads to
$$u^{*}=-\lambda\Leftrightarrow u^{*}=1-e^{2-t}$$
and the associated state is
$${x}^{*}=\frac{1}{2}e^{2-t}-1.$$
\end{example}
\hrule

\bigskip

\begin{remark}
If the Hamiltonian\index{Hamiltonian} is linear in the control variable $u$, it can
be difficult to calculate $u^{*}$ from the optimality equation,
since $\frac{\partial H}{\partial u}$ would not contain $u$.
Specific ways of solving these kind of problems can be found in
\cite{Lenhart2007}.
\end{remark}

Until here we have showed necessary conditions to solve basic optimal control problems.
Now, it is important to study some conditions that can guarantee the existence
of a finite objective functional value at the optimal control and state variables,
based on \cite{Fleming1975,Kamien1991,Lenhart2007,Macki1982}.
The following is an example of a sufficient condition result.

\begin{theorem}
Consider
\begin{equation*}\label{cap1_sufficient_condition}
\begin{tabular}{ll}
$\underset{u}{max}$ & $J\left[x(t),u(t)\right]
=\int_{t_0}^{t_f}f(t,x(t),u(t))dt$\\
$s.t.$ & $\dot{x}(t)=g(t,x(t),u(t))$\\
&$x(t_0)=x_0$\\
\end{tabular}
\end{equation*}
\noindent Suppose that $f(t,x,u)$ and $g(t,x,u)$ are both continuously differentiable
functions in their three arguments and concave in $x$ and $u$. Suppose $u^{*}$
is a control with associated state $x^{*}$, and $\lambda$ a piecewise differentiable function,
such that $u^{*}$, $x^{*}$ and $\lambda$ together satisfy on $t_0\leq t \leq t_f$:
\begin{equation*}
\begin{tabular}{l}
$f_u+\lambda g_u=0,$\\
$\lambda'=-(f_x+\lambda g_x),$\\
$\lambda(t_f)=0,$\\
$\lambda(t)\geq0.$
\end{tabular}
\end{equation*}
\noindent Then for all controls $u$, we have
$$J(u^{*})\geq J(u)$$
\end{theorem}

\begin{proof}
The proof of this theorem is available on \cite{Lenhart2007}.
\end{proof}

This result is not strong enough to guarantee that $J(u^{*})$ is finite.
Such results usually require some conditions on $f$ and/or $g$.
Next theorem is an example of an existence result from \cite{Fleming1975}.

\begin{theorem}
Let the set of controls for problem (\ref{OC_Lagrange_eq})
be Lebesgue integrable functions on $t_0\leq t\leq t_f$ in $\mathbb{R}$.
Suppose that $f(t,x,u)$ is convex in $u$, and there exist constants
$C_1$, $C_2$, $C_3 >0$, $C_4$ and $\beta >1$ such that
\begin{equation*}
\begin{tabular}{l}
$g(t,x,u)=\alpha(t,x)+\beta(t,x)u$\\
$|g(t,x,u)|\leq C_1(1+|x|+|u|)$\\
$|g(t,x_1,u)-g(t,x,u)|\leq C_2 |x_1-x|(1+|u|)$\\
$f(t,x,u)\geq C_3|u|^{\beta}-C_4$
\end{tabular}
\end{equation*}
\noindent for all $t$ with $t_0\leq t\leq t_1$, $x$, $x_1$, $u$ in $\mathbb{R}$.
Then there exists an optimal control $u^{*}$ maximizing $J(u)$, with $J(u^{*})$ finite.
\end{theorem}

\begin{proof}
The proof of this theorem is available on \cite{Fleming1975}.
\end{proof}

For a minimization problem, $g$ would have a concave property
and the inequality on $f$ would be reversed.

Note that the necessary conditions developed to this point deal
with piecewise continuous optimal controls, while this existence
theorem guarantees an optimal control which is only Lebesgue integrable.
This disconnection can be overcome by extending the necessary conditions
to Lebesgue integrable functions \cite{Lenhart2007, Macki1982},
but we did not expose this idea in the thesis.
See the existence of OC results in \cite{Filippov1968}.


\section{Optimal control with payoff terms}

In some cases it is necessary, not only minimize (or maximize) terms
over the entire time interval, but also minimize (or maximize) a function
value at one particular point in time, specifically, the end of the time interval.
There are some situations where the objective function must take into account
the value of the state at the terminal time, e.g., the number of infected
individuals at the final time in an epidemic model \cite{Lenhart2007}.

\begin{definition}[OC problem with payoff term]

An OC problem with payoff term is in the form
\begin{equation}
\label{OC_payoff}
\begin{tabular}{ll}
$\underset{u}{max}$ & $J[x(t),u(t)]=\phi(x(t_f))+\int_{t_0}^{t_f}f(t,x(t),u(t))dt$\\
$s.t.$ & $\dot{x}(t)=g(t,x(t),u(t))$\\
&$x(t_0)=x_0$\\
\end{tabular}
\end{equation}
\noindent where $\phi(x(t_f))$ is a goal with respect to the final position
or population level $x(t_f)$. The term $\phi(x(t_f))$ is called payoff or salvage.
\end{definition}

\medskip

Using the PMP\index{Pontryagin's Maximum Principle},
adapted necessary conditions can be derived for this problem.

\begin{proposition}[Necessary conditions]

If $u^{*}(t)$ and $x^{*}(t)$ are optimal for problem (\ref{OC_payoff}),
then there exists a piecewise differentiable adjoint variable
$\lambda(t)$ such that
$$H(t,x^{*}(t),u(t),\lambda(t))\leq H(t,x^{*}(t),u^{*}(t),\lambda(t))$$
\noindent for all controls $u$ at each time $t$, where $H$ is the
Hamiltonian previously defined and
\begin{equation*}
\begin{tabular}{ll}
$\displaystyle{\lambda'(t)=-\frac{\partial H(t,x^{*}(t),u^{*}(t),\lambda(t))}{\partial
x}}$ & (adjoint condition),\\
&\\
$\displaystyle\frac{\partial H}{\partial u}=0$ & (optimality condition)\index{Optimality condition}\\
& \\
$\lambda(t_f)=\phi'(x(t_f))$ & (transversality condition)\index{Transversality condition}.\\
\end{tabular}
\end{equation*}
\end{proposition}

\medskip

\begin{proof}
The proof of this result can be found in \cite{Kamien1991}.
\end{proof}

\bigskip

A new example is given to illustrate this proposition.

\begin{example}[from \cite{Neilan2010}]
\label{example_tcells2}
\hrulefill

Let $x(t)$ represent the number of tumor cells at time $t$,
with exponential growth factor $\alpha$, and $u(t)$
the drug concentration. The aim is to minimize the number
of tumor cells at the end of the treatment period and the
accumulated harmful effects of the drug on the body.
This problem is formulated as
\begin{equation*}
\begin{tabular}{ll}
$\underset{u}{minimize}$ & $x(t_f)+\int_{0}^{t_f}u^2dt$\\
$s.t.$ & $\dot{x}=\alpha x -u$\\
& $x(0)=x_0$
\end{tabular}
\end{equation*}
Let us consider the Hamiltonian
$$H(t,x,u,\lambda)=u^2+\lambda(\alpha x-u).$$
The optimality condition is given by
$$\frac{\partial H}{\partial u}=0\Rightarrow u^{*}=\frac{\lambda}{2}.$$
The adjoint condition is given by
$$\lambda'=-\frac{\partial H}{\partial x}\Leftrightarrow \lambda'=-\alpha \lambda\Rightarrow\lambda=Ce^{-\alpha t}$$
\noindent with $C$ constant.

Using the transversality condition $\lambda(t_f)=1$
(note that $\phi(s)=s$, so $\phi'(s)=1$), we obtain
$$\lambda(t)=e^{\alpha(t_f-t)}$$  and  $$u^{*}=\frac{e^{\alpha(t_f-t)}}{2}.$$

The optimally state trajectory is (using $\dot{x}=\alpha x-u$ and $x(0)=x_0$):

$$x^{*}=x_0e^{\alpha t}+e^{\alpha t_f}\frac{e^{-\alpha t}-e^{\alpha t}}{4\alpha}$$
\end{example}
\hrule

\bigskip

Many problems require bounds on the control to achieve a realistic solution.
For example, the amount of drugs in the organism must be non-negative
and it is necessary to impose a limit. For the last example,
despite being simplistic, makes more sense to constraint the control as $0 \leq u \leq 1$.


\section{Optimal control with bounded controls}
\label{sec:1:4}

\begin{definition}[OC with bounded control]

An OC with bounded control can be written in the form
\begin{equation}
\label{OC_bounded_control}
\begin{tabular}{ll}
$\underset{u}{max}$ & $J[x(t),u(t)]=\int_{t_0}^{t_f}f(t,x(t),u(t))dt$\\
$s.t.$ & $\dot{x}(t)=g(t,x(t),u(t))$\\
&$x(t_0)=x_0$\\
& $a\leq u(t)\leq b$
\end{tabular}
\end{equation}
\noindent where $a, b$ are fixed real constants and $a<b$.
\end{definition}

To solve problems with bounds on the control,
it is necessary to develop alternative necessary conditions.

\begin{proposition}[Necessary conditions]

If $u^{*}(t)$ and $x^{*}(t)$ are optimal for problem (\ref{OC_bounded_control}),
then there exists a piecewise differentiable adjoint variable $\lambda(t)$ such that
$$H(t,x^{*}(t),u(t),\lambda(t))\leq H(t,x^{*}(t),u^{*}(t),\lambda(t))$$
\noindent for all controls $u$ at each time $t$, where $H$ is the
Hamiltonian previously defined and
\begin{equation*}
\begin{tabular}{ll}
$\displaystyle{\lambda'(t)
=-\frac{\partial H(t,x^{*}(t),u^{*}(t),\lambda(t))}{\partial x}}$
& (adjoint condition),\\
$\lambda(t_f)=0$ & (transversality condition)\index{Transversality condition}.\\
\end{tabular}
\end{equation*}

By an adaptation of the PMP\index{Pontryagin's Maximum Principle},
the OC must satisfy (optimality condition)\index{Optimality condition}:
\begin{center}
\begin{tabular}{l}
$u^{*}=
\left\{
\begin{array}{lll}
a & \textrm{if} & \frac{\partial H}{\partial u}<0\\
a\leq  \tilde{u} \leq b & \text{if} & \frac{\partial H}{\partial u}=0\\
b & if & \frac{\partial H}{\partial u}>0
\end{array}
\right. $
\end{tabular}
\end{center}
\noindent \emph{i.e.}, the maximization is over all admissible controls,
and $\tilde{u}$ is obtained by the expression $\frac{\partial H}{\partial u}=0$.
In particular, the optimal control $u^{*}$ maximizes $H$ pointwise with respect to $a \leq u \leq b$.
\end{proposition}

\begin{proof}
The proof of this result can be found in \cite{Kamien1991}.
\end{proof}

\bigskip

If we have a minimization problem, then $u^{*}$ is instead chosen to minimize $H$ pointwise.
This has the effect of reversing $<$ and $>$ in the first and third lines
of optimality condition\index{Optimality condition}.

\begin{remark}
In some software packages there are no specific characterization for the bounds
of the control. In those cases, and when the implementation allows,
we can write in a compact way the optimal control $\tilde{u}$
obtained without truncation, bounded by $a$ and $b$:
$$u^{*}(t)=min\{a,max\{b,\tilde{u}\}\}.$$
\end{remark}

So far, we have only examined problems with one control
and with one dependent state variable. Often,
it is necessary to consider more variables.


\section{Optimal control of several variables}
\label{sec:1:5}

\begin{definition}[OC with several variables and several controls]
An OC with $n$ state variables, $m$ control variables
and a payoff function $\phi$ can be written in the form
\begin{equation}
\begin{tabular}{ll}
$\underset{u_1,\ldots, u_m}{max}$ & $\phi(x_1(t_f),\ldots,x_n(t_f))
+\int_{t_0}^{t_f}f(t,x_1(t),\ldots,x_n(t),u_1(t),\ldots,u_m(t))dt$\\
$s.t.$ & $\dot{x}_i(t)=g_i(t,x_1(t),\ldots,x_n(t),u_1(t),\ldots,u_m(t))$\\
&$x_i(t_0)=x_{i0},$  $i=1,2,\ldots,n$\\
\end{tabular}
\end{equation}
where the functions $f$, $g_{i}$ are continuously differentiable in all variables.
\end{definition}

From now on, to simplify the notation, let $\vec{x}(t)=\left[x_1(t),\ldots,x_n(t)\right]$,
$\vec{u}(t)=\left[u_1(t),\ldots,u_m(t)\right]$, $\vec{x_0}=\left[x_{10},\ldots,x_{n0}\right]$,
and $\vec{g}(t,\vec{x},\vec{u})=\left[g_1(t,\vec{x},\vec{u}),\ldots,g_n(t,\vec{x},\vec{u})\right]$.

So, the previous problem can be rewritten in a compact way as

\begin{equation}
\label{OC_several_variables}
\begin{tabular}{ll}
$\underset{\vec{u}}{max}$ & $\phi(\vec{x}(t_f))+\int_{t_0}^{t_f}f(t,\vec{x}(t),\vec{u}(t))dt$\\
$s.t.$ & $\dot{\vec{x}}(t)=g_i(t,\vec{x}(t),\vec{u}(t))$\\
& $\vec{x}(t_0)=\vec{x}_{0},$ $i=1,2,\ldots,n$\\
\end{tabular}
\end{equation}

Using the same approach of the previous subsections,
it is possible to derive generalized ne\-cessary conditions.

\begin{proposition}[Necessary conditions]
Let $\vec{u}^{*}$ be a vector of optimal control functions and $\vec{x}^{*}$
be the vector of corresponding optimal state variables.
With $n$ states, we will need $n$ adjoints, one for each state.
There is a piecewise differentiable vector-valued function
$\vec{\lambda}(t)=[\lambda_1(t),\ldots,\lambda_n(t)]$ , where each $\lambda_{i}$
is the adjoint variable corresponding to $x_i$, and the Hamiltonian is\index{Hamiltonian}
\begin{center}
\begin{tabular}{l}
$H(t,\vec{x},\vec{u},\vec{\lambda})=f(t,\vec{x},\vec{u})
+\sum\limits_{i=1}^{n}\lambda_i(t)g_i(t,\vec{x},\vec{u})$.
\end{tabular}
\end{center}

It is possible to find the variables satisfying identical optimality,
adjoint and transversality conditions\index{Transversality condition}
in each vector component. Namely, $\vec{u}^{*}$ maximizes $H(t,\vec{x}^{*},\vec{u},\vec{\lambda})$
with respect to $\vec{u}$ at each $t$, and  $\vec{u}^{*}$, $\vec{x}^{*}$ and $\vec{\lambda}$ satisfy
\begin{center}
\begin{tabular}{ll}
$\lambda_{j}'(t)=\displaystyle -\frac{\partial H}{\partial x_j}$,
for $j=1,\ldots,n$ & (adjoint conditions)\\
&\\
$\lambda_j(t_f)=\phi_{x_j}(\vec{x}(t_f))$ for $j=1,\ldots,n$
& (transversality conditions\index{Transversality condition})\\
&\\
$\displaystyle \frac{\partial H}{\partial u_k}=0$ at $u_{k}^{*}$
for $k=1,\ldots,m$ & (optimality conditions)\index{Optimality condition}\\
\end{tabular}
\end{center}
By $\phi_{x_j}$, it is meant the partial derivative in the $x_j$ component.
Note, if $\phi\equiv 0$, then $\lambda_j(t_f)=0$ for all $j$, as usual.
\end{proposition}

\bigskip

\begin{remark}
Similarly to the previous section, if bounds are placed on a control variable,
$a_k\leq u_k \leq b_k$ (for $k=1,\ldots,m$), then the optimality
condition\index{Optimality condition} is changed from $\frac{\partial H}{\partial u_k}=0$ to
\begin{center}
\begin{tabular}{l}
$ u_k^{*}=\left\{
\begin{array}{lll}
a_k & if & \frac{\partial H}{\partial u_k}<0\\
a_k\leq \tilde{u_k}\leq b_k & if & \frac{\partial H}{\partial u_k}=0\\
b_k & if & \frac{\partial H}{\partial u_k}>0
\end{array}
\right. $
\end{tabular}
\end{center}
\end{remark}

Below, an optimal control problem related to rubella is presented.

\bigskip

\begin{example}\hrulefill\label{example_rubeola}

Rubella, commonly known as German measles, is a most common in child age,
caused by the rubella virus. Children recover more quickly than adults,
and can be very serious in pregnancy. The virus is contracted through
the respiratory tract and has an incubation period of 2 to 3 weeks.
The primary symptom of rubella virus infection is the appearance
of a rash on the face which spreads to the trunk and limbs
and usually fades after three days. Other symptoms include low grade fever,
swollen glands, joint pains, headache and conjunctivitis.

It is presented now an optimal control problem to study the dynamics
of rubella in China over three years, using a vaccination process ($u$)
as a measure to control the disease (more details can be found in \cite{Buonomo2011}).
Let $x_1$ represent the susceptible population, $x_2$ the proportion
of population that is in the incubation period, $x_3$ the proportion
of population that is infected with rubella and $x_4$ the rule
that remains the population constant. The optimal control problem can be defined as:
\begin{equation}
\label{cap1:ode_rubeola}
\begin{tabular}{ll}
$min$& $\mathlarger{\int}_{0}^{3}(Ax_3+u^2)dt$\\
& \\
$s.t.$&$\dot{x}_1=b-b(p x_2+q x_2)-b x_1-\beta x_1 x_3 - u x_1$\\
&$\dot{x}_2=b p x_2 +\beta x_1 x_3 -(e+b)x_2$\\
&$\dot{x}_3=e x_2-(g+b)x_3$\\
&$\dot{x}_4=b-b x_4$
\end{tabular}
\end{equation}
\noindent with initial conditions $x_1(0)=0.0555$, $x_2(0)=0.0003$,
$x_3(0)=0.0004$, $x_4(0)=1$ and the parameters $b=0.012$, $e=36.5$,
$g=30.417$, $p=0.65$, $q=0.65$, $\beta=527.59$ and $A=100$.

The control $u$ is defined in $[0,0.9]$.
\end{example}
\hrule

\bigskip

It is very difficult to solve analytically this problem. For most of the
epidemiologic problems it is necessary to employ numerical methods.
Some of them will be described in the next chapter.

\clearpage{\thispagestyle{empty}\cleardoublepage}


\chapter{Methods to solve Optimal Control Problems}
\label{chp2}

\begin{flushright}
\begin{minipage}[r]{9cm}

\bigskip
\small {\emph{In this chapter, some numerical approaches to solve a system of ordinary differential equations,
such as shooting me\-thods and multi-steps methods, are introduced. Then, two distinct philosophies
to solve OC problems are presented: indirect methods centered in the PMP and the direct ones focus
on problem discretization solved by nonlinear optimization codes.
A set of software packages used all over the thesis is summarily exposed.}}
\bigskip

\hrule
\end{minipage}
\end{flushright}

\bigskip
\bigskip

In the last decades the computational world has been developed in an amazing way.
Not only in hardware issues such as efficiency, memory capacity, speed,
but also in terms of the software robustness. Groundbreaking achievements
in the field of numerical solution techniques for differential and integral
equations have enabled the simulation of highly complex real world scenarios.
This way, OC also won with these improvements and numerical methods
and algorithms have evolved significantly.

The next section concerns on the resolution of differential equations systems.


\section{Numerical solutions for dynamic systems}
\label{sec:2:1}

A dynamic system is mathematically characterized by a set of ordinary differential equations (ODE).
Specifically, the dynamics are described for $t_0\leq t\leq t_f$, by a system of $n$ ODEs
\begin{equation}
\label{cap2:ODE}
\dot{y}=\left[
            \begin{array}{c}
              \dot{y}_1 \\
              \dot{y}_2 \\
              \vdots \\
              \dot{y}_n\\
            \end{array}
          \right]
          =\left[
            \begin{array}{c}
              f_1(y_1(t),\ldots,y_n(t),t) \\
              f_2(y_1(t),\ldots,y_n(t),t) \\
              \vdots \\
              f_n(y_1(t),\ldots,y_n(t),t)\\
            \end{array}
          \right].
\end{equation}

The problems of solving an ODE are classified into \emph{initial value problems} (IVP)
and \emph{boundary value problems} (BVP), depending on how the conditions
at the endpoints of the domain are specified. All the conditions
of an initial-value problem are specified at the initial point.
On the other hand, the problem becomes a boundary-value problem
if the conditions are needed for both initial and final points.

There exist many numerical methods to solve initial value problems ---
such as Euler, Runge-Kutta or adaptive methods --- and boundary value problems,
such as shooting methods.


\subsection*{Shooting method}
\index{Shooting method}

One can visualize the shooting method as the simplest technique for solving BVP.
Supposing it is desired to determine the initial angle of a cannon so that,
when a cannonball is fired, it strikes a desired target. An initial guess
is made for the angle of the cannon, and the cannon is fired.
If the cannon does not hit the target, the angle is adjusted based
on the amount of the miss and another cannon is fired.
The process is repeated until the target is hit \cite{Rao2009}.

Suppose we want to find $y(t_0)=y_0$ such that $y(t_f)=b$.
The shooting method can be summarized as follows \cite{Betts2001}:

\begin{description}
  \item[Step 1.] guess initial conditions $x=y(t_0)$;
  \item[Step 2.] propagate the differential system equations from $t_0$ to $t_f$, i.e., shoot;
  \item[Step 3.] evaluate the error in the boundary conditions $c(x)=y(t_f)-b$;
  \item[Step 4.] use a nonlinear program to adjust the variables
  $x$ to satisfy the constraints $c(x)=0$, i.e., repeat steps 1--3.
\end{description}

Figure~\ref{cap2_shooting} presents a shooting method scheme.

\begin{figure}[ptbh]
\center
\includegraphics [scale=0.5]{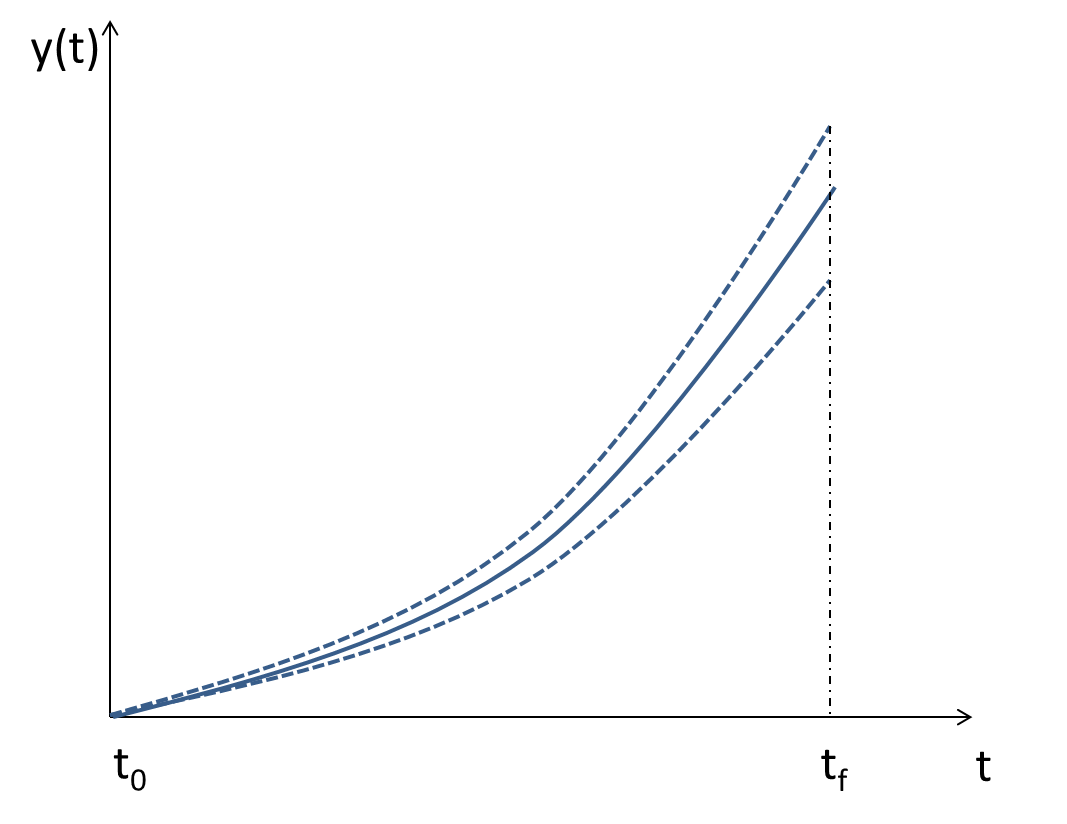}\\
{\caption{\label{cap2_shooting} Shooting method representation (adapted from \cite{Betts2001}).
The solid line represents the solution and the dashed lines are related to the shootings}}
\end{figure}

Despite its simplicity, from a practical standpoint, the shooting method\index{Shooting method}
is used only when the problem has a small number of variables. This method has a major disadvantage:
a small change in the initial condition can produce a very large change in the final conditions.
In order to overcome the numerical difficulties of the simple method,
the multiple shooting method is presented.


\subsection*{Multiple shooting method}
\index{Multiple shooting method}

In a multiple shooting method, the time-interval $[t_0,t_f]$ is divided into $M-1$ subintervals.
Then is applied over each subinterval $[t_i,t_{i+1}]$ with the initial values
of the differential equations in the interior intervals being unknown that need to be determined.
In order to enforce continuity, the following conditions are enforced at the interface of each subinterval:
$$y(t_i^{-})-y(t_i^{+})=0.$$

A scheme of the multiple shooting method\index{Multiple shooting method}
is shown in Figure~\ref{cap2_multiple}.

\begin{figure}[ptbh]
\center
\includegraphics [scale=0.5]{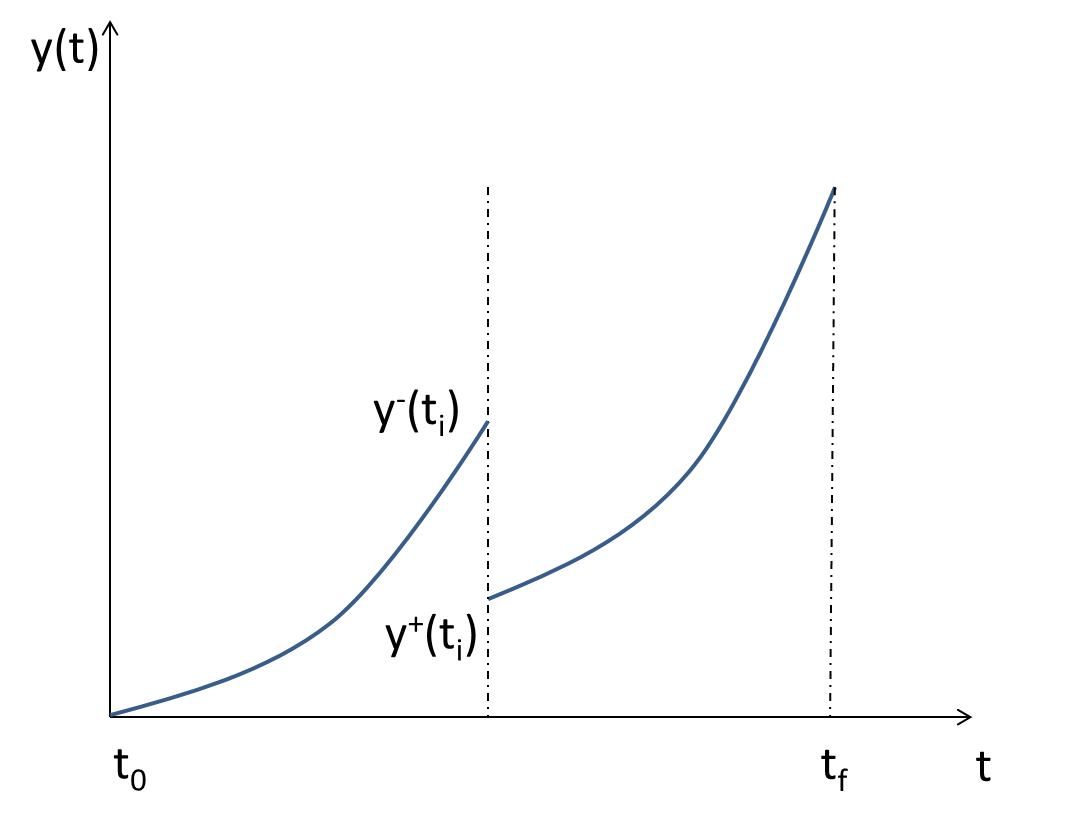}\\
{\caption{\label{cap2_multiple} Multiple shooting method representation (adapted from \cite{Betts2001})}}
\end{figure}

With the multiple shooting approach\index{Multiple shooting method} the problem size is increased:
additional variables and constraints are introduced for each shooting segment.
In particular, the number of nonlinear variables and constraints for a multiple shooting
application is $n=n_y(M-1)$, where $n_y$ is the number of dynamic variables
$y$ and $M-1$ is the number of segments \cite{Betts2001}.

Both shooting and multiple shooting methods require a good guess for initial conditions
and the propagation of the shoots for problem of high dimension is not feasible.
For this reason, other methods can be implemented, using initial value problems.

The numerical solution of the IVP is fundamental to most optimal control methods.
The problem can stand as follows: compute the value of $y(t_f)$
for some value of $t<t_f$ that satisfies (\ref{cap2:ODE})
with the known initial value $y(t_0)=y_0$.

Numerical methods for solving the ODE IVP are relatively mature in comparison
to other fields in Optimal Control. It will be considered two methods:
single-step and multiple-step methods. In both, the solution of the differential
system at each step $t_k$ is sequentially obtained using current and/or previous
information about the solution. In both cases, it is assumed that the time $t=nh$
moves ahead in uniform steps of length $h$ \cite{Betts2001,Edite1998}.


\subsection*{Euler scheme}
\index{Euler scheme}

The most common single-step method is Euler method. In this discretization scheme,
if a differential equation is written like $\dot{x}=f(x(t),t)$,
is possible to make a convenient approximation of this:
$$x_{n+1}\simeq x_n+hf(x(t_n),t_n).$$
This approximation $x_{n+1}$ of $x(t)$ at the point $t_{n+1}$ has an error of order $h^2$.
Clearly, there is a trade-off between accuracy and complexity of calculation
which depends heavily on the chosen value for $h$. In general,
as $h$ is decreased the calculation takes longer but is more accurate.

For many higher order systems it is very difficult to make Euler approximation effective.
For this reason more accurate and elaborate techniques were developed.
One of these methods is the Runge-Kutta method.


\subsection*{Runge-Kutta scheme}

A Runge-Kutta\index{Runge Kutta scheme}
method is a multiple-step method, where the solution
at time $t_{k+1}$ is obtained from a defined set of previous values
$t_{j-k},\ldots,t_k$ and $j$ is the number of steps.

If a differential equation is written like $\dot{x}=f(x(t),t)$,
it is possible to make a convenient approximation of this,
using the second order Runge-Kutta method
$$x_{n+1}\simeq x_n+\frac{h}{2}\left[f(x_n(t),t_n)+f(x_{n+1},t_{n+1})\right],$$
\noindent or the fourth order Runge-Kutta method
$$x_{n+1}\simeq x_n+\frac{h}{6}\left(k_1+2k_2+2k_3+k4\right)$$
\noindent where
\begin{equation*}
\begin{tabular}{l}
$k_1=f\left(x(t),t\right)$\\
$k_2=f\left(x(t)+\frac{h}{2}k_1,t+\frac{h}{2}\right)$\\
$k_3=f\left(x(t)+\frac{h}{2}k_2,t+\frac{h}{2}\right)$\\
$k_4=f\left(x(t)+hk_3,t+h\right).$\\
\end{tabular}
\end{equation*}
This approximation $x_{n+1}$ of $x(t)$ at the point $t_{n+1}$ has an error depending
on $h^3$ and $h^5$, for the Runge-Kutta methods of second and fourth order, respectively.

Numerical methods for solving OC problems date back to the 1950s with Bellman investigation.
From that time to present, the complexity of methods and corresponding complexity
and variety of applications has substantially increased \cite{Rao2009}.

There are two major classes of numerical methods for solving OC problems: indirect
methods and direct methods. The first ones, indirectly solve the problem by converting the optimal
control problem to a boundary-value problem, using the PMP. On the other hand,
in a direct method, the optimal solution is found by transcribing an infinite-dimensional
optimization problem to a finite-dimensional optimization problem.


\section{Indirect methods}
\label{sec:2:2}
\index{Indirect methods}

In an indirect method, the PMP\index{Pontryagin's Maximum Principle}
is used to determine the first-order optimality conditions\index{Optimality condition}
of the original OC problem. The indirect approach leads to a multiple-point boundary-value
problem that is solved to determine candidate optimal trajectories called extremals.

For an indirect method it is necessary to explicitly get the adjoint equations,
the control equations and all the transversality conditions\index{Transversality condition},
if there exist. Notice that there is no correlation between the method used to solve
the problem and its formulation: one may consider applying a multiple shooting method
solution technique to either an indirect or a direct formulation.
In the following subsection a numerical approach using the indirect method is presented.


\subsection*{Backward-Forward sweep method}
\label{sec:2:2:1}

This method is described in a recent book by Lenhart and Workman \cite{Lenhart2007}
and it is known as forward--backward sweep method.
The process begins with an initial guess on the control variable. Then, the
state equations are simultaneously solved forward in time and the adjoint
equations are solved backward in time. The control is updated
by inserting the new values of states and adjoints into its characterization, and
the process is repeated until convergence occurs.

Considering $\vec{x}=(x_1,\ldots,x_N+1)$ and $\vec{\lambda}=(\lambda_1,\ldots,\lambda_N+1)$
the vector approximations for the state and the adjoint.
The main idea of the algorithm is described as follows:

\begin{description}
\item[Step 1.] Make an initial guess for $\vec{u}$
over the interval ($\vec{u}\equiv 0$ is almost always sufficient);
\item[Step 2.] Using the initial condition $x_1=x(t_0)=a$
and the values for $\vec{u}$, solve $\vec{x}$ forward
in time according to its differential equation in the optimality system;
\item[Step 3.] Using the transversality condition $\lambda_{N+1}=\lambda(t_f)=0$
and the values for $\vec{u}$ and $\vec{x}$, solve $\vec{\lambda}$ backward
in time according to its differential equation in the optimality system;
\item[Step 4.] Update $\vec{u}$ by entering the new $\vec{x}$ and $\vec{\lambda}$
values into the characterization of the optimal control;
\item[Step 5.] Verify convergence: if the variables are sufficiently close
to the corresponding in the previous iteration, then output the current
values as solutions, else return to Step 2.
\end{description}

For Steps 2 and 3, Lenhart and Workman used for the state
and adjoint systems the Runge-Kutta fourth order
procedure to make the discretization process.

On the other hand, Wang \cite{Wang2009}, applied the same philosophy
but solving the differential equations with the solver \texttt{ode45}\index{ode45 routine}
for \texttt{Matlab}\index{Matlab}. This solver is based on an explicit
Runge-Kutta (4,5) formula, the Dormand-Prince pair. That means
the numerical solver \texttt{ode45} combines a fourth and a fifth order methods, both of which
are similar to the classical fourth order Runge-Kutta method discussed above.
These vary the step size, choosing it at each step an attempt to achieve the desired
accuracy. Therefore, the solver \texttt{ode45} is suitable for a wide variety of initial value problems in
practical applications. In general, \texttt{ode45}
is the best method to apply as a first attempt for most problems \cite{Houcque}.

\newpage

\begin{example}\hrulefill\label{ex_rubeola_Lenhart}

Let consider the open problem defined in Chapter 1 (Example \ref{example_rubeola})
about rubella disease. With $\vec{x}(t)=\left(x_1(t),x_2(t),x_3(t),x_4(t)\right)$
and $\vec{\lambda}(t)=\left(\lambda_1(t),\lambda_2(t),\lambda_3(t),\lambda_4(t)\right)$,
the Hamiltonian of this problem can be written as\index{Hamiltonian}
\begin{equation*}
\label{cap2:hamiltonian_rubeola}
\begin{tabular}{ll}
$H(t,\vec{x}(t),u(t),\vec{\lambda}(t))=$ & $Ax_3+u^2$\\
& $+\lambda_1\left(b-b(p x_2+q x_2)-b x_1-\beta x_1 x_3 - u x_1\right)$\\
& $+\lambda_2\left(b p x_2 +\beta x_1 x_3 -(e+b)x_2\right)$\\
& $+\lambda_3\left(e x_2-(g+b)x_3\right)$\\
& $+\lambda_4\left(b-b x_4\right).$
\end{tabular}
\end{equation*}
Using the PMP the optimal control problem can be studied with the state variables
\begin{equation*}
\label{cap2:ode2_rubeola}
\begin{tabular}{l}
$\dot{x}_1=b-b(p x_2+q x_2)-b x_1-\beta x_1 x_3 - u x_1$\\
$\dot{x}_2=b p x_2 +\beta x_1 x_3 -(e+b)x_2$\\
$\dot{x}_3=e x_2-(g+b)x_3$\\
$\dot{x}_4=b-b x_4$
\end{tabular}
\end{equation*}
\noindent with initial conditions $x_1(0)=0.0555$, $x_2(0)=0.0003$,
$x_3(0)=0.0004$ and $x_4(0)=1$ and the adjoint variables:
\begin{equation*}
\label{cap2:ode3_rubeola}
\begin{tabular}{l}
$\dot{\lambda}_1=\lambda_1(b+u+\beta x_3) - \lambda_2\beta x_3$\\
$\dot{\lambda}_2=\lambda_1 b p + \lambda_2(e+b+p b)-\lambda_3 e$\\
$\dot{\lambda}_3=-A+\lambda_1(b q +\beta x_1)-\lambda_2\beta x_1+\lambda_3(g+b)$\\
$\dot{\lambda}_4=\lambda_4 b$
\end{tabular}
\end{equation*}
\noindent with transversality conditions\index{Transversality condition} $\lambda_i(3)=0$, $i=1,\ldots,4$.

The optimal control is
\begin{center}
\begin{tabular}{l}
$ u^{*}=\left\{
\begin{array}{lll}
0 & if & \frac{\partial H}{\partial u}<0\\
\frac{\lambda_1 x_1}{2} & if & \frac{\partial H}{\partial u}=0\\
0.9 & if & \frac{\partial H}{\partial u}>0
\end{array}
\right. $
\end{tabular}
\end{center}

Here it is only presented the main part of the code using the backward-forward
sweep method with fourth order Runge-Kutta. The completed
one can be found in the website \cite{SofiaSITE}.

\footnotesize{
\begin{verbatim}
for i = 1:M
    m11 = b-b*(p*x2(i)+q*x3(i))-b*x1(i)-beta*x1(i)*x3(i)-u(i)*x1(i);
    m12 = b*p*x2(i)+beta*x1(i)*x3(i)-(e+b)*x2(i);
    m13 = e*x2(i)-(g+b)*x3(i);
    m14 = b-b*x4(i);

    m21 = b-b*(p*(x2(i)+h2*m12)+q*(x3(i)+h2*m13))-b*(x1(i)+h2*m11)-...
        beta*(x1(i)+h2*m11)*(x3(i)+h2*m13)-(0.5*(u(i) + u(i+1)))*(x1(i)+h2*m11);
    m22 = b*p*(x2(i)+h2*m12)+beta*(x1(i)+h2*m11)*(x3(i)+h2*m13)-(e+b)*(x2(i)+h2*m12);
    m23 = e*(x2(i)+h2*m12)-(g+b)*(x3(i)+h2*m13);
    m24 = b-b*(x4(i)+h2*m14);

    m31 = b-b*(p*(x2(i)+h2*m22)+q*(x3(i)+h2*m23))-b*(x1(i)+h2*m21)-...
        beta*(x1(i)+h2*m21)*(x3(i)+h2*m23)-(0.5*(u(i) + u(i+1)))*(x1(i)+h2*m21);
    m32 = b*p*(x2(i)+h2*m22)+beta*(x1(i)+h2*m21)*(x3(i)+h2*m23)-(e+b)*(x2(i)+h2*m22);
    m33 = e*(x2(i)+h2*m22)-(g+b)*(x3(i)+h2*m23);
    m34 = b-b*(x4(i)+h2*m24);

    m41 = b-b*(p*(x2(i)+h2*m32)+q*(x3(i)+h2*m33))-b*(x1(i)+h2*m31)-...
        beta*(x1(i)+h2*m31)*(x3(i)+h2*m33)-u(i+1)*(x1(i)+h2*m31);
    m42 = b*p*(x2(i)+h2*m32)+beta*(x1(i)+h2*m31)*(x3(i)+h2*m33)-(e+b)*(x2(i)+h2*m32);
    m43 = e*(x2(i)+h2*m32)-(g+b)*(x3(i)+h2*m33);
    m44 = b-b*(x4(i)+h2*m34);

    x1(i+1) = x1(i) + (h/6)*(m11 + 2*m21 + 2*m31 + m41);
    x2(i+1) = x2(i) + (h/6)*(m12 + 2*m22 + 2*m32 + m42);
    x3(i+1) = x3(i) + (h/6)*(m13 + 2*m23 + 2*m33 + m43);
    x4(i+1) = x4(i) + (h/6)*(m14 + 2*m24 + 2*m34 + m44);
end

for i = 1:M
    j = M + 2 - i;

    n11 = lambda1(j)*(b+u(j)+beta*x3(j))-lambda2(j)*beta*x3(j);
    n12 = lambda1(j)*b*p+lambda2(j)*(e+b-p*b)-lambda3(j)*e;
    n13 = -A+lambda1(j)*(b*q+beta*x1(j))-lambda2(j)*beta*x1(j)+lambda3(j)*(g+b);
    n14 = b*lambda4(j);

    n21 = (lambda1(j) - h2*n11)*(b+u(j)+beta*(0.5*(x3(j)+x3(j-1))))-...
        (lambda2(j) - h2*n12)*beta*(0.5*(x3(j)+x3(j-1)));
    n22 = (lambda1(j) - h2*n11)*b*p+(lambda2(j) - h2*n12)*(e+b-p*b)-(lambda3(j) - h2*n13)*e;
    n23 = -A+(lambda1(j) - h2*n11)*(b*q+beta*(0.5*(x1(j)+x1(j-1))))-...
        (lambda2(j) - h2*n12)*beta*(0.5*(x1(j)+x1(j-1)))+(lambda3(j) - h2*n13)*(g+b);
    n24 = b*(lambda4(j) - h2*n14);

    n31 = (lambda1(j) - h2*n21)*(b+u(j)+beta*(0.5*(x3(j)+x3(j-1))))-...
        (lambda2(j) - h2*n22)*beta*(0.5*(x3(j)+x3(j-1)));
    n32 = (lambda1(j) - h2*n21)*b*p+(lambda2(j) - h2*n22)*(e+b-p*b)-(lambda3(j) - h2*n23)*e;
    n33 = -A+(lambda1(j) - h2*n21)*(b*q+beta*(0.5*(x1(j)+x1(j-1))))-...
        (lambda2(j) - h2*n22)*beta*(0.5*(x1(j)+x1(j-1)))+(lambda3(j) - h2*n23)*(g+b);
    n34 = b*(lambda4(j) - h2*n24);

    n41 = (lambda1(j) - h2*n31)*(b+u(j)+beta*x3(j-1))-(lambda2(j) - h2*n32)*beta*x3(j-1);
    n42 = (lambda1(j) - h2*n31)*b*p+(lambda2(j) - h2*n32)*(e+b-p*b)-(lambda3(j) - h2*n33)*e;
    n43 = -A+(lambda1(j) - h2*n31)*(b*q+beta*x1(j-1))-...
        (lambda2(j) - h2*n32)*beta*x1(j-1)+(lambda3(j) - h2*n33)*(g+b);
    n44 = b*(lambda4(j) - h2*n34);

    lambda1(j-1) = lambda1(j) - h/6*(n11 + 2*n21 + 2*n31 + n41);
    lambda2(j-1) = lambda2(j) - h/6*(n12 + 2*n22 + 2*n32 + n42);
    lambda3(j-1) = lambda3(j) - h/6*(n13 + 2*n23 + 2*n33 + n43);
    lambda4(j-1) = lambda4(j) - h/6*(n14 + 2*n24 + 2*n34 + n44);
end
u1 = min(0.9,max(0,lambda1.*x1/2));

\end{verbatim}
}

\normalsize
The optimal curves for the states variables
and optimal control are shown in Figure~\ref{cap2_rubeola}.

\begin{figure}[ptbh]
\center
  \includegraphics [scale=0.95]{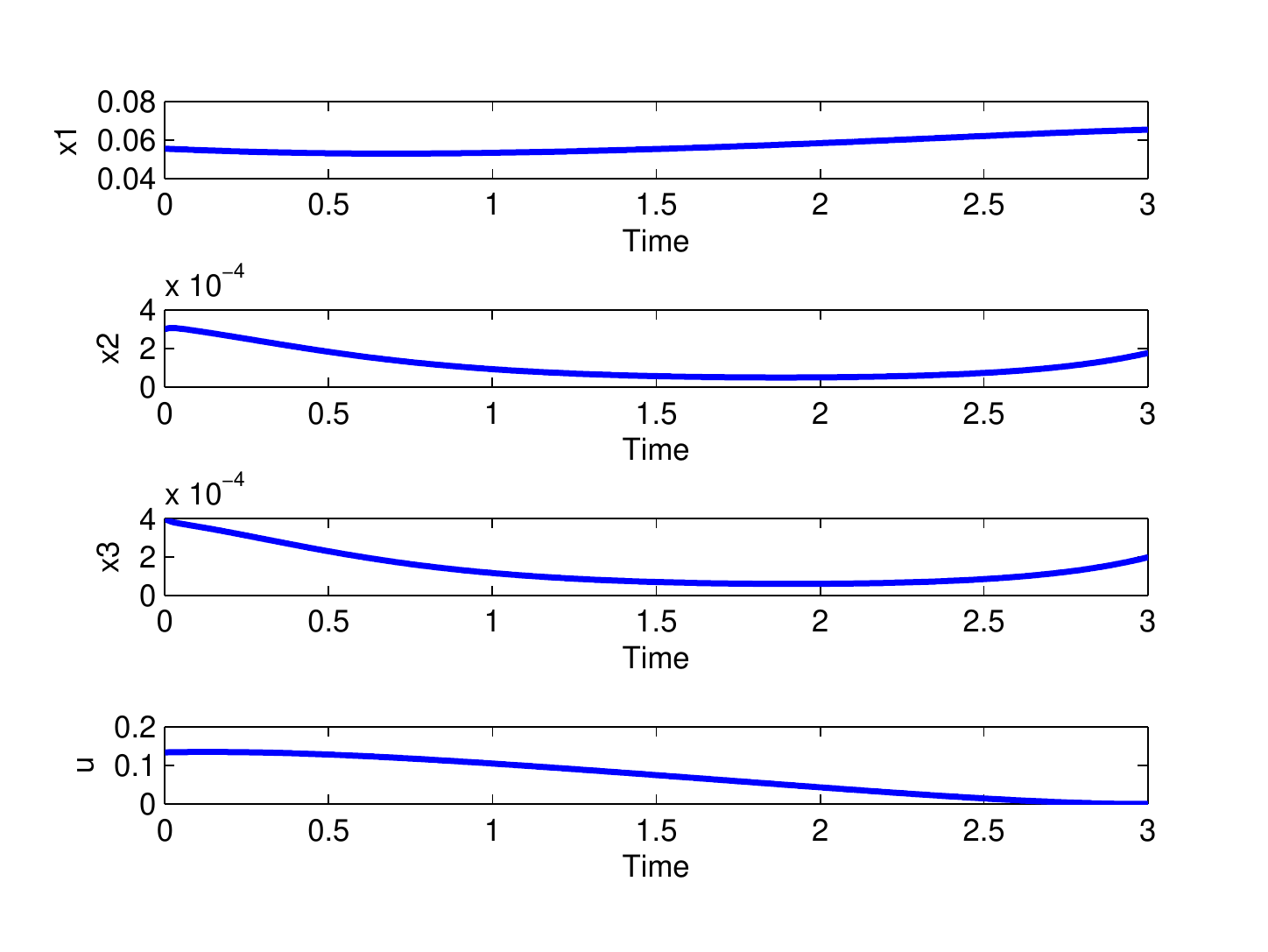}\\
   {\caption{\label{cap2_rubeola} The optimal curves for rubella problem }}
\end{figure}
\end{example}
\hrule

\bigskip

There are several difficulties to overcome when
an optimal control problem is solved by indirect methods\index{Indirect methods}.
Firstly, is necessary to calculate the hamiltonian\index{Hamiltonian},
adjoint equations, optimality condition \index{Optimality condition}
and transversality conditions\index{Transversality condition}.
Besides, the approach is not flexible, since each time a new problem is formulated,
a new derivation is required. In contrast, a direct method\index{Direct method}
does not require explicit derivation neither the necessary conditions.

Due to these practical difficulties with the indirect formulation, the main focus
will be centered on the direct methods. This approach has been gaining popularity
in numerical optimal control over the past three decades \cite{Betts2001}.


\section{Direct methods}
\label{sec:2:3}

\index{Direct method}

A new family of numerical methods for dynamic
optimization has emerged, referred to as direct methods. This development
has been driven by the industrial need to solve large-scale optimization
problems and it has also been supported by the rapidly increasing
computational power.

A direct method\index{Direct method} constructs a sequence of points
$x_1, x_2,\ldots,x^{*}$ such that the objective function in minimized
and typical $F(x_1)>F(x_2)> \cdots >F(x^{*})$. Here the state and/or control
are approximated using an appropriate function approximation
(e.g., polynomial approximation or piecewise constant parameterization).
Simultaneously, the cost functional is approximated as a cost function.
Then, the coefficients of the function approximations are treated
as optimization variables and the problem is reformulated
to a standard nonlinear optimization problem (NLP) in the form:
\begin{equation*} \label{nlp}
\begin{tabular}{ll}
$\underset{x,u}{min}$ & $F(x)$ \\
$s.t.$ & $c_{i}(x)=0, i\in E$ \\
& $c_{j}(x)\geq0, j\in I $\\
\end{tabular}
\end{equation*}
\noindent where $c_{i},i\in E$ e $c_{j},j\in I$
are the set of equality and inequality constraints, respectively.

In fact, the NLP is easier to solve than the boundary-value problem,
mainly due to the sparsity of the NLP and the many well-known software programs
that can handle with this feature. As a result, the range of problems that
can be solved via direct methods\index{Direct method} is significantly larger
than the range of problems that can be solved via indirect methods. Direct methods
have become so popular these days that many people have written sophisticated
software programs that employ these methods. Here we present two types of codes/packages: specific solvers
for OC problems and standard NLP solvers used after a discretization process.

\bigskip


\subsection{Specific Optimal Control software}

\subsubsection*{OC-ODE}

\index{OC-ODE}

The \texttt{OC-ODE} \cite{Matthias2009}, \emph{Optimal Control of Ordinary-Differential Equations},
by Matthias Gerdts, is a collection of \texttt{Fortran 77} routines for optimal control problems subject
to ordinary differential equations. It uses an automatic direct discretization method
for the transformation of the OC problem into a finite-dimensional NLP.
\texttt{OC-ODE} includes procedures for numerical adjoint estimation and sensitivity analysis.

\begin{example}\hrulefill\label{ex_rubeola_OC-ODE}

Considering the same problem (Example \ref{example_rubeola}), here is the main part
of the code in \texttt{OC-ODE}. The completed one can be found in the website \cite{SofiaSITE}.
The achieved solution is similar to the indirect approach, therefore we will not present it.

\footnotesize{
\begin{verbatim}
c     Call to OC-ODE
c      OPEN( INFO(9),FILE='OUT',STATUS='UNKNOWN')
      CALL OCODE( T, XL, XU, UL, UU, P, G, BC,
     +     TOL, TAUU, TAUX, LIW, LRW, IRES,
     +     IREALTIME, NREALTIME, HREALTIME,
     +     IADJOINT, RWADJ, LRWADJ, IWADJ, LIWADJ, .FALSE.,
     +     MERIT,IUPDATE,LENACTIVE,ACTIVE,IPARAM,PARAM,
     +     DIM,INFO,IWORK,RWORK,SOL,NVAR,IUSER,USER)
      PRINT*,'Ausgabe der Loesung: NVAR=',NVAR
      WRITE(*,'(E30.16)') (SOL(I),I=1,NVAR)
c      CLOSE(INFO(9))
c      READ(*,*)
      END
c-------------------------------------------------------------------------
c     Objective Function
c-------------------------------------------------------------------------
      SUBROUTINE OBJ( X0, XF, TF, P, V, IUSER, USER )
      IMPLICIT NONE
      INTEGER IUSER(*)
      DOUBLEPRECISION X0(*),XF(*),TF,P(*),V,USER(*)
      V = XF(5)
      RETURN
      END
c-------------------------------------------------------------------------
c     Differential Equation
c-------------------------------------------------------------------------
 SUBROUTINE DAE( T, X, XP, U, P, F, IFLAG, IUSER, USER )
      IMPLICIT NONE
      INTEGER IFLAG,IUSER(*)
      DOUBLEPRECISION T,X(*),XP(*),U(*),P(*),F(*),USER(*)
c     INTEGER NONE
      DOUBLEPRECISION B, E, G, P, Q, BETA, A

      B = 0.012D0
      E = 36.5D0
      G = 30.417D0
      P = 0.65D0
      Q = 0.65D0
      BETA = 527.59D0
      A = 100.0D0

      F(1) = B-B*(P*X(2)+Q*X(3))-B*X(1)-BETA*X(1)*X(3)-U(1)*X(1)
      F(2) = B*P*X(2)+BETA*X(1)*X(3)-(E+B)*X(2)
      F(3) = E*X(2)-(G+B)*X(3)
      F(4) = B-B*X(4)
      F(5) = A*X(3))+U(1)**2
      RETURN
      END
\end{verbatim}
}
\end{example}
\hrule

\bigskip

\bigskip


\subsubsection*{DOTcvp}

\index{DOTcvp}

The \texttt{DOTcvp} \cite{Dotcvp}, \emph{Dynamic Optimization Toolbox with Vector Control Parametrization}
is a dynamic optimization toolbox for \texttt{Matlab}\index{Matlab}. The toolbox provides
environment for a \texttt{Fortran} compiler to create the '.dll' files of the ODE, Jacobian, and sensitivities.
However, a \texttt{Fortran} compiler has to be installed in a \texttt{Matlab} environment.
The toolbox uses the control vector parametrization approach for the calculation
of the optimal control profiles, giving a piecewise solution for the control.
The OC problem has to be defined in Mayer form\index{Mayer form}.
For solving the NLP, the user can choose several deterministic solvers
--- \texttt{Ipopt}\index{Ipopt}, \texttt{Fmincon}, \texttt{FSQP} ---
or stochastic solvers --- \texttt{DE}, \texttt{SRES}.

The modified \texttt{SUNDIALS} tool \cite{Hindmarsh2005} is used for solving
the IVP and for the gradients and Jacobian automatic generation. Forward integration
of the ODE system is ensured by CVODES, a part of \texttt{SUNDIALS}, which is able to perform
the simultaneous or staggered sensitivity analysis too. The IVP problem can be solved with the Newton
or Functional iteration module and with the Adams or BDF linear multistep method. Note that the
sensitivity equations are analytically provided and the error control strategy for the sensitivity variables
could be enabled. \texttt{DOTcvp} has a user friendly graphical interface (GUI).

\begin{example}\hrulefill

Considering the same problem (Example~\ref{example_rubeola}),
here is a part of the code used in \texttt{DOTcvp}. The completed
one can be found in the website \cite{SofiaSITE}. The solution,
despite being piecewise continuous, follows the curve obtained by the previous programs.

\footnotesize{
\begin{verbatim}
% --------------------------------------------------- %
% Settings for IVP (ODEs, sensitivities):
% --------------------------------------------------- %
data.odes.Def_FORTRAN     = {''}; %this option is needed only for FORTRAN parameters definition,
                            e.g. {'double precision k10, k20, ..'}
data.odes.parameters      = {'b=0.012',' e=36.5',' g=30.417',' p=0.65',' q=0.65',' beta=527.59',
                            ' d=0',' phi1=0','phi2=0','A=100 '};
data.odes.Def_MATLAB      = {''}; %this option is needed only for MATLAB parameters definition
data.odes.res(1)          = {'b-b*(p*y(2)+q*y(3))-b*y(1)-beta*y(1)*y(3)-u(1)*y(1)'};
data.odes.res(2)          = {'b*(p*y(2)+q*phi1*y(3))+beta*y(1)*y(3)-(e+b)*y(2)'};
data.odes.res(3)          = {'b*q*phi2*y(3)+e*y(2)-(g+b)*y(3)'};
data.odes.res(4)          = {'b-b*y(4)'};
data.odes.res(5)          = {'A*y(3)+u(1)*u(1)'};
data.odes.black_box       = {'None','1.0','FunctionName'}; %['None'|'Full'],[penalty coefficient
                            for all constraints],...
                            [a black box model function name]
data.odes.ic              = [0.0555 0.0003 0.0004 1 0];
data.odes.NUMs            = size(data.odes.res,2); %number of state variables (y)
data.odes.t0              = 0.0; %initial time
data.odes.tf              = 3; %final time
data.odes.NonlinearSolver = 'Newton'; %['Newton'|'Functional'] /Newton for stiff problems;
                            Functional for non-stiff problems
data.odes.LinearSolver    = 'Dense'; %direct ['Dense'|'Diag'|'Band']; iterative
                            ['GMRES'|'BiCGStab'|'TFQMR'] /for the Newton NLS
data.odes.LMM             = 'Adams'; %['Adams'|'BDF'] /Adams for non-stiff problems;
                            BDF for stiff problems
data.odes.MaxNumStep      = 500; %maximum number of steps
data.odes.RelTol          = 1e-007; %IVP relative tolerance level
data.odes.AbsTol          = 1e-007; %IVP absolute tolerance level
data.sens.SensAbsTol      = 1e-007; %absolute tolerance for sensitivity variables
data.sens.SensMethod      = 'Staggered'; %['Staggered'|'Staggered1'|'Simultaneous']
data.sens.SensErrorControl= 'on'; %['on'|'off']
% --------------------------------------------------- %
% NLP definition:
% --------------------------------------------------- %
data.nlp.RHO              = 10; %number of time intervals
data.nlp.problem          = 'min'; %['min'|'max']
data.nlp.J0               = 'y(5)'; %cost function: min-max(cost function)
data.nlp.u0               = [0 ]; %initial value for control values
data.nlp.lb               = [0 ]; %lower bounds for control values
data.nlp.ub               = [0.9]; %upper bounds for control values
data.nlp.p0               = []; %initial values for time-independent parameters
data.nlp.lbp              = []; %lower bounds for time-independent parameters
data.nlp.ubp              = []; %upper bounds for time-independent parameters
data.nlp.solver           = 'IPOPT'; %['FMINCON'|'IPOPT'|'SRES'|'DE'|'ACOMI'|'MISQP'|'MITS']
data.nlp.SolverSettings   = 'None'; %insert the name of the file that contains settings
                            for NLP solver, if does not exist use ['None']
data.nlp.NLPtol           = 1e-005; %NLP tolerance level
data.nlp.GradMethod       = 'FiniteDifference'; %['SensitivityEq'|'FiniteDifference'|'None']
data.nlp.MaxIter          = 1000; %Maximum number of iterations
data.nlp.MaxCPUTime       = 60*60*0.25; %Maximum CPU time of the optimization
                            (60*60*0.25) = 15 minutes
data.nlp.approximation    = 'PWC'; %['PWC'|'PWL'] PWL only for: FMINCON & without the
                            free time problem
data.nlp.FreeTime         = 'off'; %['on'|'off'] set 'on' if free time is considered
data.nlp.t0Time           = [data.odes.tf/data.nlp.RHO]; %initial size of the time intervals
data.nlp.lbTime           = 0.01; %lower bound of the time intervals
data.nlp.ubTime           = data.odes.tf; %upper bound of the time intervals
data.nlp.NUMc             = size(data.nlp.u0,2); %number of control variables (u)
data.nlp.NUMi             = 0; %number of integer variables (u) taken from the last
                            control variables,
if not equal to 0 you need to use some MINLP solver ['ACOMI'|'MISQP'|'MITS']
data.nlp.NUMp             = size(data.nlp.p0,2); %number of time-independent parameters (p)
\end{verbatim}
}
\normalsize

With the GUI interface this method was the preferred to be tested,
due to its simple way to implement the code.
\end{example}
\hrule

\bigskip
\bigskip


\subsubsection*{Muscod-II}
\index{Muscod-II}

In NEOS \cite{NEOS}\index{NEOS platform} platform there is a large set
of software packages. NEOS is considered as the state of the art in optimization.
One recent solver is \texttt{Muscod-II} \cite{Muscod} (Multiple Shooting CODe for Optimal Control)
for the solution of mixed integer nonlinear ODE or DAE constrained optimal
control problems in an extended \texttt{AMPL}\index{AMPL} format.

\texttt{AMPL} \cite{AMPL} is a modelling language for mathematical programming
and was created by Fourer, Gay and Kernighan. The modelling languages organize
and automate the tasks of modelling, which can handle a large volume of data and,
moreover, can be used in machines and independent solvers, allowing the user
to concentrate on the model instead of the methodology to reach solution.
However, the \texttt{AMPL} modelling language itself does not allow the formulation
of differential equations. Hence, the \texttt{TACO Toolkit} has been designed
to implement a small set of extensions for easy and convenient modeling
of optimal control problems in \texttt{AMPL}, without the need for explicit encoding
of discretization schemes. Both the \texttt{TACO Toolkit} and the
NEOS interface to \texttt{Muscod-II} are still under development.

Probably for this reason, the Example~\ref{example_rubeola}
could not be solved by this software which crashed after some runs.
Anyway, we opted to also put the code, for the same example that is being used,
to show the differences of modelling language used in each program.

\begin{example}\hrulefill

\footnotesize{
\begin{verbatim}
include OptimalControl.mod;
var t ;									
var x1, >=0 <=1;		
var x2, >=0 <=1;
var x3, >=0 <=1;
var x4, >=0 <=1;		
var u >=0, <=0.9 suffix type "u0";

minimize
cost: integral (A*x3+u^2,3);	

subject to
	c1: diff(x1,t) = b-b*(p*x2+q*x3)-b*x1-beta*x1*x3-u*x1;
	c2: diff(x2,t) = b*p*x2+beta*x1*x3-(e+b)*x2;
	c3: diff(x3,t) = e*x2-(g+b)*x3;
	c4: diff(x4,t) = b-b*x4;

\end{verbatim}
}
\normalsize
\end{example}
\hrule

\bigskip
\bigskip


\subsection{Nonlinear Optimization software}

The three nonlinear optimization software packages presented here,
were used through the NEOS platform with codes formulated in \texttt{AMPL}\index{AMPL}.

\bigskip


\subsubsection*{Ipopt}

\index{Ipopt}

The \texttt{Ipopt} \cite{Ipopt}, \emph{Interior Point OPTimizer},
is a software package for large-scale nonlinear optimization.
It is written in \texttt{Fortran} and \texttt{C}. \texttt{Ipopt} implements
a primal-dual interior point method and uses a line search strategy based on filter method.
\texttt{Ipopt} can be used from various modeling environments.

\texttt{Ipopt} is designed to exploit 1st and 2nd derivative information
if provided, usually via automatic differentiation routines in modeling
environments such as \texttt{AMPL}\index{AMPL}. If no Hessians are provided,
\texttt{Ipopt} will approximate them using a quasi-Newton methods, specifically a BFGS update.

\begin{example}\hrulefill\label{ex_rubeola_IPOPT}

Continuing with Example~\ref{example_rubeola} the \texttt{AMPL} code
is shown for \texttt{Ipopt}. The Euler discretization was the selected.
Indeed, this code can also be implemented in other nonlinear software packages
available in NEOS platform, reason why the code for the next two software packages
will not be shown. The full version can be found on the website \cite{SofiaSITE}.

\footnotesize{
\begin{verbatim}
#### OBJECTIVE FUNCTION ###
minimize cost: fc[N];

#### CONSTRAINTS ###
subject to
	i1: x1[0] = x1_0;
	i2: x2[0] = x2_0;
	i3: x3[0] = x3_0;
	i4: x4[0] = x4_0;
	i5: fc[0] = fc_0;
	
	f1 {i in 0..N-1}: x1[i+1] = x1[i] + (tf/N)*(b-b*(p*x2[i]+q*x3[i])-b*x1[i]
                                -beta*x1[i]*x3[i]-u[i]*x1[i]);
	f2 {i in 0..N-1}: x2[i+1] = x2[i] + (tf/N)*(b*p*x2[i]+beta*x1[i]*x3[i]-(e+b)*x2[i]);
	f3 {i in 0..N-1}: x3[i+1] = x3[i] + (tf/N)*(e*x2[i]-(g+b)*x3[i]);
	f4 {i in 0..N-1}: x4[i+1] = x4[i] + (tf/N)*(b-b*x4[i]);
	f5 {i in 0..N-1}: fc[i+1] = fc[i] + (tf/N)*(A*x3[i]+u[i]^2);
\end{verbatim}
}
\normalsize
\end{example}
\hrule

\bigskip


\subsubsection*{Knitro}

\index{Knitro}

\texttt{Knitro} \cite{Knitro}, short for ``Nonlinear Interior point Trust Region Optimization'',
was created primarily by Richard Waltz, Jorge Nocedal, Todd Plantenga and Richard Byrd.
It was introduced in 2001 as a derivative of academic research at Northwestern,
and has undergone continual improvement since then.

\texttt{Knitro} is also a software for solving large scale mathematical optimization problems
based mainly on the two Interior Point (IP) methods and one active set algorithm.
\texttt{Knitro} is specialized for nonlinear optimization, but also solves linear programming problems,
quadratic programming problems, and systems of nonlinear equations. The unknowns in these problems
must be continuous variables in continuous functions; however, functions can be convex or nonconvex.
The code also provides a multistart option for promoting the computation of the global minimum.
This software was tested through the NEOS platform\index{NEOS platform}.


\subsubsection*{Snopt}

\index{Snopt}

\texttt{Snopt} \cite{Snopt}, by Philip Gill, Walter Murray and Michael Saunders,
is a software package for solving large-scale optimization problems (linear and nonlinear programs).
It is specially effective for nonlinear problems whose functions and gradients are expensive to evaluate.
The functions should be smooth but do not need to be convex.
\texttt{Snopt} is implemented in \texttt{Fortran 77} and distributed as source code.
It uses the SQP (Sequential Quadratic Programming) philosophy,
with an augmented Lagrangian approach combining a trust region approach
adapted to handle the bound constraints. \texttt{Snopt} is also available
in NEOS platform\index{NEOS platform}.

\bigskip

\bigskip

Choosing a method for solving an OC problem depends largely on the type of problem to be
solved and the amount of time that can be invested in coding. An indirect shooting method has the advantage
of being simple to understand and produces highly accurate solutions when it converges \cite{Rao2009}.
The accuracy and robustness of a direct method\index{Direct method} is highly dependent upon the method used.
Nevertheless, it is easier to formulate highly complex problems in a direct way and can be used standard
NLP solvers as an extra advantage. This last feature has the benefit of converging with poor initial guesses and being extremely
computationally efficient since most of the solvers exploit the sparsity of the derivatives in the constraints and objective function.

In the next chapter the basic concepts from epidemiology are provided,
in order to formulate/implement more complex OC problems in the health area.

\clearpage{\thispagestyle{empty}\cleardoublepage}


\chapter{Epidemiological models}
\label{chp3}

\begin{flushright}
\begin{minipage}[r]{9cm}

\bigskip
\small {\emph{In this chapter, the simplest epidemiologic models composed
by mutually exclusive compartments are introduced. Based on the models SIS
(susceptible--infected--susceptible) and SIR (susceptible--infected--recovered),
other models are presented introducing new issues related to maternal immunity
or the latent period, fitting the features to distinct diseases.
Illustrative examples are presented, with diseases that can be described by each model.
The basic reproduction number is calculated and presented as a threshold value
for the eradication or the persistence of the disease in a population.}}

\bigskip

\hrule
\end{minipage}
\end{flushright}

\bigskip
\bigskip

In the 14th century, occurred one of most famous epidemic events: the Black Death.
It killed approximately one third of the European population. From 1918-19,
twenty to forty percent of the world's population suffered from the Spanish Flu,
the most severe pandemic in history. In 1978, the United Nations promoted
an ambitious agreement between the countries forecasting that in the year
2000 infectious diseases would be eradicated. This conjecture failed,
mainly due to the assumption that the microorganisms were biologically stationary
and consequently they were not modified and became resistant to the medicines.
Besides, the improvements in the transportation allowing for a faster movement
of individuals and the population growing especially in developing countries,
led to the appearance of new diseases and the resurgence
of old ones in distinct places. Nowadays, AIDS is the most scrutinized.
In 2007 there were an estimated 33.2 million sufferers worldwide,
and 2.1 million deaths with over three quarters of these occurring
in sub-Saharan Africa \cite{Murray2002}.

Epidemiology --- the study of patterns of diseases including those which
are non-communicable of infections in population --- has become more relevant
and indispensable in the development of new models and explanations for the outbreaks,
namely due to their propagation and causes. In epidemiology, an infection
is said to be endemic in a population when it is maintained in the population
without the need for external inputs. An epidemic occurs when new cases
of a certain disease appears, in a given human population during a given period,
and then essentially disappears.

There are several types of diseases, depending on their type
of transmission mechanism, of which stand out:
\begin{itemize}
\item bacteria, which do not confer immunity against the reinfection and
frequently produce harmful toxins
to the host; in case of infection, the antibiotics are usually efficient
(examples: tuberculosis, meningitis, gonorrhea, syphilis, tetanus);

\item viral agents, that confer immunity against reinfection;
here the antibiotics do not produce effects and usually it is hoped that the immune system
of the host responds to an infection by the virus or it will be necessary to take
antiviral drugs that retard the multiplication of the virus
(examples influenza, chicken pox, measles, rubella, mumps, HIV / AIDS, smallpox);

\item vectors, that are usually mosquitoes or ticks and
are infected by humans and then transmit the disease to other humans
(examples: malaria, yellow fever, dengue, chikungunya).
\end{itemize}

The transmission can happen in a direct or indirect way.
The direct transmission of a disease can happen by physical proximity
(such as sneezing, coughing, kissing, sexual contact) or even by
a specific parasite that penetrates the host through ingestion or the skin.
The indirect transmission involves the vectors
that are intermediaries or carriers of the infection.

In most of the cases, the direct and indirect transmission
of the disease happens between the member that coexists in the host population;
this is called the horizontal transmission. When the direct transmission occurs
from one ascendent to a descendent not yet born (egg or embryo)
it is said that vertical transmission happens \cite{Keeling2008}.

\bigskip

When formulating a model for a particular disease, we should make
a trade-off between simple models --- that omit several details
and generally are used for specific situations in a short time,
but have the disadvantage of possibly being naive and unrealistic
--- and more complex models, with more details and more realistic,
but generally more difficult to solve or could contain parameters
which their estimates cannot be obtained.

Choosing the most appropriated model depends on the precision
or generality required, the available data, and the time frame
in which the results are needed. By definition, all models are ``wrong'',
in the sense that even the most complex will make some simplifying assumptions.
It is, therefore, difficult to definitively express which model is right,
though naturally we are interested in developing models that capture
the essential features of a system. The art of epidemiological modelling
is to make suitable choices in the model formulation making
it as simple as possible and yet suitable for the question
being considered \cite{Hethcote2008}.


\section{Basic Terminology}

Mathematical models are a simplified representation of how an infection
spreads across a po\-pulation over time, and generally come in two forms:
stochastic and deterministic models. The first ones,
employ randomness, with variables being described by
probability distributions. Deterministic models split the population into
subclasses, and an ODE with respect to time is formulated
for each. The state variable are determined using parameters and
initial conditions. The main focus in this chapter
will be the deterministic models, neglecting the others.

Most epidemic models are based on dividing the population
into a small number of compartments. Each containing individuals
that are identical in terms of their status with respect
to the disease in question. Here are some
of the main compartments that a model can contain.

\begin{itemize}
\item \emph{Passive immune} ($M$): is composed by newborns
that are temporary passively immune due to antibodies transferred by their mothers;

\item \emph{Susceptible} ($S$): is the class of individuals who are susceptible to infection;
this can include the passively immune once they lose their immunity or,
more commonly, any newborn infant whose mother has never been infected
and therefore has not passed on any immunity;

\item \emph{Exposed or Latent} ($E$): compartment referred to the individuals
that despite infected, do not exhibit obvious signs of infection
and the abundance of the pathogen may be too low to allow further transmission;

\item \emph{Infected} ($I$): in this class, the level of parasite is sufficiently
large within the host and there is potential in transmitting
the infection to other susceptible individuals;

\item \emph{Recovered or Resistant} ($R$):
includes all individuals who have been infected and have reco\-vered.
\end{itemize}

The choice of which compartments to include in a model depends on the characteristics
of the particular disease being modelled and the purpose of the model. The exposed
compartment is sometimes neglected, when the latent period is considered very short.
Besides, the compartment of the recovered individuals cannot always be considered
since there are diseases where the host has never became resistent. Acronyms
for epidemiology models are often based on the flow patterns between the compartments
such as MSEIR, MSEIRS, SEIR, SEIRS, SIR, SIRS, SEI, SEIS, SI, SIS.


\section{Threshold Values}

There are three commonly used threshold values in epidemiology:
$\mathcal{R}_0$, $\sigma$ and $R$. The most common and probably
the most important is the basic reproduction number
\cite{Heffernan2005,Hethcote2000,Hethcote2008}.

\begin{definition}[Basic reproduction number]\index{Basic reproduction number}

The basic reproduction number,
denoted by $\mathcal{R}_0$, is defined as the
average number of se\-condary infections that occurs when one
infective is introduced into a completely susceptible population.
\end{definition}

This threshold, $\mathcal{R}_0$, is a famous result due
to Kermack and McKendrick \cite{Kermack1927} and is referred to
as the ``threshold phenomenon'', giving a borderline between
a persistence or a disease death.  $\mathcal{R}_0$ it is also
called the basic reproduction ratio or basic reproductive rate.

\medskip

\begin{definition}[Contact number]

The contact number, $\sigma$ is the average number
of adequate contacts of a typical infective during
the infectious period.
\end{definition}

An adequate contact is one that is sufficient for transmission,
if the individual contacted by the susceptible is an infective.
It is implicitly assumed that the infected outsider is in the host population for
the entire infectious period and mixes with the host population in exactly
the same way that a population native would mix.

\medskip

\begin{definition}[Replacement number]

The replacement number, $R$, is the average number of secondary
infections produced by a typical infective
during the entire period of infectiousness.
\end{definition}

Note that the replacement number $R$ changes
as a function of time $t$ as the disease evolves after the initial invasion.

These three quantities $\mathcal{R}_0$, $\sigma$ and $R$
are all equal at the beginning of the spreading
of an infectious disease when the entire population (except the infective
invader) is susceptible. $\mathcal{R}_0$ is only defined at the time of invasion,
whereas $\sigma$ and $R$ are defined at all times.

The replacement number $R$ is the actual number of secondary cases from
a typical infective, so that after the infection has invaded a population and
everyone is no longer susceptible, $R$ is always less than the basic reproduction
number $\mathcal{R}_0$. Also after the invasion, the susceptible fraction is less
than one, and as such not all adequate contacts result in a new case. Thus the
replacement number $R$ is always less than the contact number $\sigma$ after the
invasion \cite{Hethcote2000}. Combining these results leads to
$$\mathcal{R}_0 \geq \sigma \geq R.$$

Note that $\mathcal{R}_0 = \sigma$ for most models,
and $\sigma > R$ after the invasion for all models.

For the models throughout this study the basic reproduction number,
$\mathcal{R}_0$, will be applied. When
$$\mathcal{R}_0 < 1$$
\noindent the disease cannot invade the population and the infection will die out over
a period of time. The amount of time this will take generally depends on how
small $\mathcal{R}_0$ is. When
$$\mathcal{R}_0 > 1$$
\noindent invasion is possible and infection can spread through the population. Generally,
the larger the value of $\mathcal{R}_0$ the more severe, and possibly widespread,
the epidemic will be \cite{Driessche2002}.

In Table~\ref{cap3_R0_examples} are some example diseases with their estimated
basic reproduction number\index{Basic reproduction number}. Due to differences
in demographic rates, rural-urban gradients, and contact structure, different
human populations may be associated with different values of $\mathcal{R}_0$
for the same disease \cite{Anderson1982}.

\begin{table}[ptbh]
\begin{center}
\begin{tabular}{lc}
\hline
Infectious disease & Estimated $\mathcal{R}_0$\\
\hline
Influenza & 3-4\\
Foot and mouth disease & 3.5-4.5\\
Smallpox & 3.5-6\\
Rubella & 6-7\\
Dengue & 1.3-11.6\\
Chickenpox & 10-12\\
Measles & 16-18\\
Whooping Cough & 16-18\\
\hline
\end{tabular}
\caption{Estimated $\mathcal{R}_0$ for some infectious diseases
\cite{Anderson1991,Keeling2008,Nishiura2006}}
\label{cap3_R0_examples}
\end{center}
\end{table}

In the next section some of the epidemiologic models  will be presented.


\section{The SIS model}

Numerous infectious diseases confer no long-lasting immunity. The SIS models
are suitable for some bacterial agent diseases like meningitis, sexually
transmitted diseases such as gonorrhea and for protozoan agent diseases
where malaria and the sleeping sickness are good examples. For these diseases,
an individual can be infected multiple times throughout their lives,
with no apparent immunity. Here, recovery from infection is followed
by an instant return to the susceptible compartment.

Throughout this chapter we will consider that population is constant,
neglecting the tourism and immigration factors. Also it is considered
that the population is homogeneously mixed, which means that every
individual interacts with another at the same level and therefore
all individuals have the same risk of contracting the disease.

The number of individuals in each compartment must be integer,
but if the population size $N$ is sufficiently large,
it is possible to treat $S$ and $I$ as continuous variables.
Calculating the proportion of these compartments, varying from 0 to 1,
it is considered that the total population is constant over time,
\emph{i.e.}, $1=S+I$. The compartment changes are expressed
by a system of differential equations.

The SIS model can be mathematically represented as follows.

\begin{definition}[SIS model]

The SIS model can be formulated as
\begin{equation}
\label{cap3:SIS_EDO}
\begin{tabular}{l}
$\dfrac{dS}{dt}=\gamma I-\beta S I$\\
\\
$\dfrac{dI}{dt}=\beta S I -\gamma I$\\
\end{tabular}
\end{equation}
\noindent subject to initial conditions $S(0)>0$ and $I(0)\geq0$.
\end{definition}

It is considered $\beta$ the transmission rate (per capita) and $\gamma$ the recovery rate,
so the mean infectious period is $1/\gamma$.

Figure~\ref{cap3_SIS_scheme} shows the epidemiological scheme of this model.
Each arrow pointing towards the inside of the compartment represents
a positive term in the differential equation,
and the opposite direction introduces a negative term.

\begin{figure}[ptbh]
\center
\includegraphics [scale=0.6]{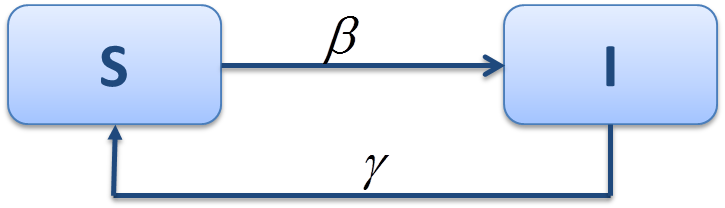}\\
{\caption{\label{cap3_SIS_scheme} The SIS schematic model}}
\end{figure}

The vital dynamics (births and deaths) were not considered,
but a similar model can be cons\-tructed with these effects \cite{Hethcote2008}.
Despite this lack of susceptible births, the disease can still persist because
the recovery of infected individuals replenishes the susceptible class
and will guarantee the long-term persistence of the disease.

\begin{remark}
The SI model is a particular case of the SIS model
when the recovery rate ($\gamma$) is null.
\end{remark}

A new infected individual is expected to infect another at a transmission rate $\beta$,
during $\frac{1}{\gamma}$ that is the infectious period, so the expected
basic reproduction number\index{Basic reproduction number} is
$$\mathcal{R}_0=\frac{\beta}{\gamma}.$$

\bigskip

\begin{example}\hrulefill\label{example_SIS}

Trachoma is an infectious disease causing a characteristic roughening
of the inner surface of the eyelids. Also called granular conjunctivitis
or Egyptian ophthalmia, it is the leading cause of infectious blindness in the world.
Adapting a model from \cite{Ray2007}, and using $\beta=0.047$ as transmission rate
and the recovery rate $\gamma=0.017$, we have the basic reproduction number $\mathcal{R}_0$
approximately 2.76 and the following representation
of the state variables (Figure~\ref{cap3_SIS_trachoma}).

\begin{figure}[ptbh]
\center
\includegraphics [scale=0.8]{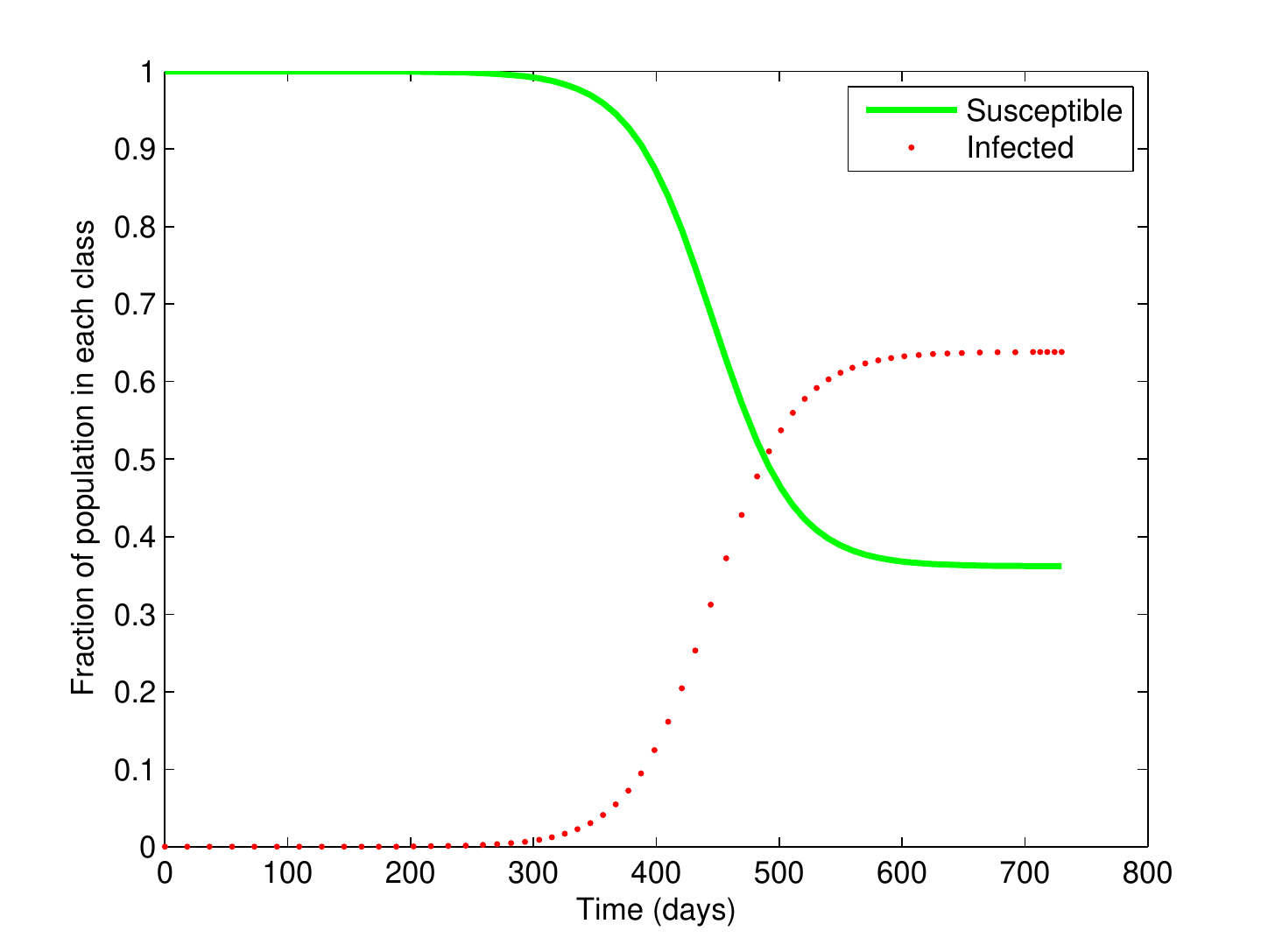}\\
{\caption{\label{cap3_SIS_trachoma} Deterministic SIS model for trachoma disease,
using as initial conditions $S(0)=1-10^{-6}$ and $I(0)=10^{-6}$}}
\end{figure}
\end{example}
\hrule

\bigskip


\section{The SIR model}

The SIR model was initially studied in depth by Kermack and McKendrick
and categorizes hosts within a population as Susceptible, Infected
and Recovered \cite{Kermack1927}. It captures the dynamics of acute infections
that confers lifelong immunity once recovered. Diseases where individuals
acquire permanent immunity, and for which this model may be applied,
include measles, smallpox, chickenpox, mumps, typhoid fever and diphtheria.

Once again we will consider the total population size constant, \emph{i.e.}, $1=S+I+R$.
Two cases will be studied, distinguished by the inclusion or exclusion of demographic factors.


\subsection{The SIR model without demography}

Having compartmentalized the population, we now need a set
of equations that specify how the sizes of compartments change over time.

\medskip

\begin{definition}[SIR model without demography]

The SIR model, excluding births and deaths, can be defined as
\begin{equation}
\label{cap3:SIR_EDO_without_demography}
\begin{tabular}{l}
$\displaystyle\frac{dS}{dt}=-\beta S I$\\
\\
$\displaystyle\frac{dI}{dt}=\beta S I -\gamma I$\\
\\
$\displaystyle\frac{dR}{dt}=\gamma I$\\
\end{tabular}
\end{equation}
\noindent subject to initial conditions $S(0)>0$, $I(0)\geq0$ and $R(0)\geq0$.
\end{definition}

In addition, the transmission rate, per capita, is $\beta$ and the recovery rate
is $\gamma$. Figure~\ref{cap3_SIR_scheme_without}
presents a epidemiological scheme for this model.

\begin{figure}[ptbh]
\center
\includegraphics [scale=0.6]{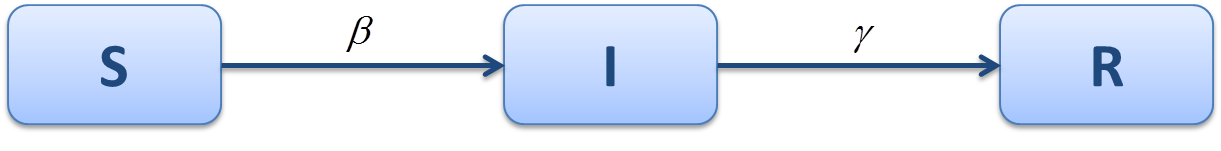}\\
{\caption{\label{cap3_SIR_scheme_without} The SIR schematic model without demography effects}}
\end{figure}

Since population is constant and as $R$ does not appear in the first two differential equations,
most of the times the last equation is omitted, indeed, $R(t)=1-S(t)-I(t)$.
By assuming equations from (\ref{cap3:SIR_EDO_without_demography}) is possible to notice that
$$\frac{dS}{dt}+\frac{dI}{dt}+\frac{dR}{dt}=0.$$
Then a newly introduced infected individual can be expected to infect other people
at the rate $\beta$ during the expected infectious period $1/\gamma$. Thus,
this first infective individual can be expected to infect\index{Basic reproduction number}
$$\mathcal{R}_0=\frac{\beta}{\gamma}.$$

\bigskip

\begin{example}\hrulefill\label{example_influenza}

Consider an epidemic of influenza in a British boarding school in early 1978 \cite{Keeling2008}.
Three boys were reported to the school infirmary with the typical symptoms of influenza.
Over the next few days, a very large fraction of the 763 boys in the school had contact
with the infection. Within two weeks, the infection had become extinguished.
The best fit parameters yield an estimated active infectious period of $1/\gamma=2.2$
days and a mean transmission rate $\beta=1.66$ per day. Therefore, the estimated $\mathcal{R}_0$
is 3.652. Figure~\ref{cap3_SIR_without_demography} represents the dynamics of the three state variables.

\begin{figure}[ptbh]
\center
  \includegraphics [scale=0.6]{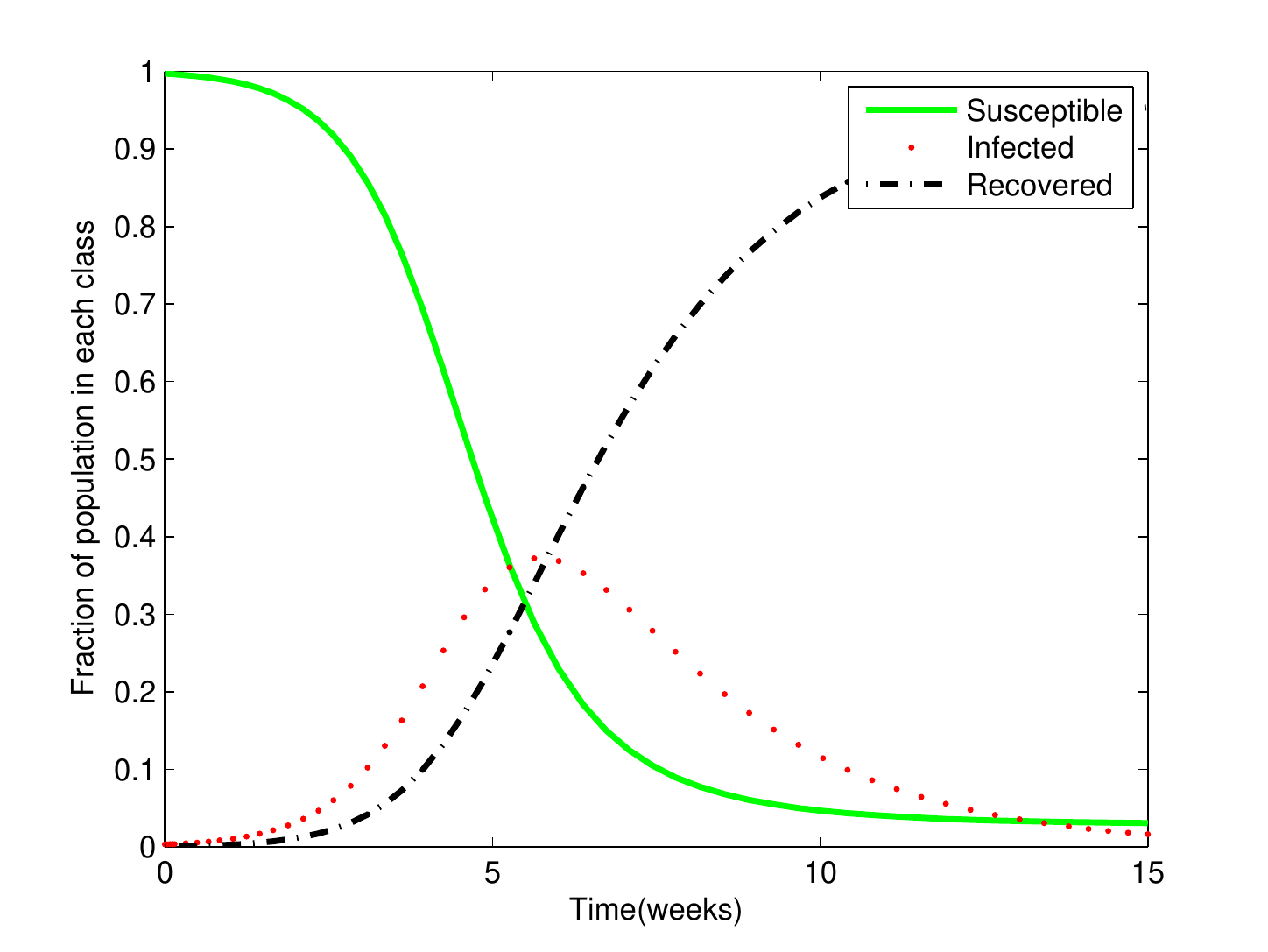}\\
   {\caption{\label{cap3_SIR_without_demography} The time-evolution of influenza over 15 days }}
\end{figure}

\end{example}
\hrule

\bigskip

It was presented a SIR model given the assumption that the time scale of disease spread
was sufficiently fast that births and deaths can be neglected. In the next subsection
we explore the long-term persistence and endemic dynamics of an infectious disease.


\subsection{The SIR model with demography}

The simplest and most common way of introducing demography into the SIR model
is to assume there is a natural host lifespan, $1/\mu$ years. Then, the rate
at which individuals, at any epidemiological compartment, suffer natural mortality
is given by $\mu$. It is important to emphasize that this factor is independent
of the disease and is not intended to reflect the pathogenicity of the infectious agent.
Historically, it has been assumed that $\mu$ also represents the population's crude birth rate,
thus ensuring that total population size does not change through time, or in other words,
$\frac{dS}{dt}+\frac{dI}{dt}+\frac{dR}{dt}=0$.

Putting all these assumptions together, we have a new definition.

\begin{definition}[SIR model with demography]

The SIR model, including births and deaths, can be defined as
\begin{equation}
\label{cap3:SIR_EDO_with_demography}
\begin{tabular}{l}
$\displaystyle\frac{dS}{dt}=\mu-\beta S I-\mu S$\\
\\
$\displaystyle\frac{dI}{dt}=\beta S I -\gamma I-\mu I$\\
\\
$\displaystyle\frac{dR}{dt}=\gamma I-\mu R$\\
\end{tabular}
\end{equation}
\noindent with initial conditions $S(0)>0$, $I(0)\geq 0$ and $R(0)\geq 0$.
\end{definition}

The epidemiological scheme is in Figure~\ref{cap3_SIR_scheme_with}.

\begin{figure}[ptbh]
\center
\includegraphics [scale=0.6]{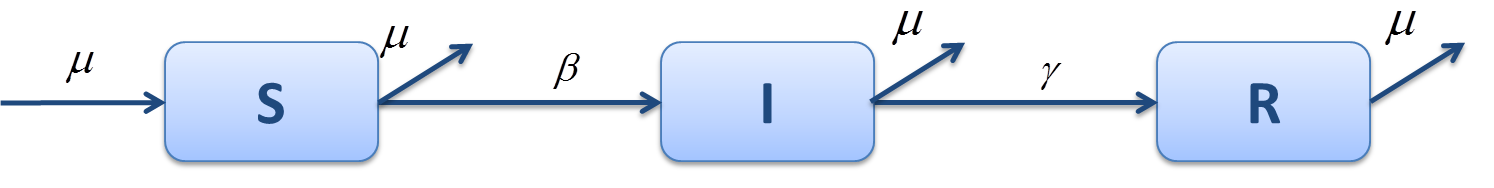}\\
{\caption{\label{cap3_SIR_scheme_with} The SIR schematic model with demography effects}}
\end{figure}

It is important to introduce the $\mathcal{R}_0$ expression for this model.
The parameter $\beta$ represents the transmission rate per infective and the negative terms
in the equation tell us that each individual spends an average $\frac{1}{\gamma+\mu}$
time units in this class. Therefore, if we assume the entire population is susceptible,
then the average number of new infectious per infectious individual
is determined by\index{Basic reproduction number}
$$\mathcal{R}_0=\frac{\beta}{\gamma+\mu}.$$

The inclusion of demographic dynamics may allow a disease to die out or persist
in a population in the long term. For this reason it is important to explore
what happens when the system is at equilibrium.

\begin{definition}[Equilibrium points]\index{Equilibrium point}

A model defined SIR has an equilibrium point,
if a triple $E^{*}=\left(S^{*},I^{*},R^{*}\right)$ satisfies the following system:
$$
\begin{cases}
\frac{dS}{dt}=0\\
\frac{dI}{dt}=0\\
\frac{dR}{dt}=0\\
\end{cases}.
$$

If the equilibrium point has the infectious component equal to zero ($I^{*}=0$),
this means that the pathogen suffered extinction and $E^{*}$
is called Disease Free equilibrium (DFE).\index{Disease Free Equilibrium}

If $I^{*}>0$ the disease persist in the population
and $E^{*}$ is called Endemic Equilibrium (EE)\index{Endemic Equilibrium}.
\end{definition}

\bigskip

With some calculations and algebraic manipulations, it is possible
to obtain two equilibria for the system (\ref{cap3:SIR_EDO_with_demography}):
\begin{equation*}
\label{cap3:SIR_equilibrium}\index{Disease Free Equilibrium}\index{Endemic Equilibrium}
\begin{tabular}{ll}
DFE: & $E^{*}_{1}=(1,0,0)$ \\
EE: & $E^{*}_{2}=\left(\frac{1}{\mathcal{R}_0},\frac{\mu}{\beta}\left(\mathcal{R}_0-1\right),
1-\frac{1}{\mathcal{R}_0}-\frac{\mu}{\beta}\left(\mathcal{R}_0-1\right)\right)$ \\
\end{tabular}
\end{equation*}

When $\mathcal{R}_0<1$, each infected individual produces, on average,
less than one new infected individual, and therefore, predictable that the infection
will be cleared from the population. If $\mathcal{R}_0>1$, the pathogen
is able to invade the susceptible population \cite{Heffernan2005,Hethcote2000}.
It is possible to prove that for the Endemic Equilibrium\index{Endemic Equilibrium}
to be stable, $\mathcal{R}_0$ must be greater than one, otherwise the
Disease Free Equilibrium is stable. More detailed information about local
and global stability of the equilibrium point\index{Equilibrium point}
can be found in \cite{Chavez2002,Kamgang2008,Li1996,Muldowney1999}.

This threshold behavior is very useful, once we can determine which control measures,
and at what magnitude, would be most effective in reducing $\mathcal{R}_0$
below one, providing important guidance for public health initiatives.

\bigskip

\begin{example}\hrulefill\label{example_SIR_with}

The SIR model below, Figure \ref{cap3:SIR_simulations}, is plotted using
similar parameters and initial conditions, except the transmission
rate $\beta$ (adapted from \cite{Keeling2008}).

It is shown that in case (a), using $\beta=520$, the basic reproduction number
is greater than one. With demographic effects could have damped oscillations
with a decreasing amplitude, in the EE\index{Endemic Equilibrium} direction.
In case (b), with $\beta=10$, we obtain $\mathcal{R}_0<1$ and the system tends
to go to a DFE\index{Disease Free Equilibrium}.

\begin{figure}[ptbh]
\centering
\subfloat[$\beta=520$ per year, $1/\mu=70$ years and $1/\gamma=7$ days,
giving $\mathcal{R}_0$ approximately 9.97]{\label{cap3_SIR_with_demography}
\includegraphics[width=0.65\textwidth]{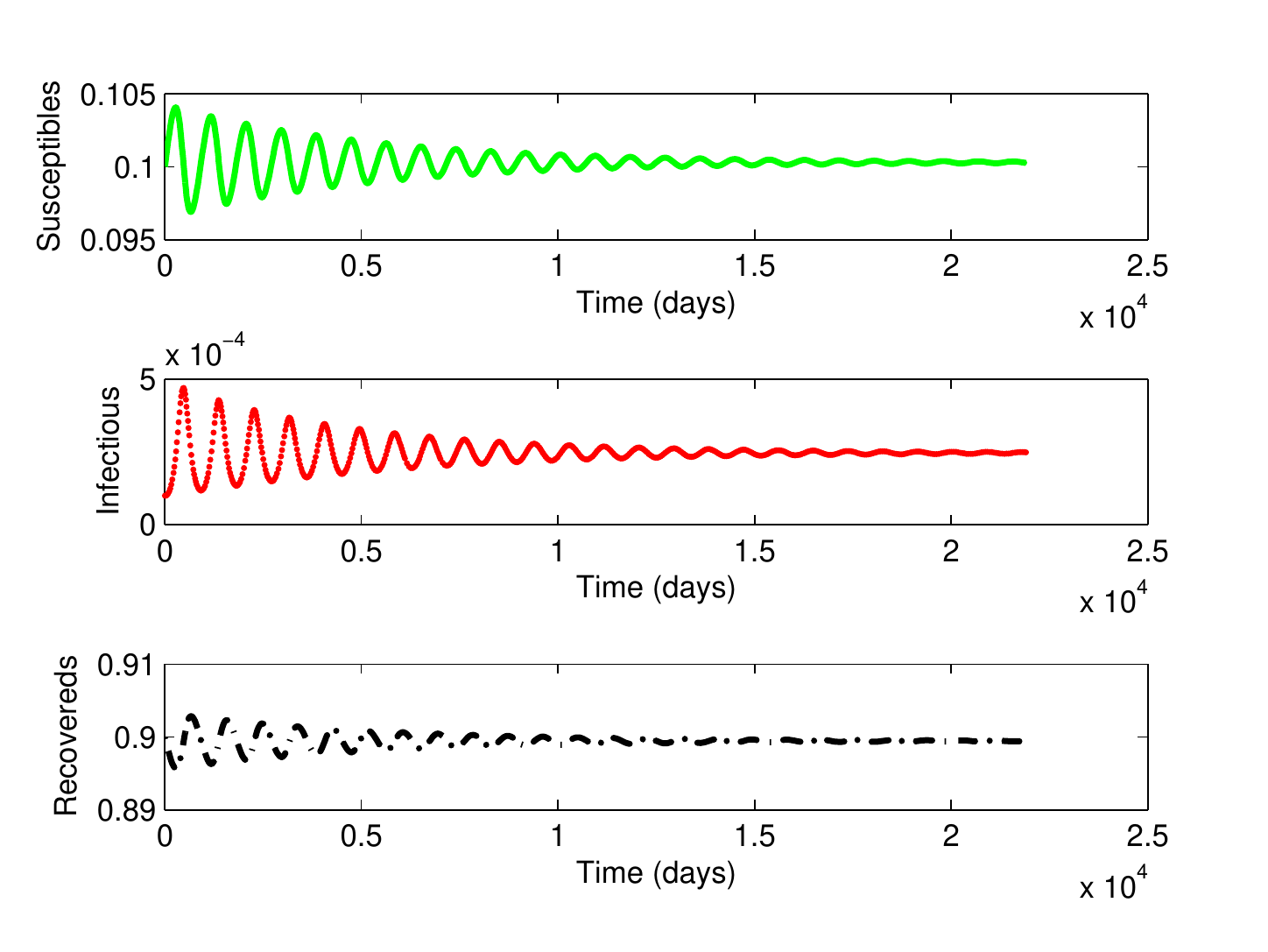}}
\\
\subfloat[$\beta=10$ per year, $1/\mu=70$ years and $1/\gamma=7$ days,
giving $\mathcal{R}_0$ approximately 0.19]{\label{cap3_SIR_with_demography_caso_free}
\includegraphics[width=0.65\textwidth]{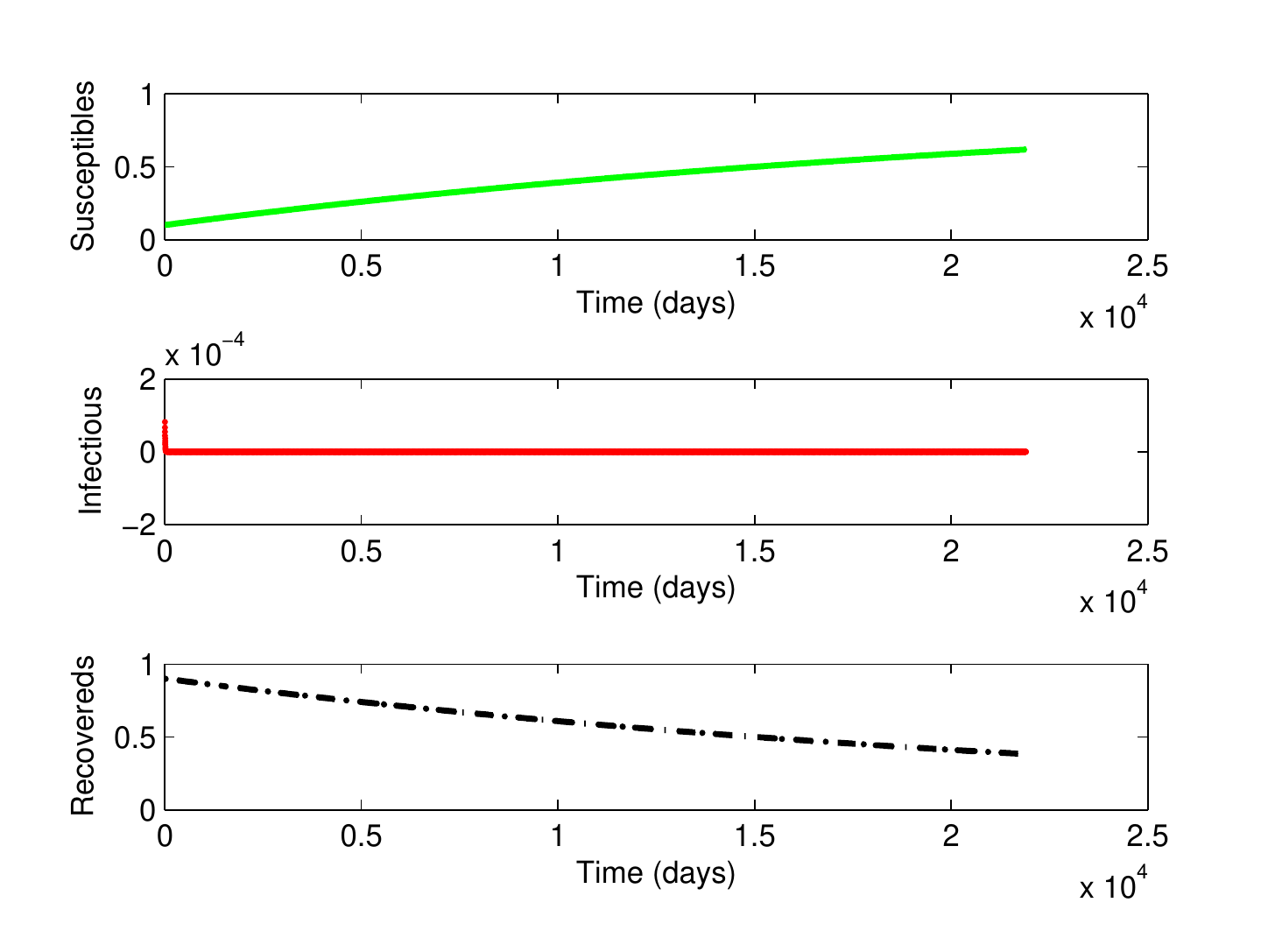}}
\caption{Dynamics of SIR model with distinct basic reproduction numbers\index{Basic reproduction number}}
\label{cap3:SIR_simulations}
\end{figure}
\end{example}
\hrule

\bigskip

The next three sections present a brief description
of other possible refinements of the basic models SIS and SIR.


\section{The SEIR model}

In the SEIR case, the transmission process occurs with an initial inoculation
with a very small number of pathogens. Then, during a period of time the pathogen
reproduces rapidly within the host, relatively unchallenged by the immune system.
During this stage, pathogen abundance is too low for active transmission
to other susceptible host, and yet the pathogen is present. The time in this stage
is very difficult to quantify, since there is no symptomatic features of the disease.
It is called the latent or exposed period. Assuming the average duration
of the latent period is given by $\frac{1}{\nu}$, the SEIR model can be described as follow.

\begin{definition}[SEIR model]

The SEIR model is formulated as
\begin{equation}
\label{cap3:SEIR_EDO}
\begin{tabular}{l}
$\displaystyle\frac{dS}{dt}=\mu-(\beta I + \mu)S$\\
\\
$\displaystyle\frac{dE}{dt}=\beta S I -(\nu+\mu)E$\\
\\
$\displaystyle\frac{dI}{dt}=\nu E -(\gamma+\mu)I$\\
\\
$\displaystyle\frac{dR}{dt}=\gamma I -\mu R$
\end{tabular}
\end{equation}
\noindent with initial conditions $S(0)>0$, $E(0)\geq 0$, $I(0)\geq 0$ and $R(0)\geq 0$.
\end{definition}

The parameter $\beta$ and $\gamma$ were defined in the previous section.
The epidemiological scheme for SEIR model is presented in Figure~\ref{cap3_SEIR_scheme}.

\begin{figure}[ptbh]
\center
  \includegraphics [scale=0.6]{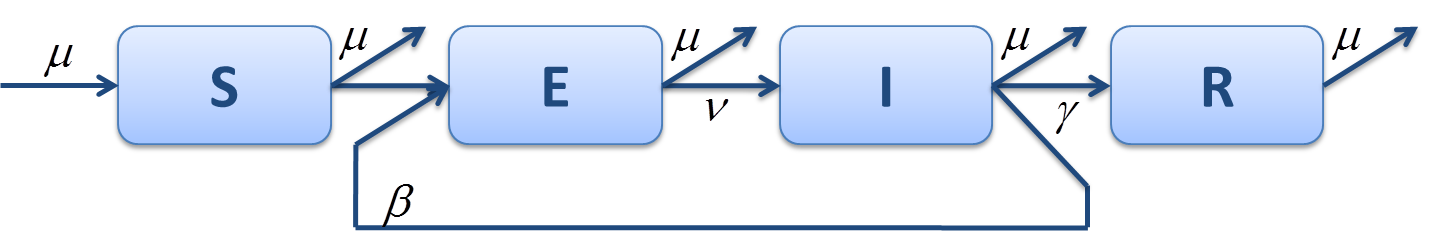}\\
   {\caption{\label{cap3_SEIR_scheme} The SEIR schematic model}}
\end{figure}

The expression for the basic reproduction number
\index{Basic reproduction number} is
$$\mathcal{R}_0=\frac{\beta \nu}{(\gamma+\mu)(\nu+\mu)}.$$

This threshold is the product of the contact rate $\beta$ per unit time,
the average infectious period adjusts to the population growth of $\frac{1}{\gamma+\mu}$,
and the fraction $\frac{\nu}{\nu+\mu}$ of exposed people surviving the latent class $E$.

Finding steady states of the system, we obtain\index{Disease Free Equilibrium}\index{Endemic Equilibrium}
the following equilibrium points:\index{Equilibrium point}
\begin{equation*}
\label{cap3:SEIR_equilibrium}
\begin{tabular}{ll}
DFE: & $E^{*}_{1}=(1,0,0,0)$ \\
EE: & $E^{*}_{2}=\left(S^{*},E^{*},I^{*},R^{*}\right)$\\
\end{tabular}
\end{equation*}
\noindent with
\begin{equation*}
\label{cap3:SEIR_equilibrium2}
\begin{tabular}{l}
$\displaystyle S^{*}=\frac{1}{\mathcal{R}_0}$\\
\\
$\displaystyle E^{*}=\frac{\mu(\mu+\gamma)}{\beta \nu}\left(\mathcal{R}_0-1\right)$\\
\\
$\displaystyle I^{*}=\frac{\mu}{\beta}\left(\mathcal{R}_0-1\right)$\\
\\
$\displaystyle R^{*}=\frac{\gamma}{\beta}\left(\mathcal{R}_0-1\right)$\\
\end{tabular}
\end{equation*}

Although the SIR and SEIR models behave similarly at equilibrium, when the parameters
are suitably adapted, the SEIR model has a slower growth rate after pathogen invasion.
This is due to the fact that individuals need to stay some time in the exposed class before
their contribution in the disease transmission chain \cite{Keeling2008}.


\section{The MSEIR model}

An infected or vaccinated mother transfers some antibodies across the placenta to her fetus,
so that the newborn infant has temporary passive immunity to an infection. Since the infant
can not produce new antibodies, when these passive antibodies are gone at a rate $\delta$,
the baby passes from the immune state $M$ to the susceptible state $S$. Some childhood
diseases have this feature. The birth rate $\mu S$ into the susceptible class of size $S$
corresponds to newborns whose mothers are susceptible, and the other newborns $\mu(1-S)$
enter passively immune class of size $M$, since their mothers were infected
or had some type of immunity \cite{Hethcote2008}.

The transfer diagram for the MSEIR model is shown in Figure~\ref{cap3_MSEIR_scheme}.

\begin{figure}[ptbh]
\center
  \includegraphics [scale=0.55]{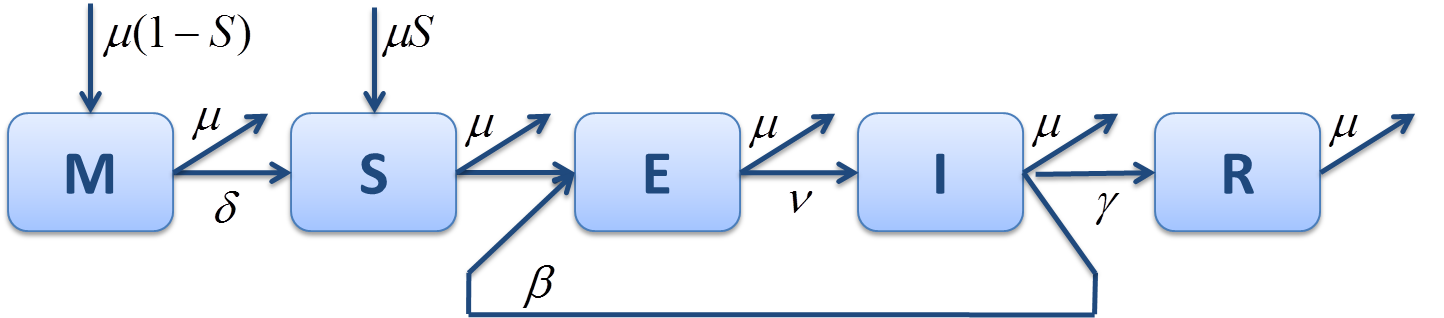}\\
   {\caption{\label{cap3_MSEIR_scheme} The MSEIR schematic model }}
\end{figure}

The MSEIR is also composed by a system of differential equations.

\begin{definition}[MSEIR model]

The SEIR model can be described as
\begin{equation}
\label{cap3:MSEIR_EDO}
\begin{tabular}{l}
$\displaystyle\frac{dM}{dt}=\mu-(\delta + \mu)M$\\
\\
$\displaystyle\frac{dS}{dt}=\delta M-(\beta I + \mu)S$\\
\\
$\displaystyle\frac{dE}{dt}=\beta S I -(\nu+\mu)E$\\
\\
$\displaystyle\frac{dI}{dt}=\nu E -(\gamma+\mu)I$\\
\\
$\displaystyle\frac{dR}{dt}=\gamma I -\mu R.$
\end{tabular}
\end{equation}
\noindent with initial conditions $M(0)\geq 0$,
$S(0)>0$, $E(0)\geq 0$, $I(0)\geq 0$ and $R(0)\geq 0$.
\end{definition}

Thus, the basic reproduction number\index{Basic reproduction number}
is equal to previous model $SEIR$, once the $M$ compartment
does not affect the transmission chain of the disease:
$$\mathcal{R}_0=\frac{\beta \nu}{(\gamma+\mu)(\nu+\mu)}.$$

The equations (\ref{cap3:MSEIR_EDO}) always have
a DFE given by\index{Disease Free Equilibrium}

$$E_{1}^{*}=(0,1,0,0,0).$$

If $\mathcal{R}_0>1$, there is also a unique EE\index{Endemic Equilibrium}
given by $E_2^{*}=\left(M^{*},S^{*},E^{*},I^{*},R^{*}\right)$, where
\begin{equation*}
\label{cap3:MSEIR_equilibrium}
\begin{tabular}{l}
$\displaystyle M^{*}=\frac{\mu}{\delta+\mu}\left(1-\frac{1}{\mathcal{R}_0}\right)$\\
\\
$\displaystyle S^{*}=\frac{1}{\mathcal{R}_0}$\\
\\
$\displaystyle E^{*}=\frac{\delta \mu}{(\delta+\mu)(\nu+\mu)}\left(1-\frac{1}{\mathcal{R}_0}\right)$\\
\\
$\displaystyle I^{*}=\frac{\delta\nu\mu}{(\delta+\mu)(\nu+\mu)(\gamma+\mu)}\left(1-\frac{1}{\mathcal{R}_0}\right)$\\
\\
$\displaystyle R^{*}=1-M^{*}-S^{*}-E^{*}-I^{*}$\\
\end{tabular}
\end{equation*}

\bigskip

There are several models with different epidemiological states,
such as MSEIRS, SEIRS, SEI, SI, SIRS, depending on the specific features
and the level of detail that is wished to introduce in the model. These models
are similar to the previous ones presented. Other epidemiological models
can be studied using more compartments such as Quarantine-Isolation
($Q$), Treatment ($T$), Carrier ($C$) or Vaccination ($V$). Besides,
most of the populations can be subdivided into different groups
(sex, age, health weaknesses, ...), depending upon characteristics
that may influence the risk of catching and/or transmitting an infection.
More details and examples can be found in \cite{Allman2004,Brauer2008,Hethcote2008,Vynnycky2010}.
Other diseases can be caught and transmitted by numerous hosts or even is required
a second different population to complete the transmission cycle, such as vector borne-diseases,
using double models. This last case will be explored in the second part of the thesis.

In cases that include multiple compartments of infected individuals, in which vital
and epidemiological parameters depend on factors as stage of the disease,
spatial position, age, behavior, multigroups, the next generation method
is the more generalized approach to calculate $\mathcal{R}_0$.


\section{$\mathcal{R}_0$ using the next generation method}

The definition of $\mathcal{R}_0$ has more than one possible interpretation,
depending on the field (ecology, demography or epidemiology) and there exist
distinct methods and estimations to calculate this threshold.

The next generation method, introduced by Diekmann \emph{et al.}\cite{Diekmann1990},
defines $\mathcal{R}_0$ as the spectral radius of the next generator operator.
The formation of the operator involves the determination of two compartments,
infected and non-infected from the model. Recent examples of this method
are given in \cite{Diekmann2000,Driessche2002,Heffernan2005,Wonham2004}.

Let us assume that there are $n$ compartments of which $m$ are infected.
We defined the vector $x=(x_1,\ldots,x_n)^{T}$, where $x_i \geq 0$
denotes the number of proportion of individuals in the $i$th compartment.
For clarity we sort the compartments so that the first $m$ compartments
correspond to infected individuals. The distinction between infected
and uninfected compartments must be determined from the epidemiological
interpretation, and not by the mathematical expressions.

It is necessary to define the set

$$X_{s}=\{x\geq0|x_i=0, i=1,\ldots,m\}.$$

Let $\mathcal{F}_i(x)$ be the rate of appearance of new infections
in compartment $i$ and let $\mathcal{V}_i(x)=V_i^{-}(x)-V_{i}^{+}(x)$,
where $V_{i}^{+}$ is the rate of transfer of individuals into compartment $i$
by all other means and $V_{i}^{-}$ is the rate of transfer
of individuals out of the $i$th compartment.

The disease transmission model consists of nonnegative initial conditions
together with the following system of equations:
\begin{equation}
\label{cap3:disease_transmission}
\dot{x}_{i}=f_i(x)=\mathcal{F}_{i}(x)-\mathcal{V}_{i}(x)
\end{equation}

Note that $F_{i}$ should include only infections that are newly arising,
but does not include terms which describe the transfer of infectious
individuals from one infected compartment to another.

Let us consider the following assumptions:

\begin{description}
\item[(A1)] If $x\geq0$, then $\mathcal{F}_{i}$,
$\mathcal{V}_{i}^{+}$, $\mathcal{V}_{i}^{-}\geq0$ for $i=1,\ldots,n$.
\item[(A2)] If $x_{i}=0$ then $\mathcal{V}_{i}^{-}=0$.
In particular, if $x \in X_{s}$ then $\mathcal{V}_{i}^{-}=0$ for $i=1,\ldots,m$.
\item[(A3)] $\mathcal{F}_{i}=0$ if $i>m$.
\item[(A4)] If $x_{i}\in X_{s}$ then $\mathcal{F}_{i}(x)=0$
and $\mathcal{V}_{i}^{+}(x)=0$ for $i=1,\ldots,m$.
\item[(A5)] If $\mathcal{F}(x)$ is set to zero, then all eigenvalues
of $Df(x_0)$ have negative real parts and $Df(x_0)$ is the derivative
$\left[\frac{\partial f_{i}}{\partial x_j}\right]$ evaluated at the DFE $x_0$.
\end{description}

Assuming that $\mathcal{F}_{i}$ and $\mathcal{V}_{i}$ meet the assumptions above,
we can form the next generation matrix $FV^{-1}$ from matrices of partial derivatives
of $\mathcal{F}_{i}$ and $\mathcal{V}_{i}$. Specially,
\begin{equation*}
\begin{tabular}{lll}
$F=\left[\frac{\partial \mathcal{F}_{i}}{\partial x_j}(x_0)\right]$
& and & $V=\left[\frac{\partial \mathcal{V}_{i}}{\partial x_j}(x_0)\right]$
\end{tabular}
\end{equation*}
\noindent where $i,j=1,\ldots,m$ and $x_0$ is the DFE.\index{Disease Free Equilibrium}

The entries of $FV^{-1}$ give the rate at which infected individuals
in $x_j$ produce new infections in $x_i$, times the average length
of time an individual spends on a single visit to compartment $j$.

\begin{definition}[Basic reproduction number using the next generator operator]\index{Basic reproduction number}

The basic reproduction number is given by
\begin{equation}
\label{cap3:R0_next_generator}
\mathcal{R}_{0}=\rho(FV^{-1})
\end{equation}
where $\rho$ denotes the spectral radius (dominant eigenvalue) of the matrix $FV^{-1}$.
\end{definition}

The following theorem states that $\mathcal{R}_0$ is a threshold parameter
for the stability of the DFE.\index{Disease Free Equilibrium}

\begin{theorem}
Consider the disease transmission model given by (\ref{cap3:disease_transmission})
with $f(x)$ satisfying conditions (A1)--(A5). If $x_0$ is a DFE of the model,
then $x_0$ is locally asymptotically stable if $\mathcal{R}_0<1$, but unstable
if $\mathcal{R}_0>1$, where $\mathcal{R}_0$ is defined by (\ref{cap3:R0_next_generator}).
\end{theorem}

\begin{proof}
The proof of this theorem can be found in \cite{Driessche2002}.
\end{proof}

\bigskip

\begin{example}\hrulefill\label{example_R0_next_generation}

Consider a simple model SEIT for tuberculosis with a treated compartment
(adapted from \cite{Blower1996}). Tuberculosis, is a common, and in many cases lethal,
infectious disease caused by various strains of mycobacteria. Tuberculosis typically
attacks the lungs but can also affect other parts of the body. Most infections
are asymptomatic and latent, but about one in ten latent infections eventually
progresses to active disease which, if left untreated, kills more than 50\% of those so infected.

In this model a constant population is considered where $N=S+E+I+T$.
Exposed individuals progress to the infectious compartment at a rate $\nu$.
The treatment rates are $r_1$ for exposed individuals and $r_2$ for infectious
individuals. However only a fraction $q$ of the treatments of infectious individuals
are successful. Unsuccessfully treated infectious individuals re-enter
the exposed compartment ($1-q$). The dynamics are illustrated
in Figure~\ref{cap3_SEIT_scheme} and the differential equations are the following:

\begin{figure}[ptbh]
\center
  \includegraphics [scale=0.5]{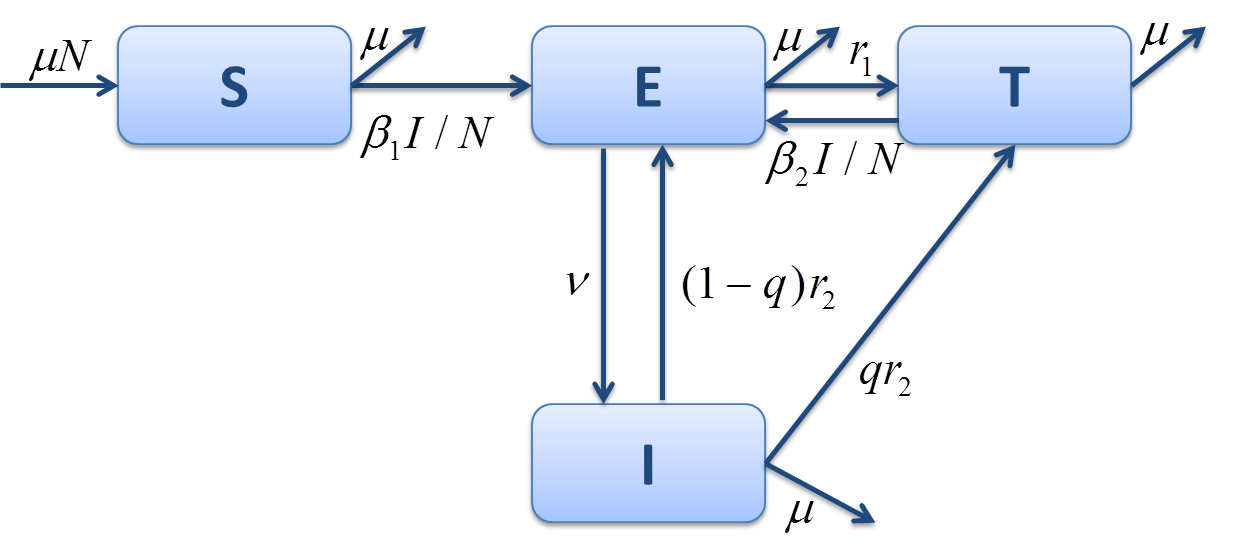}\\
   {\caption{\label{cap3_SEIT_scheme} The SEIT schematic model }}
\end{figure}

\begin{equation*}
\label{cap3:SEIT_equilibrium}
\begin{tabular}{l}
$\displaystyle\frac{dS}{dt}=\mu N - \left(\beta_{1}\frac{I}{N}+\mu\right)S$\\
\\
$\displaystyle\frac{dE}{dt}=\beta_{1}\frac{IS}{N}+\beta_{2}\frac{IT}{N}+(1-q)r_{2}I-\left(\nu + r_{1}+\mu\right)E$\\
\\
$\displaystyle\frac{dI}{dt}=\nu E -\left(r_{2}+\mu \right)I$\\
\\
$\displaystyle\frac{dT}{dt}=r_{1}E+q r_{2}I-\left(\beta_{2}\frac{I}{N}+\mu\right)T$\\
\end{tabular}
\end{equation*}

A Disease Free Equilibrium\index{Disease Free Equilibrium}
is $x_{0}=(N, 0, 0 ,0)^{T}$. Note that progression from $E$ to $I$
and failure of treatment are not considered to be new infections. Thus,
let only consider the infected compartments $E$ and $I$, which gives $m=2$,
and we will construct the matrices $\mathcal{F}$ and $\mathcal{V}$
only related to these compartments. Let
\begin{center}
$\mathcal{F}(x)=\left(
                  \begin{array}{c}
                    \beta_{1}\frac{IS}{N}+\beta_{2}\frac{IT}{N} \\
                    0 \\
                  \end{array}
                \right)
$
and
$\mathcal{V}(x)=\left(
                  \begin{array}{c}
                    -(1-q)r_{2}I+\left(\nu + r_{1}+\mu\right)E\\
                     -\nu E +\left(r_{2}+\mu \right)I\\
                  \end{array}
                \right).
$
\end{center}

Hence
\begin{center}
$F(x_0)=\left(
                  \begin{array}{cc}
                    0 & \beta_{1} \\
                    0  & 0 \\
                  \end{array}
                \right)
$
and
$V(x_0)=\left(
                  \begin{array}{cc}
                    \nu + r_{1}+\mu & -(1-q)r_{2}\\
                     -\nu & r_{2}+\mu \\
                  \end{array}
                \right).
$
\end{center}

This way,
$$\mathcal{R}_{0}=\rho(FV^{-1})=\frac{\beta_1\nu}{(\nu+r_{1}+\mu)(r_{2}+\mu)-\nu(1-q)r_{2}}.$$
\end{example}
\hrule

\bigskip

Epidemiological models can be understood as a framework to explain the mechanisms
of the disease procedures and test ideas to implement control measures.
$\mathcal{R}_0$ is a key concept, used as a threshold parameter
to predict where an infection will spread and how these controls can be effective.

In the second part of the thesis, the epidemiological models will be used
for studying Dengue disease. Applying OC theory, the repercussions
of several controls in the development of the disease will be analyzed.

\clearpage{\thispagestyle{empty}\cleardoublepage}


\part{Original Results}
\clearpage{\thispagestyle{empty}\cleardoublepage}


\chapter{A first contact with the Dengue disease}
\label{chp4}

\begin{flushright}
\begin{minipage}[r]{9cm}

\bigskip
\small {\emph{During the last decades, the global prevalence of Dengue progressed dramatically.
In this chapter, Dengue details are given, such as disease symptoms,
transmission and epidemiological trends. An old OC model for Dengue is revisited,
as a first approach to the disease. Due to the software robustness improvements
and the higher computational capacity, a better solution for this problem is proposed.
In order to study different discretization schemes for an OC problem,
taking into account time performance and resources used,
some numerical simulations are made using Euler and Runge-Kutta methods.}}

\bigskip

\hrule
\end{minipage}
\end{flushright}

\bigskip
\bigskip

\onehalfspacing

The origins of the word Dengue are not clear. Some researchers think that it is derived
from the Swahili phrase ``Ka-dinga pepo'', meaning ``cramp-like seizure caused by an evil spirit''.
The first recognized Dengue epidemics occurred almost simultaneously in Asia, Africa,
and North America in the 1780s, shortly after the identification and naming
of the disease in 1779 \cite{DengueNet}.

Dengue transcends international borders and can be found in tropical
and subtropical regions around the world, predominantly in urban and semi-urban areas.
Dengue is a disease which is now endemic in more than one hundred countries of Africa,
America, Asia and the Western Pacific. In Figure~\ref{World_Dengue}
it is possible to see the areas that in 2008 have more surveillance.

Nevertheless, some studies have indicated that countries with a mild climate,
such as in the Mediterranean, are at risk due to future climate conditions
that may be favorable to this kind of disease \cite{Hopp2001}.
In Europe there are no registered cases,
but the main vector of the disease is already in the old continent
and has been followed on Madeira island \cite{Relatorio2011, SOL2009}.
This risk may be aggravated further due to climate changes and to the globalization,
as a consequence of the huge volume of international tourism and trade \cite{Semenza2009}.
Travelers play an essential role in the global epidemiology: they act as
viremic travelers, carrying the disease into areas where mosquitos can transmit the infection.

\begin{figure}[ptbh]
\center
\includegraphics[scale=1.3]{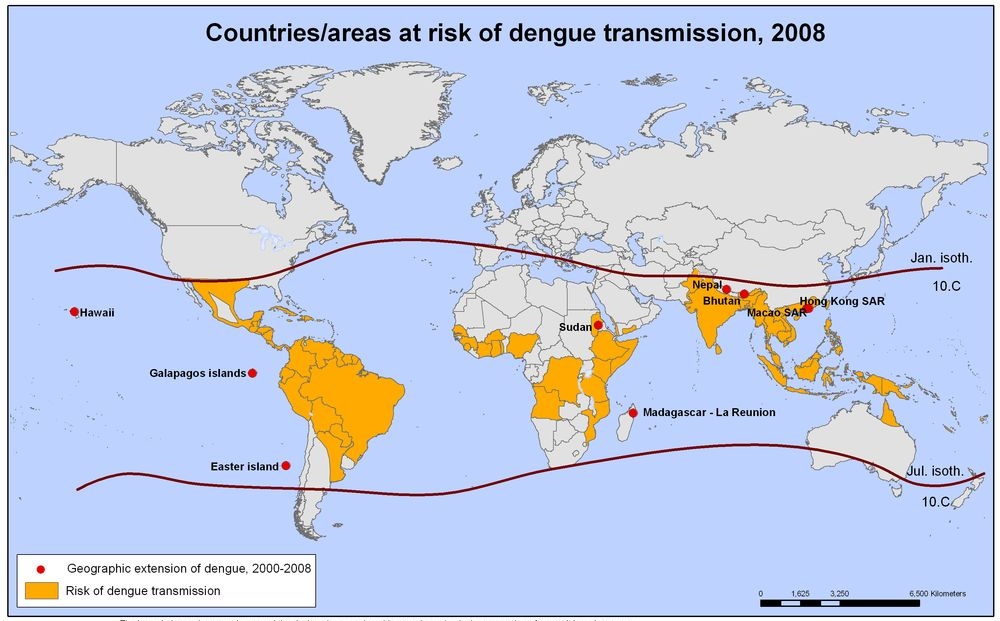}
\caption{Countries/areas at risk of Dengue transmission, 2008 \emph{(Source WHO)}\label{World_Dengue}}
\end{figure}


\section{Background for Dengue}
\label{sec:4:1}

\subsection{Dengue Fever and Dengue Hemorrhagic Fever}
\label{sec:4:1:1}

Dengue is a vector-borne disease transmitted from an infected human to a female
\emph{Aedes} mosquito by a bite. Then, the mosquito, that needs regular meals
of blood to feed their eggs, bites a potential healthy human and transmits the disease making it a cycle.

There are two forms of Dengue: Dengue Fever (DF) and Dengue Hemorrhagic Fever (DHF).
The first one is characterized by a sudden high fever without respiratory symptoms,
accompanied by intense headaches, painful joints and muscles and lasts between
three to seven days. Humans may only transmit the virus during the febrile
stage \cite{DengueNet}. DHF initially exhibits a similar,
if more severe pathology as DF, but deviates from the classic pattern at the end of
the febrile stage \cite{CDC2010}. The hemorrhagic form has an additionally
bleeding from the nose, mouth and gums or skin bruising, nausea,
vomiting and fainting due to low blood pressure by fluid leakage.
It usually lasts between two to three days and can lead to death \cite{Derouich2003}.
Nowadays, Dengue is the mosquito-borne infection that has become a major international
public health concern. According to the World Health Organization (WHO), 50 to 100 million
Dengue Fever infections occur yearly, including 500000 Dengue Hemorrhagic
Fever cases and 22000 deaths, mostly among children \cite{Who}.

There are four distinct, but closely related, viruses that cause
Dengue. The four serotypes, named DEN-1 to DEN-4,
belong to the \emph{Flavivirus} family, but they are antigenically distinct.
Recovery from infection by one virus provides lifelong
immunity against that virus but confers only partial and transient
protection against subsequent infection by the other three
viruses. There is good evidence that a sequential infection
increases the risk of developing DHF \cite{Wearing2006}.

Unfortunately, there is no specific effective treatment for Dengue.
Activities, such as triage and management, are critical in
determining the clinical outcome of Dengue. A rapid and efficient
front-line response not only reduces the number of unnecessary
hospital admissions but also saves lives.

Although up until now there is no effective and safe vaccine for Dengue,
a number of candidates are undergoing various phases
of clinical trials \cite{Who2009}. With four closely
related viruses that can cause the disease, there is a need for
a vaccine that would immunize against all four types to be effective.
The main difficulty in the vaccine production is that there
is a limited understanding of how the disease typically behaves and how
the virus interacts with the immune system. Another
challenge is that some studies show that some secondary Dengue
infection can leave to DHF, and theoretically a vaccine could be a
potential cause of severe disease if a solid immunity is not
established against the four serotypes. Research to develop
a vaccine is ongoing and the incentives to study the
mechanism of protective immunity are gaining more support, now that
the number of outbreaks around the world is increasing
\cite{DengueNet}.

The spread of Dengue is attributed to the geographic expansion of the mosquitoes
responsible for the disease: \emph{Aedes aegypti} and \emph{Aedes albopictus}
\cite{Cattand2006}. Due to its higher interaction with humans and its urban behavior,
the first mosquito is considered the major responsible
for Dengue transmission and, our attention will be focused on it.


\subsection{Biological notes on \emph{Aedes aegypti}}
\label{sec:4:1:2}

\emph{Aedes aegypti}, in Figure~\ref{aedes_aegypti}, is an insect species
closely associated with humans and their dwellings, thriving in
crowded cities and biting primarily during the day. Humans not only
provide blood meals for mosquitoes, but also nutrients needed
for them to reproduce through water-holding containers,
in and around their homes. In urban areas, \emph{Aedes} mosquitoes breed on water
collections in artificial containers such as cans, plastic cups, used
tires, broken bottles and flower pots. With
increasing urbanization, crowded cities, poor sanitation and lack
of hygiene, environmental conditions foster the spread of the
disease that, even in the absence of fatal forms, breed
significant economic and social costs (absenteeism,
immobilization, debilitation and medication) \cite{Derouich2006}.

The mosquito \emph{Aedes aegypti} is a tropical and
subtropical specie widely distributed around the world, mostly
between latitudes $35^{o}$N and 35$^o$S, which corresponds,
approximately, to a winter isotherm of 10$^o$C \cite{Braga2007},
as Figure~\ref{World_Dengue} illustrates.

\begin{figure}[ptbh]
\center
\includegraphics[scale=1]{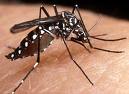}
\caption{Mosquito \emph{Aedes aegypti} \label{aedes_aegypti}}
\end{figure}

Dengue is spread only by adult females that require a blood meal
for the development of their eggs, whereas male mosquitoes feed
on fruit nectar and other sources of sugar. In this process the
female acquire the virus while feeding on the blood of an infected person.
After virus incubation from eight to twelve days (extrinsic period),
an infected mosquito is capable, during probing and blood feeding,
of transmitting the virus for the rest of its life to susceptible humans,
and the intrinsic period for humans varies from 3 to 15 days.

The life cycle of a mosquito has four distinct stages: egg, larva,
pupa and adult, as it is possible to see in Figure \ref{lifecycle}.
In the case of \emph{Aedes aegypti}, the first
three stages take place in or near water whilst air is the medium
for the adult stage \cite{Otero2008}. Female mosquitoes lay their eggs,
but usually do not lay  them all at once: it releases them in different places,
increasing the probability of new births \cite{Who1997}.

\begin{figure}[ptbh]
\center
\includegraphics[scale=0.45]{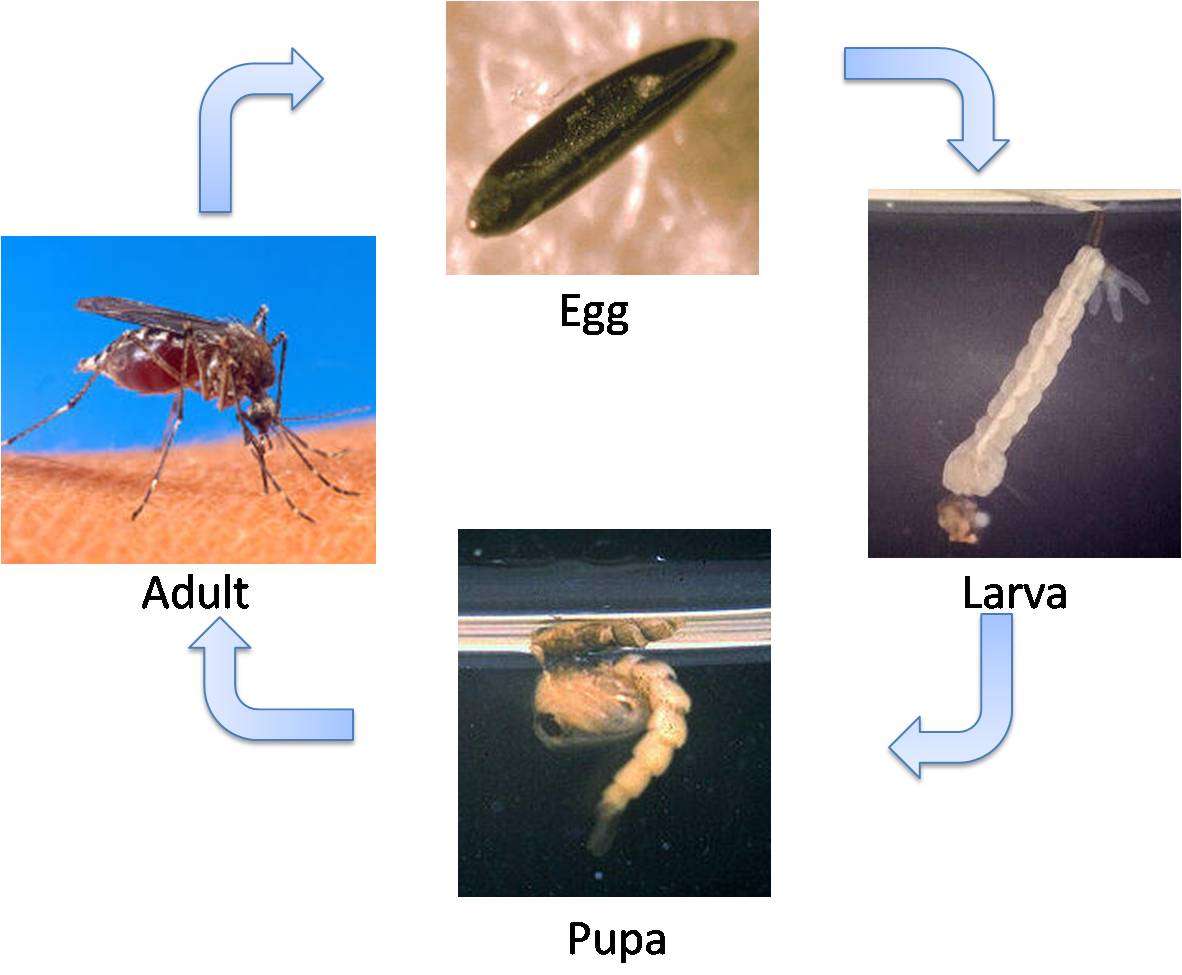}
\caption{Life cycle of \emph{Aedes aegypti}\label{lifecycle}}
\end{figure}

The eggs of \emph{Aedes aegypti} can resist to droughts and low temperatures for up to one
year. Although the hatching of mature eggs may occur spontaneously
at any time, this is greatly stimulated by flooding. Larvae hatch when water
inundates the eggs as a result of rains or an addition of water
by people. In the following days, the larvae  will feed on
microorganisms and particulate organic matter. When the larva has acquired enough
energy and size, metamorphosis is done, changing the larva into pupa.
Pupae do not feed: they just change in form until the adult body is formed.
The newly formed adult emerges from the water after breaking the pupal skin.
The entire life cycle, from the aquatic phase (eggs, larvae, pupae) to the adult phase,
lasts from 8 to 10 days at room temperature, depending on the level
of feeding \cite{Christophers1960}. The adult stage of the mosquito is considered to last an average
of eleven days in an urban environment, reaching up to 30 days in laboratory environment.

Studies suggest that most female mosquitoes may spend their
lifetime in or around the houses where they emerge as adults. This
means that people, rather than mosquitoes, rapidly move the virus
within and between communities.

\emph{Aedes aegypti} is one of the most efficient vectors
for arboviruses because it is highly anthropophilic, frequently
bites several times before complete oogenesis \cite{Who}. The extent of Dengue
transmission is determined by a wide variety of factors: the level
of herd immunity\index{Herd immunity} in the human population to circulating virus
serotype(s); virulence characteristics of the viral strain;
survival, feeding behavior, and abundance of \emph{aedes aegypti};
climate; and human density, distribution, and movement \cite{Scott2004}.

It is very difficult to control or eliminate \emph{aedes aegypti}
mosquitoes because they are highly resilient,
quickly adapting to changes in the environment and they have
the ability to rapidly bounce back to initial numbers
after disturbances resulting from natural phenomena
(\textrm{e.g.}, droughts) or human interventions
(\textrm{e.g.}, control measures). We can safely expect
that transmission thresholds will vary depending
on a range of factors.


\subsection{Measures to fight Dengue}

Primary prevention of Dengue resides mainly in mosquito control.
There are two primary methods: larval control and adult mosquito
control, depending on the intended target \cite{Natal2002}.
\emph{Larvicide} treatment is done through long-lasting chemical
in order to kill larvae and preferably have WHO clearance for use
in drinking water \cite{Derouich2003}. The application of \emph{adulticides}
can have a powerful impact on the abundance of adult mosquito vector. However,
the efficacy is often constrained by the difficulty in achieving sufficiently
high coverage of resting surfaces \cite{Devine2009}. This is the most common measure.
However the long term use of adulticide has several risks: the resistance
of the mosquito to the product reducing its efficacy, the killing of other species
that live in the same habitat and has also been linked to numerous adverse health
effects including the worsening of asthma and respiratory problems.

Larvicide treatment is an effective way to control the vector
larvae, together with \emph{mechanical control}, which is related to
educational campaigns. The mechanical control must be done by both public health
officials and residents in affected area. The participation of the entire population
is essential in removing still water from domestic recipients and eliminating
possible breeding sites \cite{Who2009}.

The most recent approach for fighting the disease is \emph{biological control}.
It is a natural process of population regulation through natural enemies.
There are techniques that combine some parasites that kill partially
the larval population; however there are some operational resistance,
because there is lack of expertise in producing these type of parasites
and there are some cultural objections in introducing something
in the water for human consumption \cite{Cattand2006}.

Another way of insect control is to change the reproduction process,
releasing sterile insects. This technique, named as \emph{Sterile Insect Technique},
consists in releasing sterile insects in natural environment, so that the result
of mating produces the non-viability of the eggs, and thus can lead to drastic
reduction of the specie. This way of control, shows two types of inconvenience:
it is expensive to produce and release insects, and can be confronted with social objections,
because an uninformed population could not correctly understand the addition
of insects as a good solution \cite{Bartlett1996,Esteva2005,Thome2007}.

Mathematical modeling became an interesting tool for
understanding epidemiological diseases and for proposing
effective strategies in fighting them. A set of mathematical models
have been developed in literature to gain insights
into the transmission dynamics of Dengue in a community.
While Zeng and Velasco-H\'{e}rnandez \cite{Feng1997}
investigate the competitive exclusion principle
in a two-strain Dengue model, Chowell \emph{et al.} \cite{Choewll2007} estimates
the basic reproduction number for Dengue using spatial
epidemic data. In \cite{Nishiura2006} the author studies the spread
of Dengue thought statistical analysis, while in Tewa \emph{et al.}
\cite{Tewa2009} global asymptotic stability of the equilibrium
of a single-strain Dengue model is established. The control
of the mosquito by the introduction of a sterile insect technique
is analyzed in Thom\'{e} \emph{et al.} \cite{Thome2010}. More recently,
a study in disease persistence was made \cite{Medeiros2011} in Brazil and Otero
\emph{et al.} \cite{Otero2011} studied Dengue outbreaks.
All these studies were made with the aim of providing a better
understanding of the nature and dynamics of Dengue infection transmission.
In the next section a temporal mathematical model that explores the dynamics
between hosts (humans) and vectors (mosquitoes) is analyzed.


\section{A first mathematical approach to Dengue epidemics}
\label{sec:4:2}

The aim of this section is to present a mathematical model
to study the dynamic of the Dengue epidemics,
in order to minimize the investments in disease's control,
since financial resources are always scarce.
Quantitative methods are applied to the optimization
of investments in the control of the epidemiologic disease,
in order to obtain a maximum of benefits
from a fixed amount of financial resources.
The used model depends on the dynamic of the mosquito growing,
but also on the efforts of public management to motivate
the population to break the reproduction cycle of the mosquitoes
by avoiding the accumulation of still water in open-air recipients
and spraying potential zones of reproduction.


\subsection{Mathematical model}
\label{sec:4:2:1}

The Dengue epidemic model described in this paper is based
on the one proposed in \cite{Caetano2001}.
It has four state variables and two control variables as follows.

\begin{tabular}{lp{10cm}}
\multicolumn{2}{l}{\textbf{State Variables:}}\\
$x_{1}(t):$ & density of mosquitoes \\
$x_{2}(t):$ & density of mosquitoes carrying the virus \\
$x_{3}(t):$ & number of individuals with the disease \\
$x_{4}(t):$ & level of popular motivation to combat mosquitoes (goodwill) \\
\end{tabular}

\medskip

\begin{tabular}{lp{10cm}}
\multicolumn{2}{l}{\textbf{Control Variables:}}\\
$u_{1}(t):$ & investments in insecticides \\
$u_{2}(t):$ & investments in educational campaigns \\
\end{tabular}

\medskip

To describe the model it is also necessary to introduce some parameters:

\begin{tabular}{lp{12.5cm}}
\multicolumn{2}{l}{\textbf{Parameters:}}\\
$\alpha_{R}:$ & average reproduction rate of mosquitoes \\
$\alpha_{M}:$ & mortality rate of mosquitoes \\
$\beta:$ & probability of contact between non-carrier mosquitoes and infected individuals\\
$\eta:$ & rate of treatment of infected individuals \\
$\mu:$ & amplitude of seasonal oscillation in the reproduction rate of mosquitoes \\
$\rho:$ & probability of individuals becoming infected \\
$\theta:$ & fear factor, reflecting the increase in the population
willingness to take actions to combat the mosquitoes as a consequence
of the high prevalence of the disease in the specific social environment \\
$\tau:$ & forgetting rate for goodwill of the target population \\
$\varphi:$ & phase angle to adjust the peak season for mosquitoes \\
$\omega:$ & angular frequency of the mosquitoes proliferation cycle,
corresponding to a 52 weeks period\\
$P:$ & population in the risk area (usually normalized to yield $P=1$) \\
$\gamma_D:$ & the instantaneous costs due to the existence of infected individuals\\
$\gamma_S:$ & the costs of each operation of spraying insecticides\\
$\gamma_E:$ & the cost associated to the instructive campaigns \\
\end{tabular}

\medskip

The model consists in minimizing
\begin{equation}
\label{cost}
J\left[u_1(\cdot),u_2(\cdot)\right]
=\int_{0}^{t_f}\{\gamma_D x_{3}^{2}(t)
+\gamma_S u_{1}^{2}(t)+\gamma_E u_{2}^{2}(t)\}dt
\end{equation}
subject to the following four nonlinear time-varying state equations
\cite{Caetano2001}:

\begin{subequations}
\begin{align}
\dot{x}_{1}(t)
=\left[\alpha_{R}\left(1-\mu \sin(\omega t+\varphi)\right)
-\alpha_M - x_4(t)\right]x_1(t)-u_1(t), \label{x1}\\
\dot{x}_{2}(t)=\left[\alpha_{R}\left(1-\mu \sin(\omega t+\varphi)\right)
-\alpha_M - x_4(t)\right]x_2(t)+\beta \left[ x_1(t)-x_2(t)\right]x_3(t)-u_1(t), \label{x2}\\
\dot{x}_{3}(t)= -\eta x_3(t)+\rho x_2(t)\left[P-x_3(t)\right], \label{x3}\\
\dot{x}_{4}(t)= -\tau x_4(t)+\theta x_3(t)+u_2(t), \label{x4}
\end{align}
\end{subequations}
\noindent where $\dot{x}_i(t) = \frac{dx_{i}(t)}{dt}$, $i = 1, \ldots, 4$.

Equation (\ref{x1}) represents the variation of the mosquitoes density per unit time
to the natural cycle of reproduction and mortality ($\alpha_R$ and $\alpha_M$),
due to seasonal effects $\mu \sin(\omega t+\varphi)$ and to human interference $- x_4(t)$ and $u_1(t)$.
Equation (\ref{x2}) expresses the variation of the mosquitoes density carrying the virus $x_2$.
The term $\left[\alpha_{R}\left(1-\mu \sin(\omega t+\varphi)\right)-\alpha_M - x_4(t)\right]x_2(t)$
denotes the rate of the infected mosquitoes and $\beta \left[ x_1(t)-x_2(t)\right]x_3(t)$
represents the increase rate of the infected mosquitoes due to the possible contact between
the uninfected mosquitoes $x_1(t)-x_2(t)$ and infected individuals denoted by $x_3(t)$.
The dynamics of the infectious transmission is presented in equation (\ref{x3}).
The term $-\eta x_3(t)$ is related to the rate of cure and $\rho x_2(t)\left[P-x_3(t)\right]$
describes the rate at which new cases spring up. The factor $\left[P-x_3(t)\right]$
is the number of individuals in the area, that are not infected.
Equation (\ref{x4}) is a model for the level of popular motivation (or goodwill)
to combat the reproductive cycle of mosquitoes.
Over time, the level of people motivated will have changed.
As a consequence, it is necessary to invest in educational
campaigns designed to increase consciousness of the population under risk.
The expression $-\tau x_4(t)$ represents the decay of the people's motivation over time,
due to forgetfulness. The term $\theta x_3(t)$ describes the natural
sensibilities of the public due to increase in the prevalence of the disease.

The goal is to minimize the cost functional (\ref{cost}).
This functional includes social costs related to the existence of ill individuals
--- like absenteeism, hospital admission, treatments ---, $\gamma_D x_{3}^{2}(t)$,
the recourses needed for the spraying of insecticide operations,
$\gamma_S u_{1}^{2}(t)$, and for educational campaigns, $\gamma_E u_{2}^{2}(t)$.
The model for the social cost is based on the concept
of goodwill explored by Nerlove and Arrow \cite{Nerlove1962}.

Due to computational issues, the optimal control problem (\ref{cost})--(\ref{x4}),
that was written in the Lagrange form\index{Lagrange form},
was converted into an equivalent Mayer problem\index{Mayer form}.
Hence, using a standard procedure (\textrm{cf.},
Section~\ref{sec:1:2}) to rewrite the cost functional,
the state vector was augmented by an extra component $x_5$,
\begin{equation}
\label{newx5}
\dot{x}_5(t)=\gamma_D x_{3}^{2}(t)
+\gamma_S u_{1}^{2}(t)+\gamma_E u_{2}^{2}(t),
\end{equation}
leading to the following equivalent terminal cost problem:
\begin{equation}
\label{x5tf}
\begin{tabular}{ll}
minimize & $J[x_5(\cdot)]=x_5(t_f)$,
\end{tabular}
\end{equation}
with given $t_f$, subject to the control system (\ref{x1})--(\ref{x4}) and (\ref{newx5}).


\subsection{Numerical implementation and computational results}
\label{sec:4:2:2}

Two different implementations were considered. In a first approach, the OC problem
is solved by a specific Optimal Control package, \texttt{OC-ODE}\index{OC-ODE},
already described in Section~\ref{sec:2:3}. The OC problem considers
(\ref{x1})--(\ref{x5tf}) and the code is available in \cite{SofiaSITE}.

A second approach uses the nonlinear solver \texttt{Ipopt}\index{Ipopt} \cite{Ipopt},
also described in Section~\ref{sec:2:3}. In order to use this software,
it was necessary to discretize the problem. The Euler discretization scheme\index{Euler scheme}
was chosen (see Section~\ref{sec:2:1} for more details). The discretization
step length was $h=1/4$, because it is a good compromise
between precision and efficiency.
Thus, the optimal control problem was discretized into
the following nonlinear programming problem:
\begin{center}
\begin{tabular}{lll}
$min$ & \multicolumn{2}{l}{$x_5(N)$}\\
$s.t.$ &  $x_{1}(i+1)=$ & $x_{1}(i)+h\left\{\left[\alpha_{R}\left(1-\mu
\,\sin(\omega i+\varphi)\right)\right.\right.$\\
& & $\left.\left.-\alpha_M - x_4(i)\right]x_1(i)-u_1(i)\right\}$\\
& $x_{2}(i+1)=$ & $x_{2}(i)+h\left\{\left[\alpha_{R}\left(1-\mu
\,\sin(\omega i+\varphi)\right)-\alpha_M - x_4(i)\right]x_2(i)\right.$\\
& & $\left. + \beta \left[ x_1(i)-x_2(i)\right]x_3(i)-u_1(i)\right\}$\\
&  $x_{3}(i+1)=$ & $x_{3}(i)+h\{-\eta x_3(t)+\rho x_2(t)\left[P-x_3(t)\right]\}$\\
&  $x_{4}(i+1)=$ & $x_{4}(i)+h\{-\tau x_4(t)+\theta x_3(t)+u_2(t)\}$\\
& $x_{5}(i+1)=$ & $x_{5}(i)+h\{\gamma_D x_{3}^{2}(t)+\gamma_S u_{1}^{2}(t)+\gamma_E u_{2}^{2}(t)\}$,\\
\end{tabular}
\end{center}
\noindent where $i \in \{0,\ldots,N-1\}$.

The error tolerance value was $10^{-8}$ using the \texttt{Ipopt}\index{Ipopt} solver.
The discretized problem, after a presolve done by the software,
has 1455 variables, 1243 of which are nonlinear;
and 1039 constraints, 828 of which are nonlinear (see the \texttt{AMPL}\index{AMPL}
code for this problem in \cite{SofiaSITE}).

The simulations were carried out using the following normalized numerical values:
$\alpha_{R}=0.20$, $\alpha_{M}=0.18$, $\beta=0.3$,
$\eta=0.15$, $\mu=0.1$, $\rho=0.1$, $\theta=0.05$, $\tau=0.1$, $\varphi=0$,
$\omega=2\pi/52$, $P=1.0$, $\gamma_D=1.0$, $\gamma_S=0.4$, $\gamma_E=0.8$,
$x_1(0)=1.0$, $x_2(0)=0.12$, $x_3(0)=0.004$, and $x_4(0)=0.05$.
These values are available on the paper \cite{Caetano2001}
and were adopted here in order to compare the obtained results
with those of \cite{Caetano2001}. It was considered $t_f=52$ weeks as final time.

The results for the state and control variables are shown in Figures~\ref{plotx1} to \ref{plotu2}.
Each figure has three graphics: \texttt{OC-ODE}\index{OC-ODE}
and \texttt{Ipopt}\index{Ipopt}, which correspond to the solutions obtained
by the solvers used, respectively; and MSM, corresponding to the
Multiple Shooting Method\index{Multiple shooting method} \cite{Pesh1989,Trelat2005}
that was used by the authors of the paper \cite{Caetano2001}. It is important to salient that,
at the time of the initial paper \cite{Caetano2001}, the authors had not the same computational resources
that exist nowadays. The results with \texttt{OC-ODE} and \texttt{Ipopt} are better since the cost to fight
the Dengue disease and the number of infected individuals are smaller.

Figures~\ref{plotx1} and \ref{plotx2} show the density of mosquitoes.
It is possible to see that in this new solution, with the same number of mosquitoes
as in the previous solution \cite{Caetano2001},
the number of infected mosquitoes falls dramatically.

\begin{figure}[ptbh]
\center
\begin{minipage}[t]{0.48\linewidth}
\center
\includegraphics[scale=0.4]{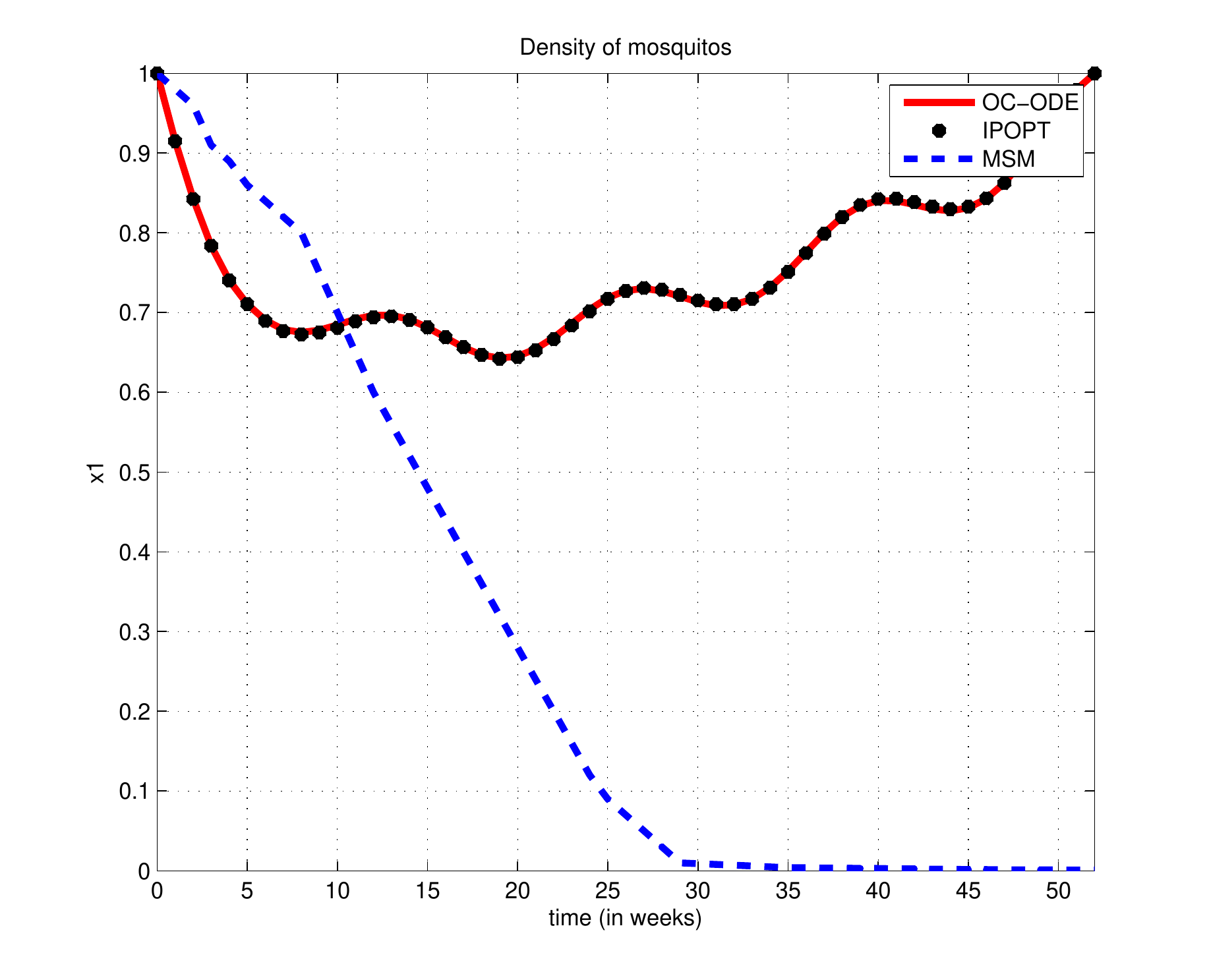}
{\caption{\label{plotx1}  Density of mosquitoes}}
\end{minipage}\hspace*{\fill}
\begin{minipage}[t]{0.48\linewidth}
\center
\includegraphics[scale=0.4]{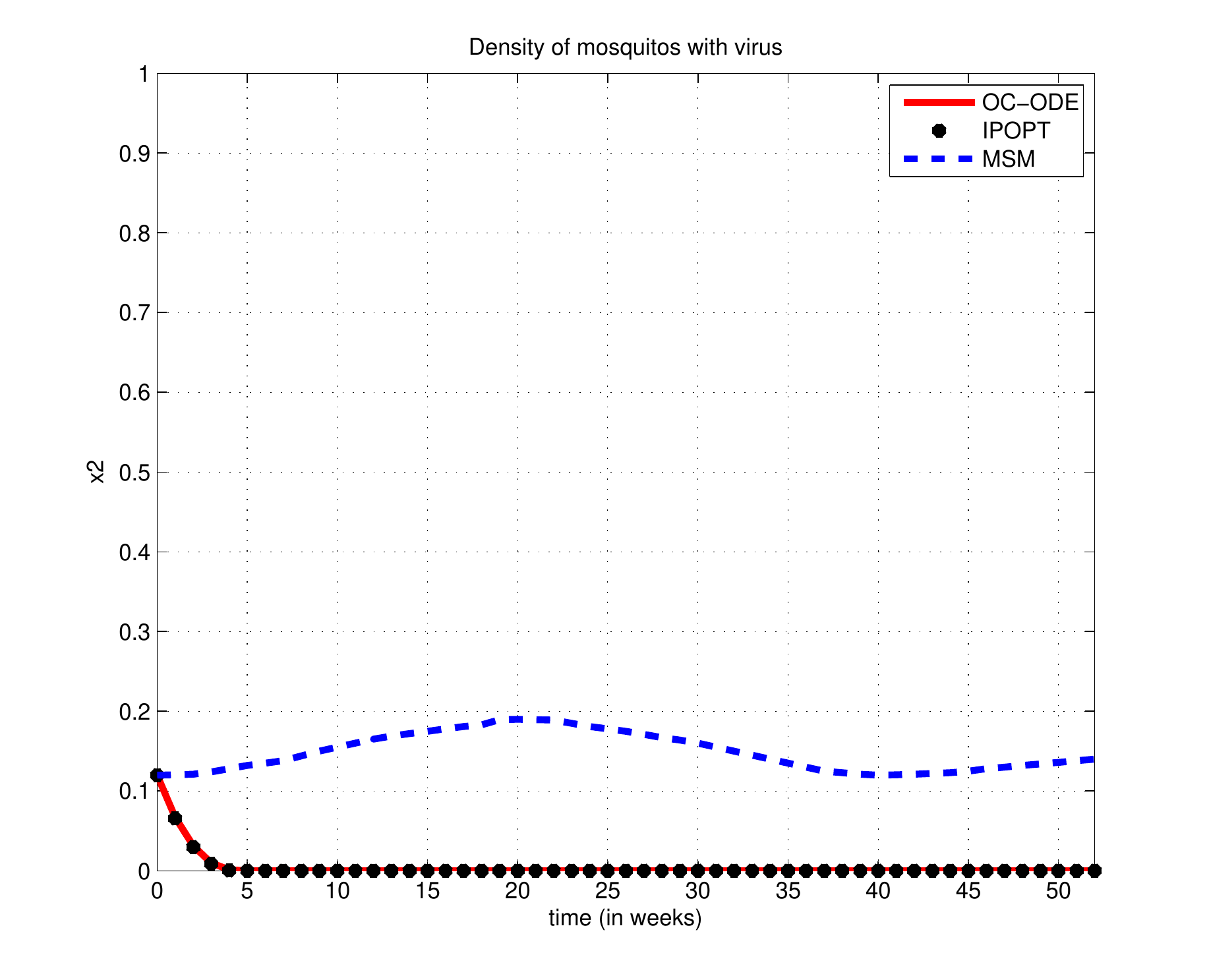}
{\caption{\label{plotx2} \small Infected mosquitoes}} \label{plot}
\end{minipage}
\end{figure}

Figures~\ref{plotx3} and \ref{plotx4} report to the population in the risk area.
Our solution shows that the number of ill people decreases quickly. This could explain
why the motivation level to fight the mosquito is lower when compared
to the previous solution proposed in \cite{Caetano2001}.

\begin{figure}[ptbh]
\center
\begin{minipage}[t]{0.48\linewidth}
\center
\includegraphics[scale=0.4]{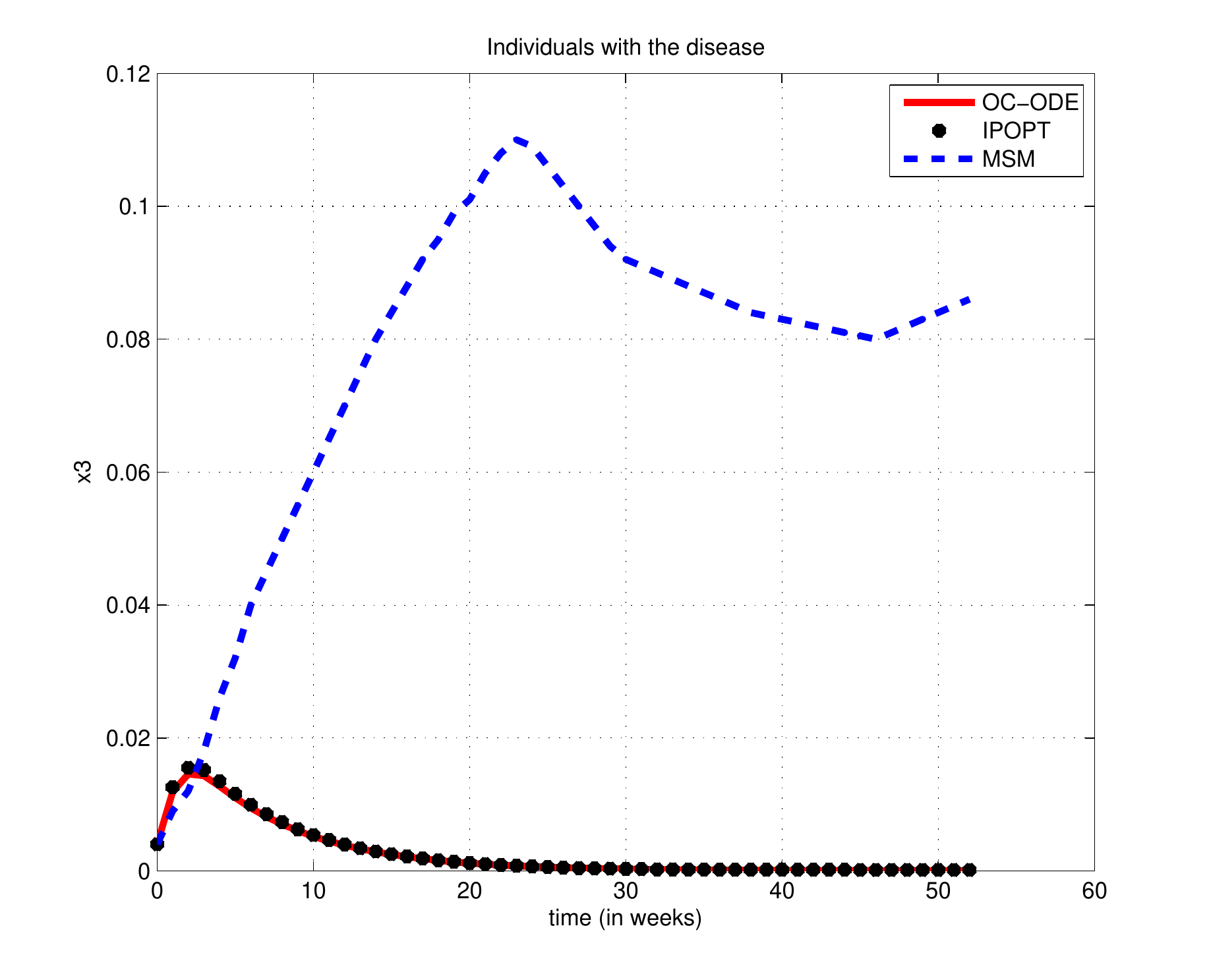}
{\caption{\label{plotx3} Infected individuals}}
\end{minipage}\hspace*{\fill}
\begin{minipage}[t]{0.48\linewidth}
\center
\includegraphics[scale=0.4]{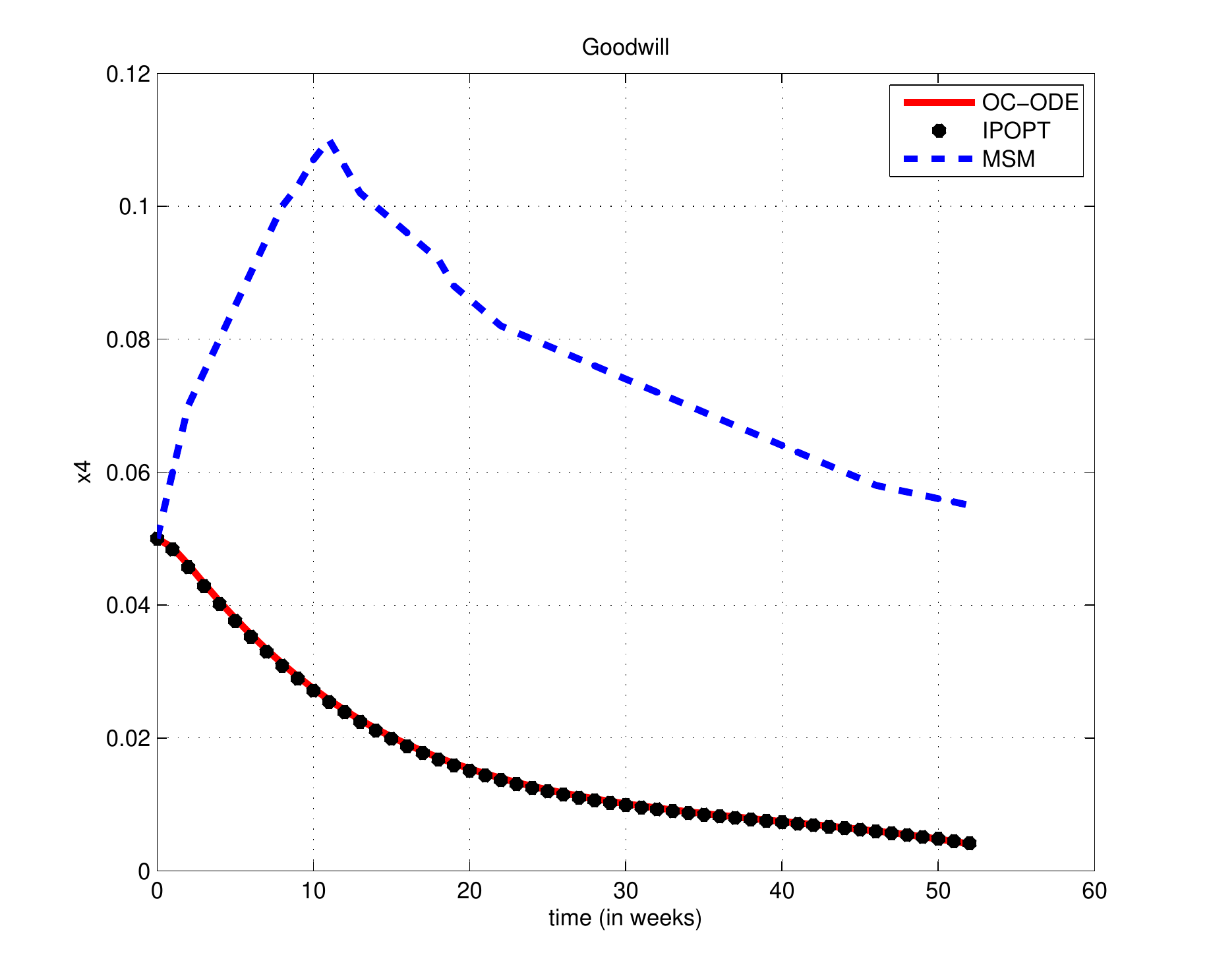}
{\caption{\label{plotx4} \small Level of popular motivation}} \label{pplot}
\end{minipage}
\end{figure}

Figure~\ref{plotx5} shows the accumulated cost. It is clear that almost all year
the cost is lower when compared with MSM \cite{Caetano2001}.
This lower cost level is a consequence of infected mosquitos
and infected individuals both falling down under our approach.

\begin{figure}[ptbh]
\center
\includegraphics[scale=0.4]{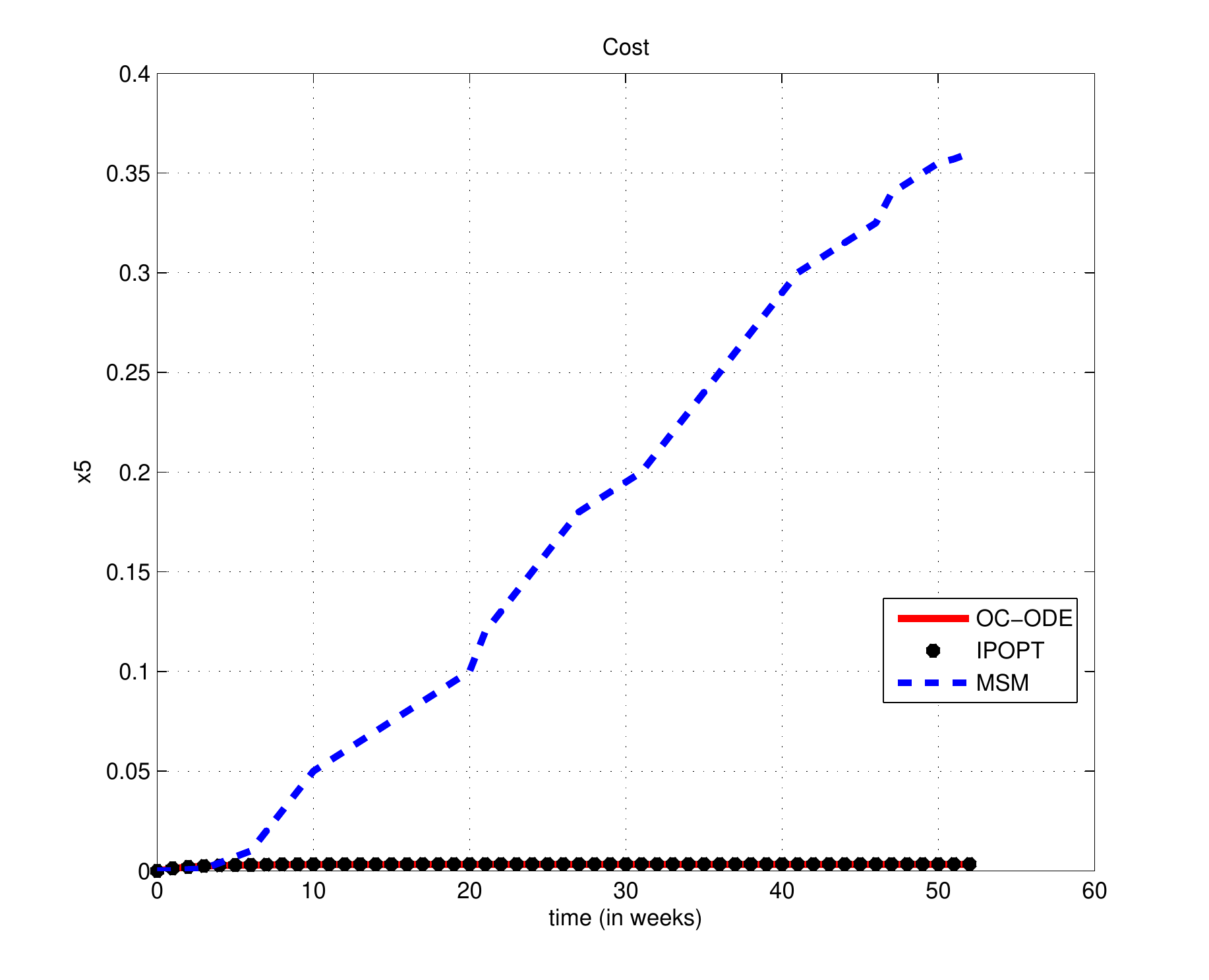}
\caption{Cost \label{plotx5}}
\end{figure}

Figures~\ref{plotu1} and \ref{plotu2} are related to the controls:
educational campaigns and application of insecticides. It is possible
to see that the new functions for the control variables are less expensive.

\begin{figure}[ptbh]
\center
  \includegraphics [scale=0.4]{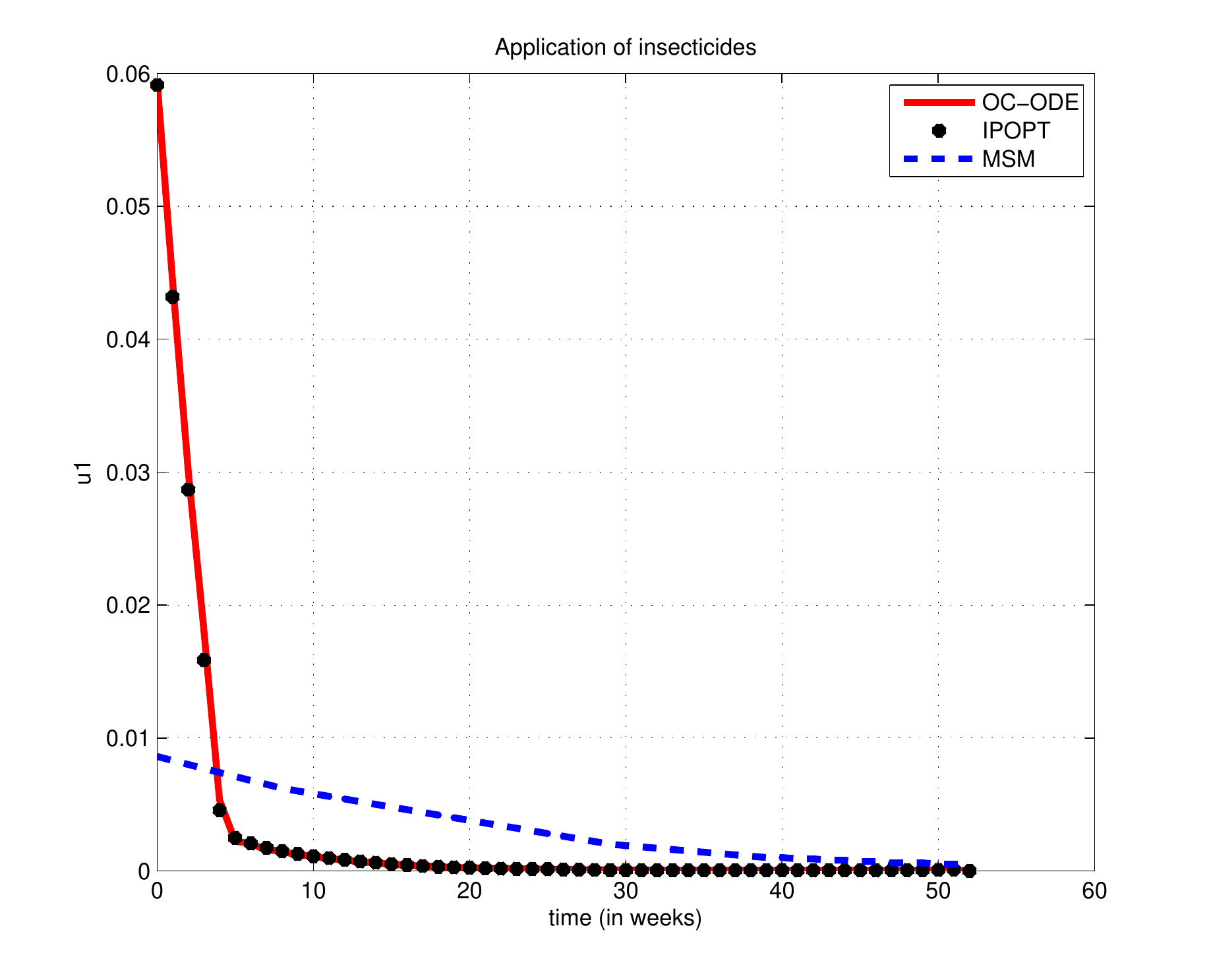}\\
   {\caption{\label{plotu1}Application of insecticide}}
\end{figure}

\begin{figure}[ptbh]
\center
  \includegraphics [scale=0.4]{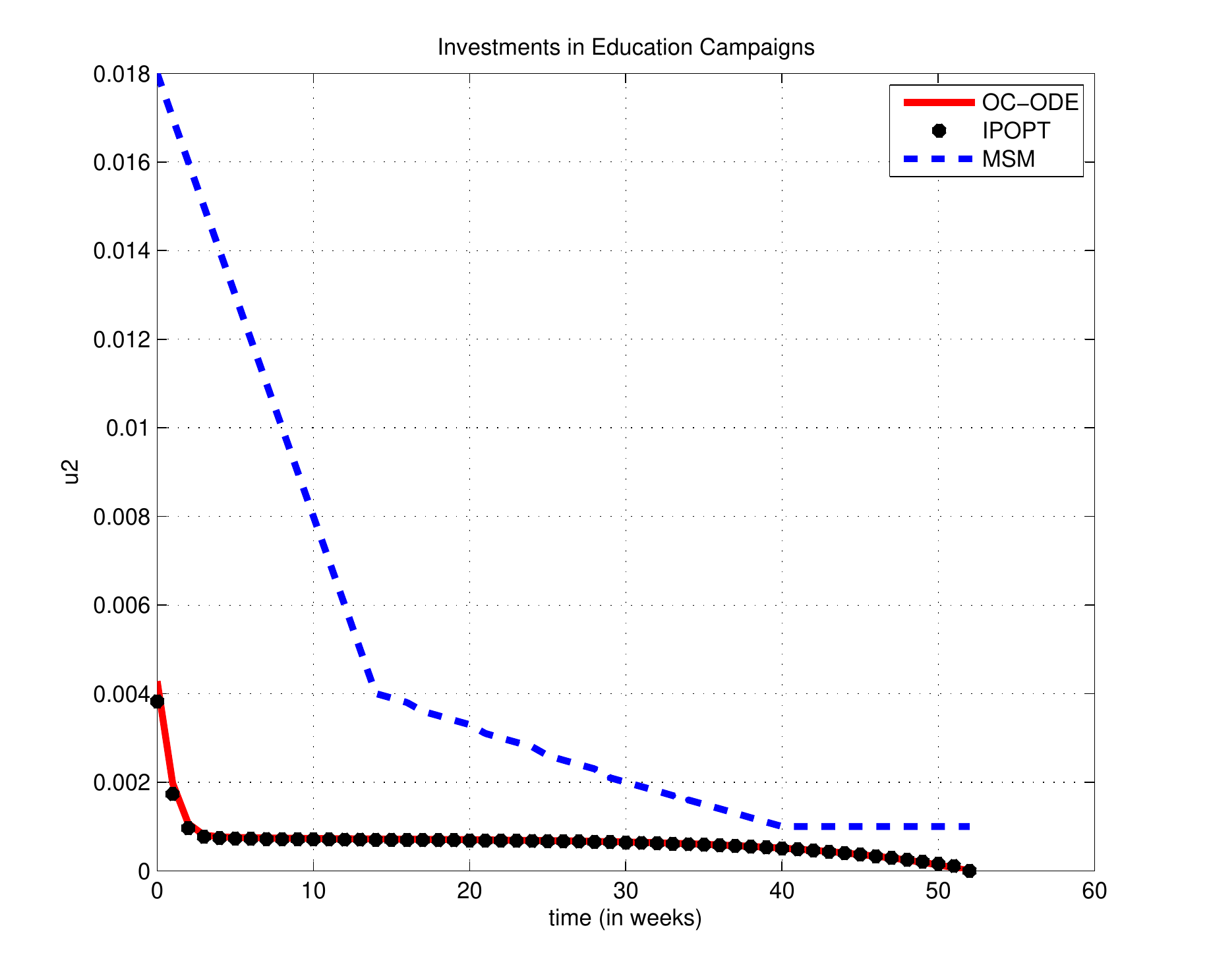}\\
   {\caption{\label{plotu2}Educational campaigns}}
\end{figure}

Despite the different philosophies of the \texttt{OC-ODE}\index{OC-ODE}
(a specific control software) and \texttt{Ipopt}\index{Ipopt}
(a standard nonlinear optimization solver),
the solution reached is similar. This fact enforces the robustness
of the obtained results. It is important to mention
that the problem under study is a difficult one.
Other nonlinear packages were tested, and they could not reach a solution
--- some crashed at the middle or some bad scaling issues were observed.

Until some years ago, due to computational limitations, most of
the models were run using codes made by the authors themselves as \cite{Caetano2001}.
Nowadays, one can choose between several proper software packages "out of the box",
that already take into account specific features of stiff problems,
scaling problems, etc. With this work it is possible
to realize that "old" problems can again be taken into account
and be better analyzed with new technology and approaches,
with the goal of finding global optimal solutions, instead of local ones.

For this purpose, and at an initial research stage, it is important
to understand if different kinds of discretization for an OC problem
have influence on the problem resolution.


\section{Using different discretization schemes}
\label{sec:4:3}

This section aims to study the costs of different discretization processes,
in terms of time performance, number of variables and iterations used.
For the purpose of this analysis two discretization schemes,
Euler\index{Euler scheme} and second order Runge-Kutta's\index{Runge Kutta scheme}
scheme \cite{Betts2001} (\textrm{cf.} Section~\ref{sec:2:1}),
are considered to solve the problem described in the previous section.

This discretization process transforms the Dengue epidemics problem into
a standard nonlinear optimization problem, with an objective function
and a set of nonlinear constraints. This NLP problem was codified,
for both discretization schemes, in the \texttt{AMPL} modeling language
\cite{AMPL}\index{AMPL} and can be checked in \cite{SofiaSITE}.

Two nonlinear solvers with distinct features were selected to solve the NLP problem:
the \texttt{Knitro}\index{Knitro} (IP method) and the \texttt{Snopt}\index{Snopt} (SQP method).
The NEOS Server \cite{NEOS}\index{NEOS platform} platform was used as interface with both solvers.
The \texttt{Ipopt} \index{Ipopt} (IP method), used in the previous section, was the first choice
for our research. However, at the time of this investigation, the NEOS platform was moved
to another Center of Research and some software packages were unavailable for long periods of time.
So, we had to choose another Interior Point robust software.

Table~\ref{cap5_resultados} reports the results for both solvers, for each discretization
method using three different discretization steps ($h=0.5, 0.25,  0.125$), rising twelve
numerical experiences. The columns \# \emph{var.} and \#\emph{ const}. mean the number
of variables and constraints, respectively. The next columns refer to the performance measures
--- number of iterations and total CPU time in seconds (time for solving the problem,
for evaluate the objective and the constraints functions and for input/output).
As the computational experiments were made in the NEOS server platform\index{NEOS platform},
the selected machine to run the program remains unknown as well as its technical specifications.

\begin{table}
\tiny{
\begin{tabular}{ccc}
\hline
& Euler's method & Runge-Kutta's method\\
\hline
\hline
Knitro & \begin{tabular}{|c|c|c|c|c|} \hline
h & \# var. & \# const. & \# iter. &  time (sec.)\\ \hline
0.5 & 727 & 519 & 113 & 2.090 \\
0.25 & 1455 & 1039 & 68 & 2.210 \\
0.125 &  2911 & 2079 & 85 & 7.240 \\ \hline
\end{tabular}
& \begin{tabular}{|c|c|c|c|c|} \hline
h & \# var. & \# const. & \# iter. &  time (sec.)\\ \hline
0.5 & 728 & 520 & 64 & 1.980 \\
0.25 & 1456 & 1040 & 82 & 5.550 \\
0.125 & 2912 & 2080 & 70 & 9.740 \\ \hline
\end{tabular}
\\ \hline

Snopt & \begin{tabular}{|c|c|c|c|c|} \hline
h & \# var. & \# const. & \# iter. & time (sec.)\\ \hline
0.5 & 727 & 519 & 175 & 4.07 \\
0.25 &  1455 & 1039 & 253 & 19.2 \\
0.125 & 2911 & 2079 & 252 & 105.4 \\ \hline
\end{tabular}
& \begin{tabular}{|c|c|c|c|c|} \hline
h & \# var. & \# const. & \# iter. &  time (sec.)\\ \hline
0.5 & 728 & 520 &  223 & 10.52 \\
0.25 & 1456 & 1040 & 219 &  39.7 \\
0.125 & 2912 & 2080  & 420 &  406.67\\ \hline
\end{tabular}
\\ \hline
\end{tabular}
\caption{Numerical results}
\label{cap5_resultados}
}
\end{table}

The optimal cost reached was $\approx 3\times 10^{-3}$ for all tests.
Comparing the general behavior of the solvers one can conclude that the IP based method
(\texttt{Knitro})\index{Knitro} presents much better performance than the SQP method
(\texttt{Snopt})\index{Snopt} in terms of the measures used. Regarding the \texttt{Knitro}
results, one realizes that the Euler's\index{Euler scheme} discretization scheme
has better times for $h=0.25$ and $h=0.125$ and similar time for $h=0.5$,
when compared to Runge-Kutta's method\index{Runge Kutta scheme}. Another obvious finding,
for both solvers, is that the CPU time increases as far as the problem dimension increases
(number of variables and constraints). With respect to the number of iterations, \texttt{Snopt}
presents more iterations as the problem dimension increases. However this conclusion cannot
be taken for \texttt{Knitro} --- in fact, there does not exist a relation between
the problem dimension and the number of iterations. The best version tested was \texttt{Knitro}
using Runge-Kutta with $h=0.5$ (best CPU time and fewer iterations), and the second one was
\texttt{Knitro} with  Euler's method using $h=0.25$. An important evidence of this numerical
experience is that it is not worth the reduction of the discretization step size
because no significative advantages are obtained.


\section{Conclusions}
\label{sec:4:4}

At this moment, as a result of major demographic changes,
rapid urbanization on a massive scale, global travel and environmental change,
the world faces enormous future challenges from emerging infectious diseases.
Dengue illustrates well these challenges \cite{Who}.
In this work we investigated an optimal control model for Dengue epidemics
proposed in \cite{Caetano2001}, that includes the mosquitoes dynamics,
the effect of educational campaigns. The cost functional reflects
a compromise between financial spending in insecticides and educational campaigns
and the population health. For comparison reasons, the same choice of data/parameters
in \cite{Caetano2001} was considered.

The results obtained from \texttt{OC-ODE}\index{OC-ODE}
and \texttt{Ipopt}\index{Ipopt} are similar, improving the ones
previously reported in \cite{Caetano2001} (\textrm{cf.} Section~\ref{sec:4:2}).
Indeed, the obtained control policy in this work presents an important progress with respect
to the previous best policy: the percentage of infected mosquitoes
vanishes just after four weeks, while mosquitoes are completely
eradicated after 30 weeks (Figures~\ref{plotx1} and \ref{plotx2});
the number of infected individuals begin to decrease after
four weeks while with the previous policy this only happens after 23 weeks
(Figure~\ref{plotx3}). Despite the fact that our results are better,
they are accomplished with a much smaller cost with insecticides
and educational campaigns (Figure~\ref{plotx5}).
The general improvement, which explains why the results are so successful,
rely on an effective control policy of insecticides.
The proposed strategy for insecticide application seems to explain
the discrepancies between the results
here obtained and the best policy of \cite{Caetano2001}.
Our results show that applying insecticides in the first four weeks
yields a substantial reduction in the cost of fighting Dengue,
in terms of the functional proposed in \cite{Caetano2001}.
The main conclusion is that health authorities should pay
attention to the epidemic from the very beginning:
effective control decisions in the first four weeks have
a decisive role in the Dengue battle, and
population and governments will both profit from it.

We successfully solved an OC problem by direct methods\index{Direct method}
using nonlinear optimization software based on IP and SQP approaches.
The problem was discretized through Euler and Runge-Kutta schemes.
The implementation efforts of higher order discretization methods bring no advantages.
The reduction of the discretization step and consequently the increase of the number
of variables and constraints do not improve the performance with respect
to the CPU time and to the number of iterations. We can point out the robustness
of both solvers in spite of the dimension problem increase. The conclusions drawn
in Section~\ref{sec:4:3} were helpful for decision-making
future processes of discretization all over the work.

As future work we intend to analyze how different
parameters/weights associated to the variables in the objective function
can influence the spread of the disease.
This chapter was based on work available in the peer reviewed journal
\cite{Sofia2010b} and the peer reviewed conference proceedings \cite{Sofia2009}.

\clearpage{\thispagestyle{empty}\cleardoublepage}


\chapter{An ODE SEIR+ASEI model with an insecticide control}
\label{chp5}

\begin{flushright}
\begin{minipage}[r]{9cm}

\bigskip
\small {\emph{A model for the Dengue disease transmission is presented.
It consists of eight mutually-exclusive compartments represen\-ting
the human and vector dynamics. It also includes a control parameter,
adulticide spray, as a measure to fight the disease. The model presents
three possible equilibria: two Disease Free Equilibria (DFE)
and an Endemic Equilibrium (EE). It has been proved that a DFE
is locally asymptotically stable, whenever a certain epidemiological threshold,
known as the basic reproduction number, is less than one. In this work we try
to understand which is the best way to apply the control in order to effectively
reduce the number of infected humans and mosquitoes.
A case study, using outbreak data in 2009 in Cape Verde, is reported.}}

\bigskip

\hrule
\end{minipage}
\end{flushright}

\bigskip
\bigskip

\onehalfspacing

In Chapter~\ref{chp4} the Dengue epidemic was studied, mostly centered in people,
specially in the goodwill of the individuals and spraying campaigns.
However, the virus transmission scheme was overlooked: it was only considered
two compartments for people and two compartments for adult mosquitoes.
Here, the aim is to deepen the relationship between human and mosquitoes,
creating a better framework to explain the development and transmission of the disease.


\section{The SEIR+ASEI model}
\label{sec:5:1}

The mathematical model is based on \cite{Dumont2010, Dumont2008},
that describes the chikungunya disease transmitted by \emph{Aedes albopictus}.

The notation used in our mathematical model includes
four epidemiological states for humans:

\begin{quote}
\begin{tabular}{ll}
$S_h(t):$ & susceptible \\
$E_h(t):$ & exposed \\
$I_h(t):$ & infected \\
$R_h(t):$ & resistant
\end{tabular}
\end{quote}

It is assumed that the total human population $(N_h)$ is constant,
so, $N_h=S_h+E_h+I_h+R_h$.
There are also four other state variables related to the female
mosquitoes (the male mosquitoes are not considered in this study
because they do not bite humans and consequently they do not
influence the dynamics of the disease):

\begin{quote}
\begin{tabular}{ll}
$A_m(t):$& aquatic phase \\
$S_m(t):$& susceptible \\
$E_m(t):$& exposed \\
$I_m(t):$& infected \\
\end{tabular}
\end{quote}

Similarly, it is assumed that the total adult mosquito population is constant,
which means $N_m=S_m+E_m+I_m$. In this way, we put our model more complex and
reliable to the reality of Dengue epidemics. For this study we introduced a control variable:

\begin{quote}
\begin{tabular}{ll}
$c(t):$& level of insecticide campaigns\\
\end{tabular}
\end{quote}

The control variable, $c(t)$, varies from 0 to 1. However, the model does not fit completely the reality.
Epidemiologist and policy makers need to be aware of both strengths and weakness
of the epidemiological modeling approach. An epidemiological model is always a simplification of reality.
So, some assumptions were made to built this model:

\begin{itemize}
\item the total human population ($N_h$) is constant;
\item there is no immigration of infected individuals into the human
population;
\item the population is homogeneous, which means that
every individual of a compartment is homogenously mixed with the
other individuals;
\item the coefficient of transmission of the
disease is fixed and does not vary seasonally;
\item both human
and mosquitoes are assumed to be born susceptible, \emph{i.e.}, there is no
natural protection;
\item for the mosquito there is no resistant
phase, due to its short lifetime.
\end{itemize}

To completely describe the model it is necessary to use parameters, which are:

\begin{quote}
\begin{tabular}{ll}
$N_h:$ & total population \\
$B:$ & average daily biting (per day)\\
$\beta_{mh}:$ & transmission probability from $I_m$ (per bite) \\
$\beta_{hm}:$ & transmission probability from $I_h$ (per bite) \\
$1/\mu_{h}:$ & average lifespan of humans (in days) \\
$1/\eta_{h}:$ & mean viremic period (in days)\\
$1/\mu_{m}:$ & average lifespan of adult mosquitoes (in days) \\
$\varphi:$ & number of eggs at each deposit per capita (per day) \\
$\mu_{A}:$ & natural mortality of larvae (per day) \\
$\eta_{A}:$ & maturation rate from larvae to adult (per day) \\
$1/\eta_{m}:$ & extrinsic incubation period (in days)  \\
$1/\nu_{h}:$ & intrinsic incubation period (in days) \\
$m:$ & female mosquitoes per human \\
$k:$ & number of larvae per human \\
$K:$ & maximal capacity of larvae
\end{tabular}
\end{quote}

For notation simplicity, the independent variable $t$ will be omitted when writing
the dependent variables, example given, will be written $S_h$ instead of $S_h (t)$.
The Dengue epidemic can be modelled by the following nonlinear
time-varying state equations:

Human Population
\begin{equation}
\label{odehuman}
\begin{tabular}{l}
$\left\{
\begin{array}{l}
\displaystyle\frac{dS_h}{dt}
= \mu_h N_h - (B\beta_{mh}\frac{I_m}{N_h}+\mu_h) S_h\\
\displaystyle\frac{dE_h}{dt}
= B\beta_{mh}\frac{I_m}{N_h}S_h - (\nu_h + \mu_h )E_h\\
\displaystyle\frac{dI_h}{dt}
= \nu_h E_h -(\eta_h  +\mu_h) I_h\\
\displaystyle\frac{dR_h}{dt}
= \eta_h I_h - \mu_h R_h
\end{array}
\right. $\\
\end{tabular}
\end{equation}
and vector population
\begin{equation}
\label{odevector}
\begin{tabular}{l}
$
\left\{
\begin{array}{l}
\displaystyle\frac{dA_m}{dt}
= \varphi(1-\frac{A_m}{kN_h})(S_m+E_m+I_m)-(\eta_A+\mu_A) A_m\\
\displaystyle\frac{dS_m}{dt}
= \eta_A A_m-(B \beta_{hm}\frac{I_h}{N_h}+\mu_m) S_m-c S_m\\
\displaystyle\frac{dE_m}{dt}
= B \beta_{hm}\frac{I_h}{N_h}S_m-(\mu_m + \eta_m) E_m-c E_m\\
\displaystyle\frac{dI_m}{dt}
= \eta_m E_m -\mu_m I_m - c I_m\\
\end{array}
\right. $
\end{tabular}
\end{equation}
with the initial conditions
\begin{equation}
\label{initial}
\begin{tabular}{llll}
$S_h(0)=S_{h0},$ & $E_h(0)=E_{h0},$ & $I_h(0)=I_{h0},$ &
$R_h(0)=R_{h0},$ \\
$A_m(0)=A_{m0},$ & $S_{m}(0)=S_{m0},$ &
$E_m(0)=E_{m0},$ & $I_m(0)=I_{m0}.$
\end{tabular}
\end{equation}

Figure~\ref{cap5_model} shows the relation between
human and mosquito and the corresponding parameters.

\begin{figure}[ptbh]
\center
\includegraphics [scale=0.45]{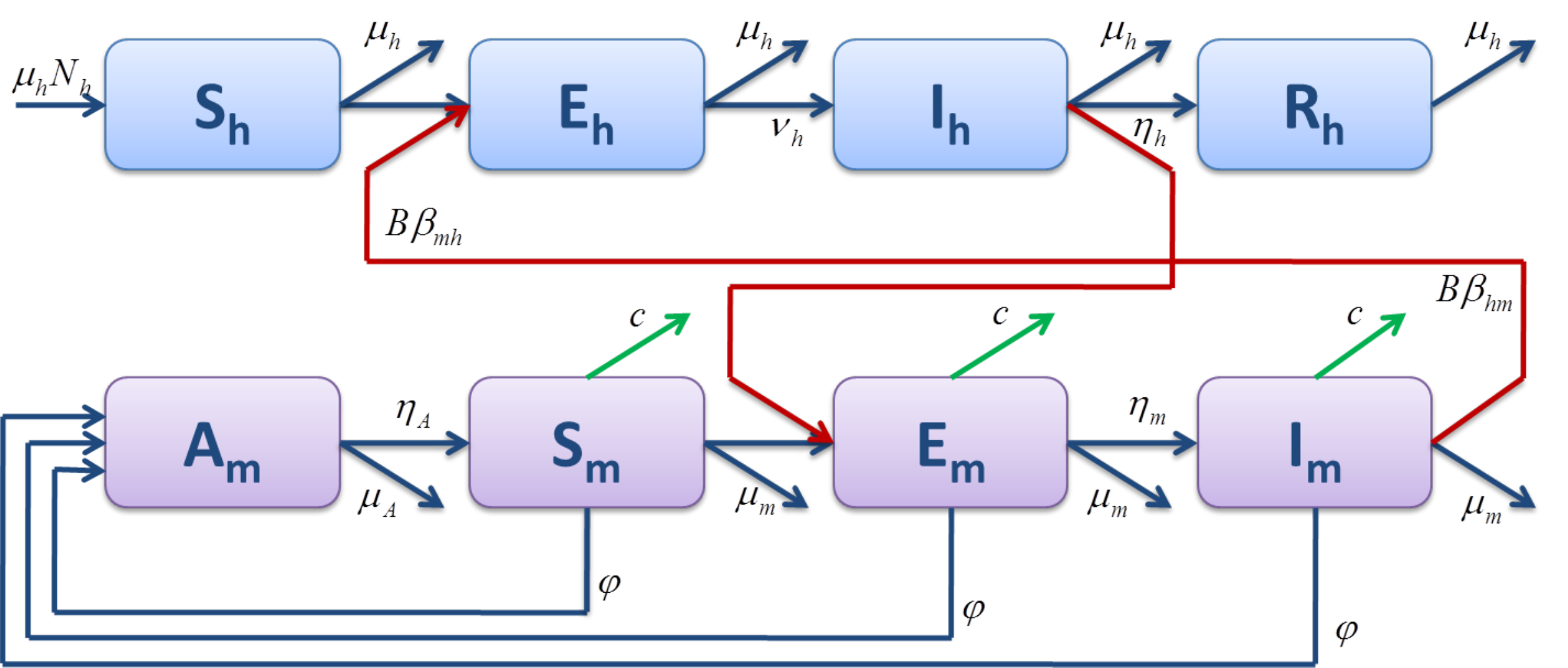}
{\caption{\label{cap5_model}  Epidemiological model SEIR+ASEI}}
\end{figure}

Notice that the equation related to the aquatic phase does not have
the control variable $c$, because the adulticide does not produce
effects in this stage of mosquito life. To fight
the larval phase it would be necessary to use larvicide.
This treatment should be long-lasting and have World Health Organization
clearance for use in drinking water. As we want to study only
a short period of time, this type of treatment has not been considered here.

With the condition $S_h+E_h+I_h+R_h=N_h$, one can, in the example given,
use $R_h=N_h-S_h-E_h-I_h$ and consider an equivalent system for human
population without considering the $R_h$ differential equation.

For the previous set of differential equations is now analyzed the
equilibrium points of the system and is determined the threshold phenomena.

\bigskip


\subsection{Basic reproduction number, equilibrium points and stability}
\label{sec:5:1:1}
\index{Equilibrium point}

Let the set
\begin{center}
\begin{tabular}{l}
\small
$\Omega=\{(S_h,E_h,I_h,A_m,S_m,E_m,I_m)\in \mathbb{R}^{7}_{+}:$\\
$ S_h+E_h+I_h\leq N_h, A_m\leq k N_h, S_m+E_m+I_m\leq m N_h \}$
\end{tabular}
\end{center}
be the region of biological interest.

\begin{proposition}
$\Omega$ is positively invariant
under the flow induced by the differential system
\eqref{odehuman}-\eqref{odevector}.
\end{proposition}

\medskip

\begin{proof}
System \eqref{odehuman}--\eqref{odevector}
can be rewritten in the following way:
\begin{equation}
\label{chap5_odematrix}
\displaystyle\frac{dX}{dt}=M(X)X+F
\end{equation}
\noindent where $X=\left(S_h,E_h,I_h,A_m,S_m,E_m,I_m\right)$,
$F=\left(\mu_h N_h,0,0,0,0,0,0\right)^{T}$ and
\small
\begin{center}
\begin{tabular}{c}
$ M(X)=\left( \begin{array}{ccccccc}
- B\beta_{mh}\frac{I_m}{N_h}-\mu_h & 0 & 0 & 0 & 0& 0& 0\\
B\beta_{mh}\frac{Im}{N_h} & -\nu_h - \mu_h & 0 &0  &0 &0 & 0\\
 0 & \nu_h & -\eta_h  -\mu_h & 0 &0 &0 &0 \\
 0& 0 &0  & \Upsilon & \mu_b & \mu_b & \mu_b \\
0 & 0 &0  & \eta_A & \Theta &0 &0 \\
 0 & 0 & 0 & 0 &B \beta_{hm}\frac{I_h}{N_h} & \Psi & 0\\
 0& 0 &0  &0  & 0  &\eta_m & -\mu_m -c\\
\end{array} \right)$,
\end{tabular}
\end{center}
\normalsize
\noindent with $\Upsilon=-\mu_b\frac{S_m+E_m+I_m}{K}-\mu_m - \eta_m$,
$\Theta=-B \beta_{hm}\frac{I_h}{N_h}-\mu_m -c$ and $\Psi=-\mu_m - \eta_m -c$.

As $M(X)$ has all off-diagonal entries nonnegative, $M(X)$ is a Metzler matrix.

Using the fact that $F\geq 0$, the system
\eqref{chap5_odematrix} is positively invariant in $\mathbb{R}^{7}_{+}$
\cite{Abate2009}, which means that any trajectory of the system
starting from an initial state in the positive orthant
$\mathbb{R}^{7}_{+}$ remains forever in $\mathbb{R}^{7}_{+}$.
\end{proof}

\medskip

\begin{theorem}
\label{thm5:thm1}
Let $\Omega$ be defined as above. Consider also
\begin{center}
\begin{tabular}{l}
$\mathcal{M}=-\left(c (\eta_A + \mu_A)
+ \mu_A \mu_m + \eta_A (-\varphi + \mu_m)\right)$.
\end{tabular}
\end{center}
The system \eqref{odehuman}--\eqref{odevector}
admits at most two Disease Free Equilibrium points:\index{Disease Free Equilibrium}
\begin{itemize}
\item if $\mathcal{M}\leq 0$, there is a Disease Free Equilibrium (DFE)\index{Disease Free Equilibrium},
$$
E_{1}^{*}=\left(N_h,0,0,0,0,0,0\right),
$$
called Trivial Equilibrium;
\item if $\mathcal{M}> 0$, there is a Biologically Realistic Disease Free Equilibrium (BRDFE),
$$
E_{2}^{*}=\left(N_h,0,0,\frac{k N_h \mathcal{M}}{\eta_A\varphi},
\frac{k N_h \mathcal{M}}{\varphi \mu_m},0,0\right).
$$
\end{itemize}
\end{theorem}

\begin{proof}
The equilibrium points\index{Equilibrium point} are reached
when the following equations hold:
\begin{equation}
\label{chap5_equilibrium}
\begin{tabular}{l}
$\left\{
\begin{array}{l}
\frac{dS_h}{dt} = 0\\
\frac{dE_h}{dt} = 0\\
\frac{dI_h}{dt} = 0\\
\frac{dR_h}{dt} = 0\\
\frac{dA_m}{dt} = 0\\
\frac{dS_m}{dt} = 0\\
\frac{dE_m}{dt} = 0\\
\frac{dI_m}{dt} = 0\\
\end{array}
\right. $\\
\end{tabular}
\end{equation}

Using the \texttt{Mathematica}\index{Mathematica} software to solve
the system \eqref{chap5_equilibrium}, we obtained four solutions.

The first one is known as the \emph{Trivial Equilibrium},
since the mosquitoes do not exist, so there is no disease:
\begin{center}
\begin{tabular}{l}
$E_{1}^{*}=\left(N_h,0,0,0,0,0,0\right)$
\end{tabular}
\end{center}

In the second one, mosquitoes and humans interact, but there
is only one outbreak of the disease, \textrm{i.e.},
over time the disease goes away without killing all the mosquitoes.
We have called this equilibrium point a
\emph{Biologically Realistic Disease Free Equilibrium} (BRDFE),\index{Disease Free Equilibrium}
since it is a more reasonable situation to find in nature than the previous one:

\begin{center}
\begin{tabular}{lr}
$E_{2}^{*}=$&$\left(N_h,0,0,\displaystyle\frac{k Nh (-\left(c (\eta_A + \mu_A)
+ \mu_A \mu_m + \eta_A (-\mu_b + \mu_m)\right))}{\eta_A\mu_b},\right.$\\
&$\displaystyle\left.\frac{k Nh (-\left(c (\eta_A + \mu_A) + \mu_A \mu_m
+ \eta_A (-\mu_b + \mu_m)\right))}{\mu_b \mu_m},0,0\right)$,
\end{tabular}
\end{center}
which is equivalent to
$E_{2}^{*}=\left(N_h,0,0,\frac{k Nh \mathcal{M}}{\eta_A\mu_b},
\frac{k Nh \mathcal{M}}{\mu_b \mu_m},0,0\right)$.
This is biologically interesting only if $\mathcal{M}$ is greater than 0.

The third solution corresponds to a situation where humans and mosquitoes
live together but the disease persists in both populations, which means
that it is not a DFE\index{Disease Free Equilibrium}. This equilibrium
will be explained after (see Theorem~\ref{thm5:thm3}).
Thus the disease is not anymore an epidemic episode, but transforms
into a endemic one. With some algebraic manipulations we obtained the following point:

$E_{3}^{*}=\left(S_h^*,E_h^*,I_h^*,A_m^*,S_m^*,E_m^*,I_m^*\right)$
where,

$S_h^*=
N_h-\displaystyle\frac{(\mu_h+\nu_h)(\mu_h+\eta_h)}{\mu_h\nu_h}I_h^{*}$,

$E_h^*=\displaystyle\frac{\mu_h+\eta_h}{\nu_h}I^{*}_{h}$,

$I_h^*=N_h \mu_h (-B^2 k \beta_{hm} \beta_{mh}\nu_h\eta_m
     \mathcal{M} + \mu_b \mu_m^2(\eta_m + \mu_m)(\mu_h + \
\nu_h)(\mu_h + \eta_h) +
    c^2\mu_b(\eta_h + \mu_h)(\mu_h + \nu_h)(c + \eta_m +
       3 \mu_m) +
    c\mu_b\mu_m(\mu_h + \nu_h)(\mu_h(3 \mu_m +
          2) + \eta_h(2 \eta_m +
          3 \mu_m)))/(B \beta_{hm} (\eta_h + \mu_h) (-\mu_b \mu_h \
(c + \mu_m) (c + \eta_m + \mu_m) -
      B k \beta_{mh} \eta_m \mathcal{M}) (\mu_h + \nu_h))$;

$A_m^*=\displaystyle\frac{\mathcal{M}}{\eta_A \mu_b}k N_h$,

$S_m^*=\displaystyle\frac{k Nh^2 \mathcal{M}}{\mu_b (c N_h
+ B I_h^{*} \beta_{hm} + N_h \mu_m)}$,

$E_m^*=\displaystyle\frac{\mu_m+c}{\eta_m}I_m^*$,

$I_m^*=\displaystyle\frac{B I_h^{*} k N_h \beta_{hm}
\eta_m \mathcal{M}}{\mu_b (c + \mu_m) (c + \eta_m + \mu_m)
(c N_h + B I_h^{*} \beta_{hm} + N_h \mu_m)}$

\medskip

As before, this equilibrium is only biologically interesting if $\mathcal{M}>0$.

With the \texttt{Mathematica}\index{Mathematica} software we obtained a fourth solution.
But some of the components are negative, which means that they
do not belong to the $\Omega$ set.
\end{proof}

\medskip

\begin{remark}
The condition $\mathcal{M}>0$ is equivalent,
by algebraic manipulation, to the condition
$$
\frac{(\eta_A+\mu_A)(\mu_m+c)}{\varphi\eta_A} < 1,
$$
which corresponds to the basic offspring number for mosquitoes.
Thus, if $\mathcal{M}<0$, then the mosquito population will collapse and
the only equilibrium for the whole system is the trivial
equilibrium. If $\mathcal{M}\geq0$, then the mosquito population is
sustainable.
\end{remark}

\medskip

The amount of mosquitoes is also related to an epidemic threshold:
the \emph{basic reproduction number}\index{Basic reproduction number}
of the disease, $\mathcal{R}_0$. Following \cite{Driessche2002}, we prove:

\begin{theorem}
\label{thm5:thm2}
If $\mathcal{M}>0$, then the square of the basic
reproduction number associated to
\eqref{odehuman}-\eqref{odevector} is
$$\mathcal{R}_{0}^2=\displaystyle\frac{B^2 S_{h0} S_{m0} \beta_{hm}
\beta_{mh} \eta_m \nu_h }{N_{h}^{2} (\eta_h + \mu_h) (c + \mu_m)
(c + \eta_m + \mu_m) (\mu_h + \nu_h)}.$$

The equilibrium point
BRDFE is locally asymptotically stable if $\mathcal{R}_{0}<1$ and
unstable if $\mathcal{R}_{0}>1$.
\end{theorem}

\medskip

\begin{proof}
To derive the basic reproduction number, we use the next generator approach.
The basic reproduction number is calculated
in a Disease Free Equilibrium\index{Disease Free Equilibrium}.
In this case we consider the most realistic one, BRDFE.\index{Disease Free Equilibrium}

Following \cite{Driessche2008,Driessche2002}, let consider the vector
$x^{T}=\left(E_h,I_h,E_m,I_m\right)$ which corresponds
to the components related to the progression of the disease.

The subsystem used is:

\begin{equation}
\label{subsystem}
\begin{tabular}{l}
$
\left\{
\begin{array}{l}
\frac{dE_h}{dt}
= B\beta_{mh}\frac{Im}{N_h}S_h - (\nu_h + \mu_h )E_h\\
\frac{dI_h}{dt}
= \nu_h E_h -(\eta_h  +\mu_h) I_h\\
\frac{dE_m}{dt}
= B \beta_{hm}\frac{I_h}{N_h}S_m-(\mu_m + \eta_m) E_m-c E_m\\
\frac{dI_m}{dt}
= \eta_m E_m -\mu_m I_m - c I_m
\end{array}
\right. $\\
\end{tabular}
\end{equation}
This subsystem can be partitioned,
$\displaystyle\frac{dx}{dt}=\mathcal{F}(x) - \mathcal{V}(x)$, where $x^{T}
=\left(E_h,I_h,E_m,I_m\right)$, $\mathcal{F}(x)$ represents the components
related to new cases of disease (in this situation in the exposed compartments)
and $\mathcal{V}(x)$ represents the other components.
Thus the subsystem \eqref{subsystem} can be rewritten as

\medskip

\begin{center}
$\mathcal{F}(x)=\left(
                  \begin{array}{c}
                    B\beta_{mh}\frac{I_m}{N_h}S_h \\
                    0 \\
                    B\beta_{hm}\frac{I_h}{N_h}S_m \\
                    0 \\
                  \end{array}
                \right)
$
and
$\mathcal{V}(x)=\left(
                  \begin{array}{c}
                    (\nu_h+\mu_h)E_h\\
                    -\nu_h E_h + (\eta_h+\mu_h)I_h \\
                    (\mu_m+\eta_m+c)E_m\\
                    -\eta_m E_m +(\mu_m+c)I_m \\
                  \end{array}
                \right).
$
\end{center}

\medskip

Let us consider the Jacobian matrices associated
with $\mathcal{F}$ and $\mathcal{V}$:

\begin{center}
\begin{tabular}{l}
$J_{\mathcal{F}(x)}=\left(
                  \begin{array}{cccc}
                    0 & 0 & 0 & B\beta_{mh}\frac{S_h}{N_h} \\
                    0 & 0 & 0 & 0\\
                    0 & B\beta_{hm}\frac{S_m}{N_h} & 0 & 0\\
                    0 & 0 & 0 & 0\\
                  \end{array}
                \right)
$
\end{tabular}
\end{center}
and
\begin{center}
\begin{tabular}{l}
$J_{\mathcal{V}(x)}=\left(
                  \begin{array}{cccc}
                    \nu_h+\mu_h & 0 & 0 & 0 \\
                    -\nu_h & \eta_h+\mu_h & 0 & 0 \\
                    0 & 0 & \mu_m+\eta_m+c & 0\\
                    0 & 0 & -\eta_m & \mu_m+c\\
                  \end{array}
                \right).
$
\end{tabular}
\end{center}

According to \cite{Driessche2002} the basic reproduction number is
$\mathcal{R}_{0}=\rho(J_{\mathcal{F}(x_{0})}J_{\mathcal{V}^{-1}(x_{0})})$,
where $x_{0}$ is a Disease Free Equilibrium (BRDFE)\index{Disease Free Equilibrium}\index{Disease Free Equilibrium}
and $\rho(A)$ defines the spectral radius of a matrix $A$. Using \texttt{Mathematica}\index{Mathematica}

\begin{center}
\begin{tabular}{l}
$\mathcal{R}_{0}^2=\displaystyle\frac{B^2 k \beta_{hm} \beta_{mh}
\eta_m \nu_h\mathcal{M} }{\mu_b (\eta_h + \
\mu_h) (c + \mu_m)^2 (c + \eta_m + \mu_m) (\mu_h + \nu_h)}$
\end{tabular}
\end{center}
and we obtain the value for the threshold parameter, with $\mathcal{M}>0$.
\end{proof}

\medskip

\begin{remark}
In the model we have two different
populations (human and vectors), so the expected basic reproduction
number\index{Basic reproduction number} reflects the infection human-vector and also
vector-human, that is, $\mathcal{R}_{0}^2=R_{hm}\times R_{mh}$.
The term $B\beta_{hm}\frac{S_{m0}}{N_h}$ represents the
product between the transmission probability of the disease from
humans to vectors and the number of susceptible mosquitoes per
human; $\frac{1}{\eta_h+\mu_h}$ is related to the human's viremic
period; and $\frac{\eta_m}{c+\eta_m+\mu_m}$ represents the
proportion of mosquitoes that survive to the incubation period.
Analogously, the term $B\beta_{mh}\frac{S_{h0}}{Nh}$ is related to
the transmission probability of the disease between mosquitos and
human, in a susceptible population; $\frac{1}{(c+\mu_m)}$
represents the lifespan of an adult mosquito; and
$\frac{\nu_h}{(\mu_h+\nu_h)}$ is the proportion of humans that
survive the incubation period.
\end{remark}

\medskip

When $\mathcal{R}_{0}<1$,
each infected individual produces, on
average, less than one new infected individual, and therefore it
is predictable that the infection will be cleared from the
population. If $\mathcal{R}_{0}>1$, the disease is able to invade
the susceptible population \cite{Driessche2002, Li1996}.

\begin{theorem}
\label{thm5:thm3}
If $\mathcal{M}>0$ and $\mathcal{R}_{0}>1$, then the
system \eqref{odehuman}-\eqref{odevector} also admits an endemic
equilibrium (EE)\index{Endemic Equilibrium}:
$E_{3}^{*}=\left(S_h^*,E_h^*,I_h^*,A_m^*,S_m^*,E_m^*,I_m^*\right)$,
where,
\begin{equation*}
\begin{split}
S_h^* &=
N_h-\displaystyle\frac{(\mu_h+\nu_h)(\mu_h+\eta_h)}{\mu_h\nu_h}I_h^{*},\\
E_h^* &=\displaystyle\frac{\mu_h+\eta_h}{\nu_h}I^{*}_{h},\\
I_h^* &= \frac{\xi}{\chi},\\
\xi &= N_h \mu_h \Bigl[-B^2 k \beta_{hm} \beta_{mh}\nu_h\eta_m
\mathcal{M} + \varphi \mu_m^2(\eta_m + \mu_m)(\mu_h + \nu_h)(\mu_h + \eta_h)\\
&\quad + c^2\varphi(\eta_h + \mu_h)(\mu_h + \nu_h)(c + \eta_m + 3 \mu_m)\\
&\quad + c\varphi\mu_m(\mu_h + \nu_h)(\mu_h(3 \mu_m + 2)
+ \eta_h(2 \eta_m + 3 \mu_m))\Bigr],\\
\chi &= B \beta_{hm} (\eta_h + \mu_h) \Bigl[-\varphi \mu_h (c +
\mu_m) (c + \eta_m + \mu_m)
- B k \beta_{mh} \eta_m \mathcal{M}\Bigr] (\mu_h + \nu_h),\\
A_m^* &= \displaystyle\frac{\mathcal{M}}{\eta_A \varphi}k N_h,\\
S_m^* &=\displaystyle\frac{k Nh^2 \mathcal{M}}{\varphi (c N_h + B I_h^{*} \beta_{hm} + N_h \mu_m)},\\
E_m^* &= \displaystyle\frac{\mu_m+c}{\eta_m}I_m^*,\\
I_m^* &=\displaystyle\frac{B I_h^{*} k N_h \beta_{hm} \eta_m
\mathcal{M}}{\varphi (c + \mu_m) (c + \eta_m + \mu_m) (c N_h + B
I_h^{*} \beta_{hm} + N_h \mu_m)} .
\end{split}
\end{equation*}
\end{theorem}

\begin{proof}
See Proof of Theorem~\ref{thm5:thm1}.
\end{proof}

\medskip

From a biological point of view, it is desirable that humans
and mosquitoes coexist without the disease reaching a level of endemicity.
Good estimates of Dengue transmission intensity are therefore necessary
to compare and interpret Dengue interventions conducted in different
places and times and to evaluate options for Dengue control.
The basic reproduction number\index{Basic reproduction number}
has played a central role in the epidemiological theory for Dengue
and other infectious diseases because it provides an index
of transmission intensity and establishes a threshold criteria.
We claim that proper use of control $c$, can result in the basic
reproduction number remaining below unity and, therefore,
making BRDFE stable.\index{Disease Free Equilibrium}
In order to make effective use of achievable insecticide control,
and simultaneously to explain more easily to the competent authorities its effectiveness,
we assume that $c$ is constant. The goal is to find $c$ such that $\mathcal{R}_{0}^2<1$.
For this purpose we have studied the reality of Cape Verde.


\subsection{Dengue in Cape Verde}
\label{sec:5:1:2}

An unprecedented outbreak was detected in the Cape
Verde archipelago in September 2009. This is the first report of
Dengue virus activity in that country. As the population had never had
contact with the virus, the herd immunity\index{Herd immunity} was very low. Dengue
type 3 spread throughout the islands of the archipelago,
reaching four of the nine islands. The worst outbreak occurred
on the Santiago island, where most people live. The number of
cases increased sharply since the beginning of November,
reaching 1000 cases per day. The Cape Verde Ministry of Health
reported more than 20000 cases of Dengue Fever within the
archipelago between October and December 2009, which is about
5\% of the total population of the country. From 173 cases of Dengue
hemorrhagic fever reported, six people died \cite{CapeVerdeSaude,CapeVerde}.

It represented a challenge to the performance of the National Health Care System.
Government officials launched a plan to eradicate the mosquito including
a national holiday during which citizens were being asked to clear out standing water
and other potential breeding areas used by the mosquitoes. The intense
and spontaneous movement of solidarity from the civil society was another noteworthy dimension,
not only cleaning but also voluntary donating blood to strengthen
the stock of the Central Hospital Dr. Agostinho Neto, in Praia.

We used the data for human population related to Cape Verde \cite{CDC2010}.
Due to low surveillance and the fact of being the first ever of Dengue outbreak
in the country, it was not possible to collect detailed data from the mosquito.
However, the authorities speak explicitly about mosquitoes coming from Brazil \cite{Africa21}.
Also the information that comes from the Ministry of Health
in the capital of Cape Verde, Praia, confirms
that the insects responsible for Dengue came most probably from Brazil,
transported by means of air transport that perform frequent
connections between Cape Verde and Brazil, as
reported by the Radio of Cape Verde.
With respect to \emph{Aedes aegypti}, we have thus considered
data from Brazil \cite{Thome2010,Yang2009}.

\medskip

The simulations were carried out using the following values:
$N_h=480000$, $B=1$, $\beta_{mh}=0.375$, $\beta_{hm}=0.375$,
$\mu_{h}=1/(71\times365)$, $\eta_{h}=1/3$, $\mu_{m}=1/11$,
$\varphi=6$, $\mu_{A}=1/4$, $\eta_A=0.08$, $\eta_m=1/11$,
$\nu_h=1/4$, , $m=6$, $k=3$. The initial conditions for
the problem were: $S_{h0}=N_h-E_h-I_h$, $E_{h0}=216$,
$I_{h0}=434$, $R_{h0}=0$, $A_{m0}=kN_h$, $S_{m0}=mN_h$ and $I_{m0}=0$.

Considering nonexistence of control, \textrm{i.e.} $c=0$, the basic
reproduction number for this outbreak in Cape Verde is
approximately $\mathcal{R}_{0}=2.396$, which is in agreement to other
studies of Dengue in other countries \cite{Nishiura2006}.
The control $c$ affects the basic reproduction number\index{Basic reproduction number},
and our aim is to find a control that puts $\mathcal{R}_{0}$ less than one.

\begin{proposition}
Let us consider the parameters listed above and consider $c$ as a constant.
Then $\mathcal{R}_{0}<1$ if and only if $c > 0.156961$.
\end{proposition}

This value was obtained by \texttt{Mathematica}\index{Mathematica},
calculating the inequality with the parameters values above.
The computational investigations were carried out using
$c=0.157$, which means that the insecticide is continuously
applied during a twelve week period.

The software used was \texttt{Scilab}\index{Scilab} \cite{Campbell2006}.
It is an open source, cross-platform numerical computational
package and a high-level, numerically oriented programming
language. For our problem we used the routine \texttt{ode} to
solve the set of differential equations. By default, \texttt{ode}
uses the \texttt{lsoda} solver of package \texttt{ODEPACK}.
It automatically selects between the nonstiff
predictor-corrector Adams method and the stiff
Backward Differentiation Formula (BDF) method. It uses the nonstiff
method initially and dynamically monitors data in order to decide
which method to use. The graphics were also obtained using this
software, with the command \texttt{plot} (see code in \cite{SofiaSITE}).

Figures~\ref{cap5_human_control} and \ref{cap5_human_nocontrol} show the curves
related to human population, with and without control, respectively.
The number of infected people, even with small control,
is much less than without any insecticide campaign.

\begin{figure}[ptbh]
\center
\includegraphics [scale=0.6]{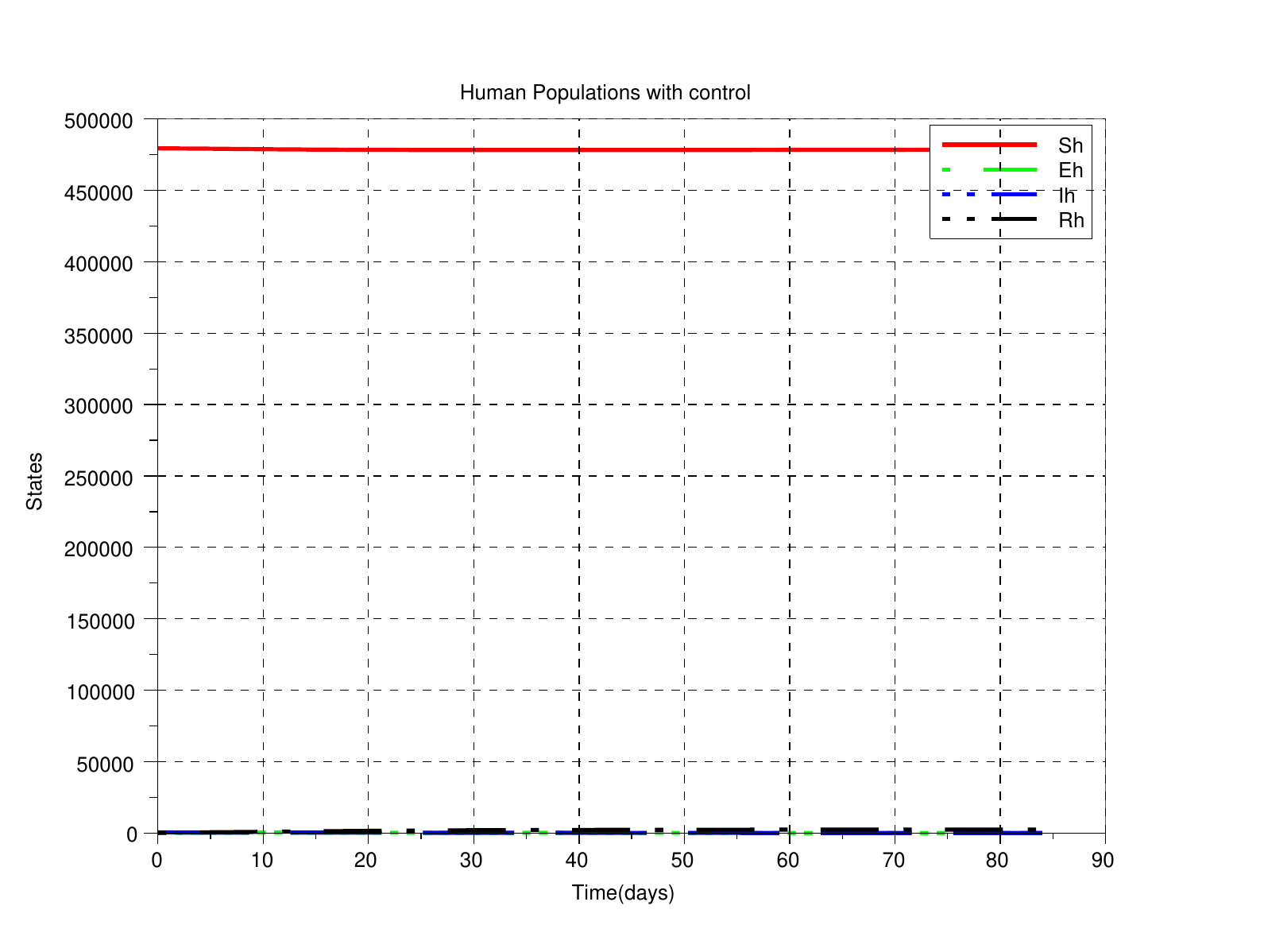}\\
{\caption{\label{cap5_human_control}Human compartments using control}}
\end{figure}

\begin{figure}[ptbh]
\center
\includegraphics [scale=0.6]{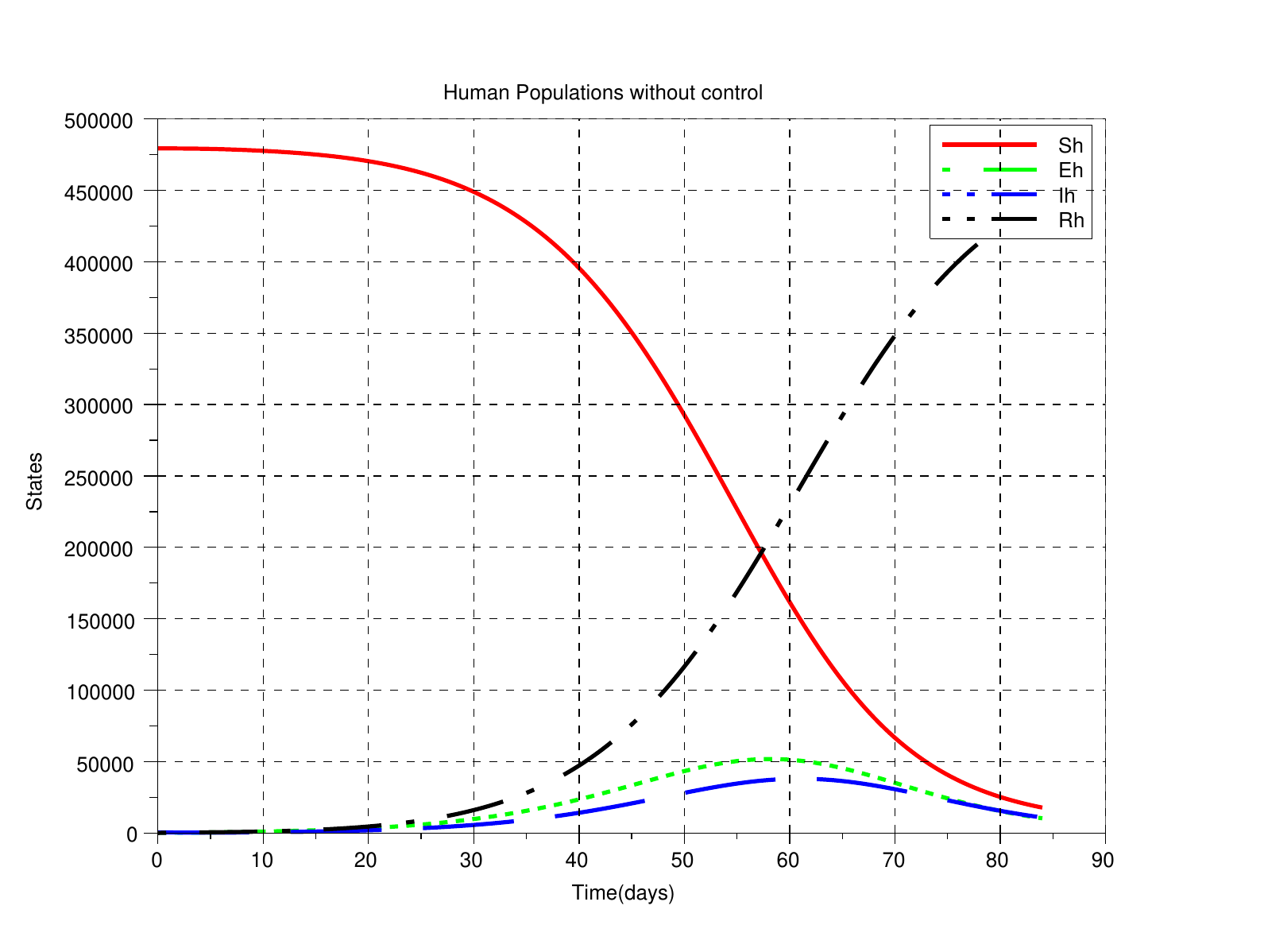}\\
{\caption{\label{cap5_human_nocontrol}Human compartments without control}}
\end{figure}

Figures~\ref{cap5_mosquito_control} and \ref{cap5_mosquito_nocontrol}
show the difference of the mosquito population with control and without control.
When the control is applied, the number of infected mosquitoes is close to zero.
Note that the intention is not to completely eradicate the mosquitoes
but instead the number of infected mosquitoes.

\begin{figure}[ptbh]
\center
\includegraphics [scale=0.6]{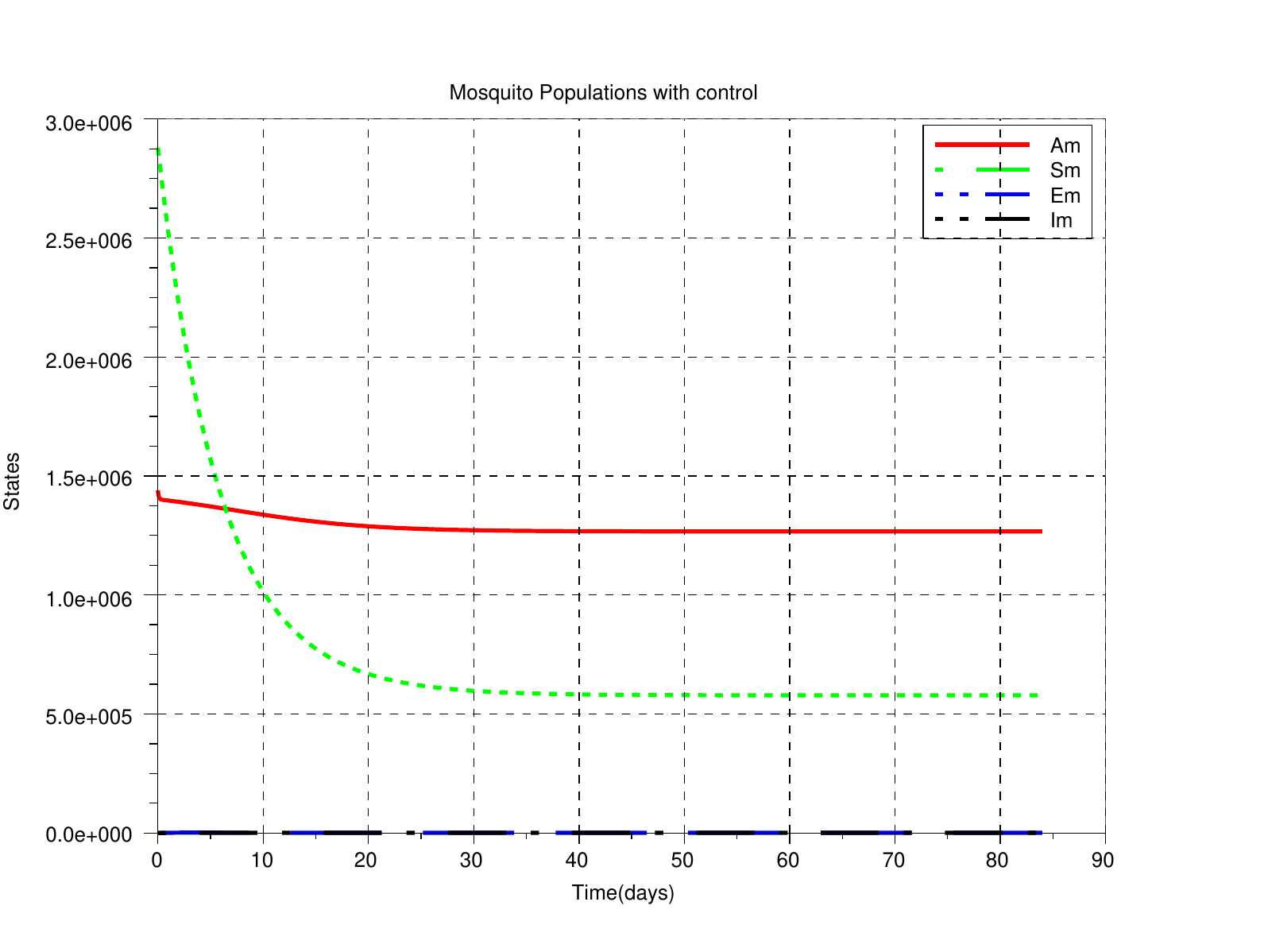}\\
{\caption{\label{cap5_mosquito_control}Mosquito compartments using control}}
\end{figure}

\begin{figure}[ptbh]
\center
\includegraphics [scale=0.6]{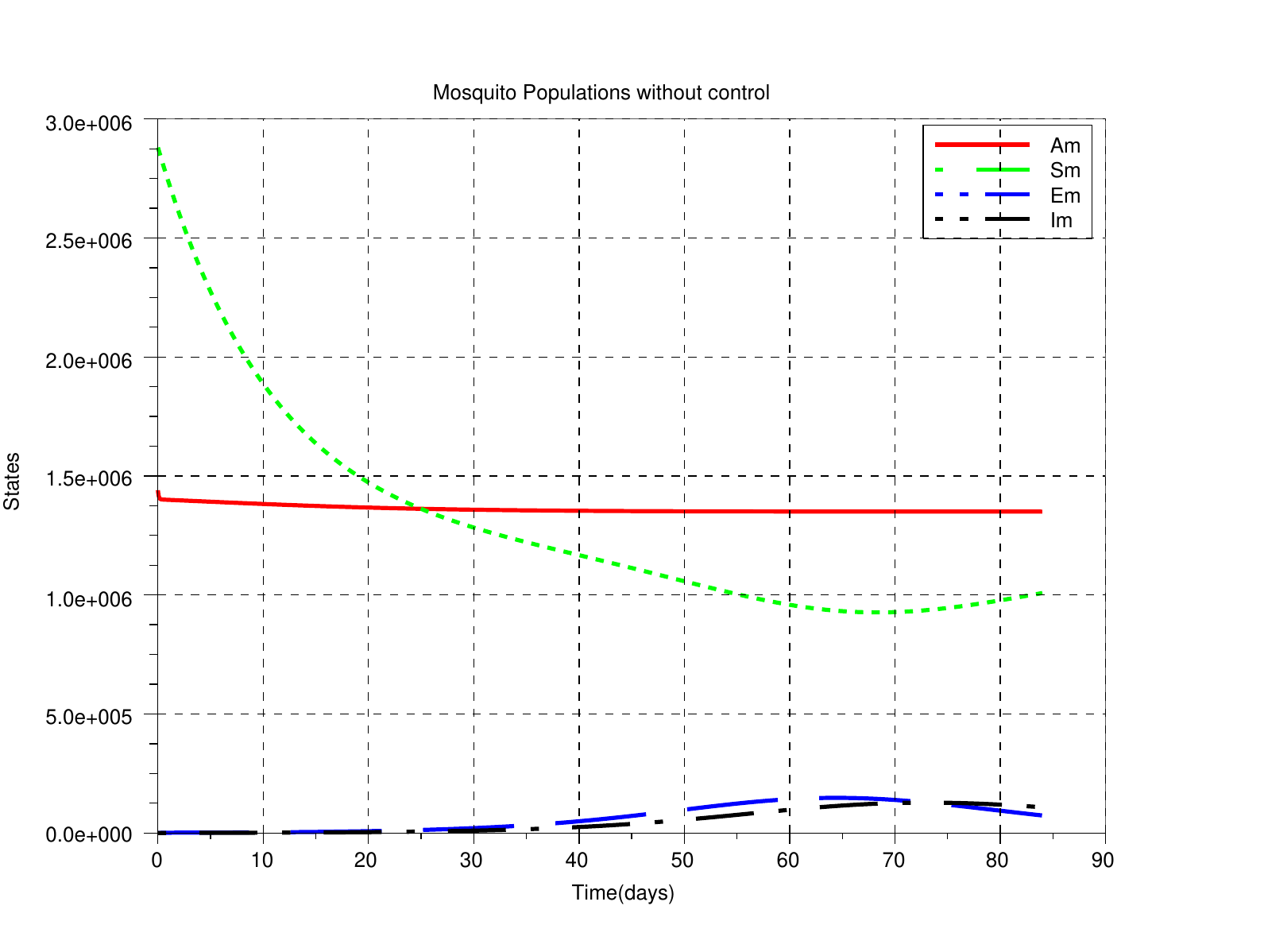}\\
{\caption{\label{cap5_mosquito_nocontrol}Mosquito compartments without control}}
\end{figure}

It has been algebraically proved that if a constant minimum
level of a control is applied ($c = 0.157$),
it is possible to maintain the basic reproduction number below unity,
guaranteeing the BRDFE. This value is corroborated\index{Disease Free Equilibrium}
in another numerical study \cite{Sofia2010c}.

The values of infected humans obtained by the model are higher
when compared to what really happened in Cape Verde. Despite
the measures taken by the government and Health authorities
were not accounted, they have had a considerable impact in the progress of the disease.

Until here, it was considered a constant control.
Using a theoretical approach \cite{Sofia2010b}, we intend to
find the best function $c(t)$, using OC theory.
Instead of finding a constant control it will be possible
to study other types of control, such as piecewise constant or
even continuous but not constant. Additionally we could
consider another strategy, a more practical one: due to logistics
and health reasons, it may be more convenient to apply insecticide
periodically and at some specific hours at night.

\medskip


\section{Insecticide control in a Dengue epidemics model}
\label{sec:5:3}

In this section we investigate the best way to apply the control in order
to effectively reduce the number of infected humans and mosquitoes, using pulse control.

In literature it has been proven that a DFE\index{Disease Free Equilibrium}
is locally asymptotically stable, whenever a certain epidemiological threshold,
known as the \emph{basic reproduction number}\index{Basic reproduction number},
is less than one. In the previous section, it was proven that if a constant minimum level
of insecticide is applied ($c=0.157$), it is possible to maintain the basic reproduction
number below unity, guaranteeing the DFE. In this section, other kinds of piecewise controls
that maintain the basic reproduction number less than one and could give a better solution
to implement by health authorities are investigated.

To solve the system (\ref{odehuman})--(\ref{odevector}), in a first step,
several strategies of control application were used: three different frequencies
(weekly, bi-weekly and monthly) for control application, a constant control ($c=0.157$)
and no control ($c=0$). The first three different frequencies mean that during one day
(per week, bi-week and month), the whole ($100\%$) capacity of insecticide ($c=1$)
is used during all day. Besides, also was used a constant control strategy ($c=0.157$)
that consists in the application of $15.7\%$  capacity of insecticide 24 hours per day
all the time (84 days). In this work, the amount of insecticide is an adimensional
value and must be considered in relative terms.

The numerical tests were carried out using Scilab \cite{Scilab}, with the same ODE
system and parameters of the previous section (code available in \cite{SofiaSITE}).

Figures~\ref{cap5_human_weekly} and \ref{cap5_mosquitoes_weekly} show the results
of these strategies regarding infected mosquitoes and individuals. Without control,
the number of infected mosquitoes and individuals increase expressively. It is also
possible to see that the weekly pulse control had the closest results to the continuous control.

\begin{figure}[ptbh]
\center
\includegraphics [scale=0.6]{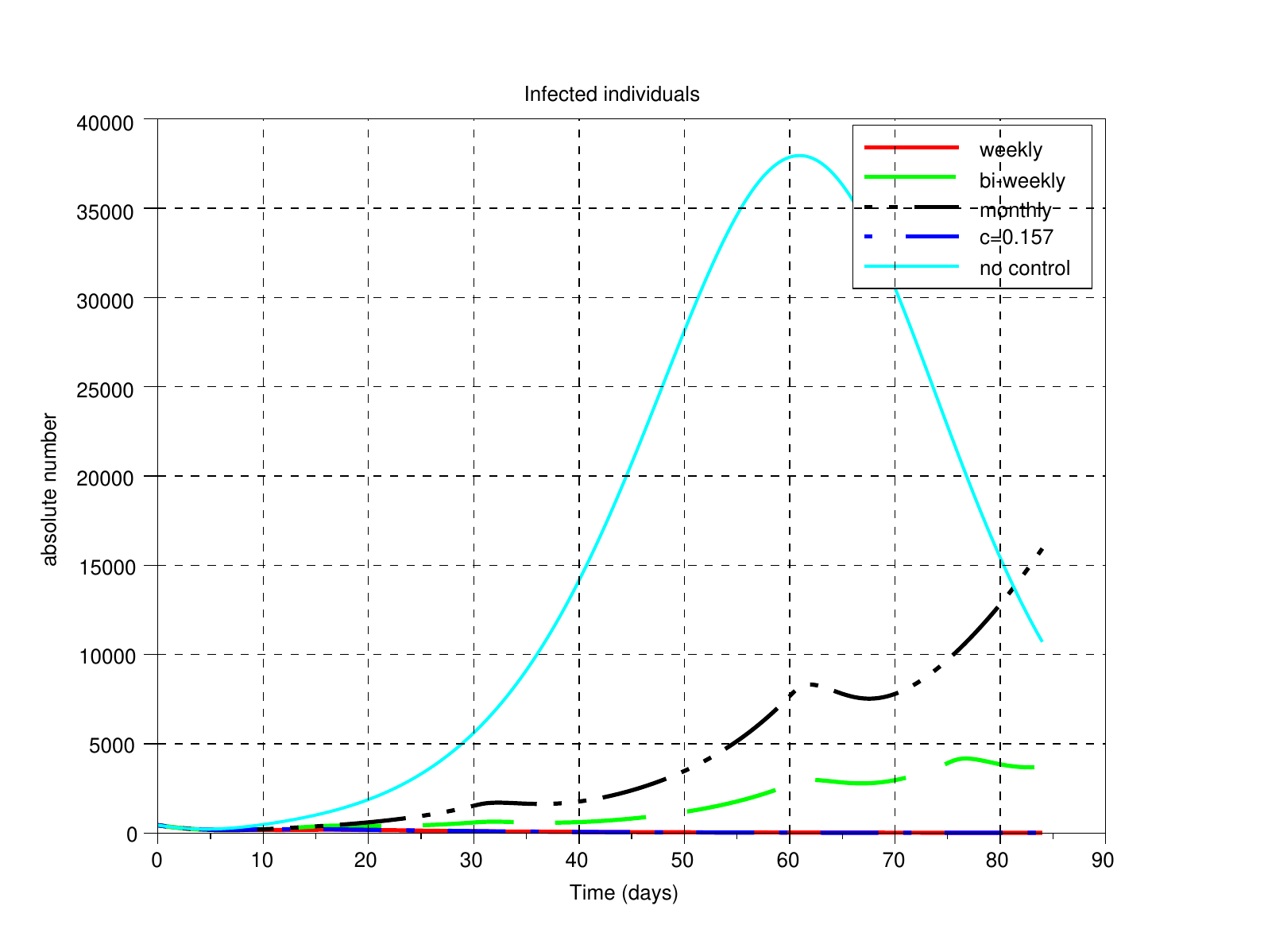}\\
{\caption{\label{cap5_human_weekly} Comparison of infected humans using no control,
continuous control and periodically control (weekly, biweekly and monthly)}}
\end{figure}

\begin{figure}[ptbh]
\center
\includegraphics [scale=0.6]{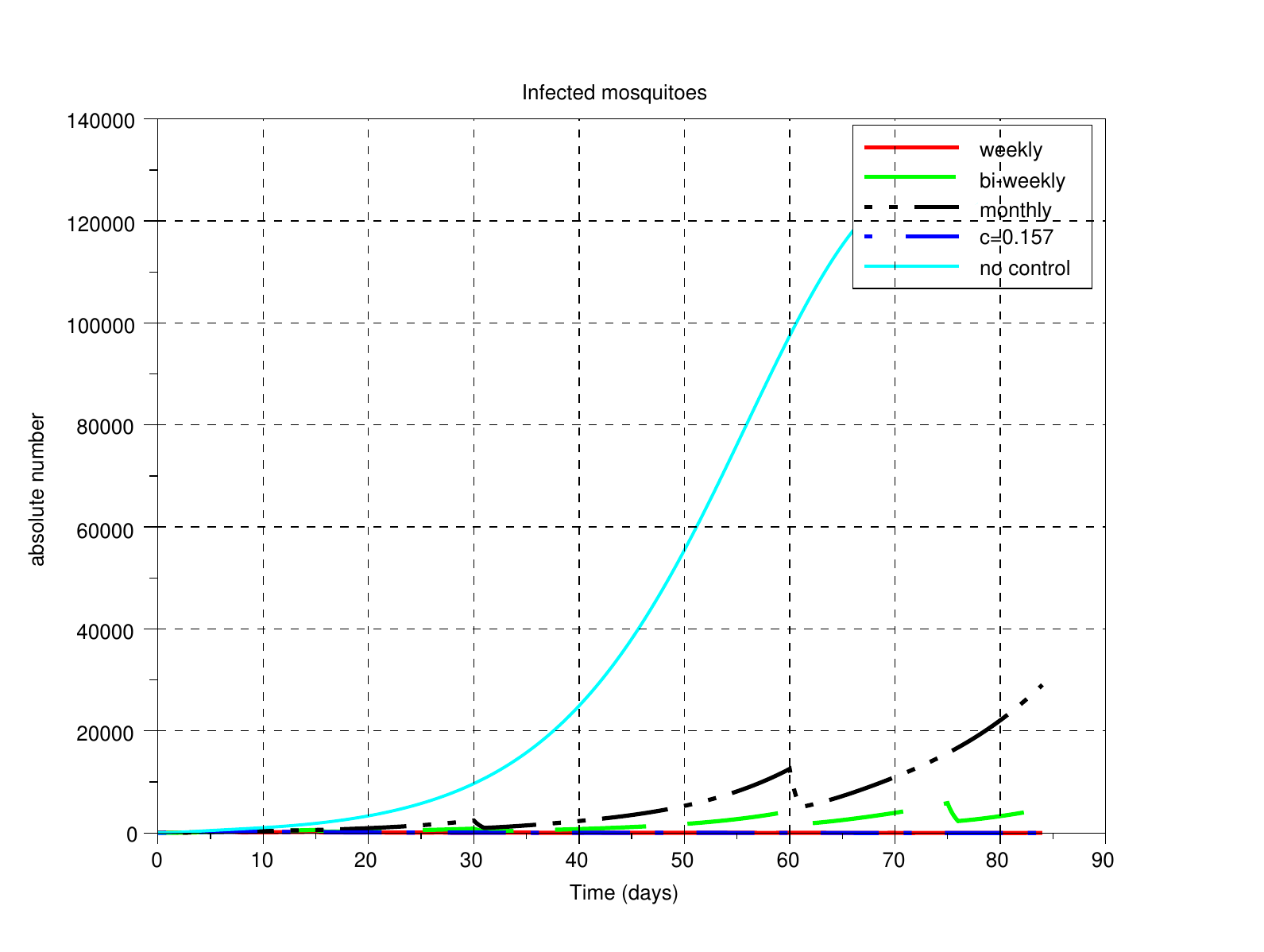}\\
{\caption{\label{cap5_mosquitoes_weekly} Comparison of infected mosquitoes using no control,
continuous control and periodically control (weekly, biweekly and monthly)}}
\end{figure}

Therefore, realizing the influence of the insecticide control, further tests were carried out
to find the optimum periodicity of administration which, from gathered results,
must rest between six and seven days. The second phase of numerical tests,
Figures~\ref{cap5_infected_individuals_weekly} and \ref{cap5_infected_mosquitoes_weekly},
considered four situations: 6 days, 7 days, 10 days and continuously $c=0.157$.
To guarantee the DFE, the curves must remain below the one corresponding
to $c=0.157$.\index{Disease Free Equilibrium}

\begin{figure}[ptbh]
\center
\includegraphics [scale=0.6]{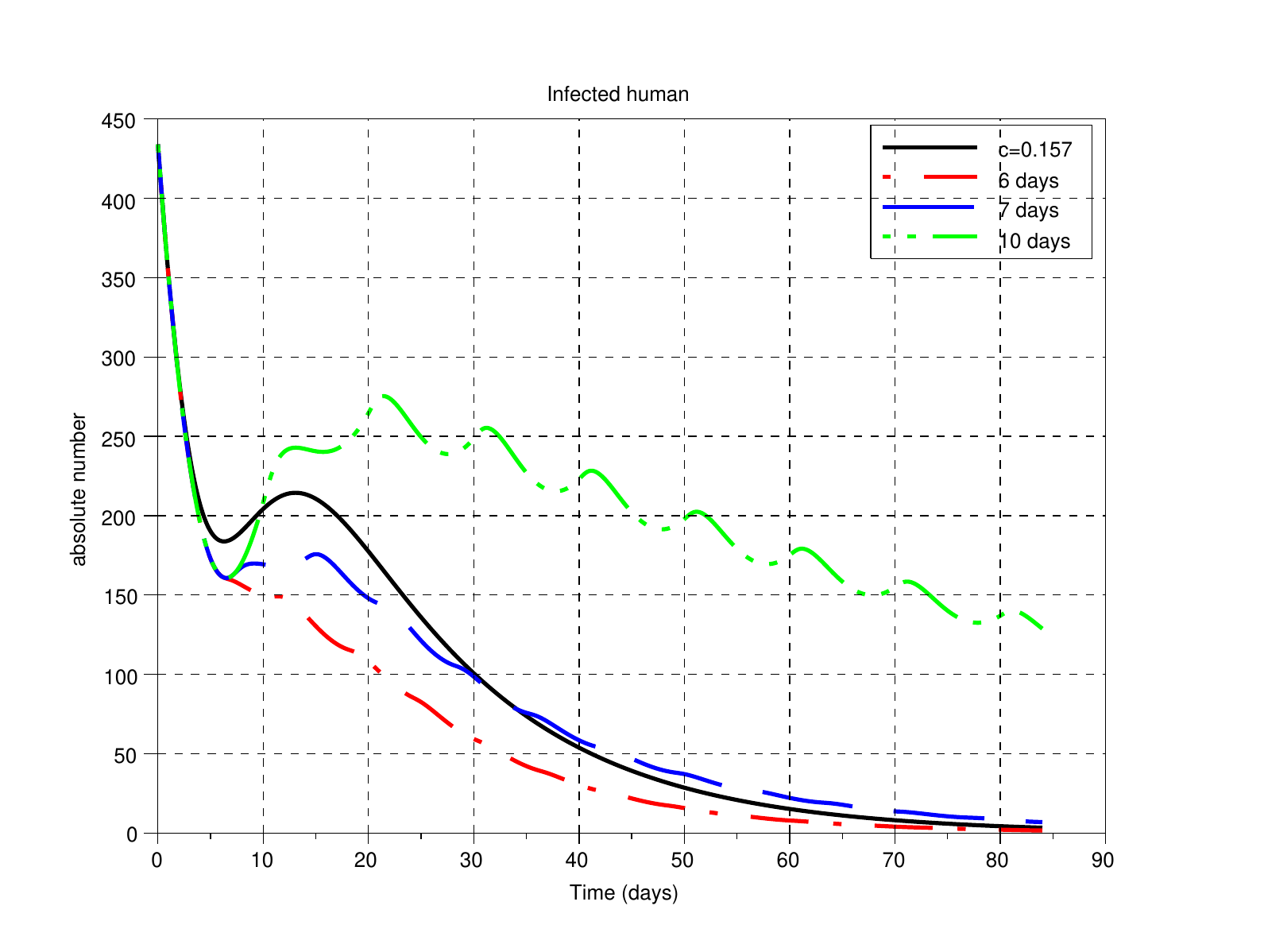}\\
{\caption{\label{cap5_infected_individuals_weekly} Infected human and different piecewise controls}}
\end{figure}

\begin{figure}[ptbh]
\center
\includegraphics [scale=0.6]{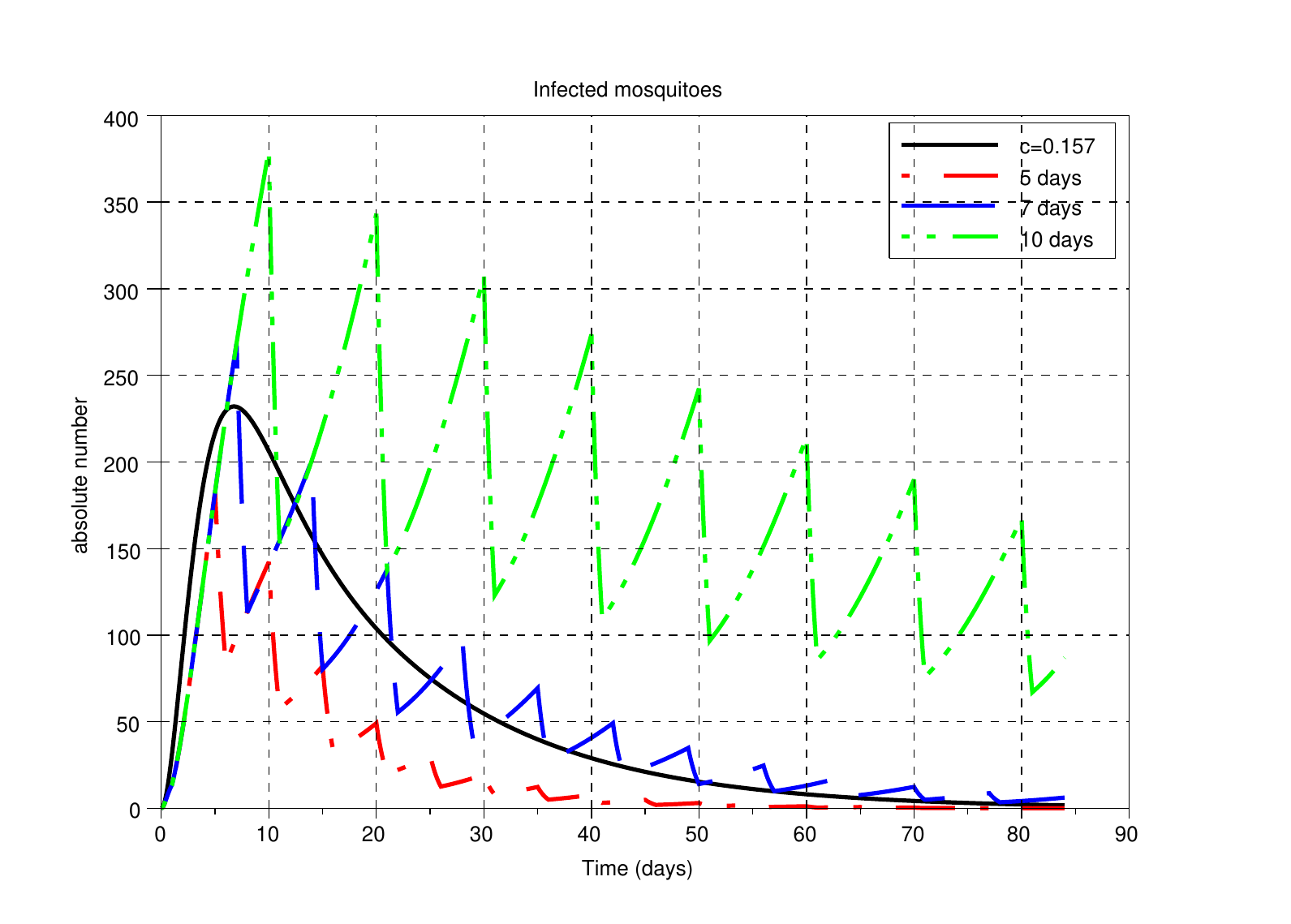}\\
{\caption{\label{cap5_infected_mosquitoes_weekly} Infected mosquitoes and different piecewise controls}}
\end{figure}

The amount of insecticide, and when to apply it, are important factors for outbreak control.
Table~\ref{cap5_result} reports the total amount of insecticide
used in each version during the 84 days.

\begin{table}[ptbh]
\begin{center}
\begin{tabular}{ccccccc}
\hline
& 6 days & 7 days & 10 days & 15 days & 30 days & $c=0.157$ \\
\hline
insecticide amount & 14 & 12 & 9 & 6 & 3 & 13.188  \\\hline
\end{tabular}
\caption{Insecticide cost}
\label{cap5_result}
\end{center}
\end{table}

The numerical tests conclude that the best strategy for the infected reduction
is every 6/7 days application. The value spent in insecticide is similar
to the continuous strategy, but is much easier to implement in terms of logistics
of a massive application. This result is consistent with what WHO recommends,
which is an application of insecticide every 3/4 days; this lower frequency
is probably due to the non-application of insecticide during all day.

In this work, several piecewise strategies were studied to find the best way
of applying insecticide, but always having in mind a periodic application of the product.
But if the best strategy is not a periodic one? In the next section the optimal
control solution for this problem will be studied.

\medskip


\section{Optimal control}

The aim of this section is the study of optimal strategies for applying insecticide,
taking into account different perspectives: thinking only on the insecticide cost,
focusing on the cost of infected humans or even combining both perspectives.

To take this approach, and after several numerical experiences, we considered that
it was better to normalize the ODE system (\ref{odehuman})--(\ref{odevector}).
The reason for this transformation is due to the bad scaling of the variables:
some vary from 0 to 480000, and others from 0 to 1, which can influence the software performance.

Consider the following transformations:

\begin{center}
\begin{tabular}{llll}
$\displaystyle s_h=\frac{S_h}{N_h}$ & $\displaystyle e_h
=\frac{E_h}{N_h}$ & $\displaystyle i_h=\frac{I_h}{N_h}$ & $\displaystyle r_h=\frac{R_h}{N_h}$\\
& & & \\
$\displaystyle a_m=\frac{A_h}{kN_h}$ & $\displaystyle s_m=\frac{S_m}{mN_h}$
& $\displaystyle e_m=\frac{E_m}{mN_h}$ & $\displaystyle i_m=\frac{I_m}{mN_h}$\\
\end{tabular}
\end{center}

The ODE system (\ref{odehuman})--(\ref{odevector}) is transformed into:
\begin{equation}
\label{ode_norm}
\begin{tabular}{l}
$\left\{
\begin{array}{l}
\displaystyle\frac{ds_h}{dt} = \mu_h - (B\beta_{mh}m i_m +\mu_h) s_h\\
\displaystyle\frac{de_h}{dt} = B\beta_{mh}m i_m s_h - (\nu_h + \mu_h )e_h\\
\displaystyle\frac{di_h}{dt} = \nu_h e_h -(\eta_h  +\mu_h) i_h\\
\displaystyle\frac{dr_h}{dt} = \eta_h i_h - \mu_h r_h\\
\displaystyle\frac{da_m}{dt} = \varphi \frac{m}{k}(1-a_m)(s_m+e_m+i_m)-(\eta_A+\mu_A) a_m\\
\displaystyle\frac{ds_m}{dt} = \eta_A \frac{k}{m}a_m-(B \beta_{hm}i_h+\mu_m) s_m-c s_m\\
\displaystyle\frac{de_m}{dt} = B \beta_{hm}i_h s_m-(\mu_m + \eta_m) e_m-c e_m\\
\displaystyle\frac{di_m}{dt} = \eta_m e_m -\mu_m i_m - c i_m\\
\end{array}
\right. $\\
\end{tabular}
\end{equation}
\noindent with the initial conditions
\begin{equation}
\label{initial_norm}
\begin{tabular}{llll}
$s_h(0)=0.99865,$ & $e_h(0)=0.00035,$ & $i_h(0)=0.001,$ &
$r_h(0)=0,$ \\
$a_m(0)=1,$ & $s_{m}(0)=1,$ &
$e_m(0)=0,$ & $i_m(0)=0.$
\end{tabular}
\end{equation}

Let $x=(s_h,e_h,i_h,r_h,a_m,s_m,e_m,i_m)$ the vector representing the state variables.

Let us consider the objective functional $J$ considering the costs
of infected humans and costs with insecticide:
\begin{equation}
\label{functional}
\text{minimize } J=\int_{0}^{t_f}\left[\gamma_D i_h(t)^2
+\gamma_S c(t)^2\right]dt,
\end{equation}
\noindent where $\gamma_D$ and $\gamma_S$ are positive constants
representing the costs weights of infected individuals and spraying campaigns, respectively.


\subsection{Pontryagin's Maximum Principle}

Using OC theory, let $\lambda_i(t)$, with $i=1,\ldots,8$, be the co-state variables.
The Hamiltonian\index{Hamiltonian} for the present OC problem is given by
\begin{equation}
\label{hamiltonian_chapter5}
\begin{tabular}{ll}
$H$&=$\gamma_D i_h^2 +\gamma_S c^2$\\
&$+\displaystyle\lambda_1 \left[\mu_h - (B\beta_{mh}m i_m +\mu_h) s_h\right]$\\
&$+\displaystyle\lambda_2\left[B\beta_{mh}m i_m s_h - (\nu_h + \mu_h )e_h\right]$\\
&$+\displaystyle\lambda_3\left[\nu_h e_h -(\eta_h  +\mu_h) i_h\right]$\\
&$+\displaystyle\lambda_4\left[\eta_h i_h - \mu_h r_h\right]$\\
&$+\displaystyle\lambda_5\left[\varphi \frac{m}{k}(1-a_m)(s_m+e_m+i_m)-(\eta_A+\mu_A) a_m\right]$\\
&$+\displaystyle\lambda_6\left[\eta_A \frac{k}{m}a_m-(B \beta_{hm}i_h+\mu_m) s_m-c s_m\right]$\\
&$+\displaystyle\lambda_7\left[B \beta_{hm}i_h s_m-(\mu_m + \eta_m) e_m-c e_m\right]$\\
&$+\displaystyle\lambda_8\left[\eta_m e_m -\mu_m i_m - c i_m\right]$\\
\end{tabular}
\end{equation}

By the Pontryagin Maximum Principle \cite{Pontryagin1962}\index{Pontryagin's Maximum Principle},
the optimal control $c^{*}$ should be the one that minimizes, at each instant $t$,
the Hamiltonian given\index{Hamiltonian} by \eqref{hamiltonian_chapter5}, that is,
$H\left(x^*(t),\lambda^*(t),c^*(t)\right)
=\min_{c \in [0,1]} H\left(x^*(t),\lambda^*(t),c\right)$.
In this way, the optimal control is given by

\begin{equation*}
c^*=\min\left\{1,\max\left\{0,\frac{\lambda_6 s_m+\lambda_7 e_m+\lambda_8 i_m}{2\gamma_S}\right\}\right\}.
\end{equation*}

It is also necessary to consider the adjoint system
$\lambda^{'}_i(t)=-\displaystyle\frac{\partial H}{\partial x_i}$, \textrm{i.e.},
\begin{equation}
\label{adjoint_system_chapter5}
\begin{tabular}{l}
$\left\{
\begin{array}{l}
\lambda_1^{'}=\lambda_1\left(B\beta_{mh}m i_m+\mu_h\right)-\lambda_2B\beta_{mh}m i_m\\
\lambda_2^{'}=\lambda_2\left(\nu_h+\mu_h\right)-\lambda_3\nu_h \\
\lambda_3^{'}= -2\gamma_D i_h +\lambda_3 \left(\eta_h+\mu_h\right)-\lambda_4\eta_h+(\lambda_6-\lambda_7)B\beta_{hm}s_m\\
\displaystyle\lambda_4^{'}(t)=\lambda_4 \mu_h\\
\lambda_5^{'}=\lambda_5\left[\varphi\frac{m}{k}\left(s_m+e_m+i_m\right)+\eta_A+\mu_A\right]-\lambda_6 \eta_A \frac{k}{m}\\
\lambda_6^{'}=-\lambda_5\varphi\frac{m}{k}\left(1-a_m\right)+\lambda_6\left(B\beta_{hm}i_h+\mu_m+c\right)-\lambda_7B\beta_{hm}i_h\\
\lambda_7^{'}=-\lambda_5\varphi\frac{m}{k}\left(1-a_m\right)+\lambda_7(\mu_m+\eta_m+c)-\lambda_8\eta_m\\
\lambda_8^{'}=\left(\lambda_1-\lambda_2\right)B\beta_{mh}ms_h-\lambda_5\varphi\frac{m}{k}\left(1-a_m\right)+\lambda_8\left(\mu_m+c\right)\\
\end{array}
\right. $\\
\end{tabular}
\end{equation}

As the OC problem just has initial conditions it is necessary to find the
transversality conditions\index{Transversality condition},
that correspond to a terminal condition in the co-state equation,
\begin{equation}
\label{terminal_conditions_chapter5}
\begin{tabular}{lll}
$\lambda_i(t_f)=0$ & & $i=1,\ldots,8$
\end{tabular}
\end{equation}
Replacing the optimal control $c^*$ in the state system \eqref{ode_norm}
and in the adjoint system (\ref{adjoint_system_chapter5}) is possible
to solve the differential system taking into account the initial
and transversality conditions.\index{Transversality condition}


\subsection{Numerical simulations}

In order to solve this OC problem three approaches were tested.
The first one is a direct method\index{Direct method}, \texttt{DOTcvp}\index{DOTcvp} toolbox \cite{Dotcvp}
for \texttt{Matlab}\index{Matlab}. It uses the differential system (\ref{ode_norm}),
the initial conditions (\ref{initial_norm}) and it is necessary to transform the problem
into a Mayer form\index{Mayer form} (see Section~\ref{sec:2:3}). To solve the discretized OC problem,
the \texttt{Ipopt}\index{Ipopt} software is chosen as an option inside the \texttt{DOTcvp}.
The functional is divided into time intervals, in this case ten intervals,
with an initial value problem tolerance of $10^{-7}$ and a NLP tolerance of $10^{-5}$.
The optimal control function given by this toolbox is piecewise constant.

Another direct method\index{Direct method} is \emph{OC-ODE}\index{OC-ODE} \cite{Matthias2009}.
It uses the differential system (\ref{ode_norm}), the initial conditions (\ref{initial_norm})
and it includes procedures for numerical adjoint estimation and sensitivity analysis
and the feasibility tolerance considered was $10^{-10}$.

The last method used, an indirect one\index{Indirect methods}, it is coded in \texttt{Matlab}\index{Matlab}
environment. It involves the backward-forward method (see Section~\ref{sec:2:2}) and the resolution
of the ODE systems are made by the \texttt{ode45}\index{ode45 routine} routine. The state differential
system \eqref{ode_norm} is solved forward with initial conditions (\ref{initial_norm}),
while the adjoint system (\ref{adjoint_system_chapter5}) is solved backwards using the terminal conditions
(\ref{terminal_conditions_chapter5}). The absolute and relative tolerances were fixed in $10^{-4}$.
The three codes are available at \cite{SofiaSITE}.

All the parameters are assumed equal to the previous section. For this first simulation,
the values for the weights of the costs are $\gamma_D=0.5$ and $\gamma_S=0.5$.

Figure~\ref{cap5_3_control_both_methods} shows that, despite having distinct philosophies
of resolution, the curves obtained by the three solvers are similar,
which reinforces the confidence in the result.

\begin{figure}[ptbh]
\center
\includegraphics [scale=0.6]{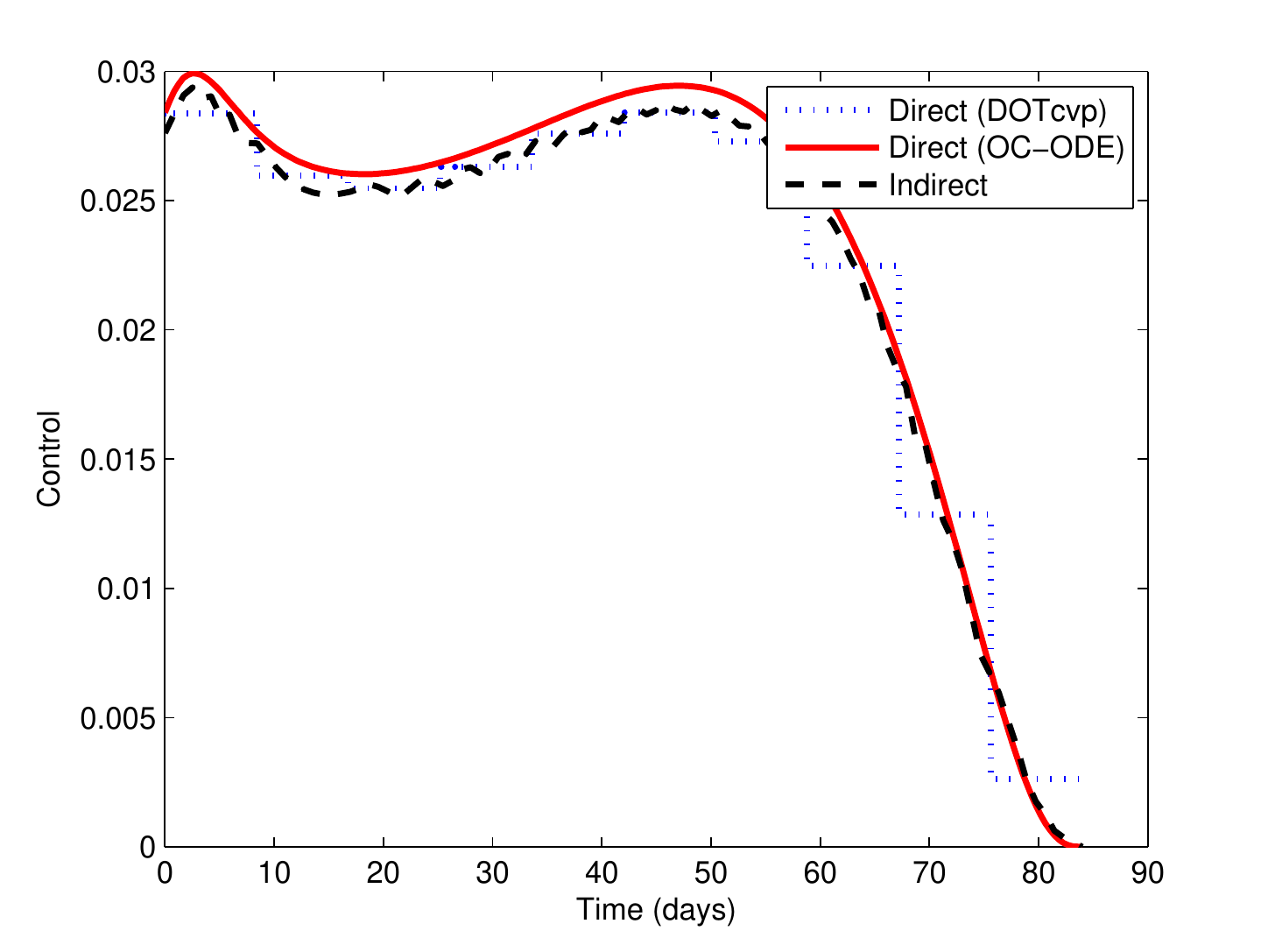}\\
{\caption{\label{cap5_3_control_both_methods} Optimal control obtained by the three methods}}
\end{figure}

Figure~\ref{cap5_3_infected_both_methods} presents the number of infected humans over the 84 days.
The tendency curve is similar, but the indirect method\index{Indirect methods}
presents a higher number of infected people.

\begin{figure}[ptbh]
\center
\includegraphics [scale=0.6]{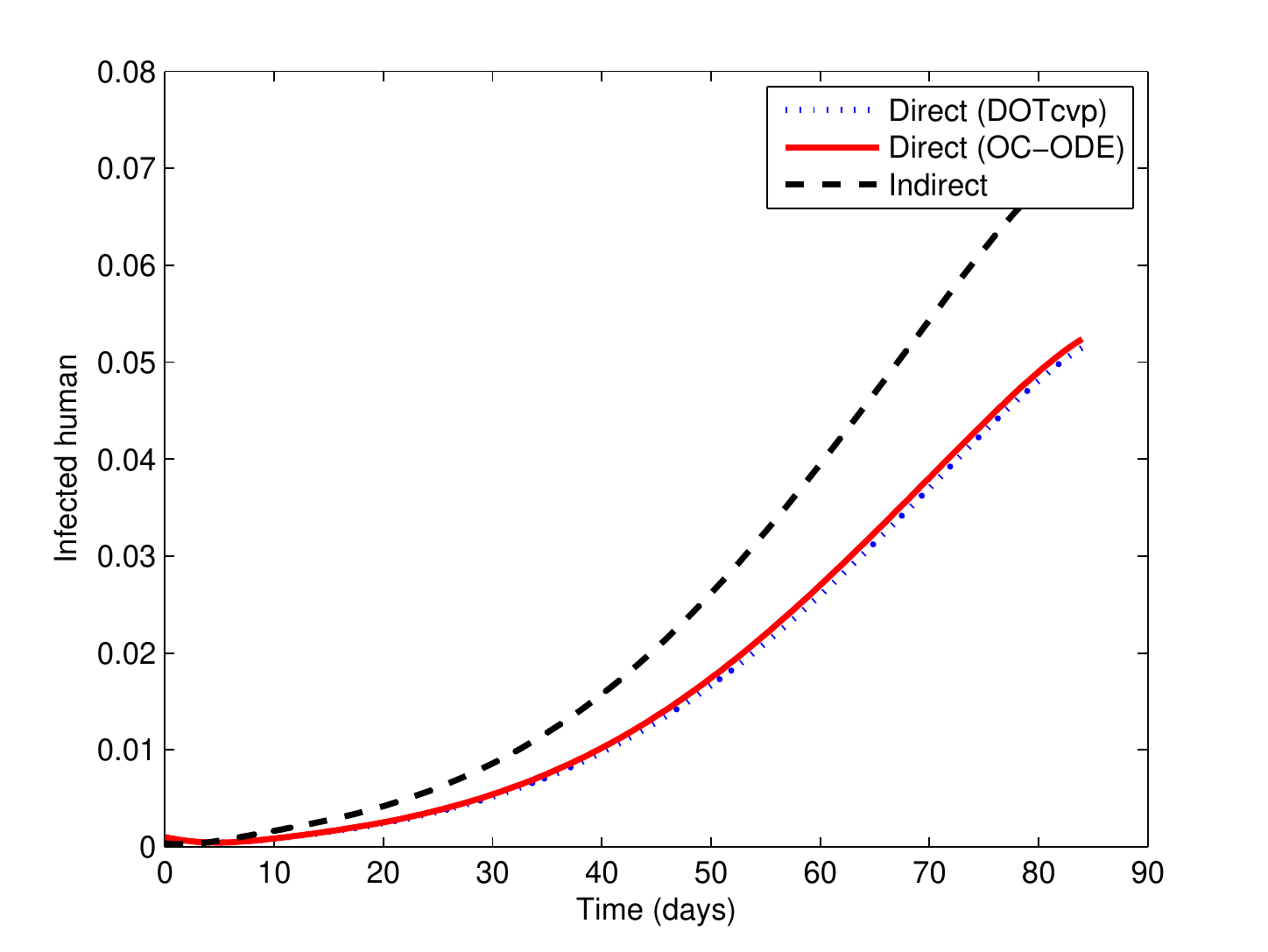}\\
{\caption{\label{cap5_3_infected_both_methods} Infected human obtained by the three methods}}
\end{figure}

The optimal functional is 0.0470, 0.0498 and 0.0445, for \texttt{DOTcvp}\index{DOTcvp},
\texttt{OC-ODE}\index{OC-ODE} and Backward-Forward, respectively. The last solver achieved
better value for the functional. This is expected, once its resolution contains more
information about the OC problem because the adjoint system is supplied.

The study considers three situations: A, B, and C. Situation A, that was previously presented,
regards both perspectives in the functional (human infected and insecticide application).
Situation B concerns only to infected humans whereas case C only considers insecticide campaigns.
Table~\ref{cases_B_C} resumes the three situations.

\begin{table}[ptbh]
\begin{center}
\small
\begin{tabular}{l l c c}\hline
& Perspective & Functional & Values for weights \\ \hline
Case A & Both perspectives &$\int_{0}^{t_f}\left[\gamma_D i_h(t)^2
+\gamma_S c(t)^2\right]dt$ & $\gamma_D=0.5$; $\gamma_S=0.5$\\
Case B & Medical perspective &$\int_{0}^{t_f}\gamma_D i_h(t)^2 dt$ & $\gamma_D=1$; $\gamma_S=0$\\
Case C & Economical perspective &$\int_{0}^{t_f}\gamma_S c(t)^2 dt$ & $\gamma_D=0$; $\gamma_S=1$\\ \hline
\end{tabular}
\caption{Functional costs}
\label{cases_B_C}
\end{center}
\end{table}

Since the three solvers presented similar solutions,
only one of them was chosen to solve these cases,
that was \texttt{DOTcvp}\index{DOTcvp}.

Figures~\ref{cap5_3_control_A_B_C}, \ref{cap5_3_infected_A_B_C}
and \ref{cap5_3_cost_A_B_C} show the results for optimal control $c$,
infected humans and total costs, respectively.

\begin{figure}[ptbh]
\center
\includegraphics [scale=0.6]{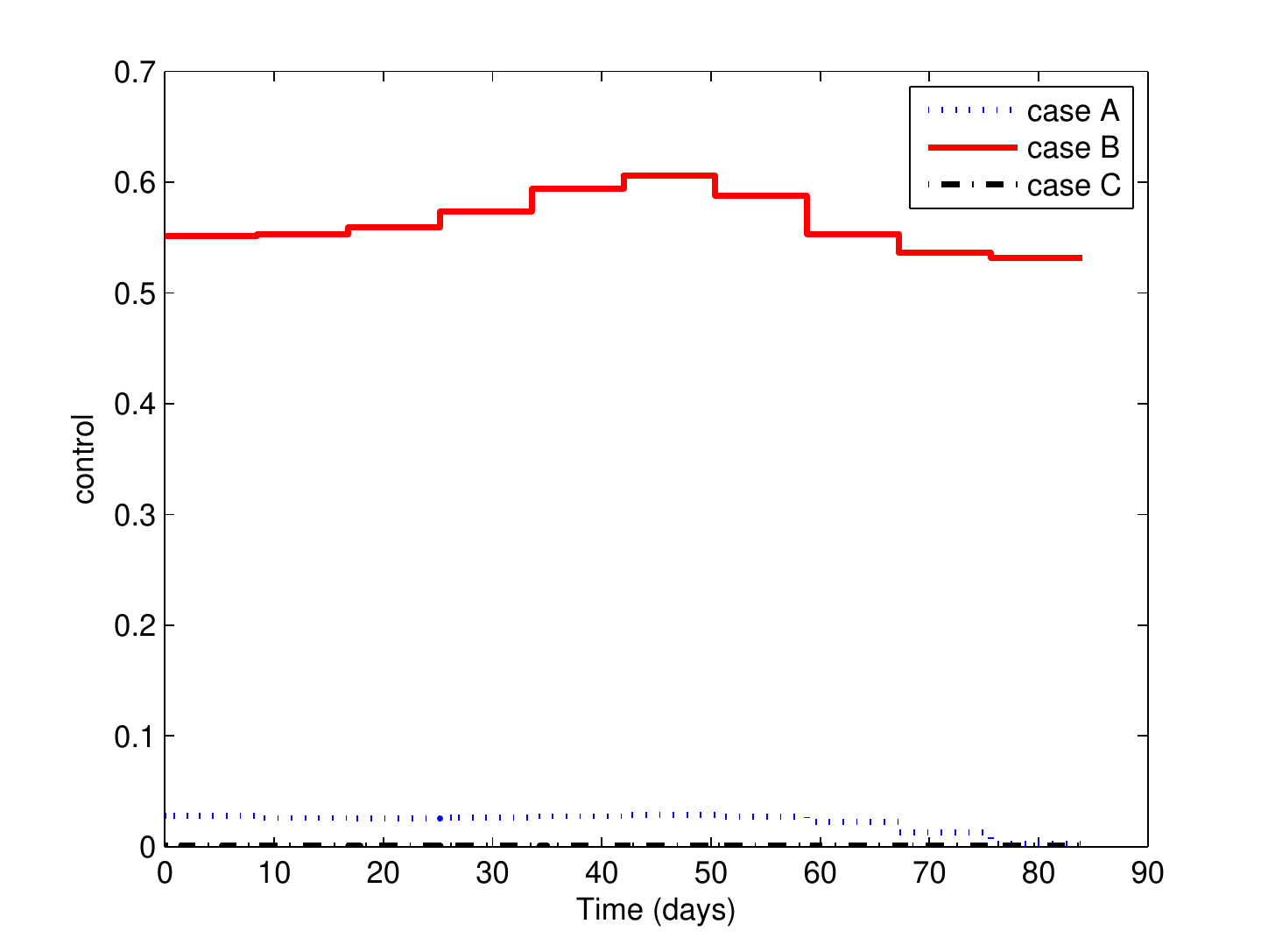}\\
{\caption{\label{cap5_3_control_A_B_C} Bioeconomic approaches for optimal control}}
\end{figure}

\begin{figure}[ptbh]
\center
\includegraphics [scale=0.6]{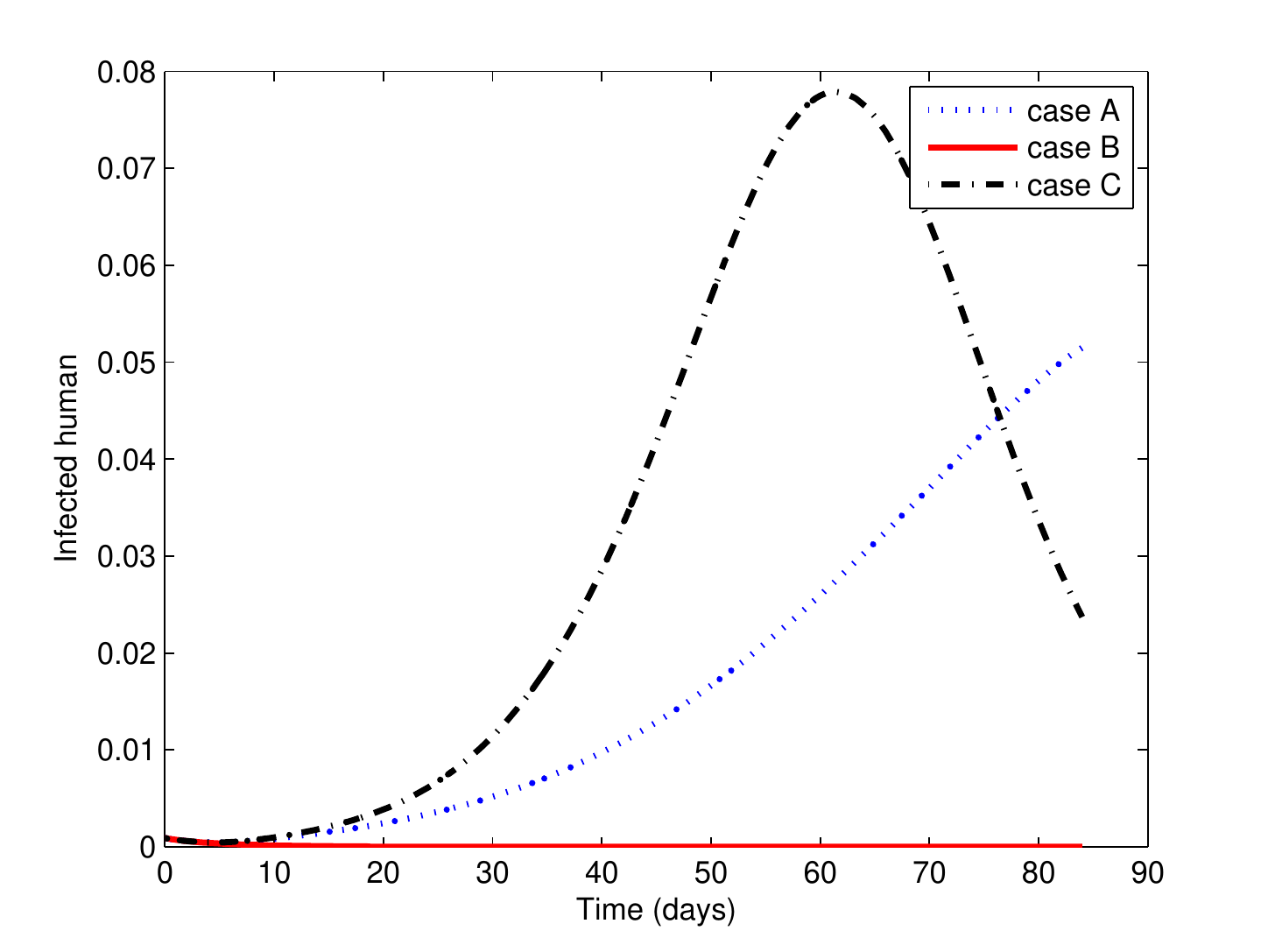}\\
{\caption{\label{cap5_3_infected_A_B_C} Bioeconomic approaches for infected humans}}
\end{figure}

\begin{figure}[ptbh]
\center
\includegraphics [scale=0.6]{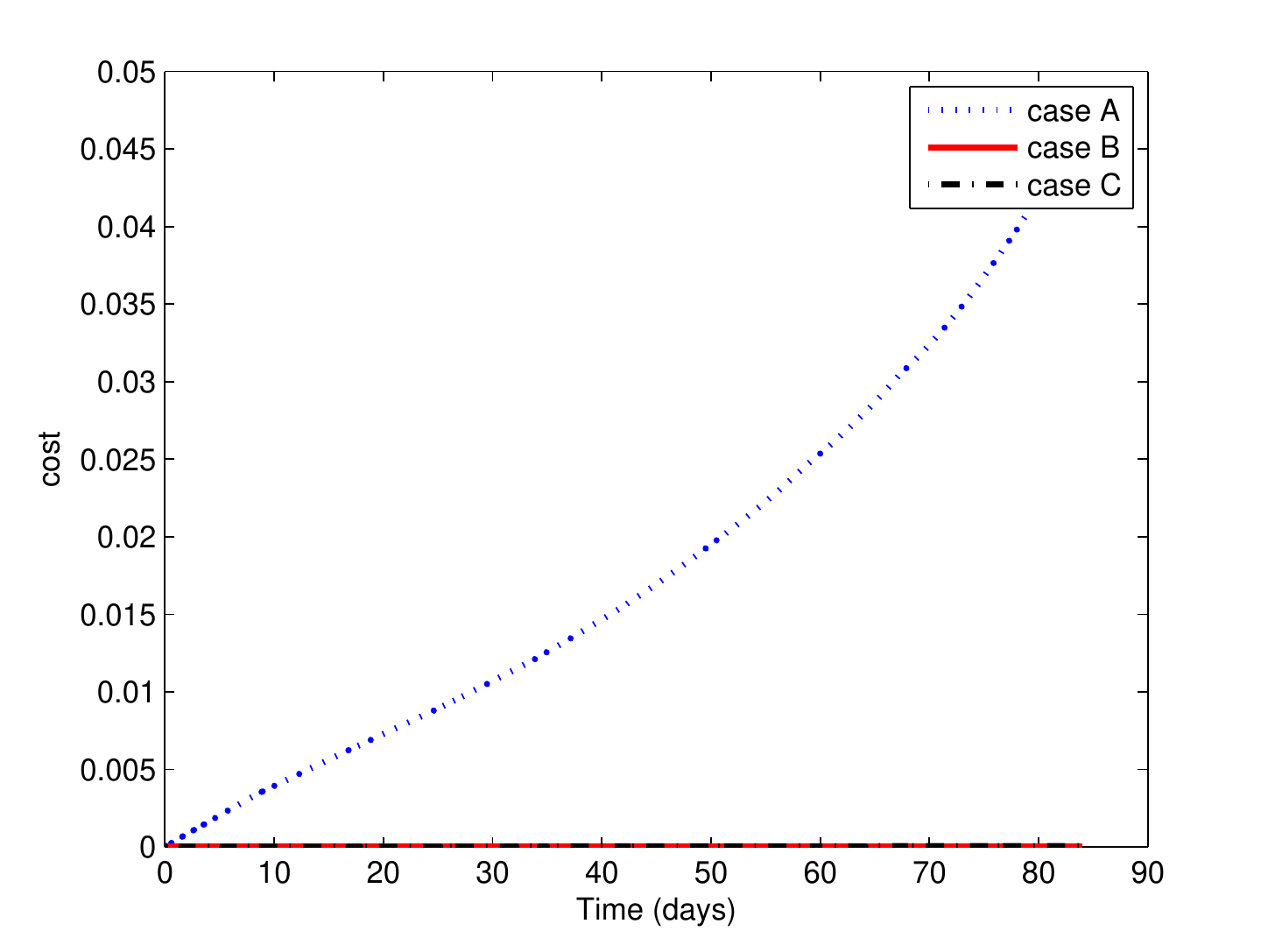}\\
{\caption{\label{cap5_3_cost_A_B_C} Bioeconomic approaches for total costs}}
\end{figure}

In the medical perspective (case B), when only the costs related with ill people
(absenteeism, drugs, ...) are considered, the number of infected  is the lowest;
however a huge quantity of insecticide is used, because it is considered cheaper.
On the other hand, when people just think on the economical perspective (case C),
the treatment for people is neglected. Optimal control is low, but the number
of infected humans is high. The total cost is higher when both perspectives are considered.

Epidemiological modelling has largely focused on identifying the mechanisms
responsible for epidemics but has taken little account on economic constraints
in analyzing control strategies. Economic models have given insight into optimal
control under constraints imposed by limited resources, but they frequently ignore
the spatial and temporal dynamics of the disease. Nowadays the combination
of epidemiological and economic factors is essential. The bioeconomic approach
requires an equilibrium between economic and epidemiological parameters in order
to give an efficient disease control and reflecting the nature of the epidemic.

For this, the study goes on implementing both perspectives, but taking into account
distinct weights of the parameters associated to the variables $i_h$ and $c$.
Table~\ref{cases_D_E} resumes these approaches.

\begin{table}[ptbh]
\begin{center}
\small
\begin{tabular}{l l c}
\hline
\multirow{2}{*}{} & \multirow{2}{*}{Bioeconomic perspective} & Values for weights\\
 & & $\int_{0}^{t_f}\left[\gamma_D i_h(t)^2 +\gamma_S c(t)^2\right]dt$ \\ \hline
Case A & Both perspectives (equal weights) & $\gamma_D=0.5$; $\gamma_S=0.5$\\
Case D & Perspective centered in insecticide  & $\gamma_D=0.2$; $\gamma_S=0.8$\\
Case E & Perspective centered in humans & $\gamma_D=0.8$; $\gamma_S=0.2$\\ \hline
\end{tabular}
\caption{Different weights for the functional}
\label{cases_D_E}
\end{center}
\end{table}

In case D, it is studied a situation where the lack of insecticide in a country could
be a reality and as a consequence its market value is high. This could happen due to an
unprecedent outbreak where the authorities were not prepared or even due to financial
reasons by the fact that the government does not have financial viability for this
kind of measure. In case E, once again the human perspective gains strength and,
as human life and quality of life is an expensive good, it was considered that it
is more expensive to treat humans then apply insecticide.

The analysis from the Figures from \ref{cap5_3_control_A_D_E} to \ref{cap5_3_cost_A_D_E}
is consistent with what we expected in reality. In case D, as insecticide is expensive,
the function for optimal control is lower than the other perspective.
As consequence, the number of infected people is higher.

In case E, where the human factor is preponderant, the number of infected humans
is low but expenses with insecticide are higher. Curiously the total cost,
in case D and E, are of the same order of magnitude, with a slightly higher cost for case E.
The total cost is reported in Table~\ref{functional_all_cases}.
When both perspectives are considered,
the total cost is higher than the single perspective.

\begin{figure}[ptbh]
\center
\includegraphics [scale=0.6]{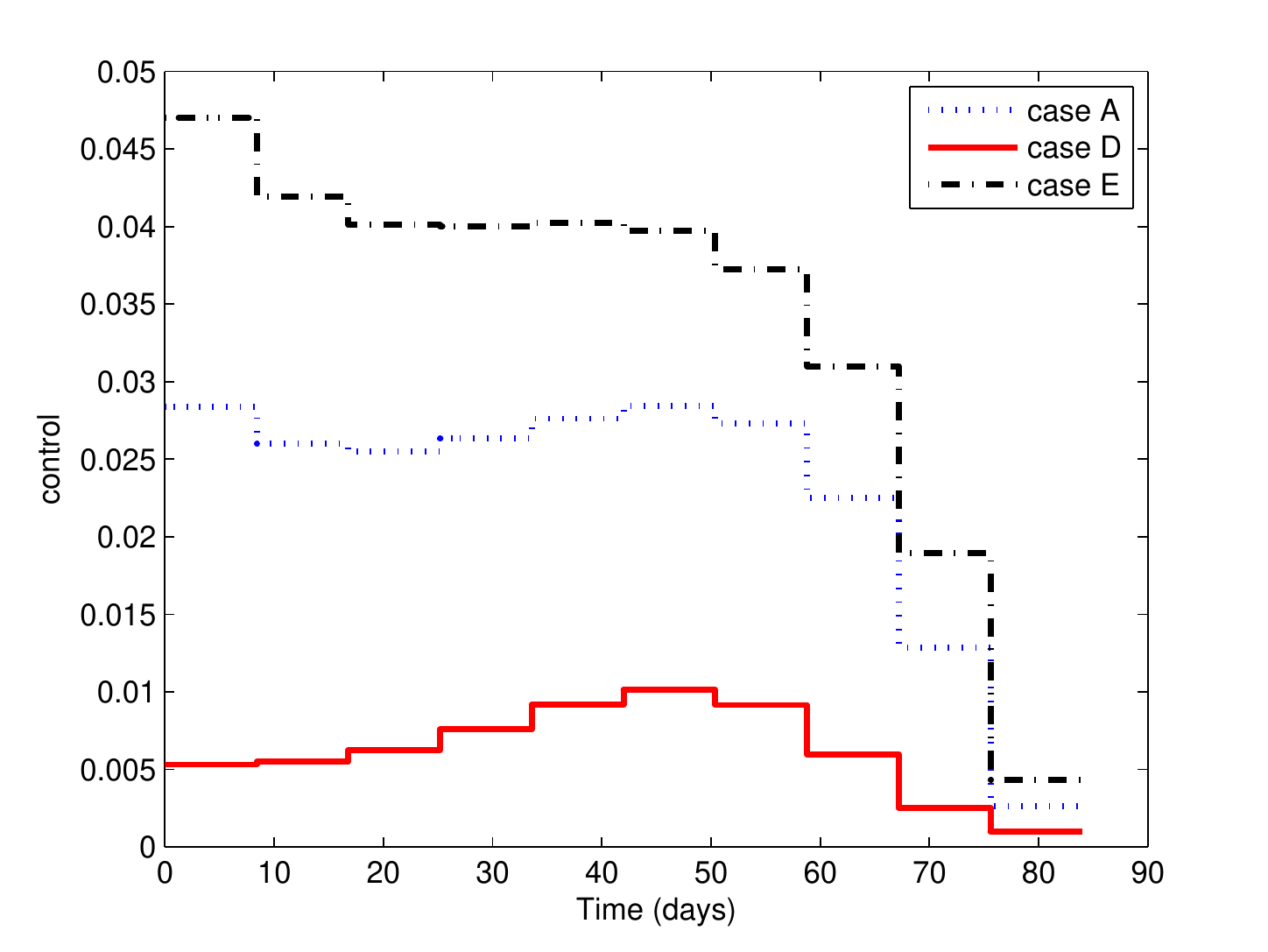}\\
{\caption{\label{cap5_3_control_A_D_E} Optimal control, using different weights in the functional}}
\end{figure}

\begin{figure}[ptbh]
\center
\includegraphics [scale=0.6]{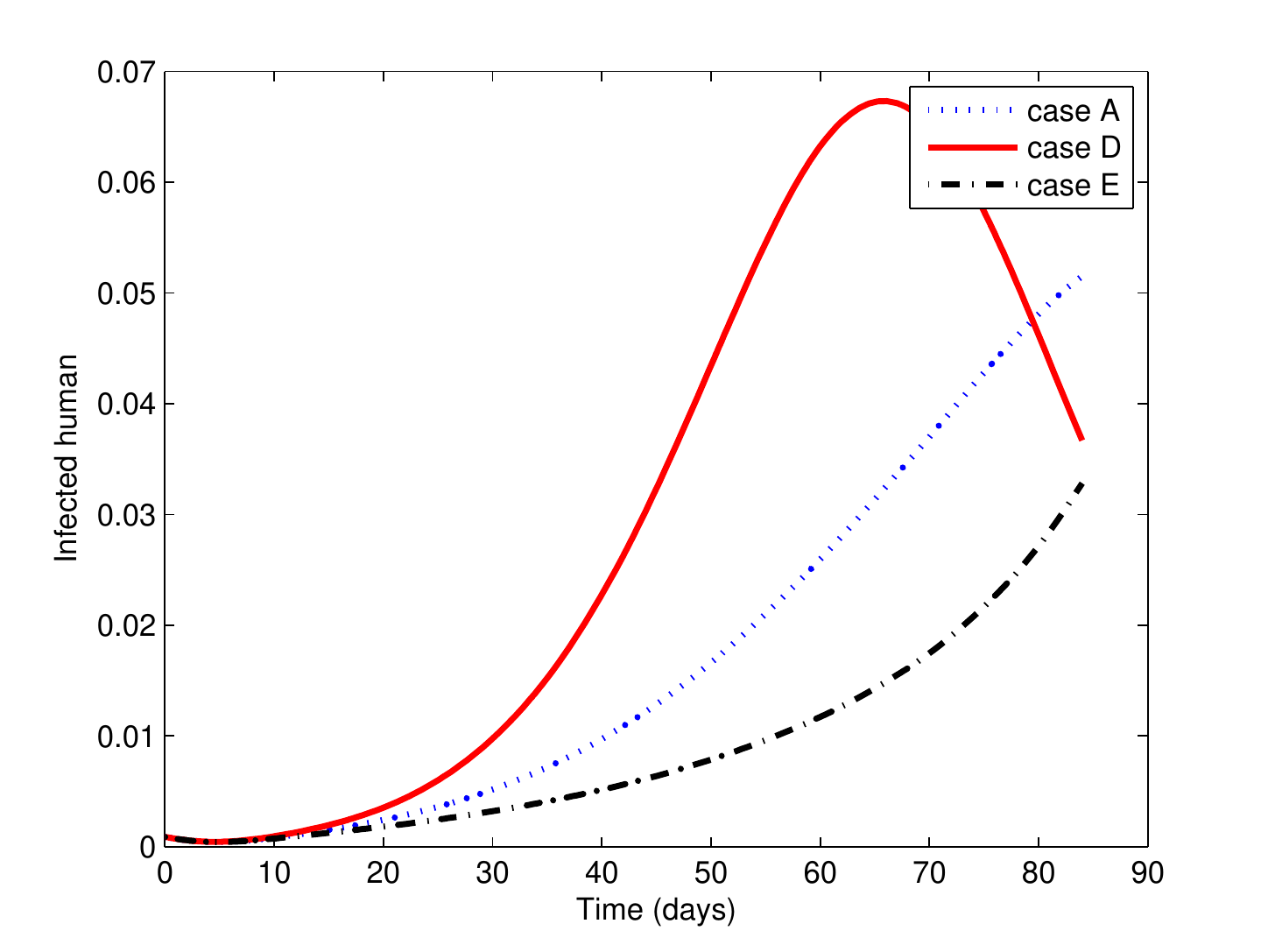}\\
{\caption{\label{cap5_3_infected_A_D_E} Infected human, using different weights in the functional}}
\end{figure}

\begin{figure}[ptbh]
\center
\includegraphics [scale=0.6]{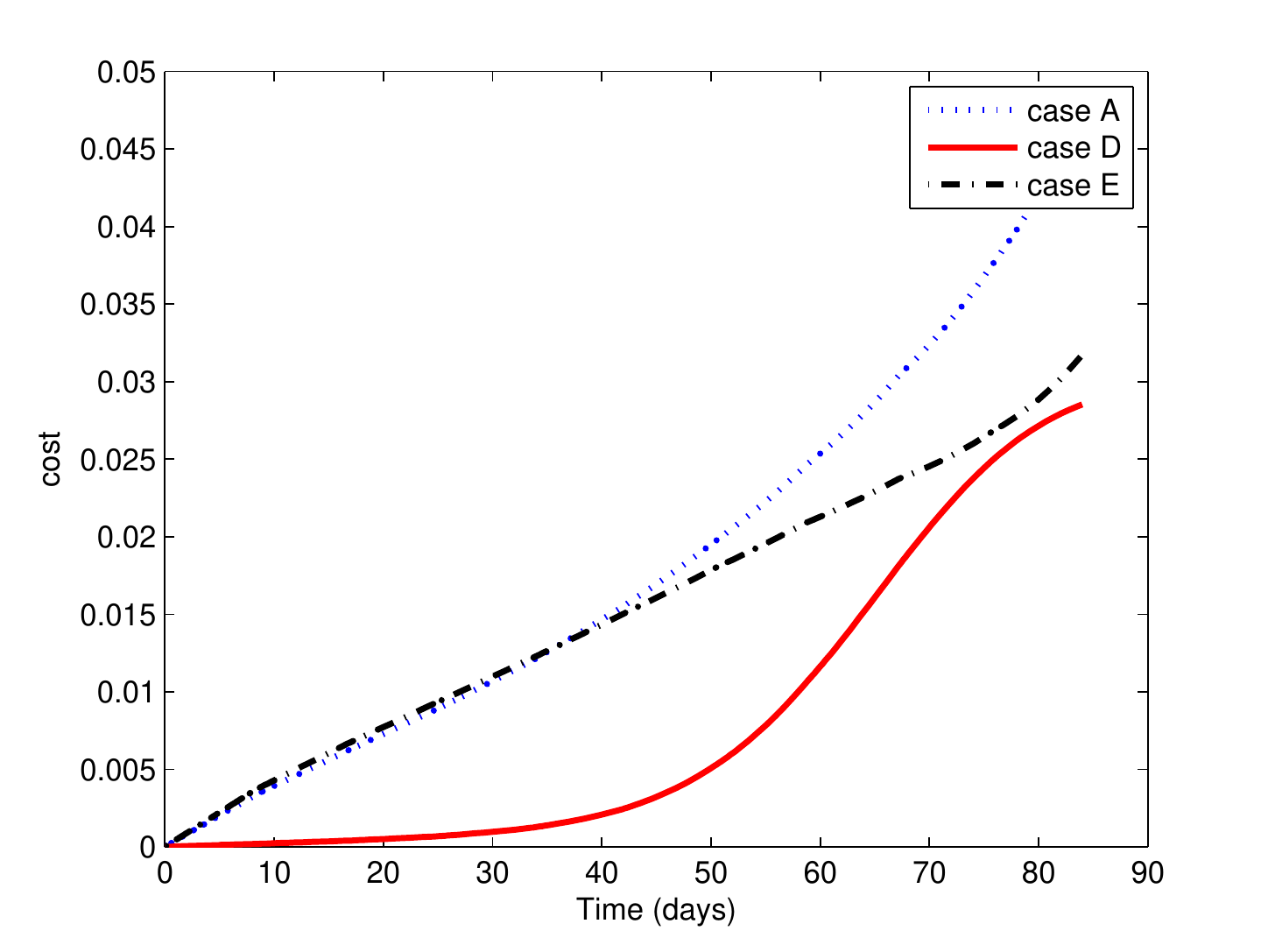}\\
{\caption{\label{cap5_3_cost_A_D_E} Total costs, using different weights in the functional}}
\end{figure}

\begin{table}[ptbh]
\begin{center}
\small
\begin{tabular}{l c}
\hline
Case & Total cost \\
\hline
Case A & 0.04700440 \\
Case B & 0.00000203 \\
Case C & 0.00004266 \\
Case D & 0.02852298 \\
Case E & 0.03172708 \\
\hline
\end{tabular}
\caption{Values for the functional in several cases}
\label{functional_all_cases}
\end{center}
\end{table}

For a last analysis, a mathematical perspective was carried out:
what values should the $\gamma_D$ and $\gamma_S$ have, in order
to minimize the functional? Let us call this perspective as Case F.
In this perspective, we not only want to minimize the control $c$,
but also the parameters $\gamma_D$ and $\gamma_S$,
enforcing the equality constraint $\gamma_D+\gamma_S=1$.

The values obtained from \texttt{DOTcvp}\index{DOTcvp} were $J=2.7\times 10^{-6}$,
$\gamma_D=9.9962\times10^{-1}$ and $\gamma_S=1.0863\times 10^{-10}$.

\begin{figure}[ptbh]
\center
\includegraphics [scale=0.6]{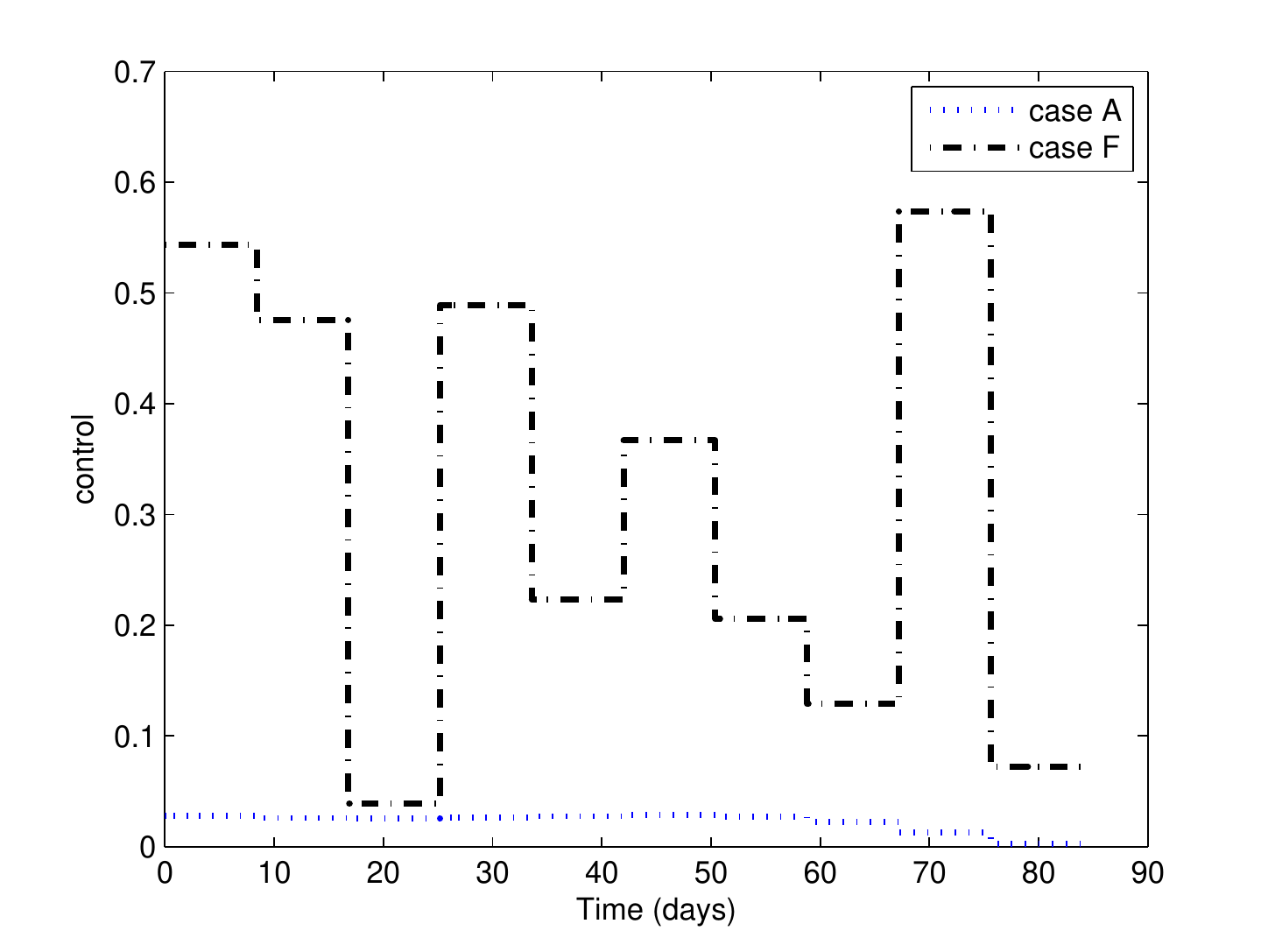}\\
{\caption{\label{cap5_3_control_A_parameters} Comparison of optimal control,
between initial and mathematical perspectives}}
\end{figure}

\begin{figure}[ptbh]
\center
\includegraphics [scale=0.6]{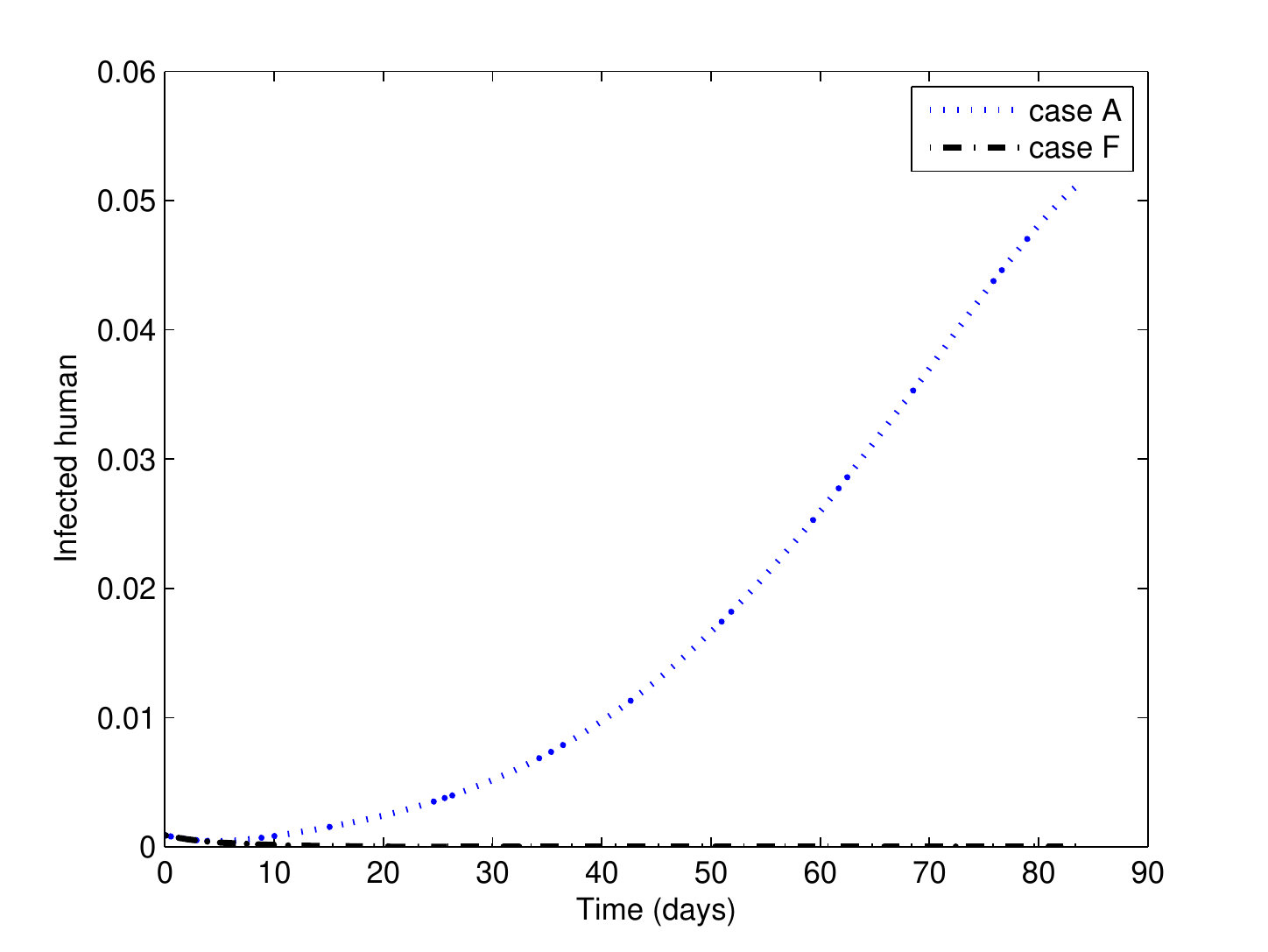}\\
{\caption{\label{cap5_3_infected_A_parameters} Comparison of infected humans,
between initial and mathematical perspectives}}
\end{figure}

\begin{figure}[ptbh]
\center
\includegraphics [scale=0.6]{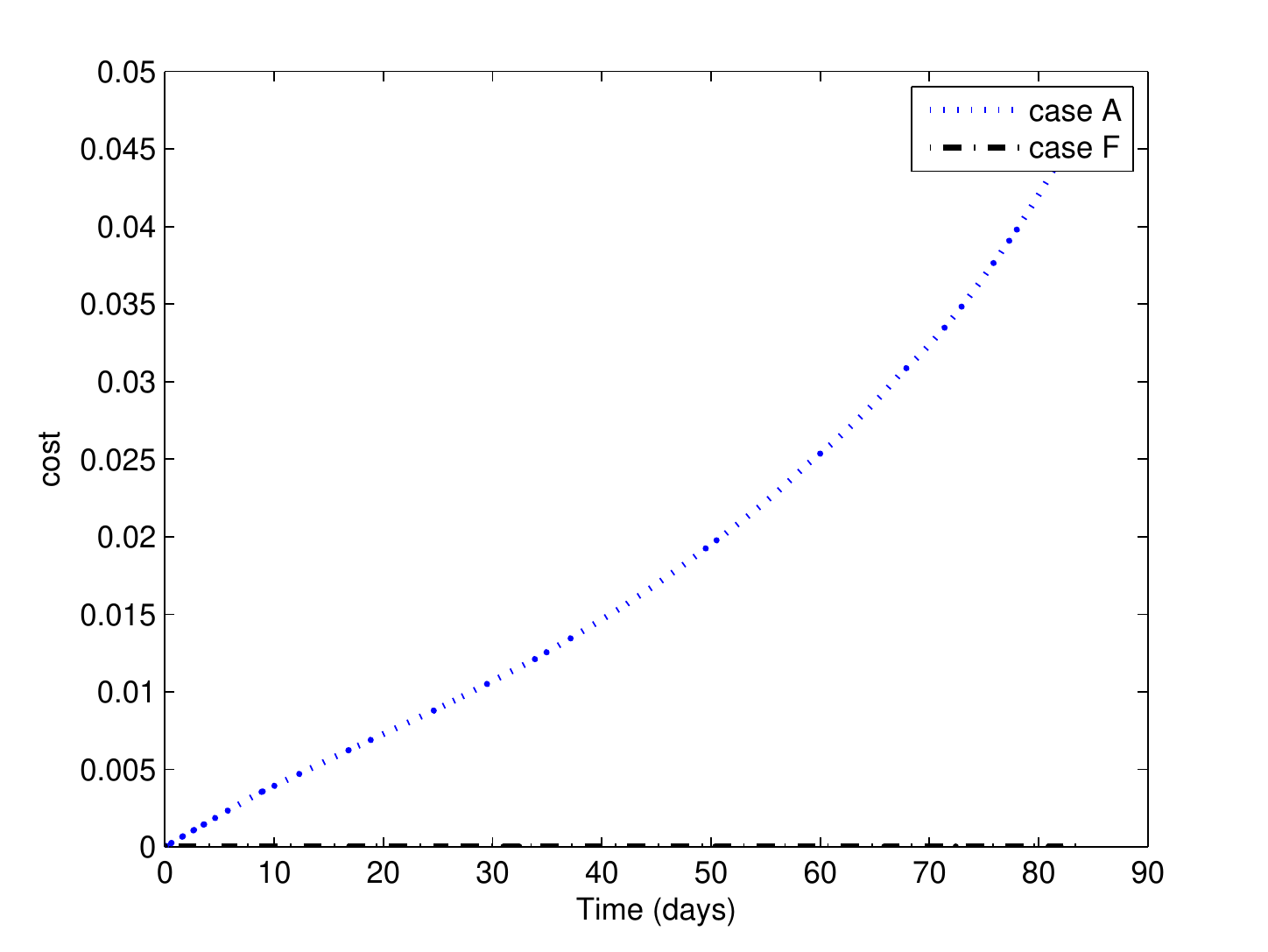}\\
{\caption{\label{cap5_3_cost_A_parameters} Comparison of total costs,
between initial and mathematical perspectives}}
\end{figure}

Figures~\ref{cap5_3_control_A_parameters} to \ref{cap5_3_cost_A_parameters}
show the comparison of this case with the first one. As expected, the total
cost and the number of infected humans are the lowest ones. Giving freedom
to the parameters it is possible to see that the optimal control function
is not periodic (as studied in the previous section),
but gives a practical solution in applying insecticide.


\section{Conclusions}
\label{sec:5:4}

In this chapter a model based on two populations, humans and mosquitoes,
with insecticide control is presented. It has shown that as time goes by,
depending on several parameters, the outbreak could disappear (leading to a DFE)\index{Disease Free Equilibrium}
or could transform the disease into an endemic one\index{Endemic Equilibrium} (leading to an EE).
This analysis can be made through the threshold basic reproduction number.

Assuming that the parameters are fixed, the only variable that can influence
this threshold is the control variable $c$, it has shown that with
a steady insecticide campaign it is possible to reduce the number
of infected humans and mosquitoes, and can prevent an outbreak
that could transform an epidemiological episode to an endemic disease.

For a steady campaign has proven that $c=0.157$ it is enough to maintain the
$\mathcal{R}_0$  below unit. However, this type of control is difficult to implement.
A pulse insecticide campaign was studied to circumvent this difficulty. It has proven
that applying insecticide every 6/7 days, is a better strategy to implement
by health authorities and has the same efficacy level and financial costs.

Finally, the OC theory was used to find the best optimal control function
for the insecticide. The optimal function varies, giving a different answer
depending on the main goal to reach, thinking in economical or human centered perspective.

As future work it is important to study different kinds of controls.
The accelerated increase in mosquito resistance to several
chemical insecticides and the damage caused by these to the
environment, has resulted in the search for new control
alternatives. Among the alternatives available,
the use of \emph{Bacillus thuringiensis israelensis (Bti)}
has been adopted by several countries \cite{ministeriodasaude2009}.
Laboratory testing shows that \emph{Bti} has a high larvicide
property and its mechanism of action is based on the production
of an endotoxin protein that, when ingested by the larvae, causes death.
To ensure the minimization of the outbreaks, educational
programmes that are customized for different levels of health care
and that reflect local capacity should be supported and
implemented widely. People should be instructed to minimize
the number of potential places for the mosquito breeding.
Educational campaigns can be included as an extra
control parameter in the model.

This chapter was based on work available in the peer reviewed journal \cite{Sofia2012}
and the peer reviewed conference proceedings \cite{Sofia2010c,Sofia2010a}.

\clearpage{\thispagestyle{empty}\cleardoublepage}


\chapter{An ODE SIR+ASI model with several controls}
\label{chp6}

\begin{flushright}
\begin{minipage}[r]{9cm}

\bigskip
\small {\emph{A new model with six mutually-exclusive compartments related
to Dengue disease is presented. In this model there are three vector control tools:
insecticides (larvicide and adulticide) and mechanical control. The human data
for the model is again related to Cape Verde. Due to the rapid development
of the outbreak on the islands, only a few control measures were taken, but not quantified.
In this chapter, some of these measures are simulated and their consequences are analyzed.}}

\bigskip

\hrule
\end{minipage}
\end{flushright}

\bigskip
\bigskip

\onehalfspacing

In Chapter~\ref{chp5}, a model with eight compartments and a single control was analyzed.
However, after discussion with some researchers in this area, many of them suggested
removing the compartment \emph{exposed} for three main reasons: first, it is difficult
to collect data for this compartment, since the disease at this stage does not show symptoms;
second, the curve obtained is similar to the infected compartment with only an advance of time,
not bringing novelty to the model, but possible difficulties to the numeric resolution;
and finally, as the main goal is to study the effects of several controls centered
in the infected humans, this compartment plays a secondary role. Thus, it was decided
to remove the exposed compartments, in humans and mosquitoes, adjusting the other parameters
to this new model and including three controls.


\section{The SIR+ASI model}
\label{sec:6:1}

Taking into account the model presented in
\cite{Dumont2010,Dumont2008} and the considerations
of \cite{Sofia2009}-\cite{Sofia2010c}, a new model more adapted
to the Dengue reality is proposed. The
notation used in the mathematical model includes three
epidemiological states for humans:

\begin{quote}
\begin{tabular}{ll}
$S_h(t):$ & susceptible \\
$I_h(t):$ & infected \\
$R_h(t):$ & resistant
\end{tabular}
\end{quote}

It is assumed that the total human population $(N_h)$
is constant and $N_h=S_h(t)+I_h(t)+R_h(t)$ at any time $t$.
The population is homogeneous, which means that every individual
of a compartment is homogeneously mixed with the other individuals.
Immigration and emigration are not considered.

There are three other state variables, related
to the female mosquitoes:

\begin{quote}
\begin{tabular}{ll}
$A_m(t):$& aquatic phase \\
$S_m(t):$& susceptible \\
$I_m(t):$& infected \\
\end{tabular}
\end{quote}

Due to the short lifespan of mosquitoes, there is no resistant phase.
It is assumed homogeneity between host and vector populations,
which means that each vector has an equal probability to bite any host.
Humans and mosquitoes are assumed to be born susceptible.

To analyze the effect of campaigns in the disease fight, three controls are considered:

\begin{quote}
\label{eq:star:3}
\begin{tabular}{ll}
$c_{A}(t):$ & proportion of larvicide, $0\leq c_A\leq1$ \\
$c_{m}(t):$ & proportion of adulticide, $0\leq c_m\leq1$\\
$\alpha(t): $ & proportion of mechanical control, $0 < \alpha \leq1$.\\
\end{tabular}
\end{quote}

Larval control targets the immature mosquitoes living in water
before they become biting adults.  A soil
bacterium, \emph{Bacillus thuringiensis israelensis} (Bti), is applied
from the ground or by air to larval habitats. This bacterium is
used because when properly applied, it has virtually no effect
on non-target organisms.

The control of adult mosquitoes is necessary when mosquito populations
cannot be treated in their larval stage. It is the most
effective way to eliminate adult female mosquitoes that are
infected with human pathogens. Depending on the size of the
area to be treated, either trucks for ground adulticide
treatments or aircraft for aerial adulticide treatments
can be used.

The purpose of mechanical control is to reduce the number of
larval habitat areas available to mosquitoes. The mosquitoes are
most easily controlled by treating, cleaning and/or emptying containers that
hold water, since the eggs of the specie are laid in water-holding
containers.

The aim is to simulate different realities in order
to find the best policy to decrease the number of infected humans.
A temporal mathematical model is introduced, with mutually-exclusive compartments,
to study the outbreak occurred on Cape Verde islands in 2009 and
improving the model described in \cite{Sofia2009}.

The model uses the following parameters:

\begin{quote}\label{eq:star:4}
\begin{tabular}{ll}
$N_h: $ & total population\\
$B:$ & average daily biting (per day)\\
$\beta_{mh}:$ & transmission probability from $I_m$ (per bite)\\
$\beta_{hm}:$ & transmission probability from $I_h$ (per bite)\\
$1/\mu_{h}:$ & average lifespan of humans (in days)\\
$1/\eta_{h}:$ & mean viremic period (in days)\\
$1/\mu_{m}:$ & average lifespan of adult mosquitoes (in days)\\
$\varphi:$ & number of eggs at each deposit per capita (per day)\\
$1/\mu_{A}:$ & natural mortality of larvae (per day)\\
$\eta_{A}:$ & maturation rate from larvae to adult (per day)\\
$m:$ & female mosquitoes per human\\
$k:$ & number of larvae per human
\end{tabular}
\end{quote}

The Dengue epidemic is modelled by the following nonlinear
time-varying state equations:
\begin{equation}
\label{cap6_ode1}
\begin{cases}
\displaystyle\frac{dS_h}{dt} = \mu_h N_h - \left(B\beta_{mh}\frac{I_m}{N_h}+\mu_h\right)S_h\\
\displaystyle\frac{dI_h}{dt} = B\beta_{mh}\frac{I_m}{N_h}S_h -(\eta_h+\mu_h) I_h\\
\displaystyle\frac{dR_h}{dt} = \eta_h I_h - \mu_h R_h
\end{cases}
\end{equation}
and
\begin{equation}
\label{cap6_ode2}
\begin{cases}
\displaystyle\frac{dA_m}{dt} = \varphi \left(1-\frac{A_m}{\alpha k N_h}\right)(S_m+I_m)
-\left(\eta_A+\mu_A + c_A\right) A_m\\
\displaystyle\frac{dS_m}{dt} = \eta_A A_m
-\left(B \beta_{hm}\frac{I_h}{N_h}+\mu_m + c_m\right) S_m\\
\displaystyle\frac{dI_m}{dt} = B \beta_{hm}\frac{I_h}{N_h}S_m
-\left(\mu_m + c_m\right) I_m
\end{cases}
\end{equation}
with the initial conditions
\begin{equation*}
\begin{tabular}{llll}
$S_h(0)=S_{h0},$ &  $I_h(0)=I_{h0},$ &
$R_h(0)=R_{h0},$ \\
$A_m(0)=A_{m0},$ & $S_{m}(0)=S_{m0},$ & $I_m(0)=I_{m0}.$
\end{tabular}
\end{equation*}
Figure~\ref{cap6_model} shows a scheme of the model.
\begin{figure}[ptbh]
\begin{center}
\includegraphics[scale=0.5]{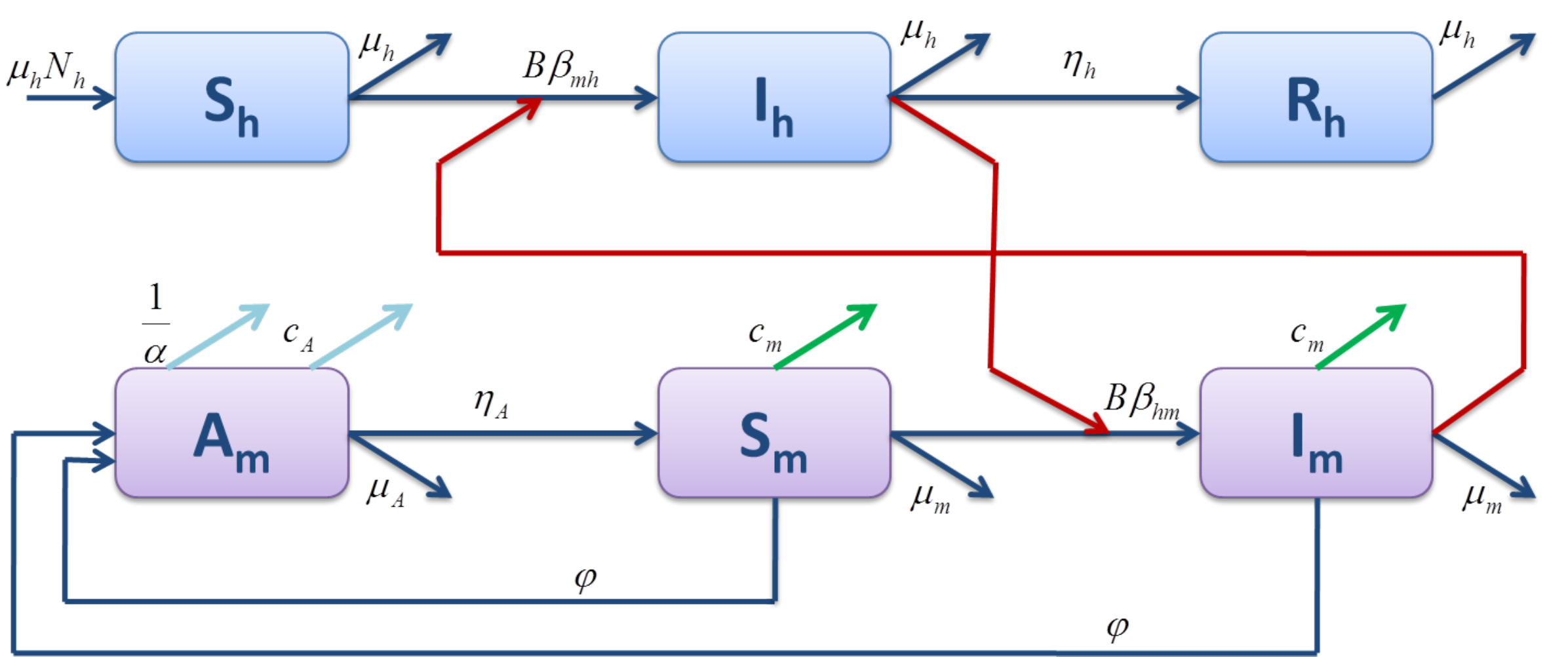}
\end{center}
\caption{Epidemiological model SIR + ASI \label{cap6_model}}
\end{figure}


\subsection*{Basic reproduction number, equilibrium points and stability}
\label{sec:6:1:1}\index{Equilibrium point}

Due to biological reasons, only nonnegative solutions of the differential system are acceptable.
More precisely, it is necessary to study the solution properties
of the system \eqref{cap6_ode1}--\eqref{cap6_ode2} in the closed set

\begin{equation*}
\Omega=\left\{(S_h,I_h,R_h,A_m,S_m,I_m)\in \mathbb{R}_{+}^{6}:
S_h+I_h+R_h\leq N_h,\,  A_m\leq k N_h, \, S_m+I_m \leq mN_h\right\}.
\end{equation*}
It can be verified that $\Omega$ is a positively invariant set
with respect to \eqref{cap6_ode1}--\eqref{cap6_ode2}. The proof
of this statement is similar as
in \cite{Sofia2012}. The system
\eqref{cap6_ode1}--\eqref{cap6_ode2} has at most three biologically
meaningful equilibrium points (\textrm{cf.} Theorem~\ref{thm:thm1}).

\medskip

\begin{definition}[Equilibrium points for SIR+ASI model]
A sextuple $E = \left(S_h,I_h,R_h,A_m,S_m,I_m\right)$
is said to be an \emph{equilibrium point} \index{Equilibrium point}for system
\eqref{cap6_ode1}--\eqref{cap6_ode2} if it satisfies the following relations:
\begin{equation}
\label{equilibrio}
\begin{cases}
\mu_h N_h - \left(B\beta_{mh}\frac{I_m}{N_h}+\mu_h\right)S_h=0\\
B\beta_{mh}\frac{I_m}{N_h}S_h -(\eta_h+\mu_h) I_h=0\\
\eta_h I_h - \mu_h R_h=0\\
\varphi \left(1-\frac{A_m}{\alpha k N_h}\right)(S_m+I_m)-(\eta_A+\mu_A+c_A) A_m=0\\
\eta_A A_m - \left(B \beta_{hm}\frac{I_h}{N_h}+\mu_m + c_m\right) S_m=0\\
B \beta_{hm}\frac{I_h}{N_h}S_m -(\mu_m + c_m) I_m=0.
\end{cases}
\end{equation}

An equilibrium point $E$ is \emph{biologically meaningful}
if and only if $E \in \Omega$. The biologically meaningful equilibrium points
are said to be disease free or endemic depending on $I_h$ and $I_m$:
if there is no disease
for both populations of humans and mosquitoes ($I_h=I_m=0$), then the equilibrium point
is said to be a \emph{Disease Free Equilibrium} (DFE);\index{Disease Free Equilibrium}
otherwise, if $I_h \ne 0$ or $I_m \ne 0$ (in other words, if $I_h > 0$ or $I_m > 0$),
then the equilibrium point is called \emph{endemic}\index{Endemic Equilibrium}.
\end{definition}

\medskip

\begin{theorem}
\label{thm:thm1}
System \eqref{cap6_ode1}--\eqref{cap6_ode2} admits
at most three biologically meaningful equilibrium points;
at most two DFE points\index{Disease Free Equilibrium}
and at most one endemic equilibrium point\index{Endemic Equilibrium}.
More precisely, let
$$
\mathcal{M} =-\left(\eta_A \mu_m+\eta_A c_m+\mu_A \mu_m+\mu_A
c_m+c_A \mu_m+c_A c_m-\varphi \eta_A\right),
$$
$$
\xi = \varphi(\mu_m+c_m)^{2}(\eta_h+\mu_h), \quad
\chi = \alpha k B^2 \beta_{hm}\beta_{mh}\mathcal{M},
$$
$$
E_{1}^{*}=\left(N_h,0,0,0,0,0\right), \quad
E_{2}^{*}=\left(N_h,0,0,\frac{\alpha
k N_h \mathcal{M}}{\eta_A\varphi}, \frac{\alpha k N_h
\mathcal{M}}{(\mu_m+c_m) \varphi},0\right),
$$
and
$E_{3}^{*}=\left(S_h^*,I_h^*,R_h^*,A_m^*,S_m^*,I_m^*\right)$ with
\begin{equation}
\label{eqE3:t2}
\begin{split}
S_{h}^{*} &= \frac{ -\varphi
N_h(\mu_m\eta_h+\mu_h B
\beta_{hm}+c_m\mu_h+c_m\eta_h+\mu_m\mu_h)(\mu_m+c_m)}{B\beta_{hm}(-\alpha
k B \beta_{mh}\mathcal{M}-\varphi\mu_h(\mu_m+c_m))},\\
I_{h}^{*} &=\frac{\mu_h
N_h\left(\xi-\chi\right)}{(\eta_h+\mu_h)B\beta_{hm}(-\alpha
k B \beta_{mh}\mathcal{M}-\varphi\mu_h(\mu_m+ c_m))},\\
R_{h}^{*} &= \frac{\eta_h
N_h\left(\xi-\chi\right)}{(\eta_h+\mu_h)B\beta_{hm}\left(-\alpha
k B \beta_{mh}\mathcal{M}-\varphi \mu_h(\mu_m+c_m)\right)},\\
A_{m}^{*} &= \frac{N_h k \alpha \mathcal{M}}{\varphi \eta_A},\\
S_{m}^{*} &=\frac{N_h\mu_h(c_m+\mu_h)(\mu_h+\eta_h)}{B \beta_{mh}(\mu_m \eta_h+\mu_h B
\beta_{hm}+c-m\mu_h+c_m\eta_h+\mu_m\mu_h)}\\
&\qquad -\frac{N_h\alpha k \left(c_m(\mu_h+\eta_h)(\mu_A+\eta_A+c_A)
+\mu_m(\mu_h(\eta_A+\mu_A)+\eta_{h}(c_A+\eta_A))\right)}{\varphi(\mu_m \eta_h+\mu_h B
\beta_{hm}+c-m\mu_h+c_m\eta_h+\mu_m\mu_h)}\\
&\qquad -\frac{N_h\alpha k \left(-\eta
\varphi (\mu_h+\eta_h)+\mu_m(\eta_h\mu_A+\mu_h
c_A)\right)}{\varphi(\mu_m \eta_h+\mu_h B
\beta_{hm}+c-m\mu_h+c_m\eta_h+\mu_m\mu_h)},\\
I_{m}^{*} &=\frac{-\mu_h N_h\left(\xi-\chi\right)}{B\beta_{mh}\left(
\varphi B\beta_{hm}\mu_h(\mu_m+c_m)+\xi\right)}.
\end{split}
\end{equation}
If $\mathcal{M} \le 0$, then there is only
one biologically meaningful equilibrium point, $E_{1}^{*}$,
which is a DFE point\index{Disease Free Equilibrium}. If $\mathcal{M} > 0$ with $\xi \ge \chi$,
then there are two biologically meaningful equilibrium points,
$E_{1}^{*}$ and $E_{2}^{*}$, both DFE points.
If $\mathcal{M} > 0$ with $\xi < \chi$,
then there are three biologically meaningful equilibrium points,
$E_{1}^{*}$, $E_{2}^{*}$, and $E_{3}^{*}$,
where $E_{1}^{*}$ and $E_{2}^{*}$ are DFEs
and $E_{3}^{*}$ endemic\index{Endemic Equilibrium}.
\end{theorem}

\medskip

\begin{proof}
System \eqref{equilibrio} has four solutions
easily obtained with a Computer
Algebra System like \texttt{Maple}\index{Maple}:
$E_{1}^{*}$, $E_{2}^{*}$, $E_{3}^{*}$ and $E_{4}^{*}$.
The equilibrium point $E_{1}^{*}$ is always a DFE
because it always belongs to $\Omega$ with $I_h=I_m=0$.
In contrast, $E_{4}^{*}$ is never biologically realistic
because it has always some negative coordinates.
The other two equilibrium points, $E_{2}^{*}$ and $E_{3}^{*}$,
are biologically realistic only for certain values of the parameters.
The equilibrium $E_{2}^{*}$ is biologically realistic if and only if
$\mathcal{M} \ge 0$, in which case it is a DFE. For the
condition $\mathcal{M} \le 0$, the third
equilibrium $E_{3}^{*}$ is not biologically realistic.
If $\mathcal{M}>0$, then three situations can occur
with respect to $E_{3}^{*}$: if
$\xi = \chi$, then $E_{3}^{*}$ degenerates into $E_2^*$,
which means that $E_{3}^{*}$ is the DFE $E_{2}^{*}$;
if $\xi > \chi$, then $E_{3}^{*}$ is not biologically realistic; otherwise,
one has $E_{3}^{*} \in \Omega$ with $I_h \ne 0$ and $I_m \ne 0$, which means that
$E_{3}^{*}$ is an endemic equilibrium point.
\end{proof}

By algebraic manipulation, $\mathcal{M}>0$ is equivalent to condition
$\displaystyle \frac{(\eta_A+\mu_A+c_A)(\mu_m+c_m)}{\varphi\eta_A}>1$,
which is related to the basic offspring number for mosquitos. Thus,
if $\mathcal{M} \leq 0$, then the mosquito population will collapse and the only
equilibrium for the whole system is the trivial DFE $E_{1}^{*}$.
If $\mathcal{M} > 0$, then the mosquito population is sustainable. From a
biological standpoint, the equilibrium $E_{2}^{*}$ is more plausible,
because the mosquito is in its habitat, but without the disease.

An important measure of transmissibility of the disease is now introduced:
the basic reproduction number\index{Basic reproduction number}.
It provides an invasion criterion for the initial spread
of the virus in a susceptible population. For this case
the following result holds.

\medskip

\begin{theorem}
\label{thm:r0}
The basic reproduction number\index{Basic reproduction number}
$\mathcal{R}_0$ associated to the
differential system \eqref{cap6_ode1}--\eqref{cap6_ode2} is
\begin{equation}
\label{eq:R0}
\mathcal{R}_0 = \left(\frac{\alpha k B^2 \beta_{hm} \beta_{mh}
\mathcal{M}}{\varphi (\eta_h + \mu_h) (c_m + \mu_m)^2}\right)^{\frac{1}{2}}
= \left(\frac{\chi}{\xi}\right)^{\frac{1}{2}}.
\end{equation}
\end{theorem}

\medskip

\begin{proof}
In agreement with \cite{Driessche2002}, just the
epidemiological compartments that have new infections,
$I_h$ and $I_m$, are considered. The two differential equations related
to these two compartments can be rewritten as
$\dfrac{dx}{dt}=\mathcal{F}-\mathcal{V}$, where $\mathcal{F}$ is
the rate of production of new infections and $\mathcal{V}$ is the
transition rates between states:
\begin{equation*}
\mathcal{F}(x)
=\left(
\begin{array}{c}
B \beta_{mh} \frac{I_{m}}{N_h}S_{h}  \\
B \beta_{hm} \frac{I_{h}}{N_h}S_{m}\\
\end{array}
\right), \quad
\mathcal{V}(x)=\left(
\begin{array}{c}
(\eta_h+\mu_h)I_{h}  \\
(c_m+\mu_m)I_{m} \\
\end{array}
\right).
\end{equation*}
The Jacobian derivatives are
\begin{equation*}
J_{\mathcal{F}(x)}
=\left(
\begin{array}{cc}
0 &  B \beta_{mh} \frac{S_{h}}{N_h}  \\
B \beta_{hm} \frac{S_{m}}{N_h} & 0\\
\end{array}
\right), \quad
J_{\mathcal{V}(x)}=\left(
\begin{array}{cc}
(\eta_h+\mu_h) & 0  \\
0 & (c_m+\mu_m)\\
\end{array}
\right).
\end{equation*}

The quantity $J_{\mathcal{F}(x)} J_{\mathcal{V}(x)}^{-1}$
gives the total production of new infections over the
course of an infection. The largest eigenvalue gives the fastest
growth of the infected population, which means that $\mathcal{R}_0$ is the
spectral radius of the matrix $J_{\mathcal{F}(x)} J_{\mathcal{V}(x)}^{-1}$
in a DFE point. \texttt{Maple}\index{Maple} was used to obtain
\begin{equation}
\label{eq:r0:m}
\mathcal{R}_0 = \left(\frac{B^2 \beta_{hm} \beta_{mh} S_{h_{DFE}}
S_{m_{DFE}}}{(\eta_h + \mu_h) (c_m + \mu_m) N_h^2}\right)^{\frac{1}{2}}.
\end{equation}
The basic reproduction number $\mathcal{R}_0$ in \eqref{eq:R0} is obtained,
replacing $S_{h_{DFE}}$ and $S_{m_{DFE}}$ in \eqref{eq:r0:m}
by those of the DFE $E_2^*$.
\end{proof}

The model has two different populations (host and vector)
and the expected basic reproduction number should reflect the
infection transmitted from host to vector and vice-versa.
Accor\-dingly, $\mathcal{R}_0$ can be seen as
$\mathcal{R}_0=(\mathcal{R}_{hm}\times\mathcal{R}_{mh})^{\frac{1}{2}}$. The
infection host-vector is represented by
$\mathcal{R}_{hm}=\frac{B \beta_{hm}S_{m_{DFE}}}{N_h(\eta_h +
\mu_h)}$, where the term $\frac{B \beta_{hm}S_{m_{DFE}}}{N_h}$
represents the product between the transmission probability of the
disease from humans to mosquitoes in a susceptible population of
vectors, and the term $\frac{1}{\eta_h + \mu_h}$ the
human viremic period. Analogously, the infection vector-host
is composed by $\mathcal{R}_{mh}
= \frac{B \beta_{mh}S_{h_{DFE}}}{N_h(c_m + \mu_m)}$,
where $\frac{B \beta_{mh}S_{h_{DFE}}}{N_h}$
describes the transmission of the disease from mosquito
to the susceptible human population, and $\frac{1}{c_m+\mu_m}$
the lifespan of an adult mosquito.

If $\mathcal{R}_0<1$, then, on average, an infected individual
produces less than one new infected individual over the course of
its infectious period, and the disease cannot grow. Conversely, if
$\mathcal{R}_0>1$, then each individual infects more than one
person, and the disease can invade the population.

\medskip

\begin{theorem}
\label{thm:thm3}
If $\mathcal{M}>0$ and $\mathcal{R}_{0}>1$,
then the system \eqref{cap6_ode1}--\eqref{cap6_ode2}
admits the endemic equilibrium\index{Endemic Equilibrium}
$E_{3}^{*}=\left(S_h^*,I_h^*,R_h^*,A_m^*,S_m^*,I_m^*\right)$
given by \eqref{eqE3:t2}.
\end{theorem}

\medskip

\begin{proof}
The only solution of \eqref{equilibrio}
with $I_h > 0$ or $I_m > 0$, the only endemic equilibrium, is $E_3^*$.
That occurs, in agreement with Theorem~\ref{thm:thm1},
in the case $\mathcal{M}>0$ and $\chi > \xi$.
The condition $\chi > \xi$ is equivalent,
by Theorem~\ref{thm:r0}, to $\mathcal{R}_{0}>1$.
\end{proof}

\medskip

Using the methods in \cite{Driessche2002, Li1996},
it is possible to prove that if $\mathcal{R}_0 \leq 1$,
then the DFE\index{Disease Free Equilibrium} is globally asymptotically stable in
$\Omega$, and thus the vector-borne disease always dies out;
if $\mathcal{R}_0 > 1$, then the unique endemic equilibrium
is globally asymptotically stable in $\Omega$,
so that the disease, if initially present, will persist
at the unique endemic equilibrium level.

Assuming that parameters are fixed, the threshold $\mathcal{R}_0$
is influenceable by the control values. Figure~\ref{cap6_R0_cm_CA_alpha}
gives this relationship. It is possible to realize that the control $c_m$
is the one that most influences the basic reproduction number\index{Basic reproduction number}
to stay below unit. Besides, the control in the aquatic phase alone is not enough to maintain
$\mathcal{R}_0$ below unit: an application close to 100\% is required.

\begin{figure}
\centering
\subfloat[$\mathcal{R}_{0}$ as a function of $c_m$ and $c_A$]{\label{cap6_R0_cm_CA}
\includegraphics[width=0.55\textwidth]{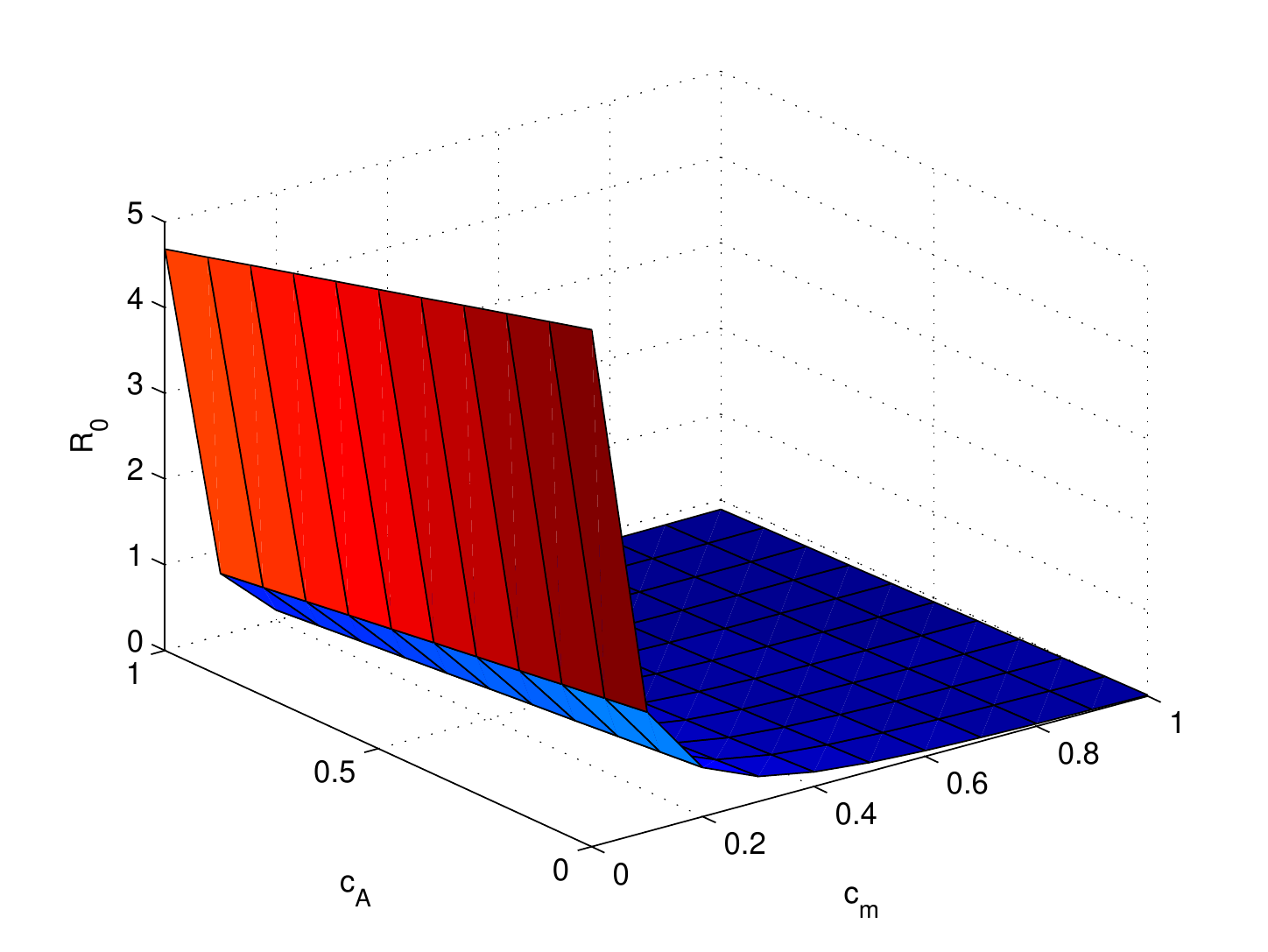}}\\
\subfloat[$\mathcal{R}_{0}$ as a function of $c_m$ and $\alpha$]{\label{cap6_R0_cm_alpha}
\includegraphics[width=0.55\textwidth]{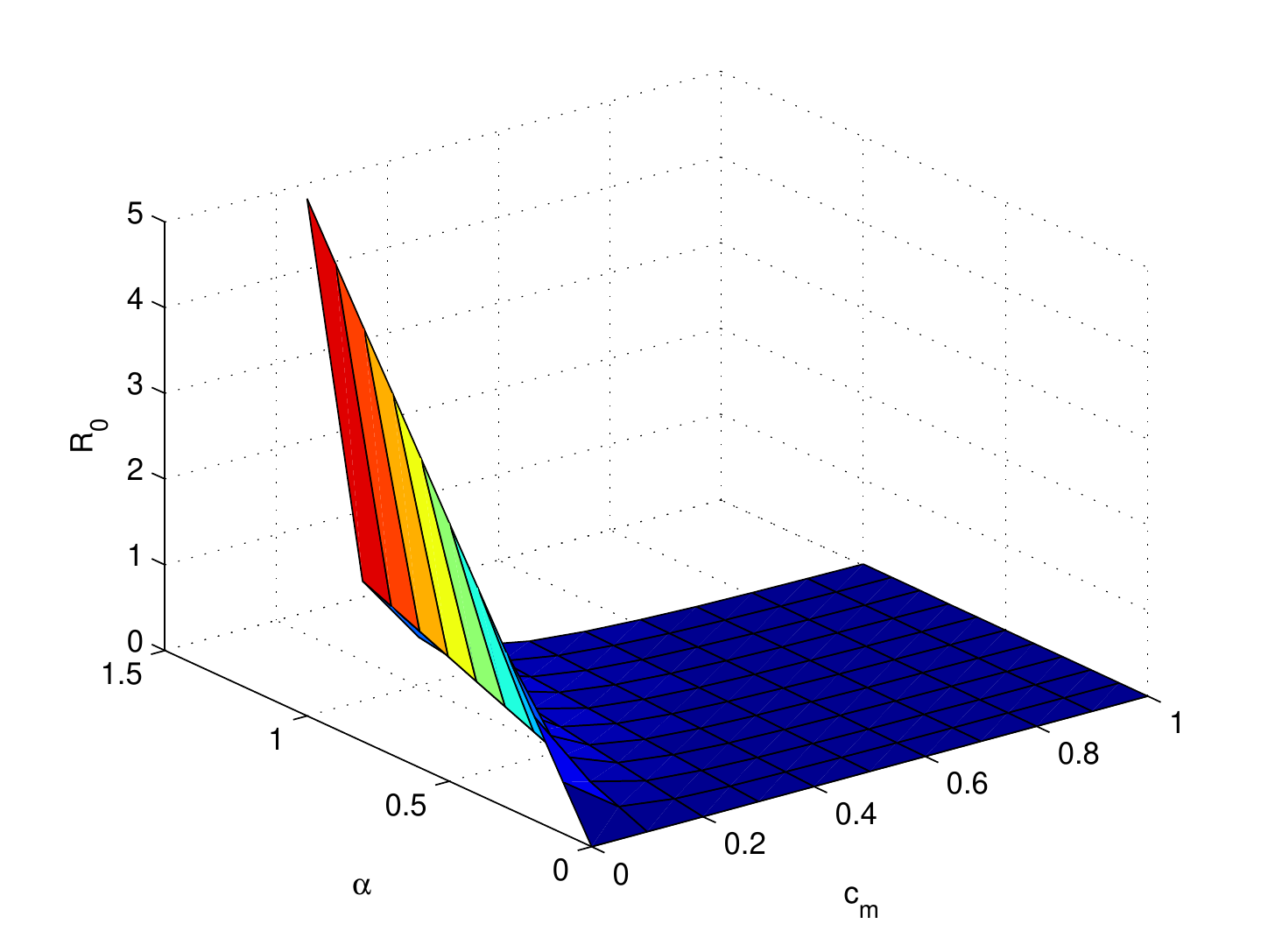}}\\
\subfloat[$\mathcal{R}_{0}$ as a function of $c_A$ and $\alpha$]{\label{cap6_R0_cA_alpha}
\includegraphics[width=0.55\textwidth]{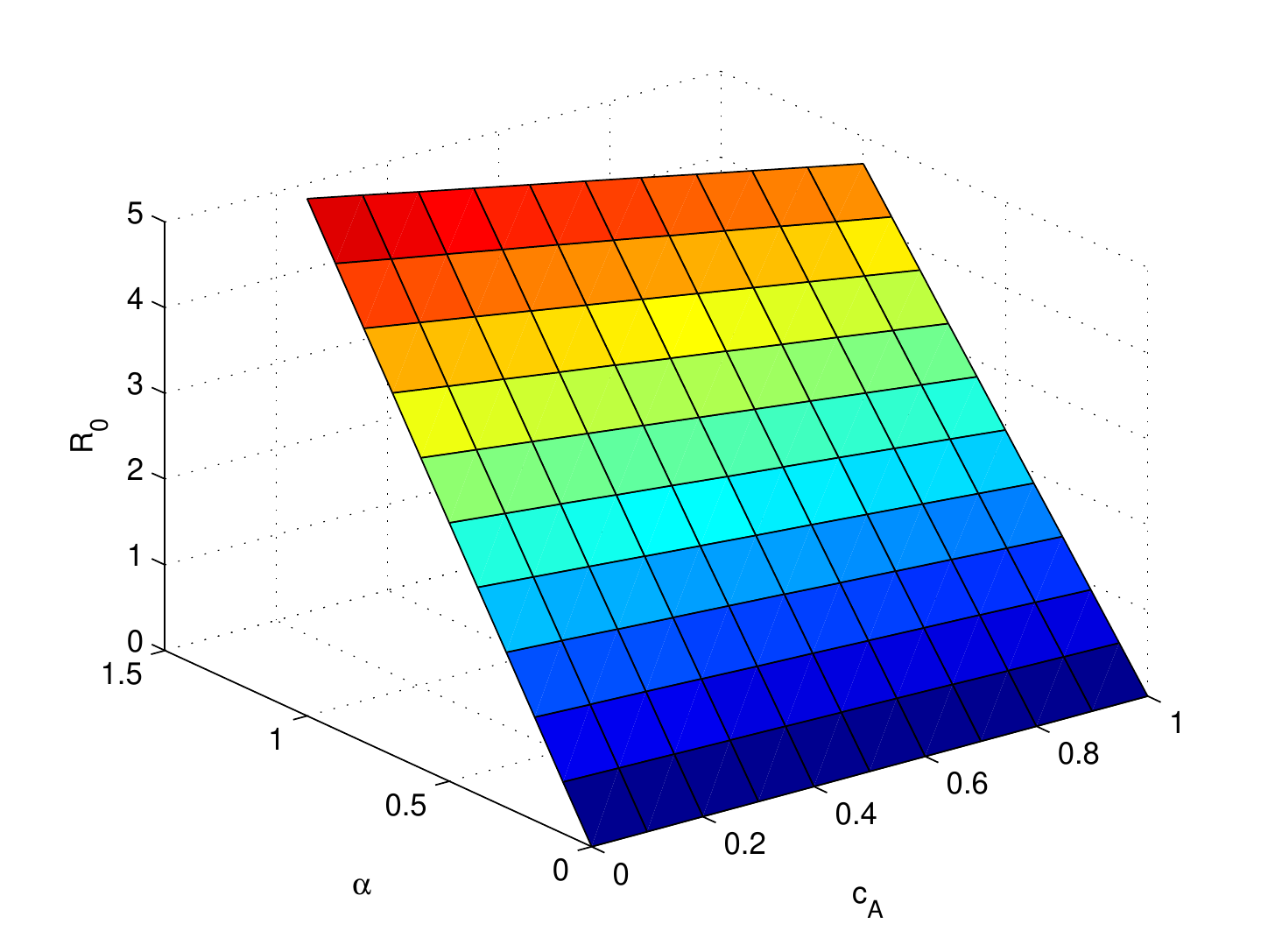}}
\caption{Influence of the controls on the basic reproduction number $\mathcal{R}_{0}$}
\label{cap6_R0_cm_CA_alpha}
\end{figure}


\section{Numerical implementation}
\label{sec:6:2}

The simulations were carried out using the following numerical
values: $N_h=480000$, $B=0.8$, $\beta_{mh}=0.375$,
$\beta_{hm}=0.375$, $\mu_{h}=1/(71\times365)$, $\eta_{h}=1/3$,
$\mu_{m}=1/10$, $\varphi=6$, $\mu_A=1/4$,
$\eta_A=0.08$, $m=3$, $k=3$. The initial conditions for the
problem were: $S_{h0}=N_h-10$, $I_{h0}=10$, $R_{h0}=0$,
$A_{m0}=k N_h$, $S_{m0}=m N_h$, $I_{m0}=0$.
With these values, one has $\mathcal{M}>0$. As in the previous chapter
the values related to humans describe the reality of Cape Verde \cite{INEcv}
and the information about mosquitoes is based on Brazil \cite{Coelho2008,Esteva2005}.

All computational calculus consider one year for time interval.
Although the final time was $t_f=365$ days, the figures
show graphics in suitable windows, in order to provide a better analysis.
All the simulations and graphics were done in \texttt{Matlab}\index{Matlab}.
To solve the differential equation system, the \texttt{ode45} routine \index{ode45 routine}was used.
This function implements a Runge-Kutta method\index{Runge Kutta scheme} with a variable time step for
efficient computation (see \cite{SofiaSITE} for more details about the code).

Figures~\ref{humanno} and \ref{mosquitono} describe the human and mosquito populations
in the absence of any control, respectively. The number of
human infection has a peak between the 30th and the 40th day.
The infection of the mosquitoes had a delay when compared to the humans.

The number of infected humans from the model \eqref{cap6_ode1}--\eqref{cap6_ode2}
is higher when compared with what really happened in Cape Verde.
As far as it was possible to investigate in the local news, the government of Cape Verde has
done their best to banish the mosquito, with media campaigns appealing
to people to remove or cover all recipients that could serve to
breed the mosquito, and to use insecticide in critical areas.
However, it was not possible to quantify those efforts in precise terms.

\begin{figure}[ptbh]
\begin{center}
\includegraphics[scale=0.45]{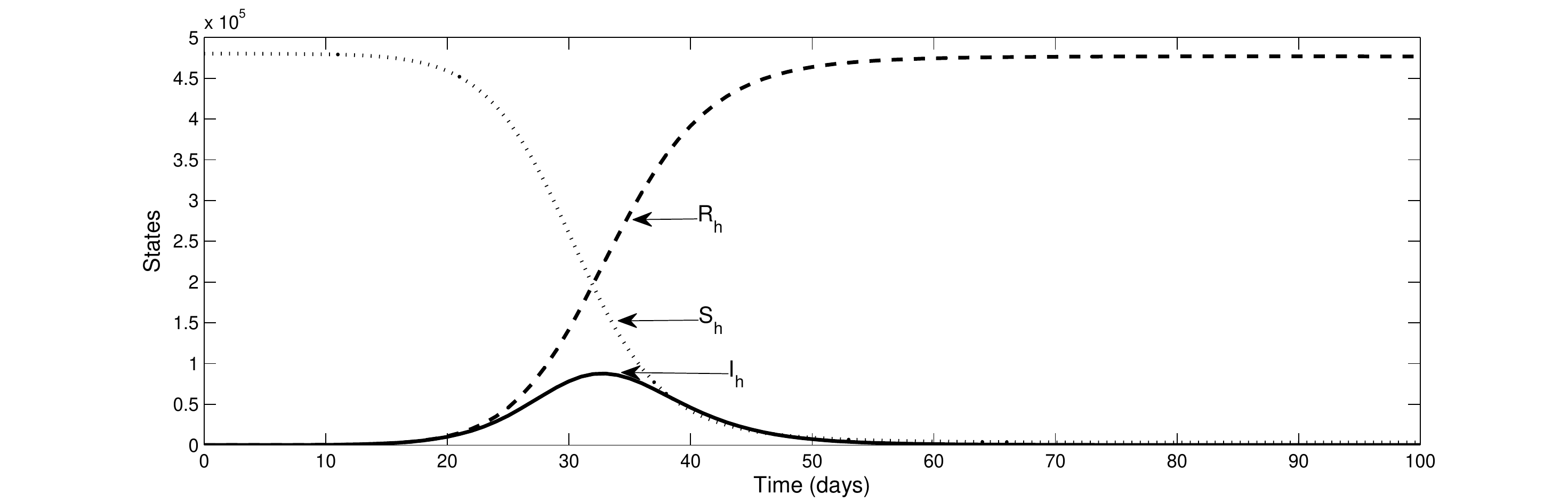}
\end{center}
\caption{Human population (without control)}\label{humanno}
\end{figure}

\begin{figure}[ptbh]
\begin{center}
\includegraphics[scale=0.45]{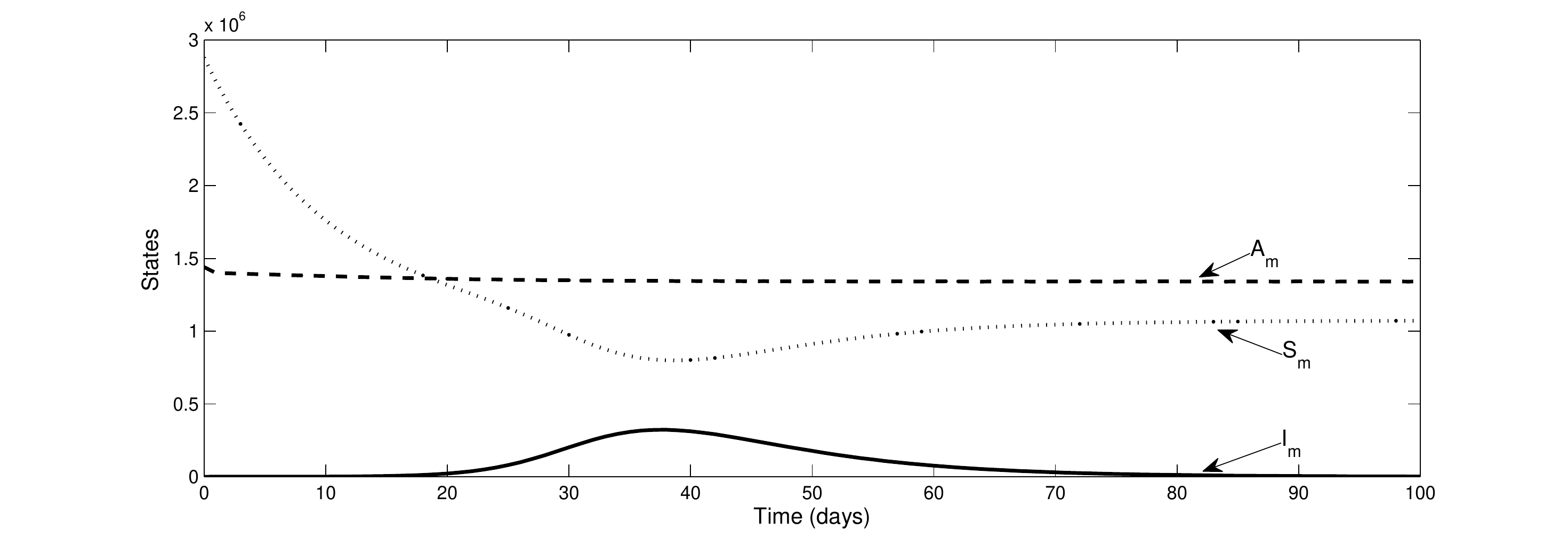}
\end{center}
\caption{Mosquito population (without control)}\label{mosquitono}
\end{figure}

Next, follows a set of simulations using different
controls. In each figure, only one control is used, continuously,
which means that the others are not applied. The aim of this
simulation is to see the importance of the control and what
repercussions has on the model. Figures~\ref{14_human_adulticide_simulation}
and \ref{15_mosquito_adulticide_simulation} concern the adulticide control,
Figures~\ref{12_human_larvicide_simulation} and \ref{13_mosquito_larvicide_simulation}
the larvicide control, and Figures~\ref{16_human_mechanical_simulation} and
\ref{17_mosquito_mechanical_simulation} the mechanical control.
Using a small quantity of each control,
the number of infected people falls dramatically. In some cases,
in spite of all graphs displaying five simulations, some curves are so close
to zero that it is difficult to distinguish them.

\begin{figure}[ptbh]
\begin{center}
\includegraphics[scale=0.45]{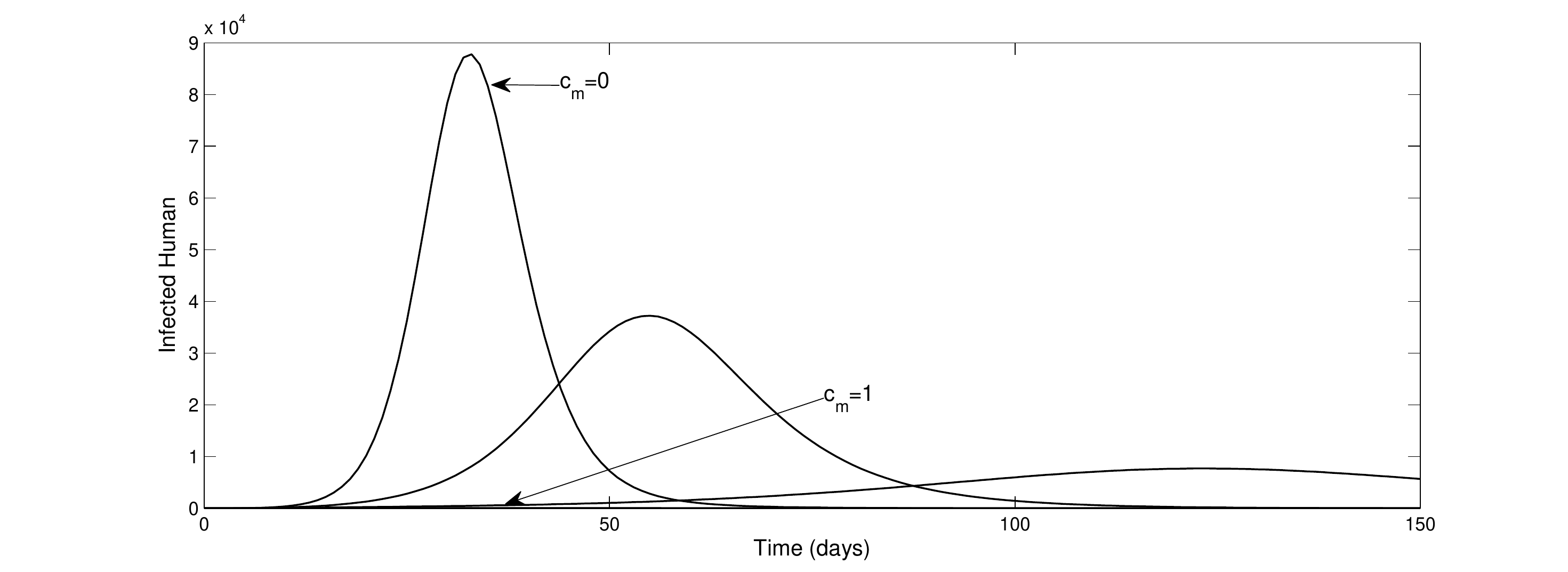}
\end{center}
\caption{Infected humans using different levels of adulticide
($c_m=0$, 0.25, 0.50, 0.75, 1)}\label{14_human_adulticide_simulation}
\end{figure}

\begin{figure}[ptbh]
\begin{center}
\includegraphics[scale=0.45]{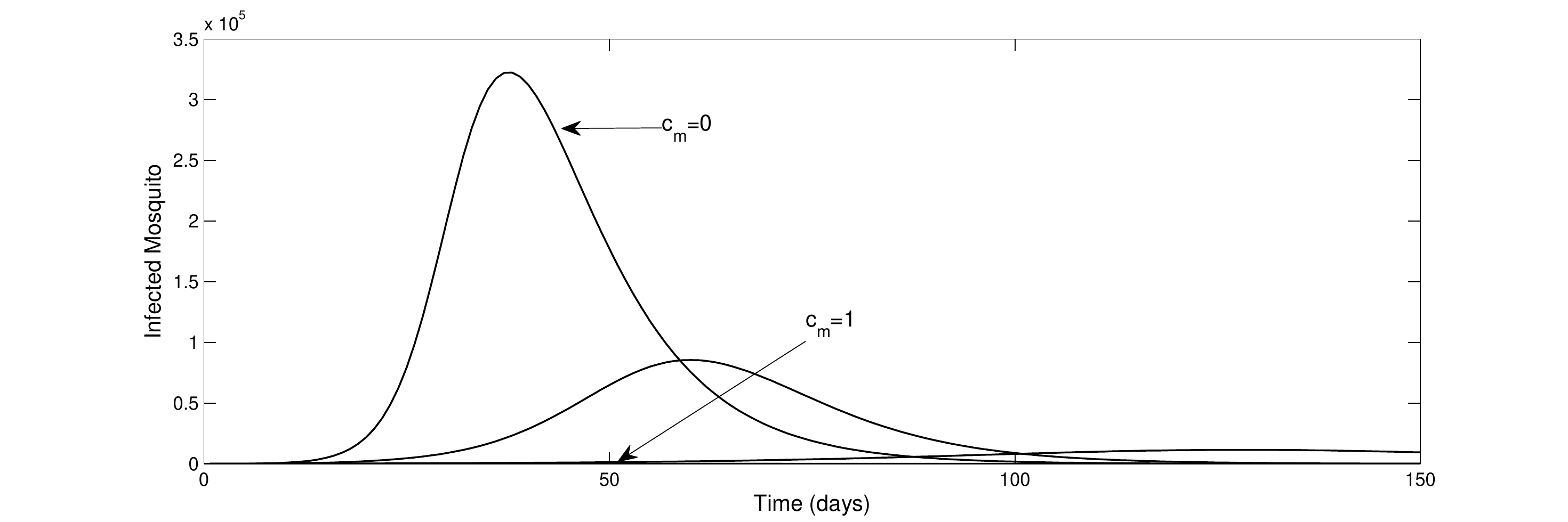}
\end{center}
\caption{Infected mosquitoes with different levels of adulticide
($c_m = 0$, 0.25, 0.50, 0.75, 1)}\label{15_mosquito_adulticide_simulation}
\end{figure}

\begin{figure}[ptbh]
\begin{center}
\includegraphics[scale=0.45]{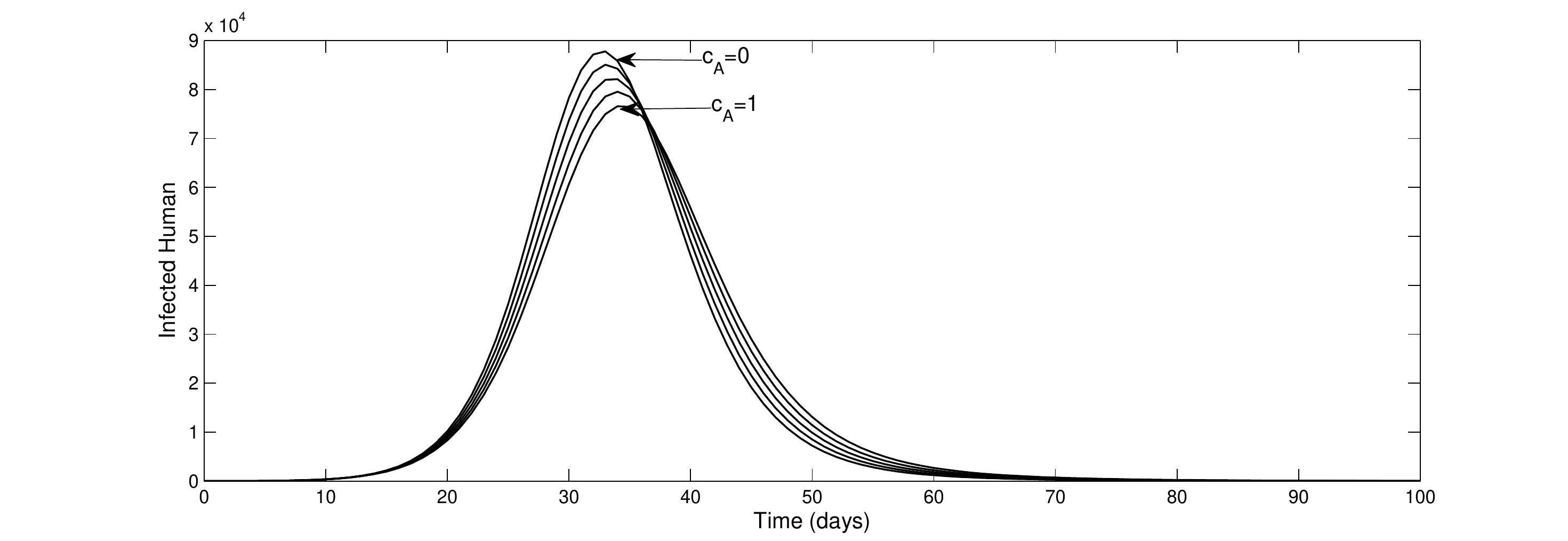}
\end{center}
\caption{Infected humans using different levels of larvicide
($c_A = 0$, 0.25, 0.50, 0.75, 1)}\label{12_human_larvicide_simulation}
\end{figure}

\begin{figure}[ptbh]
\begin{center}
\includegraphics[scale=0.45]{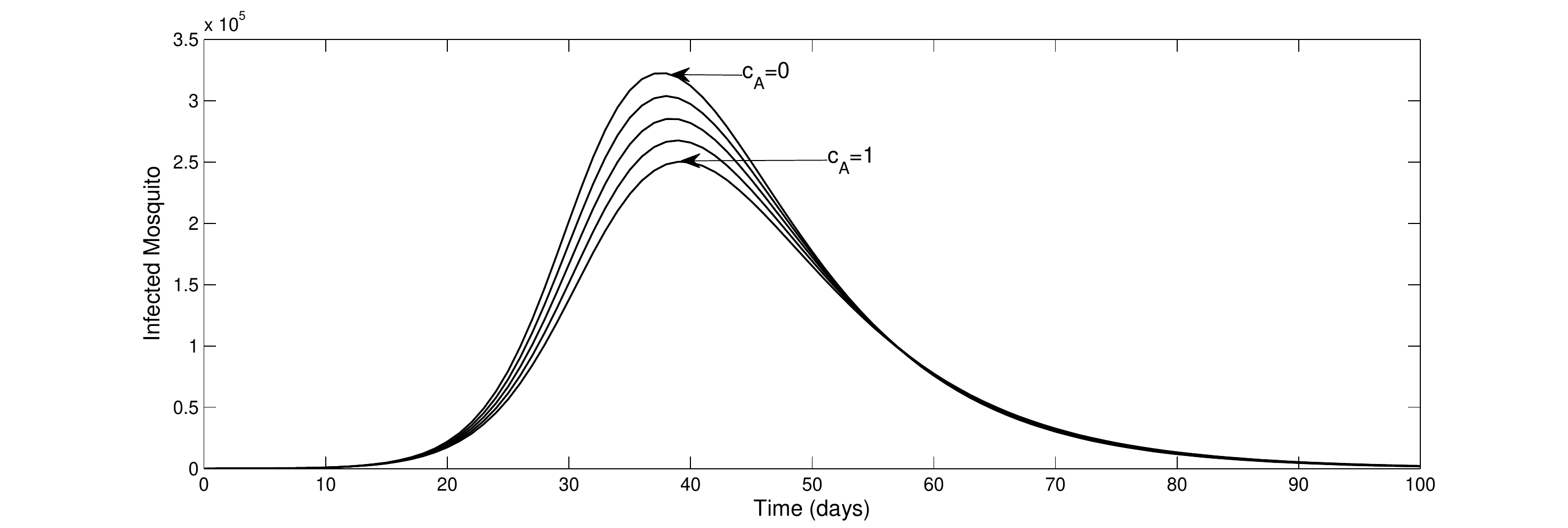}
\end{center}
\caption{Infected mosquitoes using different levels of larvicide
($c_A = 0$, 0.25, 0.50, 0.75, 1)}\label{13_mosquito_larvicide_simulation}
\end{figure}

\begin{figure}[ptbh]
\begin{center}
\includegraphics[scale=0.45]{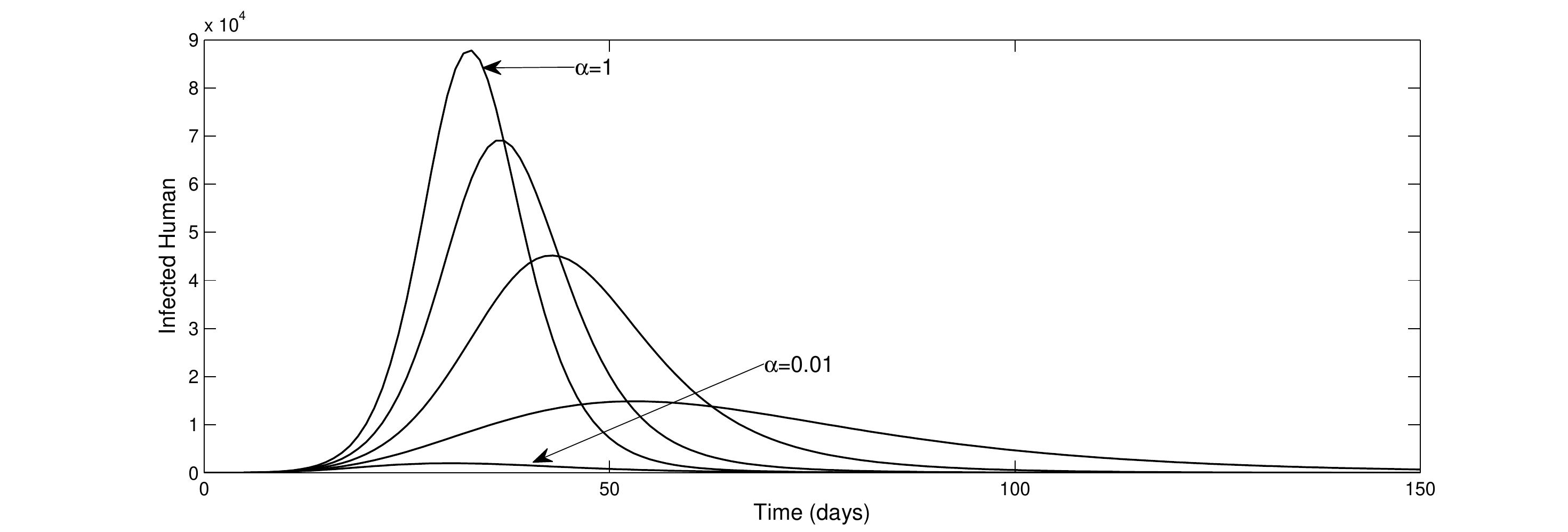}
\end{center}
\caption{Infected humans using different levels of mechanical control
($\alpha = 0.01$, 0.25, 0.5, 0.75, 1)}\label{16_human_mechanical_simulation}
\end{figure}

\begin{figure}[ptbh]
\begin{center}
\includegraphics[scale=0.45]{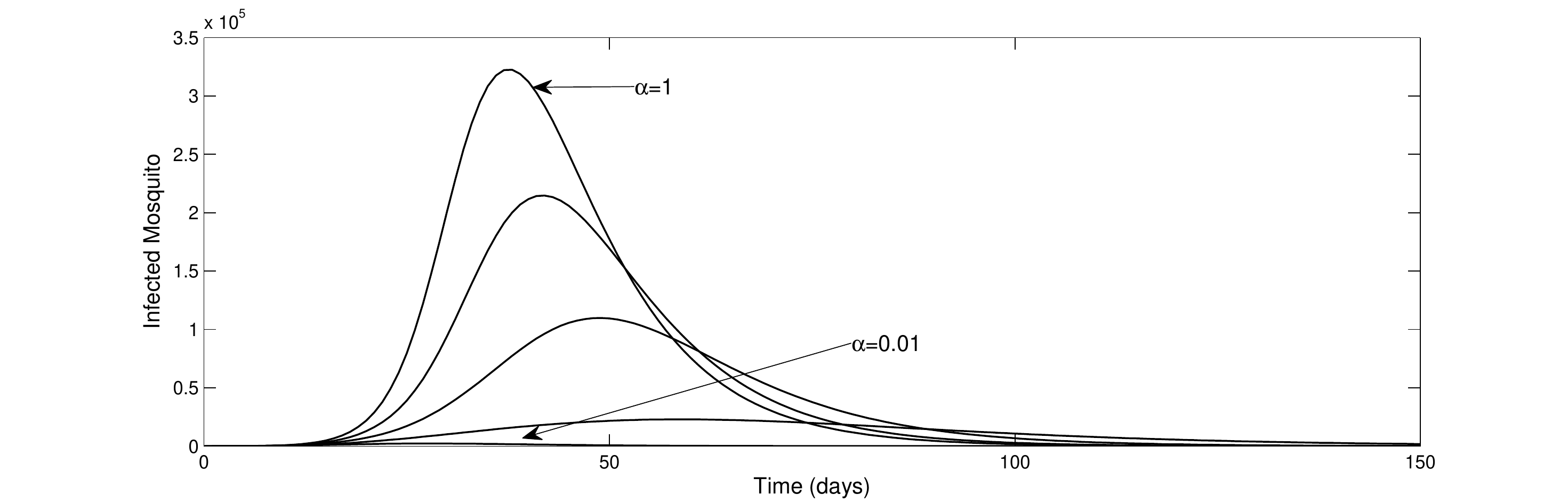}
\end{center}
\caption{Infected mosquitoes using different levels of mechanical control
($\alpha = 0.01$, 0.25, 0.5, 0.75, 1)}\label{17_mosquito_mechanical_simulation}
\end{figure}

Figures~\ref{14_human_adulticide_simulation}
and \ref{15_mosquito_adulticide_simulation} show that
excellent results for the human population are obtained
by covering only 25\% of the country with insecticide
for adult mosquitoes. The simulations were done considering that
the \emph{aedes aegypti} does not become resistant
to the insecticide and that it is financially possible to
apply insecticide during all time.
Figures~\ref{12_human_larvicide_simulation}--\ref{17_mosquito_mechanical_simulation}
are related to the applied controls in the aquatic phase of the
mosquito. In these graphics the controls were studied separately,
but one is closely related to the other. The application of these
controls are not sufficient to decrease the infected human to zero,
but the removal of breeding sites and the use of larvicide
is important to the reduction of that number.

Figures~\ref{22_human_mixed_control} and \ref{23_mosquitomixed}
show several simulations using three controls at different levels,
simultaneously. Only $10\%$ of each control,
applied continuously, is enough for the number of infected,
humans and mosquitoes, to remain near zero.

\begin{figure}[ptbh]
\begin{center}
\includegraphics[scale=0.45]{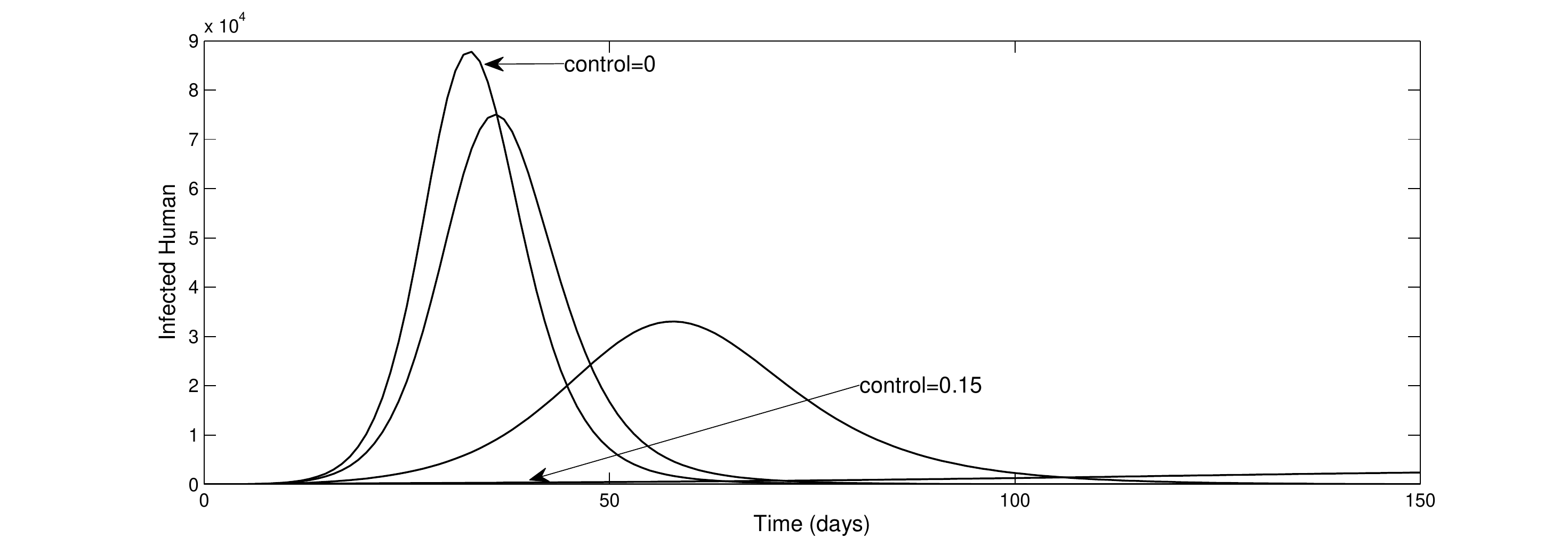}
\end{center}
\caption{Infected humans using different levels of control
($c_A=c_m=1-\alpha= 0$, 0.01, 0.05, 0.1, 0.15)}\label{22_human_mixed_control}
\end{figure}

\begin{figure}[ptbh]
\begin{center}
\includegraphics[scale=0.45]{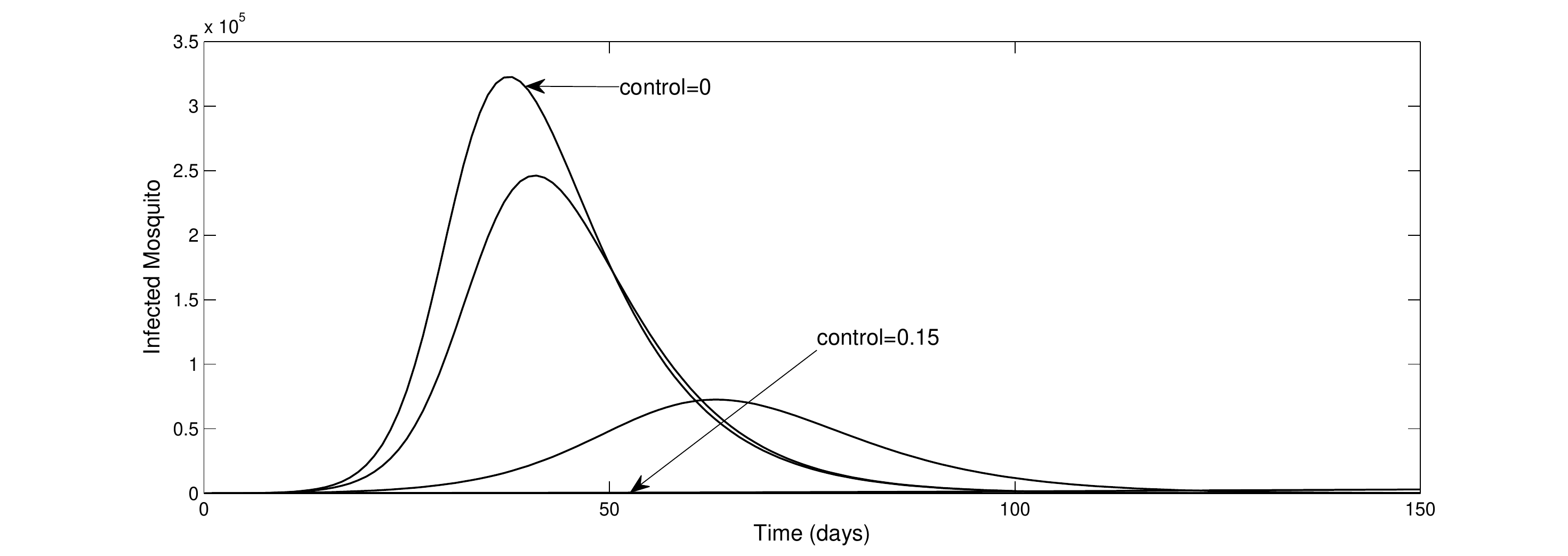}
\end{center}
\caption{Infected mosquitoes using different levels of control
($c_A=c_m=1-\alpha= 0$, 0.01, 0.05, 0.1, 0.15)}\label{23_mosquitomixed}
\end{figure}

In the next section, using OC strategy we will find the best solution for the controls.


\section{Optimal control with several controls}
\label{sec:6:3}

Epidemiological models may give some basic guidelines for public health practitioners,
comparing the effectiveness of different potential management strategies.

In reality, a range of constraints and trade-offs may substantially influence
the choice of practical strategy, and therefore their inclusion
in any modelling analysis may be important. Frequently, epidemiological models
need to be coupled to economic considerations, such that control strategies
can be judged through holistic cost-benefit analysis. Control of livestock
disease is a scenario when cost-benefit analysis can play a vital role
in choosing between cheap, weak controls that lead to a prolonged epidemic,
or expensive but more effective controls that lead to a shorter outbreak.

Normalizing the previous ODE system (\ref{cap6_ode1})--(\ref{cap6_ode2}), we obtain:

\begin{equation}
\label{cap6_ode_norm}
\begin{cases}
\dfrac{ds_h}{dt} = \mu_h - \left(B\beta_{mh}m i_m+\mu_h\right)s_h\\
\dfrac{di_h}{dt} = B\beta_{mh}m i_m s_h -(\eta_h+\mu_h) i_h\\
\dfrac{dr_h}{dt} = \eta_h i_h - \mu_h r_h\\
\dfrac{da_m}{dt} = \varphi \frac{m}{k}\left(1-\frac{a_m}{\alpha }\right)(s_m+i_m)
-\left(\eta_A+\mu_A + c_A\right) a_m\\
\dfrac{ds_m}{dt} = \eta_A \frac{k}{m}a_m -\left(B \beta_{hm}i_h+\mu_m + c_m\right) s_m\\
\dfrac{di_m}{dt} = B \beta_{hm}i_h s_m -\left(\mu_m + c_m\right) i_m
\end{cases}
\end{equation}
\noindent with the initial conditions
\begin{equation}
\begin{tabular}{llll}
$s_h(0)=0.9999,$ &  $i_h(0)=0.0001,$ &
$r_h(0)=0,$ \\
$a_m(0)=1,$ & $s_{m}(0)=1,$ & $i_m(0)=0.$
\end{tabular}
\label{chap6_initial_norm}
\end{equation}

The cost functional considered was
\begin{equation}
\label{chap6_functional}
\begin{tabular}{ll}
minimize & $\displaystyle J\left(u_1(\cdot),u_2(\cdot)\right)
=\int_{0}^{t_f}\left[\gamma_D i_h(t)^2
+\gamma_S c_m(t)^2+\gamma_L c_A(t)^2
+\gamma_E \left(1-\alpha\right)^2\right]dt$\\
\end{tabular}
\end{equation}
\noindent where $\gamma_D$, $\gamma_S$, $\gamma_L$ and $\gamma_E$
are weights related to costs of the disease, adulticide,
larvicide and mechanical control, respectively.

In a first approach to this problem, it is assumed that all weights
are the same, which means $\gamma_D= \gamma_S=\gamma_L=\gamma_E=0.25$ (Case A).

The OC problem was solved using two different packages: \texttt{DOTcvp}\index{DOTcvp}
\cite{Dotcvp} and \texttt{Muscod-II} \cite{Muscod}\index{Scilab}. The mathematical
formulation of the SIR+ASI problem, for the both packages, is available on \cite{SofiaSITE}.
The simulation behavior is similar, and we decided to show only the \texttt{DOTcvp} results.
The optimal functions for the controls are given in Figure~\ref{cap6_all_controls}.

\begin{figure}[ptbh]
\begin{center}
\includegraphics[scale=0.7]{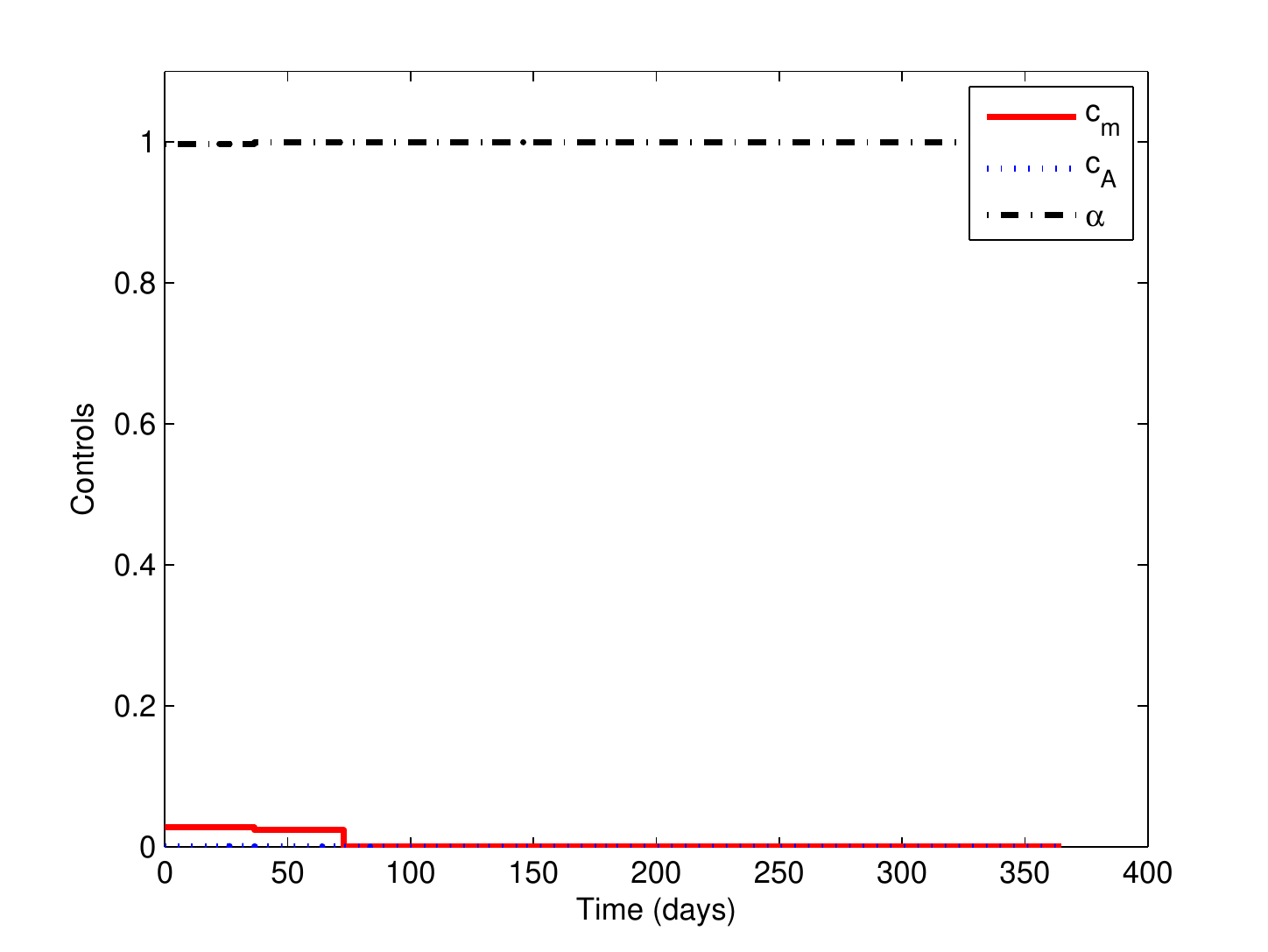}
\end{center}
\caption{Optimal control functions (Case A:
$\gamma_D=\gamma_S=\gamma_L=\gamma_E=0.25$)}\label{cap6_all_controls}
\end{figure}

The adulticide was the control that more influenced the decreasing of that ratio and,
as consequence, the decreasing of the number of infected people and mosquitoes,
matching with the results obtained for the basic reproduction number
in Section~\ref{sec:6:1}. Therefore, the adulticide was the most used.
We believe  that the other controls do not assume an important role
in the epidemic episode, because all the events happened at a short period of time,
which means that adulticide has more impact. However the mosquito control in the
aquatic phase can not be neglected. In situations of longer epidemic episodes
or even in an endemic situation, the larval control represents an important tool.

Figure~\ref{cap6_all_controls_vs_nocontrol} presents the number of infected humans.
Comparing the optimal control case with a situation with no control,
the number of infected people decreased considerably. Besides, in the situation
where OC is used, the peak of infected people is minor, which facilitates the work
in health centers, because they can provide a better medical monitoring.

\begin{figure}[ptbh]
\begin{center}
\includegraphics[scale=0.7]{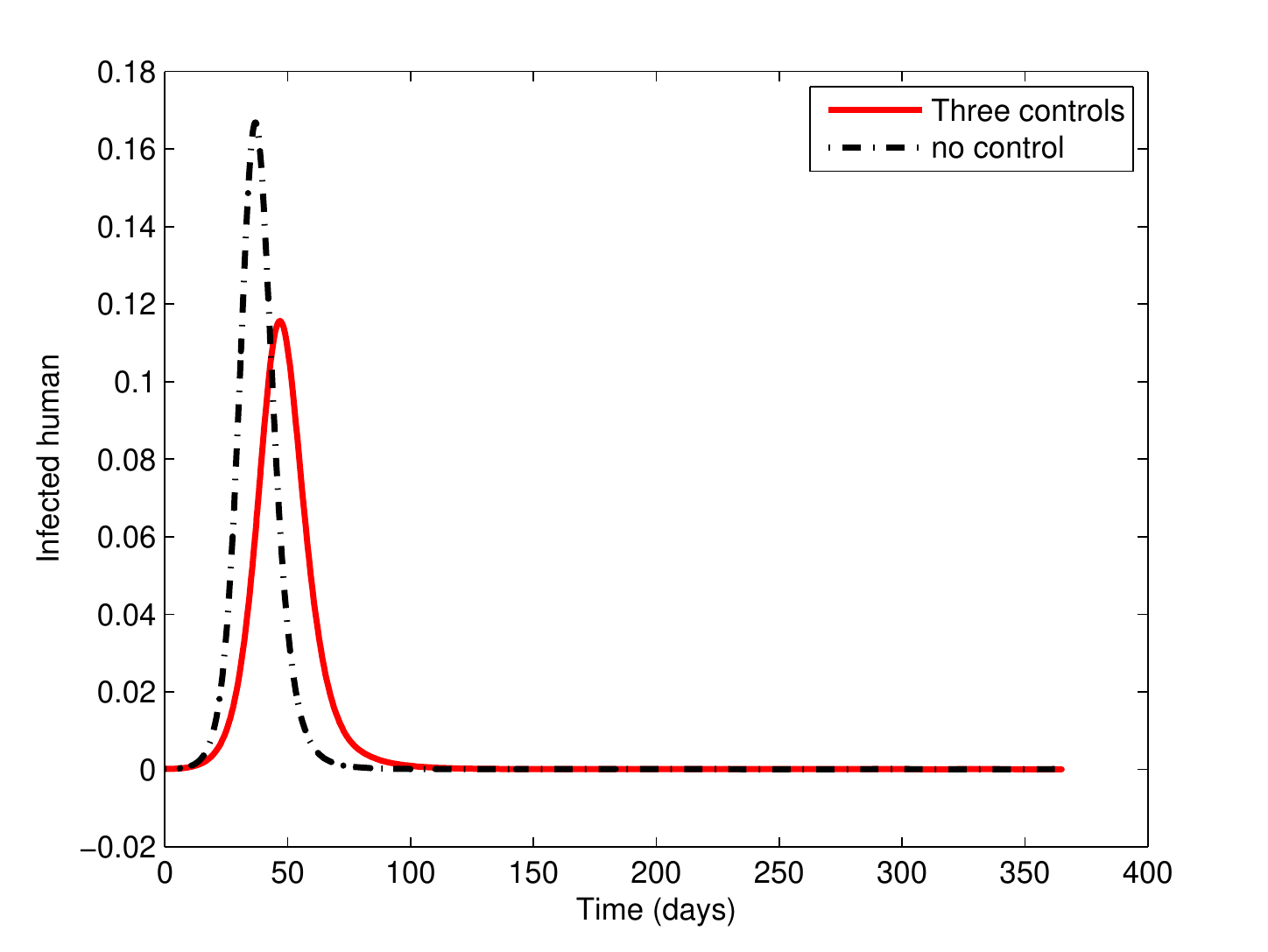}
\end{center}
\caption{Comparison of infected individuals under an optimal
control situation and no controls }\label{cap6_all_controls_vs_nocontrol}
\end{figure}

A second analysis was made, taking into account different weights on the functional.
Table~\ref{table_chap6_optimal weights} resumes the weights chosen for perspectives:
not only economic issues (cost of insecticides and educational campaigns),
but also human issues are considered. In case A, all costs were equal.
In case B is given more impact on the infected people, considering that
the treatment and absenteeism to work is very prejudicial to the country,
when compared with the cost of insecticides and educational campaigns.
In case C, the costs with killing mosquitoes and educational campaigns
have more impact in the economy. Higher total costs were obtained
when human life had more weight than controls measures as can
be checked in Table~\ref{table_chap6_optimal weights}.

\begin{table}[ptbh]
\begin{center}
\small
\begin{tabular}{l l c} \hline
 & Values for weights & Cost obtained\\ \hline
Case A & $\gamma_D=0.25$; $\gamma_S=0.25$; $\gamma_L=0.25$; $\gamma_E=0.25$ & 0.06691425\\
Case B & $\gamma_D=0.55$; $\gamma_S=0.15$; $\gamma_L=0.15$; $\gamma_E=0.15$ & 0.10431186\\
Case C & $\gamma_D=0.10$; $\gamma_S=0.30$; $\gamma_L=0.30$; $\gamma_E=0.30$ & 0.03012849\\ \hline
\end{tabular}
\caption{Different weights for the functional and respective values}
\label{table_chap6_optimal weights}
\end{center}
\end{table}

Figure~\ref{cap6_human_infected_CasesABC} shows the number of infected human
in each bioeconomic perspective. We can realize that Case A and Case C are similar.
It can be explained by the low weight given to the cost of treatment (cases A and C)
when compared with the heavy weight given in case B. Figure~\ref{cap6_all controls}
presents the behavior of the controls for the A, B and C cases.
Again, as adulticide is the one that has more influence on the model,
this is the control that most varies when the weights are changed.

\begin{figure}[ptbh]
\begin{center}
\includegraphics[scale=0.7]{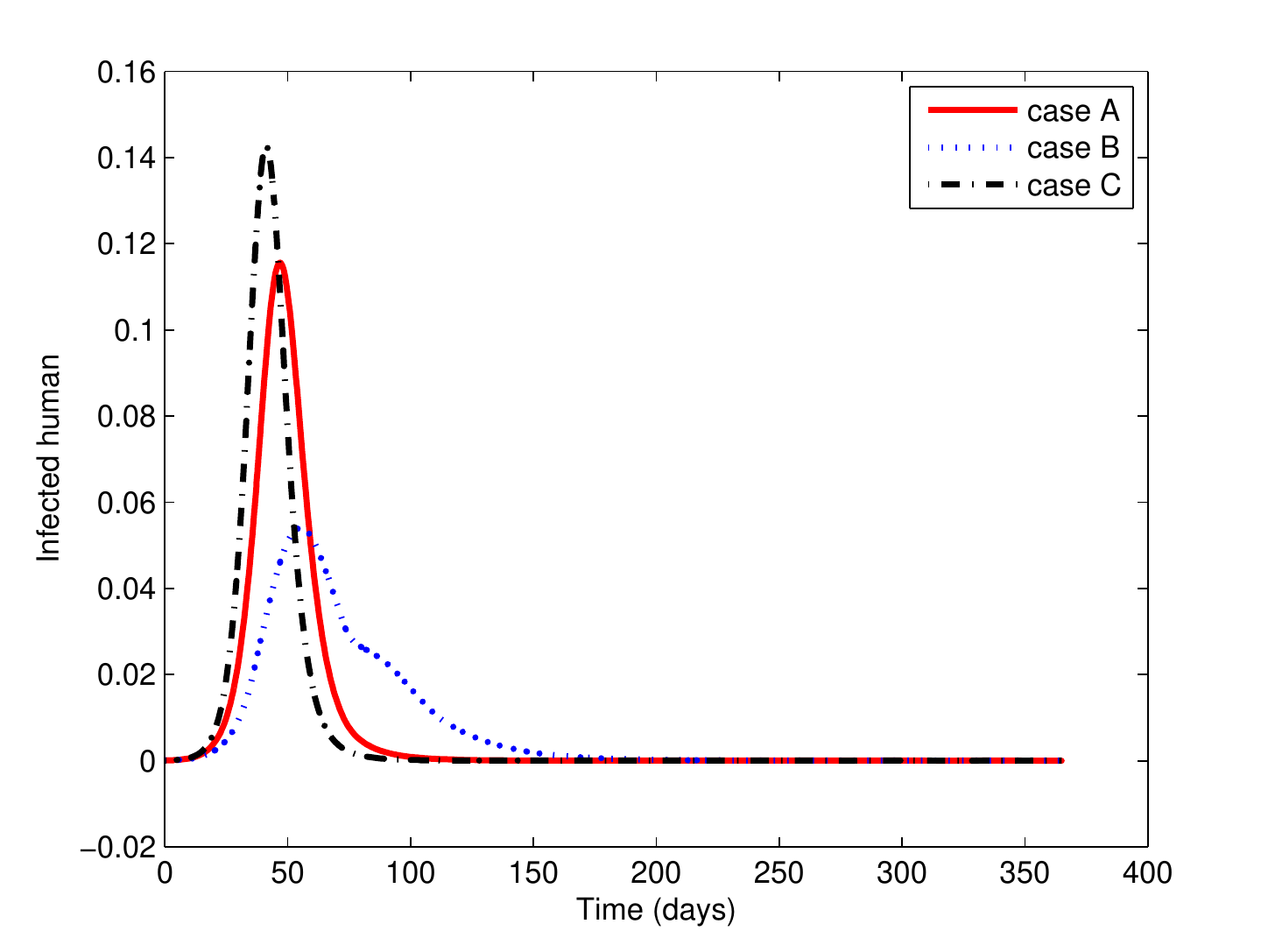}
\end{center}
\caption{Infected individuals in three bioeconomic perspectives}\label{cap6_human_infected_CasesABC}
\end{figure}

\begin{figure}
\centering
\subfloat[Adulticide]{\label{cap6_proportion_adulticide_CasesABC}
\includegraphics[width=0.55\textwidth]{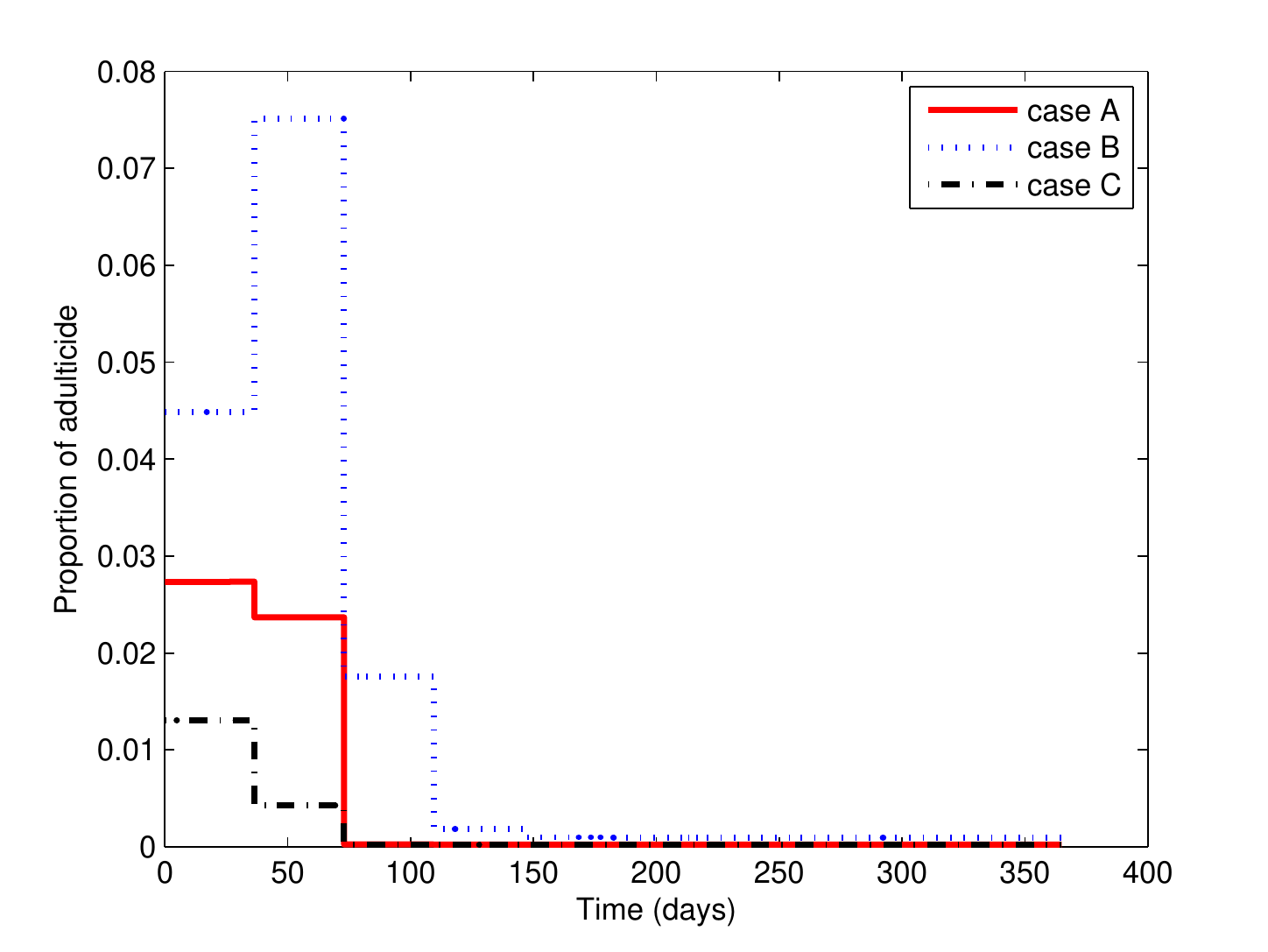}}\\
\subfloat[Larvicide]{\label{cap6_proportion_larvicide_CasesABC}
\includegraphics[width=0.55\textwidth]{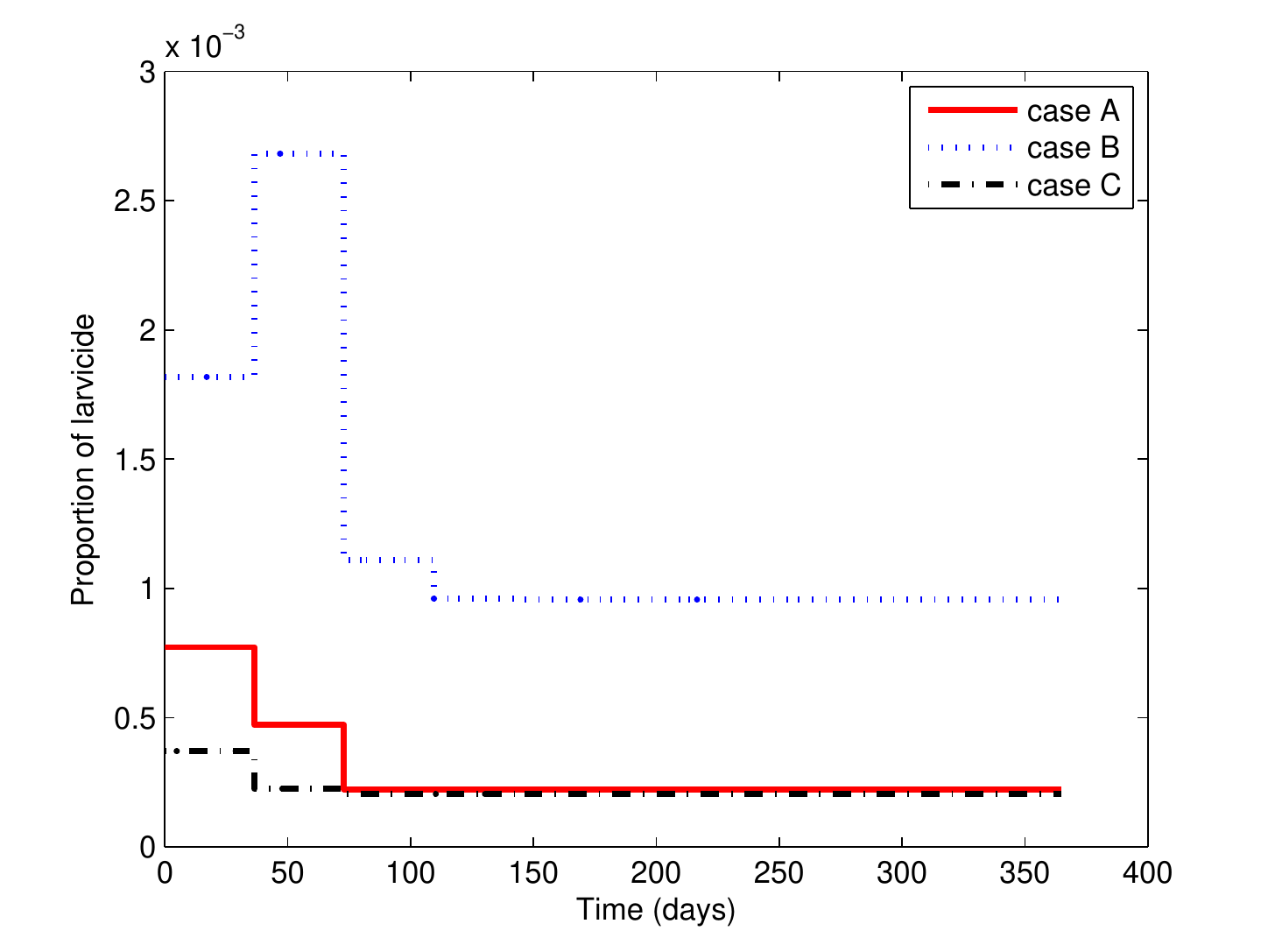}}\\
\subfloat[Mechanical control]{\label{cap6_proportion_mechanical_CasesABC}
\includegraphics[width=0.55\textwidth]{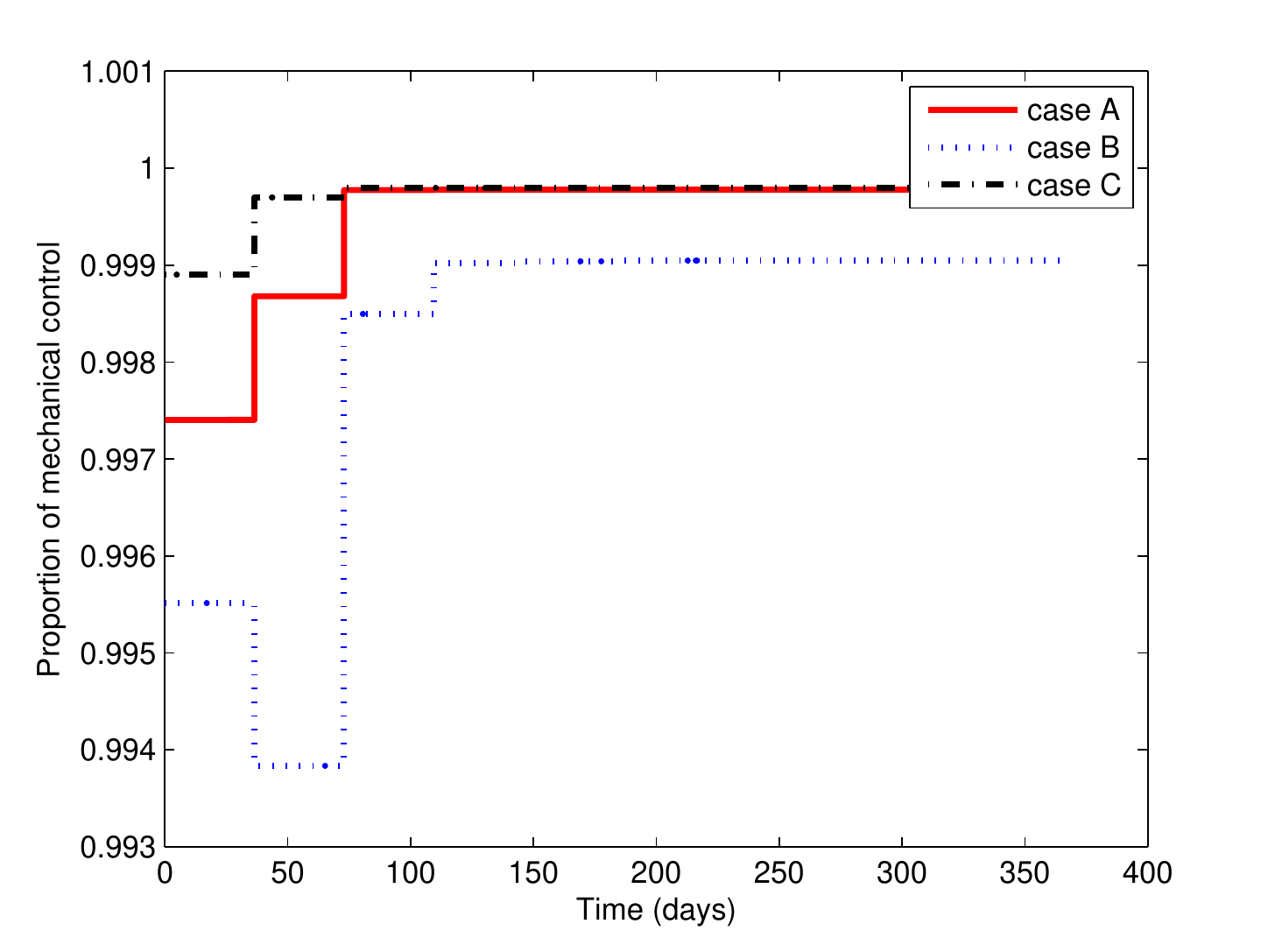}}
\caption{Proportion of control used in the three bioeconomic perspective }
\label{cap6_all controls}
\end{figure}

A third analysis was made to the model: it changed the functional in order
to study the effects of each control when separately considered. Therefore,
the new functional also considers bioeconomic perspectives, but only includes
two variables: the costs with infected humans (with $\gamma_D=0.5$)
and the costs with only one control (with $\gamma_i=0.5, i\in\{S,L,E\}$).
Thus, in Figure~\ref{cap6_adulticide} are presented the proportion
of adulticide $(a)$ and infected humans $(b)$, when the functional is
$\int_{0}^{t_f}\left[\gamma_D i_h(t)^2 +\gamma_S c_m(t)^2\right]dt$.
Figures~\ref{cap6_larvicide} and \ref{cap6_mechanical} represent
the same simulations when the controls considered are larvicide
and mechanical control, respectively. It is possible to see that
the use of larvicide and mechanical control, used alone,
does not bring relevant influence on the control of the disease.

\begin{figure}
\centering
\subfloat[Optimal control $c_m$]{\label{fcap6_adulticide_vs_all}
\includegraphics[width=0.48\textwidth]{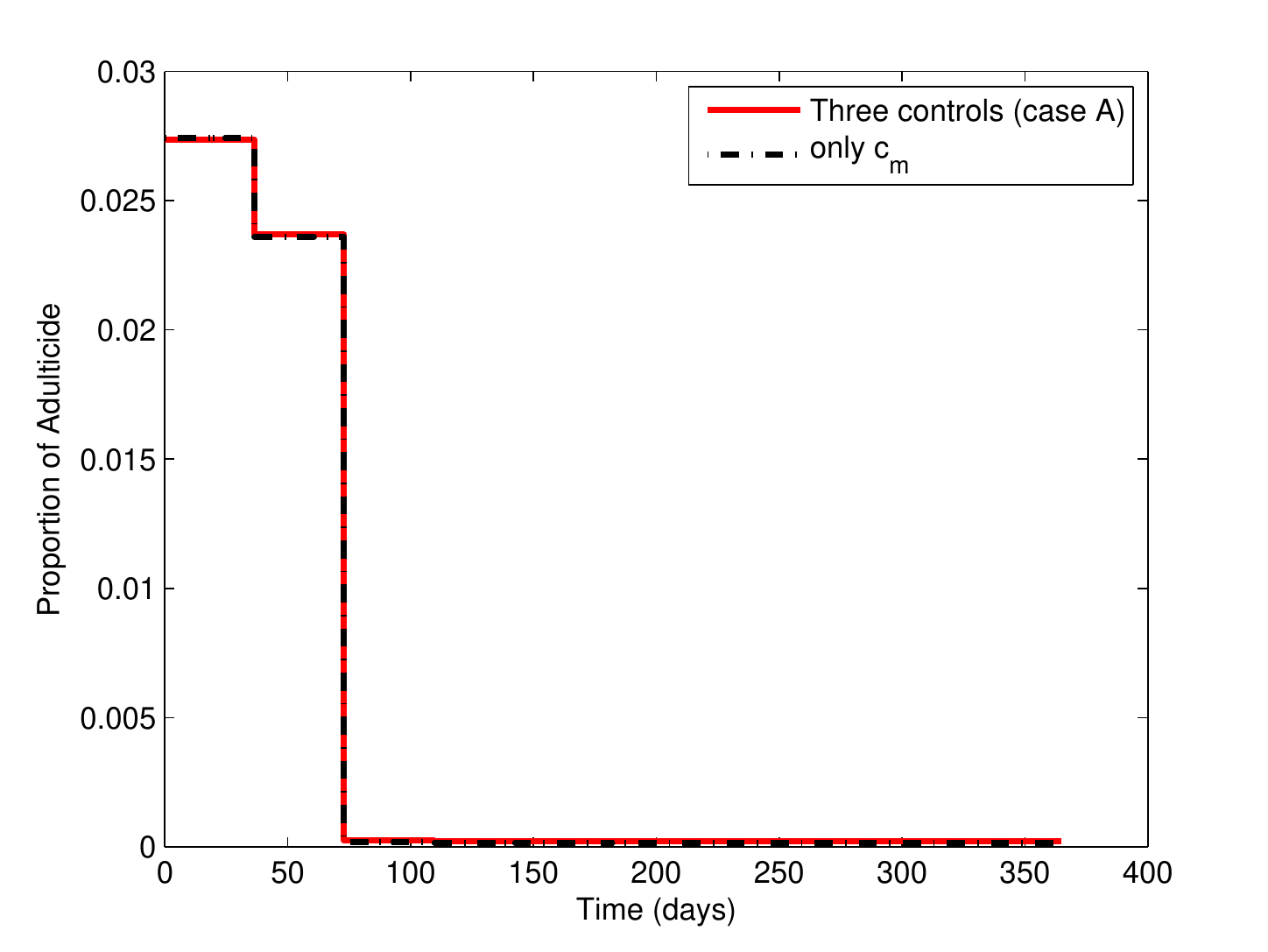}}
\subfloat[Proportion of infected human]{\label{cap6_infected_adulticide_vs_all}
\includegraphics[width=0.48\textwidth]{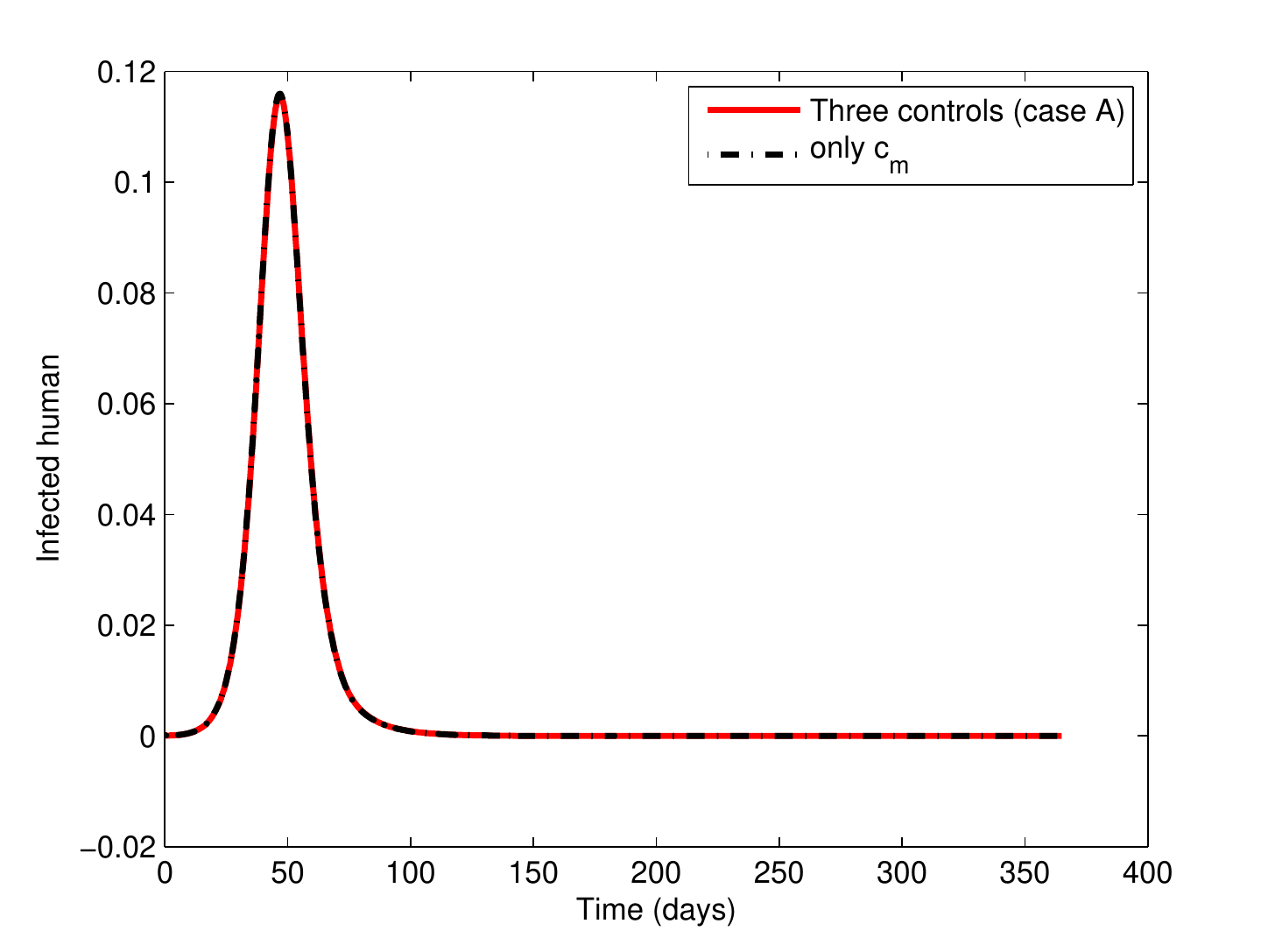}}
\caption{Optimal control and infected humans when considered all controls
(solid line) and only adulticide control (dashed line) in one year}
\label{cap6_adulticide}
\end{figure}

\begin{figure}
\centering
\subfloat[Optimal control $c_A$]{\label{cap6_larvicide_vs_all}
\includegraphics[width=0.48\textwidth]{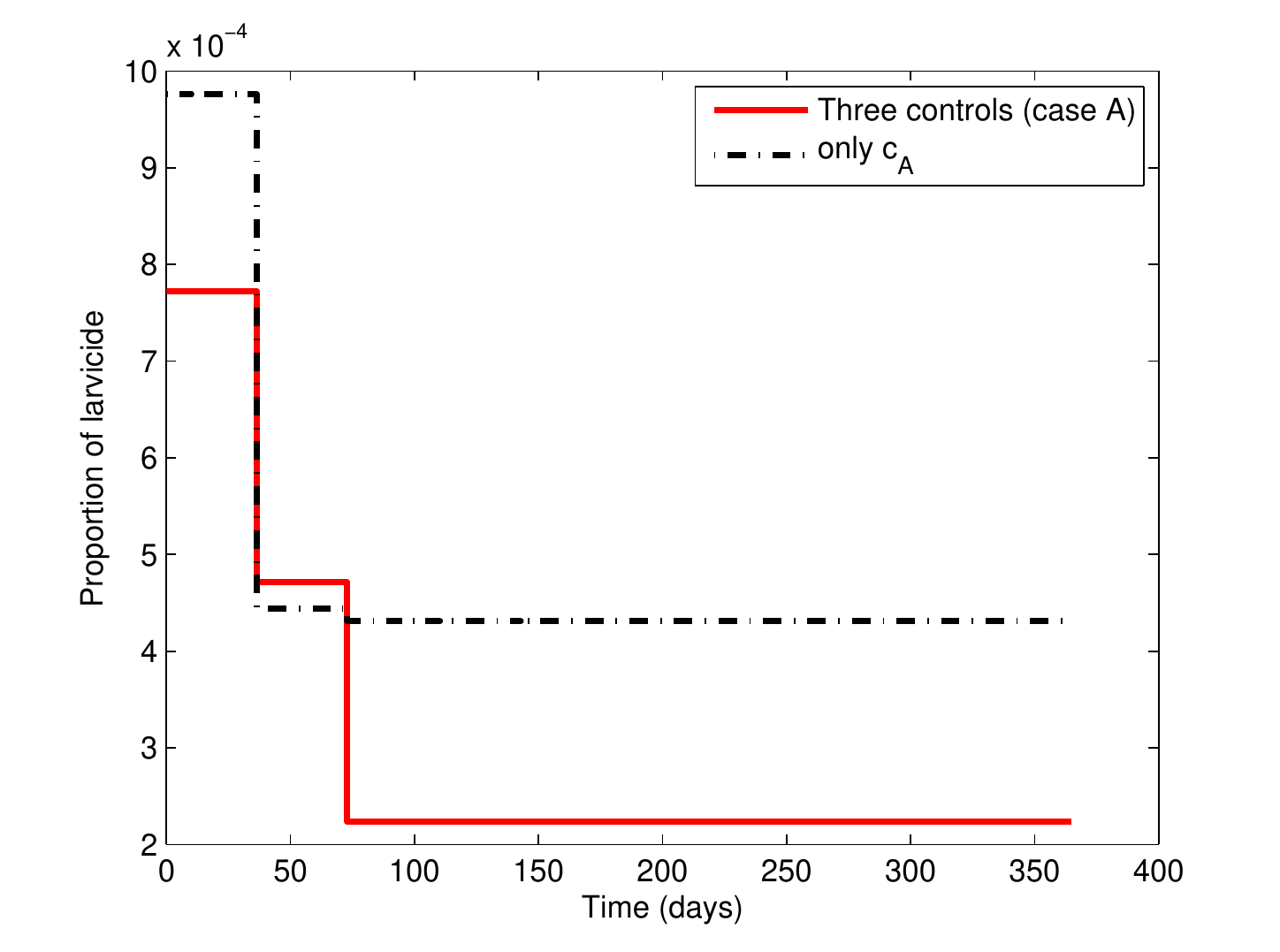}}
\subfloat[Proportion of infected human]{\label{cap6_infected_larvicide_vs_all}
\includegraphics[width=0.48\textwidth]{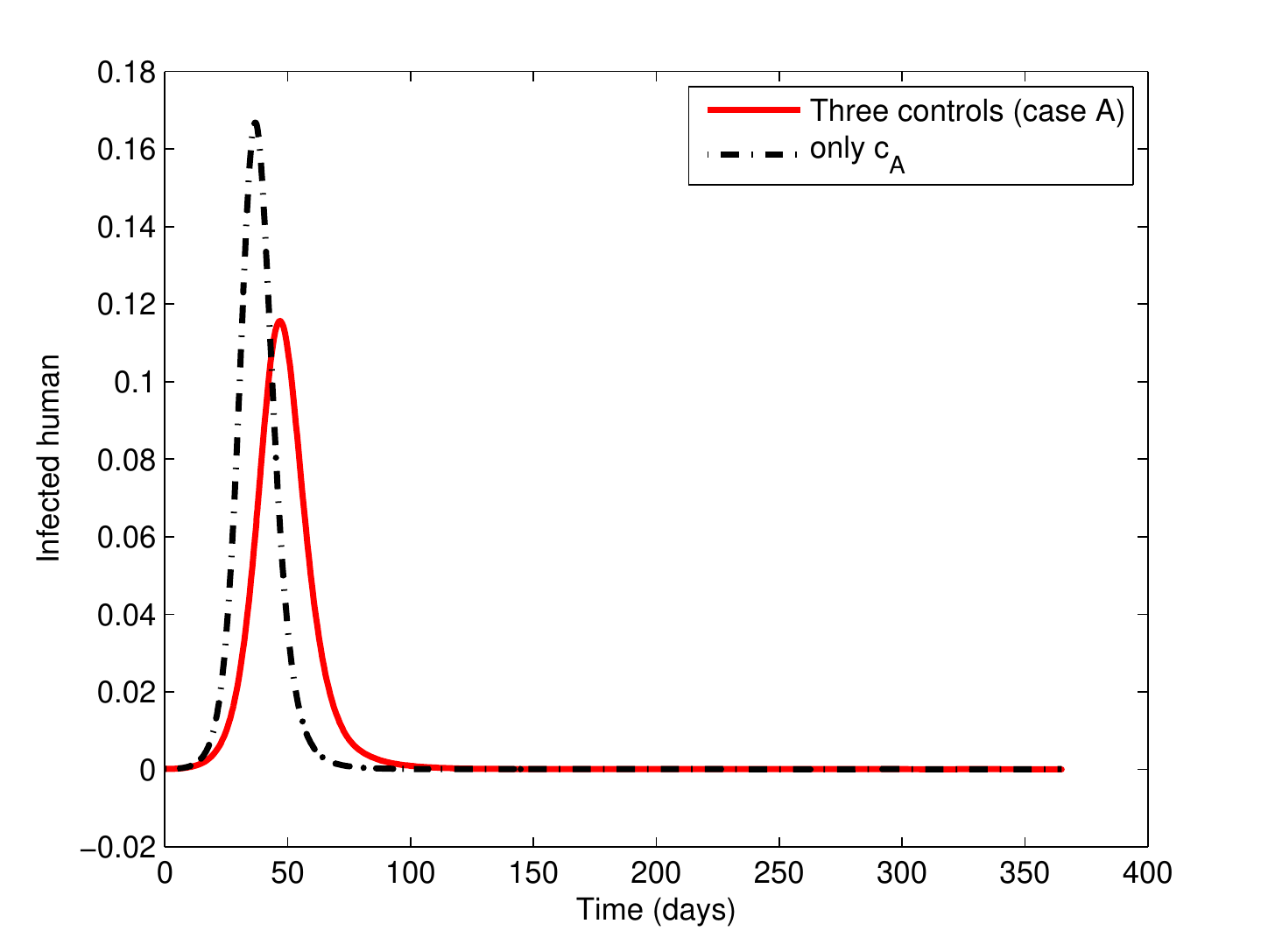}}
\caption{Optimal control and infected humans when considered all controls (solid line)
and only larvicide control (dashed line) in one year}
\label{cap6_larvicide}
\end{figure}

\begin{figure}
\centering
\subfloat[Optimal control $\alpha$]{\label{cap6_mechanical_vs_all}
\includegraphics[width=0.48\textwidth]{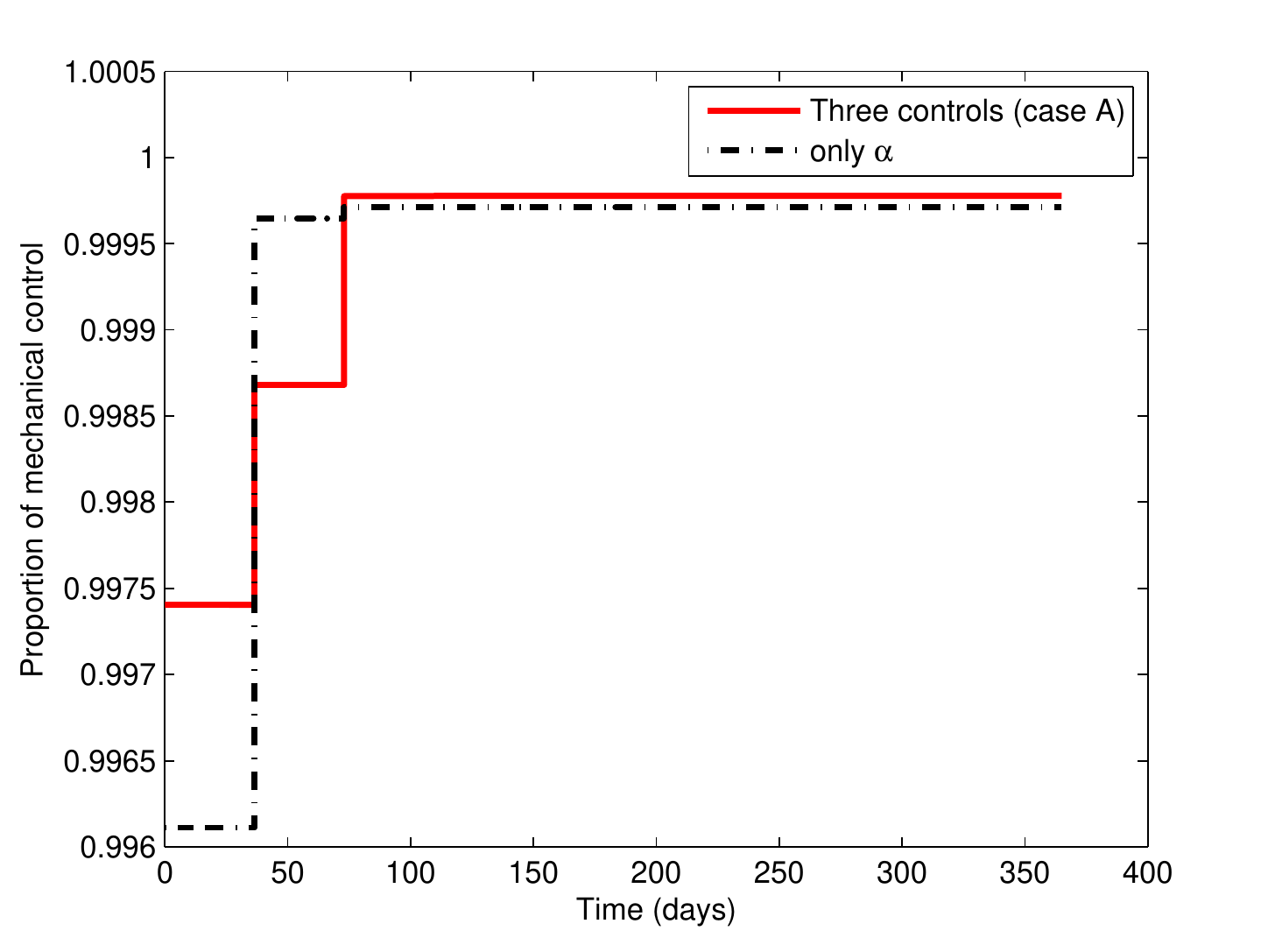}}
\subfloat[Proportion of infected human]{\label{cap6_infected_mechanical_vs_all}
\includegraphics[width=0.48\textwidth]{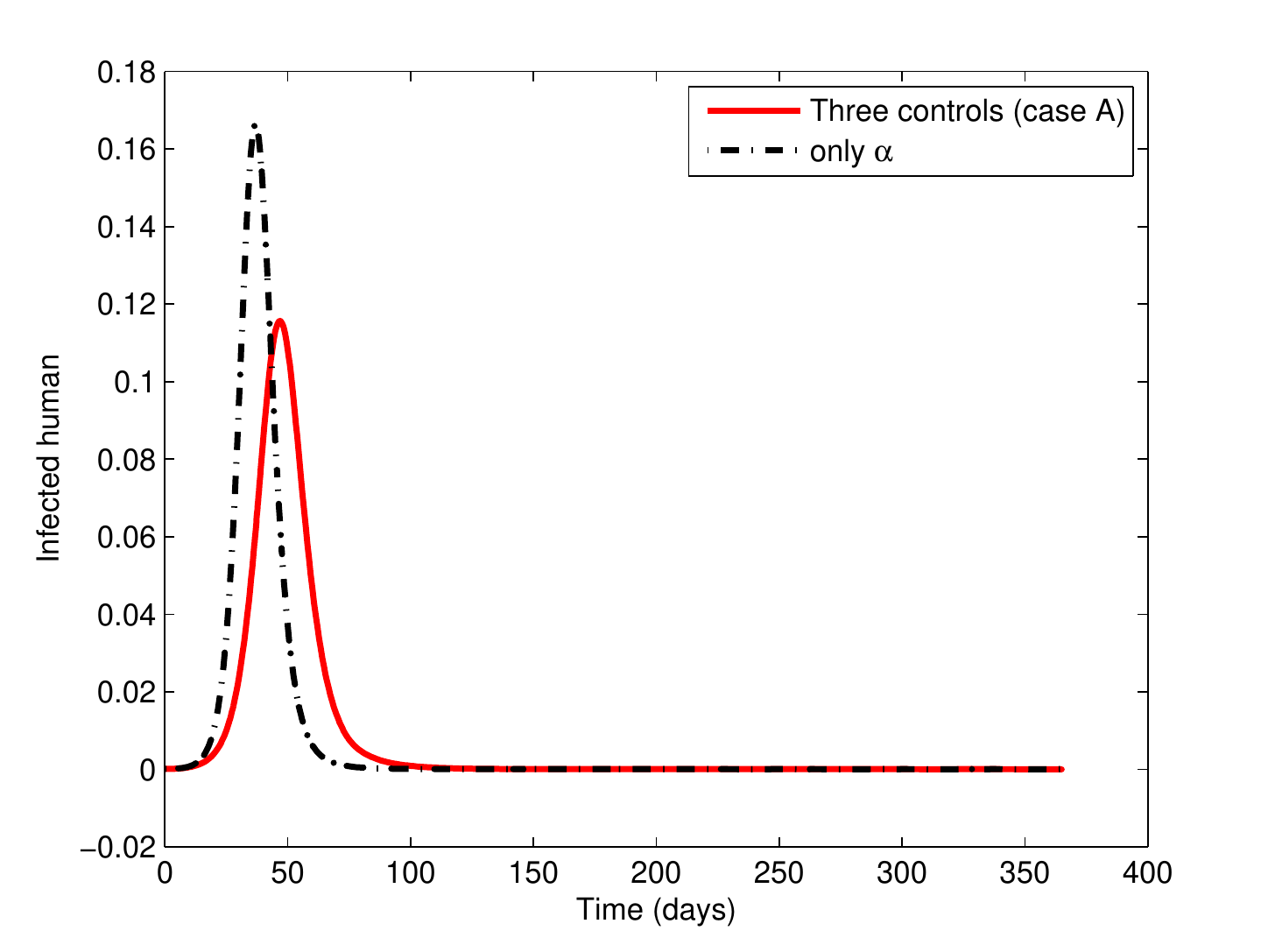}}
\caption{Optimal control and infected humans when considered all controls
(solid line) and only mechanical control (dashed line) in one year}
\label{cap6_mechanical}
\end{figure}


\section{Conclusions}
\label{sec:6:4}

Dengue disease breeds, even in the absence of fatal forms,
significant economic and social costs: absenteeism, debilitation
and medication. To observe and to act at the epidemics onset
could save lives and resources to governments.
Moreover, the under-reporting of Dengue cases is probably
the most important barrier in obtaining an accurate assessment.

A compartmental epidemiological
model for Dengue disease composed by a set of differential
equations was presented. Simulations based on clean-up
campaigns to remove the vector breeding sites,
and also simulations on the application of insecticides
(larvicide and adulticide), were made. It was shown that even with
a low, although continuous, index of control over time,
the results are surprisingly positive. The adulticide was
the most effective control, since with a low percentage of insecticide,
the basic reproduction number is kept below unit
and the number infected humans was smaller.
However, to bet only in adulticide is a risky decision. In some countries,
such as Mexico and Brazil, the prolonged use of adulticides has been increasing
the mosquito tolerance capacity to the product or even they become completely resistant.
In countries where Dengue is a permanent threat, governments must act with differentiated tools.
It will be interesting to analyze these controls in an endemic region
and with several outbreaks. We believe that the results will be quite different.
\emph{Aedes aegypti} eradication is not considered to be feasible
and, from the environmental point of view, not desirable.
The aim is to reduce the mosquito density and, simultaneously,
amount the level of immunity on the human population.
The increase of population herd immunity\index{Herd immunity}
can be reached by two ways:
increasing the resistant people to the disease
implying an increasing of infected individuals
or with a vaccination campaign.
No commercially available clinical cure or vaccine
is currently available for Dengue, but efforts are underway to
develop one \cite{Blaney2007,WHO2007}. The potential
of prevention of Dengue by immunization seems to be technically feasible
and progress is being made in the development of vaccines
that may protect against all four Dengue viruses.
In the next chapter, a model of this disease with a vaccine simulation
as a new strategy to fight the disease will be analyzed.

This chapter was based on work accepted in the peer reviewed journal \cite{Sofia2013}
and the peer reviewed conference proceedings \cite{Sofia2012b}.

\clearpage{\thispagestyle{empty}\cleardoublepage}


\chapter{Dengue with vaccine simulations}
\label{chp7}

\begin{flushright}
\begin{minipage}[r]{9cm}

\bigskip
\small {\emph{As the development of a Dengue vaccine is ongoing, it is
si\-mulated an hypothetical vaccine as an extra protection to the population.
In a first phase, the vaccination process is studied as a new compartment in the model,
and some different types of vaccines are simulated: pediatric and random mass vaccines,
with distinct levels of efficacy and durability. In a second step, the vaccination
is seen as a control variable in the epidemiological process. In both cases, epidemic
and endemic scenarios are included in order to analyze distinct outbreak realities.}}

\bigskip

\hrule
\end{minipage}
\end{flushright}

\bigskip
\bigskip

\onehalfspacing

In 1760, the Swiss mathematician Daniel Bernoulli published a study
on the impact of immunization with cowpox upon the expectation
of life of the immunized population \cite{Scherer2002}. The process
of protecting individuals from infection by immunization has become a routine,
with historical success in reducing both mortality and morbidity.

The impact of vaccination may be regarded not only as an individual
protective measure, but also as a collective one. While direct
individual protection is the major focus of a mass vaccination program,
the effects on population also contribute indirectly to other individual protection
through herd immunity\index{Herd immunity}, providing protection for unprotected
individuals \cite{Farrington2003} (see scheme in Figure~\ref{cap7_protection}).
This means that when we have a large neighborhood of vaccinated people,
a susceptible individual has a lower probability in coming into contact with the infection,
being more difficult for diseases to spread,
which decreases the relief of health facilities and can break the chain of infection.

\begin{figure}[ptbh]
\begin{center}
\includegraphics[scale=0.5]{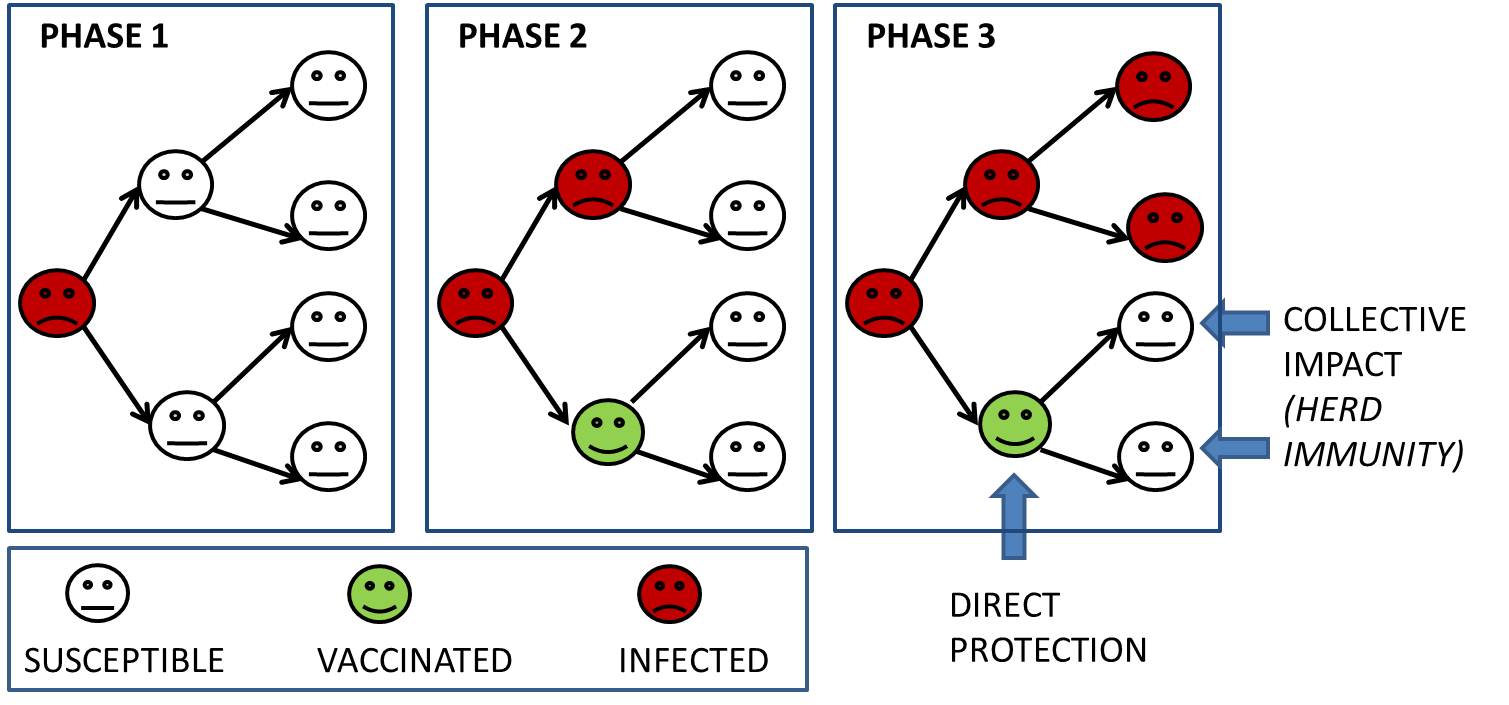}
\end{center}
\caption{Individual and collective protection provided
from a vaccination campaign \label{cap7_protection}}
\end{figure}


\section{About Dengue vaccine}
\label{sec:7:1}

Vector control remains the only available strategy against Dengue.
Despite integrated vector control with community participation,
along with active disease surveillance and insecticides,
there are only a few examples of successful Dengue prevention
and control on a national scale \cite{Cattand2006}.
Besides, the levels of resistance of \emph{Aedes aegypti} to
insecticides has increased, which implies shorter intervals between
treatments, and only few insecticide products are available in the
market due to the high costs for development and registration and low returns \cite{Keeling2008}.

Dengue vaccines have been under development since the 1940s,
but due to the limited appreciation of global disease burden
and the potential markets for Dengue vaccines, industry interest
languished throughout the 20th century. However, in recent years,
the development of Dengue vaccines has accelerated dramatically
with the increase in Dengue infections, as well as
the prevalence of all four circulating serotypes.
Faster development of a vaccine became a serious concern \cite{Murrell2011}.

Economic analysis are now conducted periodically to guide public support for vaccine development
in both industrialized and developing countries, including a previous cost-effectiveness
study of Dengue \cite{Clark2005,Shepard2004,Shepard2009}. The authors compared the cost
of the disease burden with the possibility of making a vaccination campaign;
they suggest that there is a potential economic benefit associated with promising
Dengue interventions, such as Dengue vaccines and vector control innovations,
when compared to the cost associated to the disease treatments.

Constructing a successful vaccine for Dengue has been challenging:
the knowledge of disease pathogenesis is insufficient and in addition
the vaccine must protect simultaneously against all serotypes
in order to not increase the level of DHF \cite{Supriatna2008}.

Nevertheless, several promising approaches are being investigated
in both academic and industrial laboratories. Vaccine candidates include
live attenuated vaccines obtained via cell passages or by recombinant
DNA technology (such as those being developed by the US National Institutes
of Allergy and Infectious Diseases, InViragen, Walter Reed Army Institute
of Research/GlaxoSmithKline, and Sanofi Pasteur) and subunit vaccines
(such as those developed by Merck/Hawaii Biotech) \cite{Guy2011, Whitehead2007}.
Recent studies indicate that, by the progress in clinical development of Sanofi
Pasteur's live attenuated tetravalent chimeric vaccine, a vaccine could be licensed
as early as 2014 \cite{Mahoney2011}. The team is carrying out an efficacy study
on a vaccine covering four serotypes on 4000 children aged four
to eleven years old in Muang district, Thailand.

At this time, the features of Dengue vaccine are mostly unknown.
So, in this chapter we opt to present a set of simulations with different
types of vaccines and we have explored also the vaccination process
under two different perspectives. The first one uses a new compartment
in the model and several kinds of vaccination are considered. A second perspective
is studied using the vaccination process as a disease control
in the mathematical formulation. In this case the theory of OC is applied.
Both methods assume a continuous strategy vaccination.


\section{Vaccine as a new compartment }
\label{sec:7:2}

In this section, a new compartment $V$ is added to the previous $SIR$
model related to the human population. This new compartment represents
the new group of human population that is vaccinated, in order to distinguish
the resistance obtained through vaccination and the one achieved by disease recovery.

Two forms of random vaccination are possible: the most common for human disease
is pediatric vaccination to reduce the prevalence of an endemic disease;
the alternative is random vaccination of the entire population in an outbreak situation.
In both types, the vaccination can be considered perfect conferring 100\% protection
for all life or else imperfect. This last case can be due to the difficulty
of producing an effective vaccine, the heterogeneity of the population
or even the life span of the vaccine.


\subsection{Perfect pediatric vaccine}
\label{sec:7:2:1}

For many potentially human infections, such as measles, mumps, rubella,
whooping cough, polio, there has been much focus on vaccinating newborns
or very young infants. Dengue can be a serious candidate for this type of vaccination.

In the $SVIR$ model, a continuous vaccination strategy is considered,
where a proportion of the newborn $p$ (where $0\leq p \leq 1$),
was by default vaccinated. This model also assumes that the permanent
immunity acquired through vaccination is the same as the natural immunity obtained
from infected individuals eliminating the disease naturally.
The population remains constant, \emph{i.e.}, $N_h=S_h+V_h+I_h+R_h$.
The new model for human population is represented in Figure~\ref{cap7_modeloSVIR_newborn}.
The mosquito population remains equal to the previous chapter, excluding the controls.

\begin{figure}[ptbh]
\begin{center}
\includegraphics[scale=0.6]{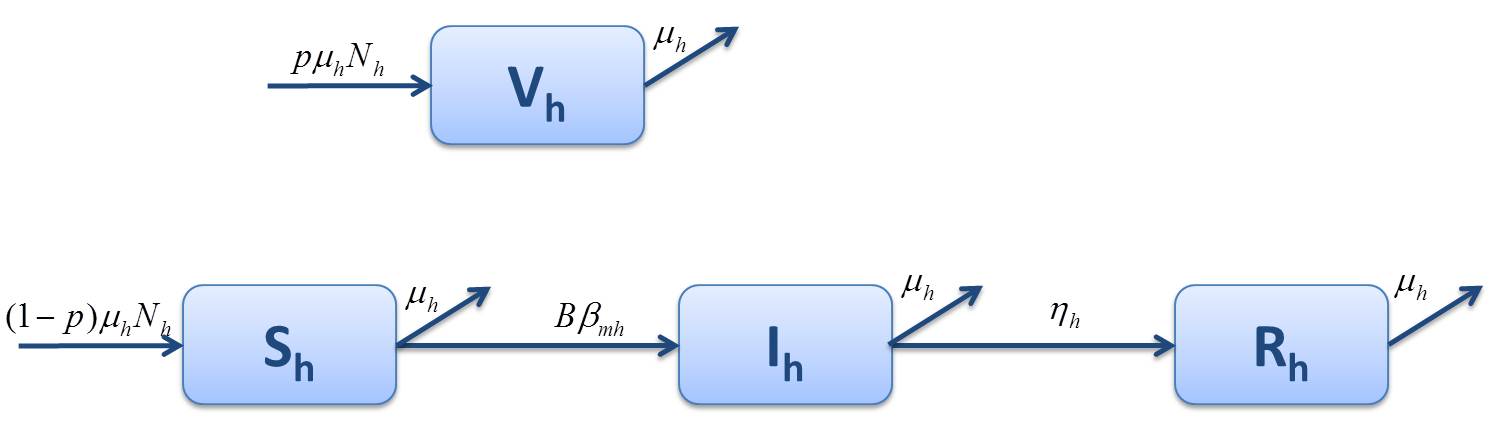}
\end{center}
\caption{Epidemic model for human population using
a pediatric vaccine \label{cap7_modeloSVIR_newborn}}
\end{figure}

All the parameters/variables remain with the same meaning
of the previous chapter. The mathematical formulation is:
\begin{equation}
\label{cap7_ode_vaccination_newborn}
\begin{cases}
\frac{dS_h}{dt} = (1-p)\mu_h N_h - \left(B\beta_{mh}\frac{I_m}{N_h}+\mu_h\right)S_h\\
\frac{dV_h}{dt} = p\mu_h N_h-\mu_h V_h\\
\frac{dI_h}{dt} = B\beta_{mh}\frac{I_m}{N_h}S_h -(\eta_h+\mu_h) I_h\\
\frac{dR_h}{dt} = \eta_h I_h - \mu_h R_h\\
\frac{dA_m}{dt} = \varphi \left(1-\frac{A_m}{ k N_h}\right)(S_m+I_m)
-\left(\eta_A+\mu_A \right) A_m\\
\frac{dS_m}{dt} = \eta_A A_m
-\left(B \beta_{hm}\frac{I_h}{N_h}+\mu_m \right) S_m\\
\frac{dI_m}{dt} = B \beta_{hm}\frac{I_h}{N_h}S_m
-\mu_m I_m
\end{cases}
\end{equation}

We are assuming that it is a perfect vaccine,
which means that it confers life-long protection.

As a first step, it is necessary to determine
the basic reproduction number without vaccination ($p=0$).

\begin{theorem}
\label{chap7_thm:r0}
The basic reproduction number\index{Basic reproduction number}, $\mathcal{R}_0$,
associated to the differential system \eqref{cap7_ode_vaccination_newborn}
without vaccination is given by
\begin{equation}
\label{cap7:eq:R0}
\mathcal{R}_0 =  \sqrt{ \frac{k B^2 \beta_{hm} \beta_{mh}\left(
-\eta_A \mu_m-\mu_A \mu_m+\varphi \eta_A\right)}{\varphi (\eta_h + \mu_h)  \mu_m^2}}.
\end{equation}
\end{theorem}

\medskip

\begin{proof}
The proof of this theorem is similar to the one in the
previous chapter (see proof of Theorem~\ref{thm:r0}).
Just consider
\begin{equation*}
\mathcal{F}(x)
=\left(
\begin{array}{c}
B \beta_{mh} \frac{I_{m}}{N_h}S_{h}  \\
B \beta_{hm} \frac{I_{h}}{N_h}S_{m}\\
\end{array}
\right), \quad
\mathcal{V}(x)=\left(
\begin{array}{c}
(\eta_h+\mu_h)I_{h}  \\
\mu_mI_{m} \\
\end{array}
\right).
\end{equation*}
\end{proof}

In this chapter, we make all the simulations in two scenarios:
an epidemic and an endemic situation (programming codes available in \cite{SofiaSITE}).
For these, the following parameter values of the differential system
and initial conditions were used (Tables~\ref{chap7_parameters}
and \ref{chap7_initial_conditions}).

\begin{table}[ptbh]
\begin{center}
\begin{tabular}{lcc}
\hline
Parameter & Epidemic scenario & Endemic scenario\\
\hline
$N_h$ & 480000 & 480000\\
$B$ & 0.8 & 0.75\\
$\beta_{mh}$ & 0.375 & 0.21\\
$\beta_{hm}$ & 0.375 & 0.21\\
$\mu_{h}$ & $\frac{1}{71\times365}$ & $\frac{1}{71\times365}$\\
$\eta_{h}$ & $\frac{1}{3}$& $\frac{1}{3}$\\
$\mu_{m}$ & $\frac{1}{10}$ & $\frac{1}{10}$\\
$\varphi$ & 6 & 6\\
$\mu_{A}$ & $\frac{1}{4}$& $\frac{1}{4}$\\
$\eta_{A}$ & 0.08 & 0.08\\
$m$ & 3 & 3\\
$k$ & 3 & 3\\
\hline
\end{tabular}
\caption{Parameters values of the differential system (\ref{cap7_ode_vaccination_newborn})}
\label{chap7_parameters}
\end{center}
\end{table}

\begin{table}[ptbh]
\begin{center}
\begin{tabular}{lcc}
\hline
Initial conditions & Epidemic scenario & Endemic scenario\\
\hline
$S_{h0}$ & 479990 & 379990\\
$V_{h0}$ & 0 & 0 \\
$I_{h0}$ & 10 & 10 \\
$R_{h0}$ & 0 & 100000\\
$A_{m0}$ & 1440000 & 1440000\\
$S_{m0}$ & 1440000 & 1440000\\
$I_{m0}$ & 0 & 0\\
\hline
\end{tabular}
\caption{Initial conditions of the differential system (\ref{cap7_ode_vaccination_newborn})}
\label{chap7_initial_conditions}
\end{center}
\end{table}

There were two main differences between the epidemic episode and an endemic situation.
Firsty, in the endemic situation there was a slight decrease in the average daily
biting $B$ and transmission probabilities
$\beta_{mh}$ and $\beta_{hm}$, that could be explained by the fact that the mosquito
could have more difficulties to find a naive individual.
The second difference is concerned with the strong increase of the initial human
population that is resistant to the disease. This may be explained by the fact
that the disease, in an endemic situation, already creates an immune resistance
to the infection, \emph{i.e.}, the population already has herd immunity\index{Herd immunity}.
With these values we obtain approximately $\mathcal{R}_{0}=2.46$ and $\mathcal{R}_{0}=1.29$
for epidemic and endemic scenarios, respectively.

During an outbreak, the disease transmission assumes different behaviors,
according to the distinct scenarios, as can been seen in Figure~\ref{cap7_population_no_vaccine}.
In one year, the peak in an epidemic situation could reach more than 80000 cases. In contrast,
instead in the endemic situation the curve of infected individuals has a more smooth behavior
and reaches a peak less than 3000 cases. Figure~\ref{cap7_population_no_vaccine_mosquito}
relates to the mosquito population. In the endemic scenario, because a substantial part
of the human population is resistant to the disease, the infected mosquitoes bite
a considerable percentage of resistant host, and as consequence,
the disease is not transmitted.

\begin{figure}
\centering
\subfloat[Epidemic scenario]{\label{cap7_epid_no_vaccine}
\includegraphics[width=0.49\textwidth]{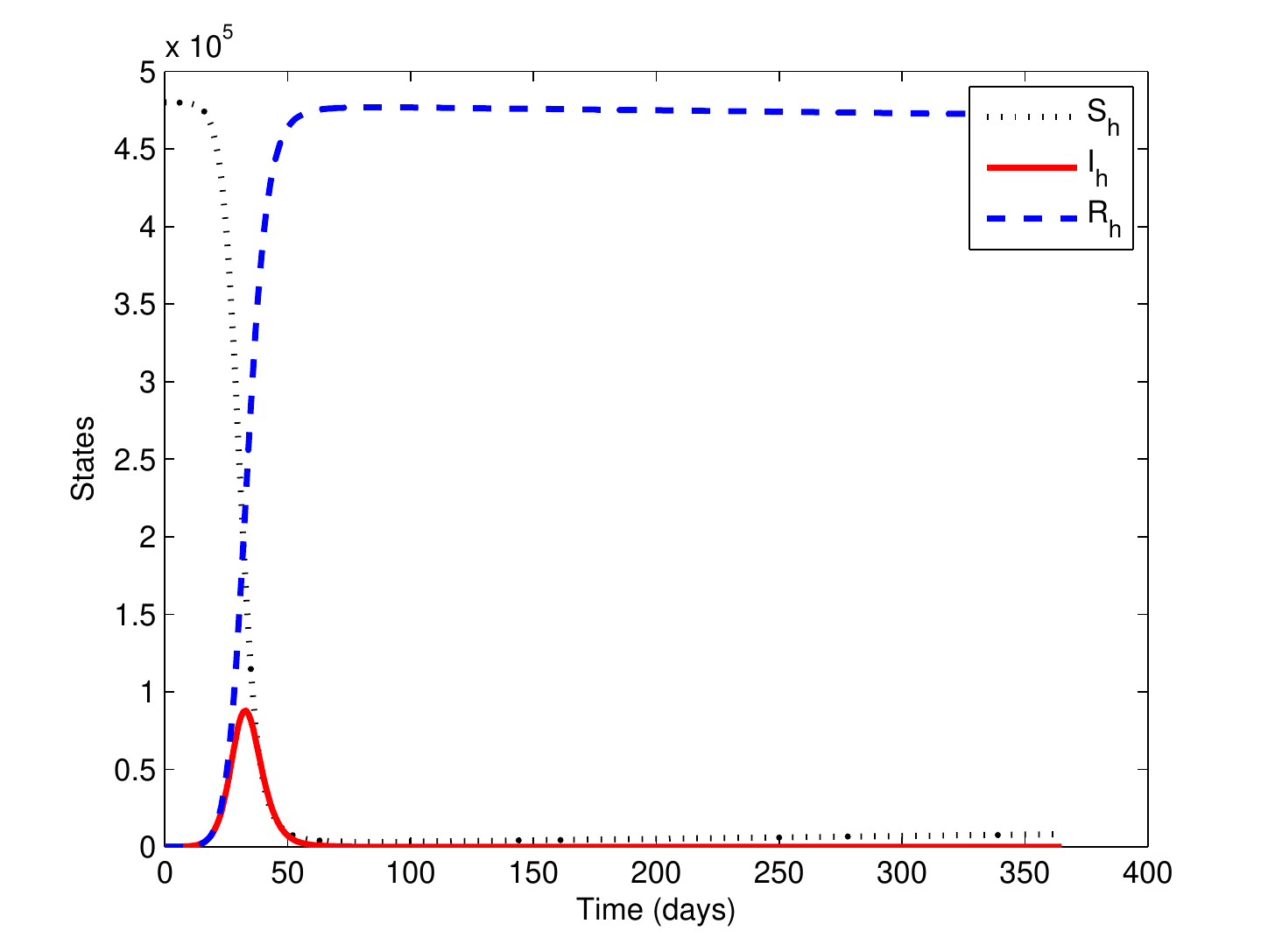}}
\subfloat[Endemic scenario]{\label{cap7_end_no_vaccine}
\includegraphics[width=0.49\textwidth]{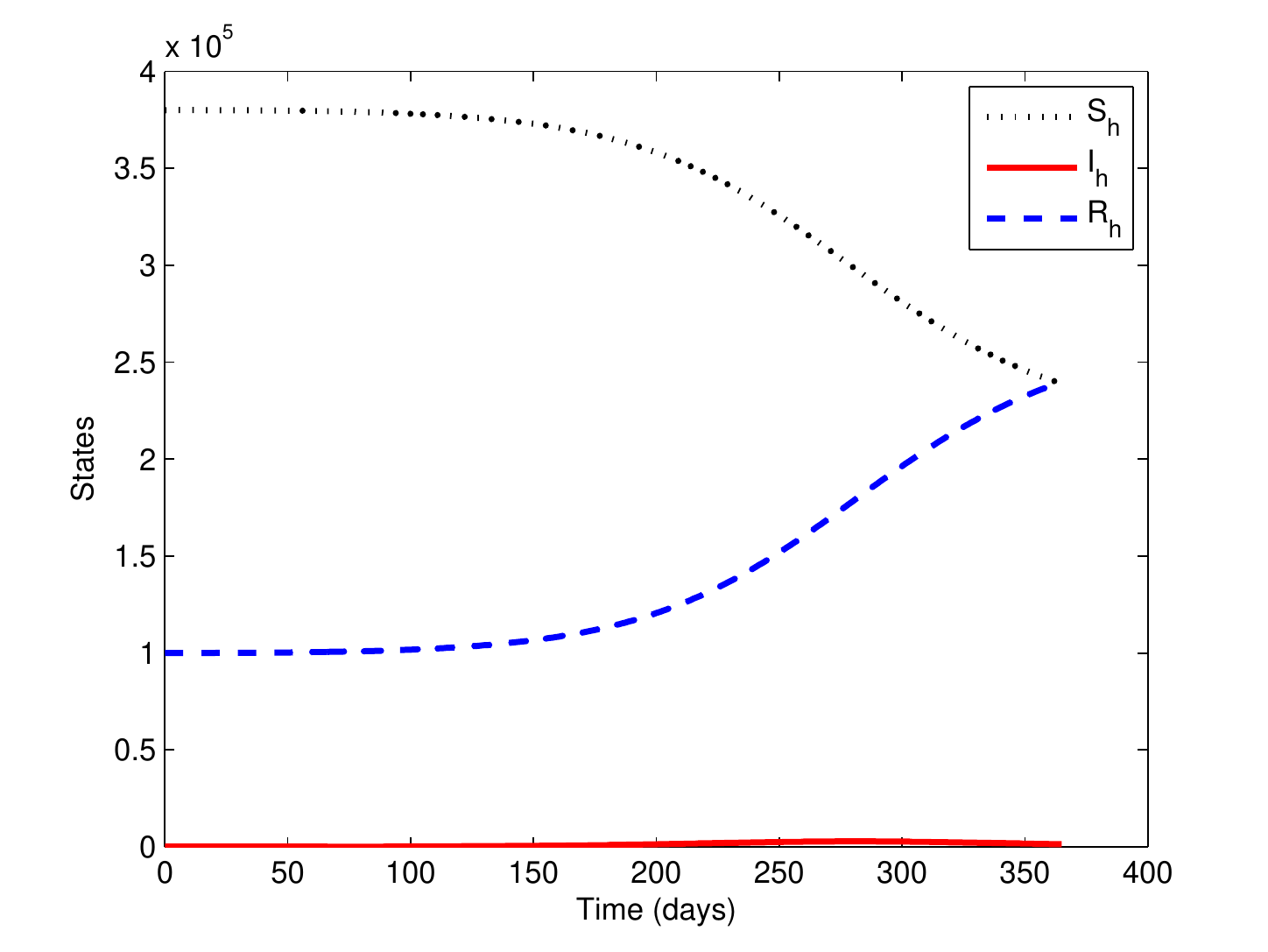}}
\caption{Human population in a Dengue outbreak, without vaccine}
\label{cap7_population_no_vaccine}
\end{figure}

\begin{figure}
\centering
\subfloat[Epidemic scenario]{\label{cap7_epid_no_vaccine_mosquito}
\includegraphics[width=0.49\textwidth]{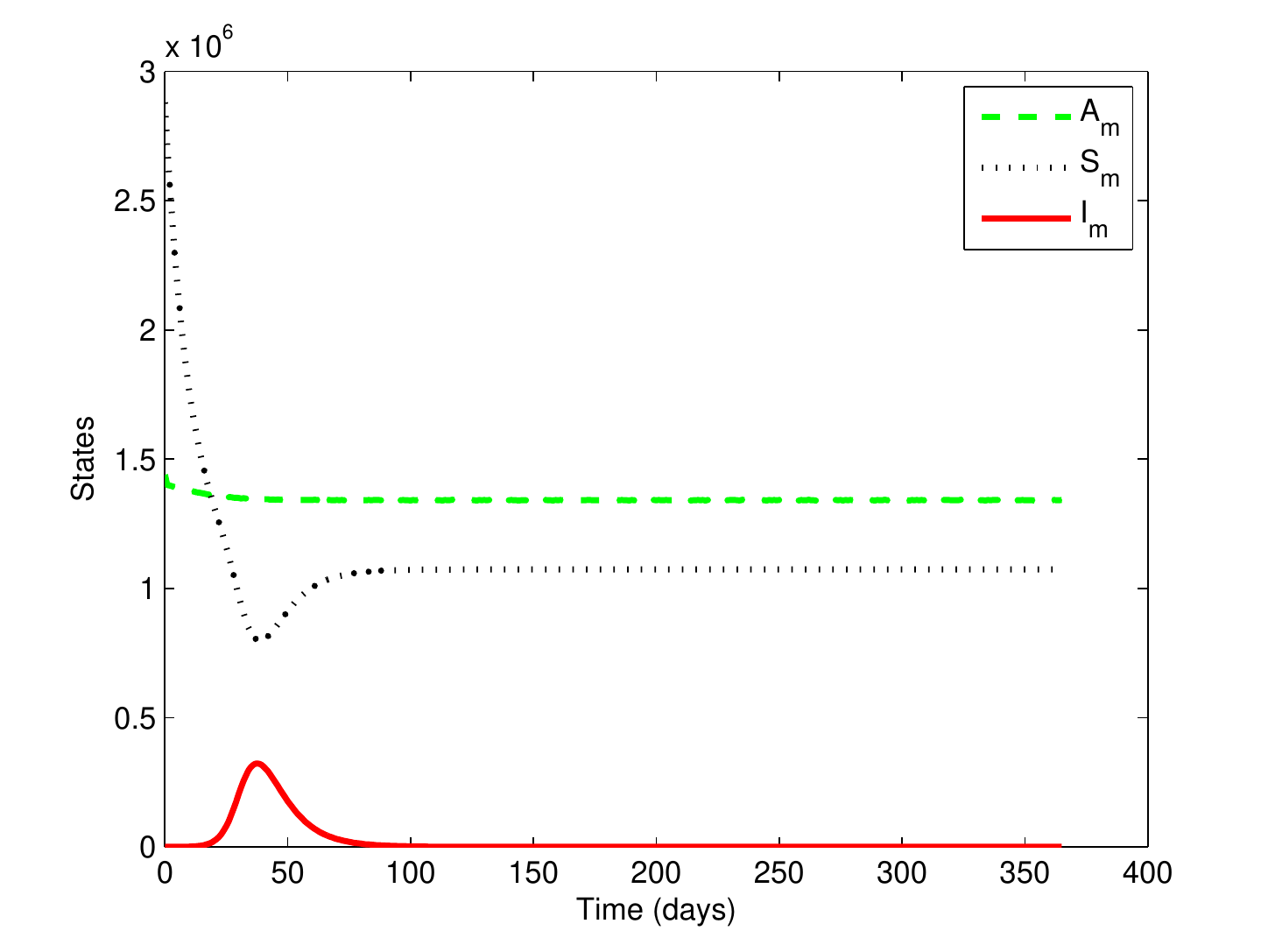}}
\subfloat[Endemic scenario]{\label{cap7_end_no_vaccine_mosquito}
\includegraphics[width=0.49\textwidth]{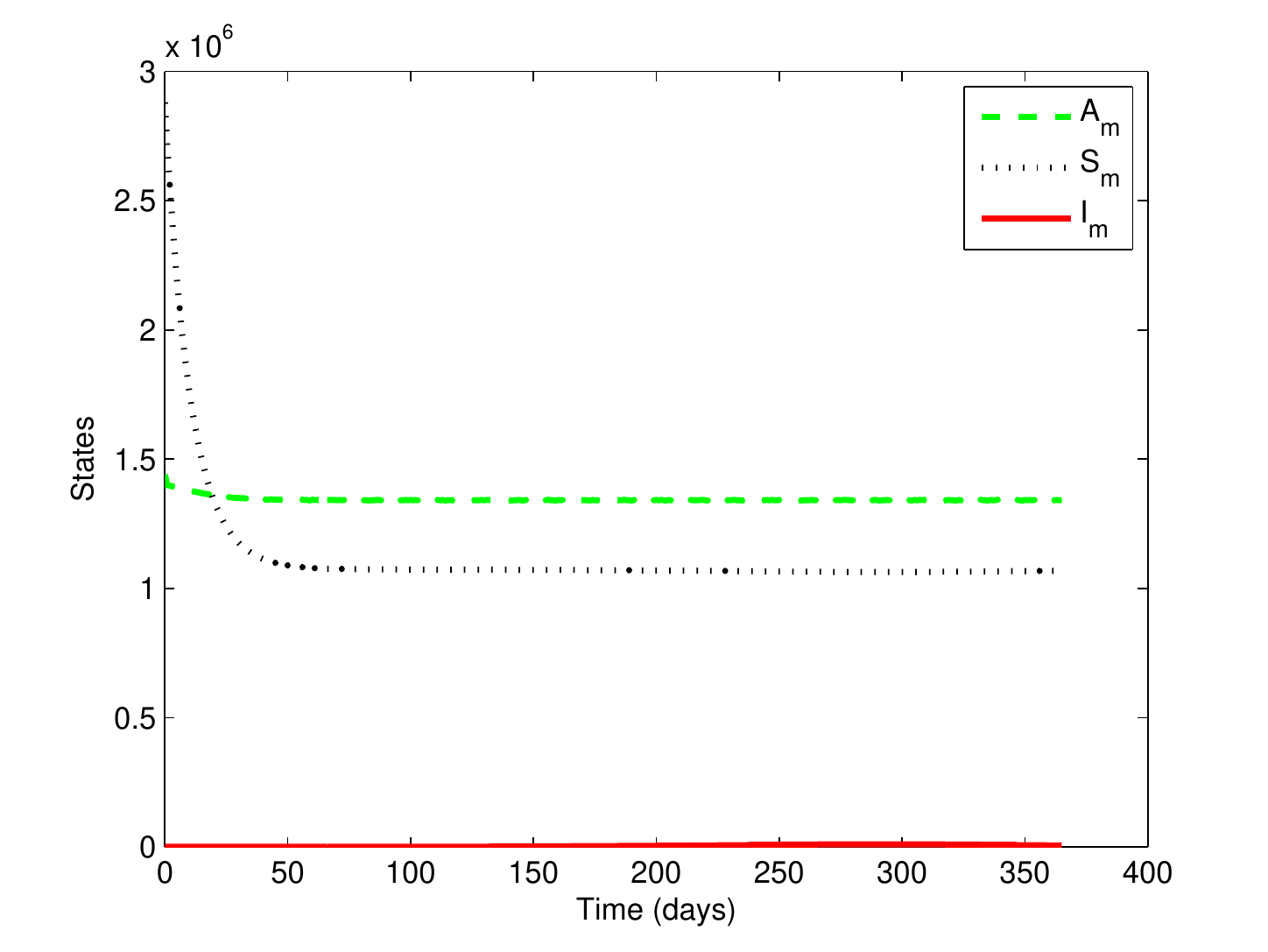}}
\caption{Mosquito population in a Dengue outbreak, without vaccine}
\label{cap7_population_no_vaccine_mosquito}
\end{figure}

Suppose that at time $t=0$ a proportion $p$ of newborns are vaccinated
with a perfect vaccine that causes no side effects. Since this proportion,
$p$, is now immune, $\mathcal{R}_{0}$ is reduced,
creating a new basic reproduction number.

\begin{definition}[Basic reproduction number with pediatric vaccination]
\label{chap7_thm:r0:newborn}\index{Basic reproduction number}
The basic reproduction number with pediatric vaccination, $\mathcal{R}_0^{p}$, associated to the
differential system \eqref{cap7_ode_vaccination_newborn} is given by
\begin{equation*}
\label{eq:R0_newborn}
\mathcal{R}_0^{p} = (1-p)\mathcal{R}_{0}
\end{equation*}
\noindent where $\mathcal{R}_{0}$ is defined in (\ref{cap7:eq:R0}).
\end{definition}

Observe that $\mathcal{R}_0^{p}\leq\mathcal{R}_0$.
Equality is only achieved when $p=0$,
\emph{i.e.}, when there is no vaccination. The constraint $\mathcal{R}_0^{p}<1$
implicitly defines a critical vaccination portion $p > p_{c}$
that must be achieved for eradication, where
$$
p_{c}=1-\frac{1}{\mathcal{R}_{0}}.
$$

Since vaccination entails costs, to choose the smallest coverage
that achieves eradication is the best option. This way, the entire
population does not need to be vaccinated in order to eradicate the disease.
This phenomenon is called herd immunity\index{Herd immunity}.

Vaccinating at the critical level, $p_{c}$, does not instantly lead
to disease eradication. The immunity level within the population
requires time to build up and at the critical level it may take
a few generations before the required herd immunity is achieved.
Thus, from a public health perspective, $p_{c}$ acts as a lower bound
on what should be achieved, with higher levels of vaccination
leading to a more rapid elimination of the disease.
Figure~\ref{cap7_p_variation} shows the simulations related
to the proportion of newborns vaccinated ($p=0, 0.25, 0.50, 0.75, 1$)
in both scenarios. Notice that at time $t=0$ no person was vaccinated.
In the epidemic situation, as the outbreak reached a peak at the beginning of the year,
the proportion of newborns vaccinated at that time is minimum
and cannot influence the curve of infected individuals,
giving the optical illusion of a single curve. On the other hand,
in the endemic case, as the outbreak occurs later, the vaccination campaign
starts to produce effects, decreasing the total number of sick humans. This last
graphic illustrates that a vaccination campaign centered in newborns
is a bet for the future of a country, but does not produce instantly results
to fight the disease. To produce immediate results, it is necessary
to use random mass vaccination, which means that it is necessary
to vaccine a significant part of the population.

\begin{figure}
\centering
\subfloat[Epidemic scenario]{\label{cap7_epid_p_variation}
\includegraphics[width=0.5\textwidth]{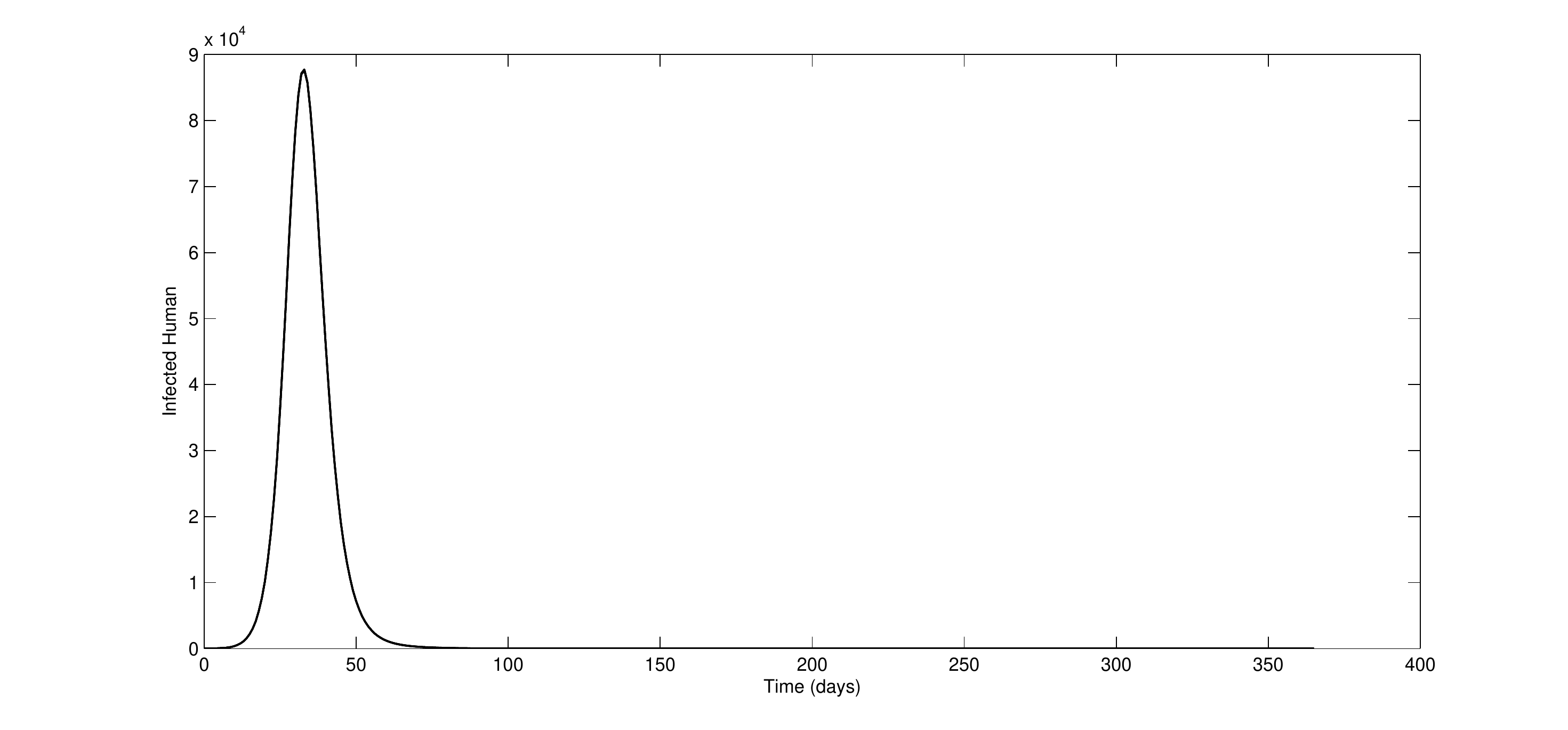}}
\subfloat[Endemic scenario]{\label{cap7_end_p_variation}
\includegraphics[width=0.5\textwidth]{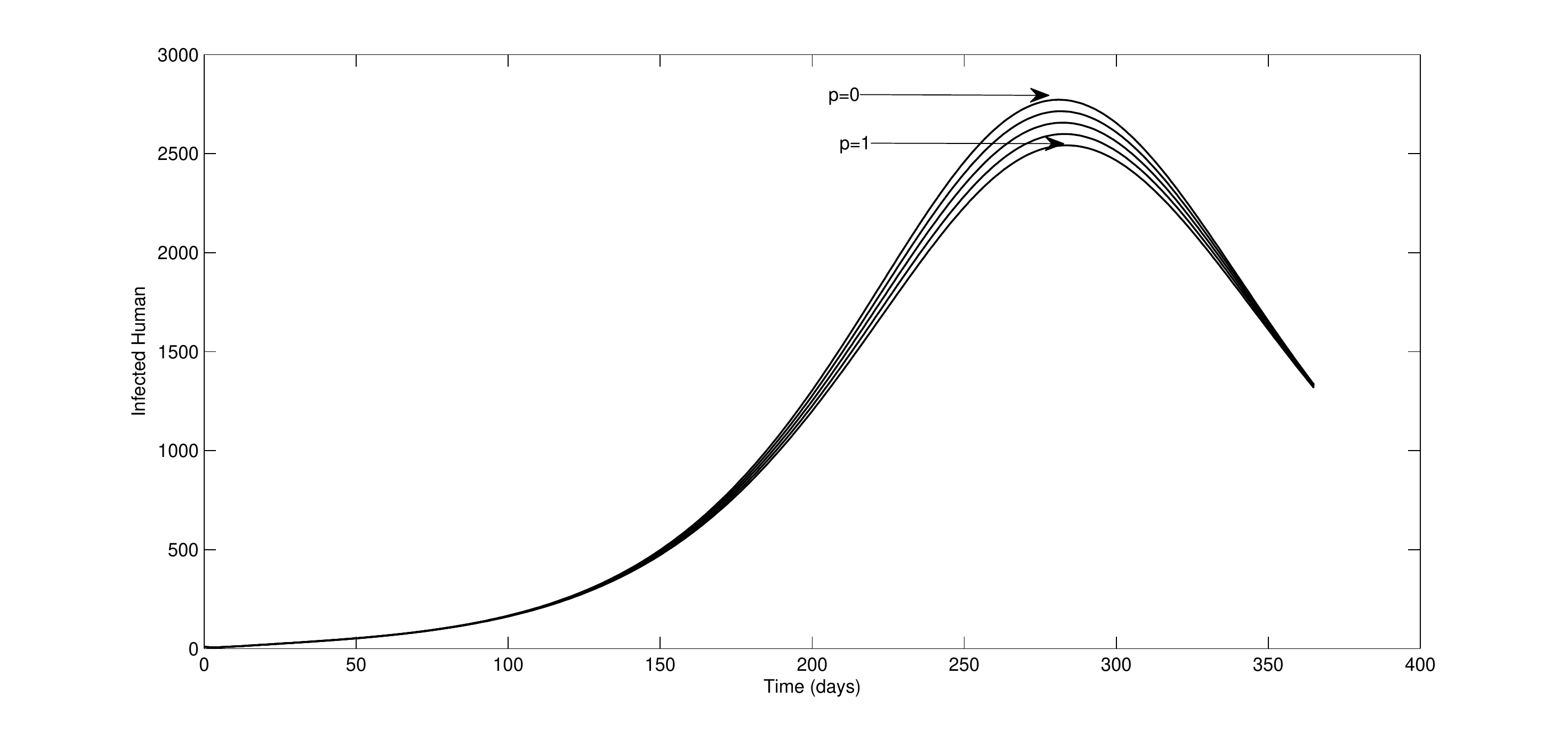}}
\caption{Infected human in an outbreak, varying the proportion
of newborns vaccinated ($p=0, 0.25, 0.50, 0.75, 1$)}
\label{cap7_p_variation}
\end{figure}


\subsection{Perfect random mass vaccination}
\label{sec:7:2:2}

A mass vaccination program may be initiated whenever there is an increase of
the risk of an epidemic. In such situations, there is a competition between
the exponential increase of the epidemic and the logistical constraints
upon mass vaccination. For most human diseases it is possible, and more efficient,
to not vaccinate those individuals who have recovered from the disease because
they are already protected. Another situation could be the introduction
of a new vaccine in a population that lives an endemic situation.

Let us consider the control technique of constant vaccination of susceptibles.
In this scheme a fraction $0 \leq \psi \leq1$ of the entire susceptible population,
not just newborns, is being continuously vaccinated. It is assumed that the permanent
immunity acquired by vaccination is the same as natural immunity obtained
from infected individuals in recovery. The epidemiological scheme
is presented in Figure~\ref{cap7_modeloSVIR_mass}.

\begin{figure}[ptbh]
\begin{center}
\includegraphics[scale=0.6]{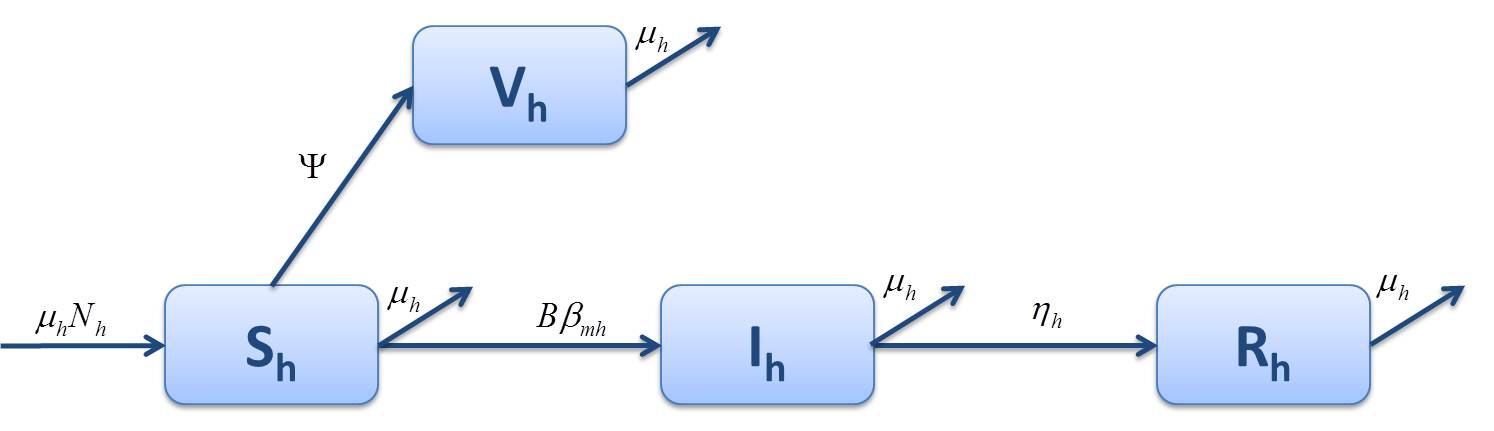}
\end{center}
\caption{Epidemic model for human population using
a mass random vaccine \label{cap7_modeloSVIR_mass}}
\end{figure}

The mathematical formulation for human population (the differential
equations related to the mosquito remain equal to the previous subsection) is given by:

\begin{equation}
\label{cap7_ode_vaccination_mass}
\begin{cases}
\frac{dS_h}{dt} = \mu_h N_h - \left(B\beta_{mh}\frac{I_m}{N_h}+\psi+\mu_h\right)S_h\\
\frac{dV_h}{dt} = \psi S_h-\mu_h V_h\\
\frac{dI_h}{dt} = B\beta_{mh}\frac{I_m}{N_h}S_h -(\eta_h+\mu_h) I_h\\
\frac{dR_h}{dt} = \eta_h I_h - \mu_h R_h\\
\end{cases}
\end{equation}

For this model, we define a new basic reproduction number.

\begin{definition}[Basic reproduction number with random mass vaccination \cite{Zhou2003}]
\label{chap7_thm:r0:mass}\index{Basic reproduction number}
The basic reproduction number with random mass vaccination, $\mathcal{R}_0^{\psi}$, associated to the
differential system \eqref{cap7_ode_vaccination_mass} is
\begin{equation}
\label{eq:R0_mass}
\mathcal{R}_0^{\psi} = \mathcal{R}_{0}\left(\frac{\mu_h}{\mu_h+\psi}\right),
\end{equation}
\noindent where $\mathcal{R}_{0}$ is defined in (\ref{cap7:eq:R0}).
\end{definition}

Comparing this model with the constant vaccination of newborns model,
it is apparent that instead of constantly vaccinating a portion of newborns,
a part of the entire susceptible population is now being continuously
vaccinated. Since the natural birth rate $\mu_h$ is usually small, the fraction $p\mu_h$
of newborns being continuously vaccinated will be small, whereas in this
model, will be also a larger group of susceptibles can be continuously vaccinated, $\psi S_h$. Due to this, we
expect that this model should require a smaller proportion of $\psi$ to achieve eradication.

Notice that $\mathcal{R}_0^{\psi}\leq\mathcal{R}_0$. Equality is only achieved
in the limit $\psi=0$, that is, when there is no vaccination. The constraint
$\mathcal{R}_0^{\psi}<1$ implicitly defines a critical
vaccination portion $\psi > \psi_{c}$ that must be achieved for eradication, where
$$
\psi_{c}=\left(\mathcal{R}_{0}-1\right)\mu_h.
$$

Figure~\ref{cap7_psi_variation} illustrates the variation of the number
of infected people when a mass vaccination is introduced.
The graphs present five simulations using different proportions
of the vaccinated population: $\psi=0.05, 0.10, 0.25, 0.50, 1$. Observe that
in spite of the calculations done in the period of 365 days, the figures
only show suitable windows, in order to provide a better analysis. In both scenarios,
even with a small coverage of population, vaccination dramatically decreases the number of infected.
The epidemic scenario has changed from less than 80000 cases (with no vaccination,
Figure~\ref{cap7_population_no_vaccine}) to less than 1200 cases vaccinating
only 5\% of population. In the endemic scenario, the decrease is more accentuated.

\begin{figure}
\centering
\subfloat[Epidemic scenario]{\label{cap7_epid_psi_variation}
\includegraphics[width=0.49\textwidth]{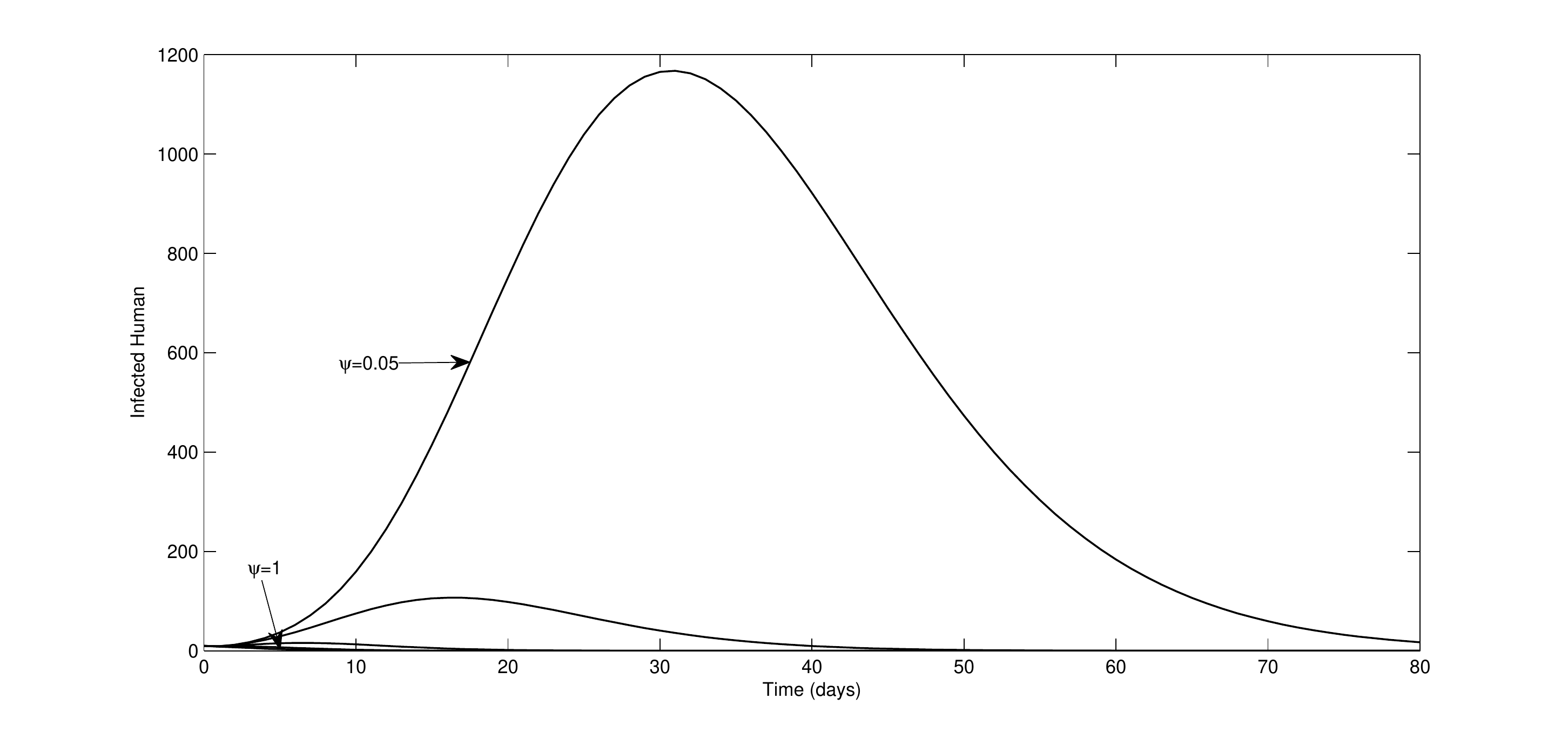}}
\subfloat[Endemic scenario]{\label{cap7_end_psi_variation}
\includegraphics[width=0.49\textwidth]{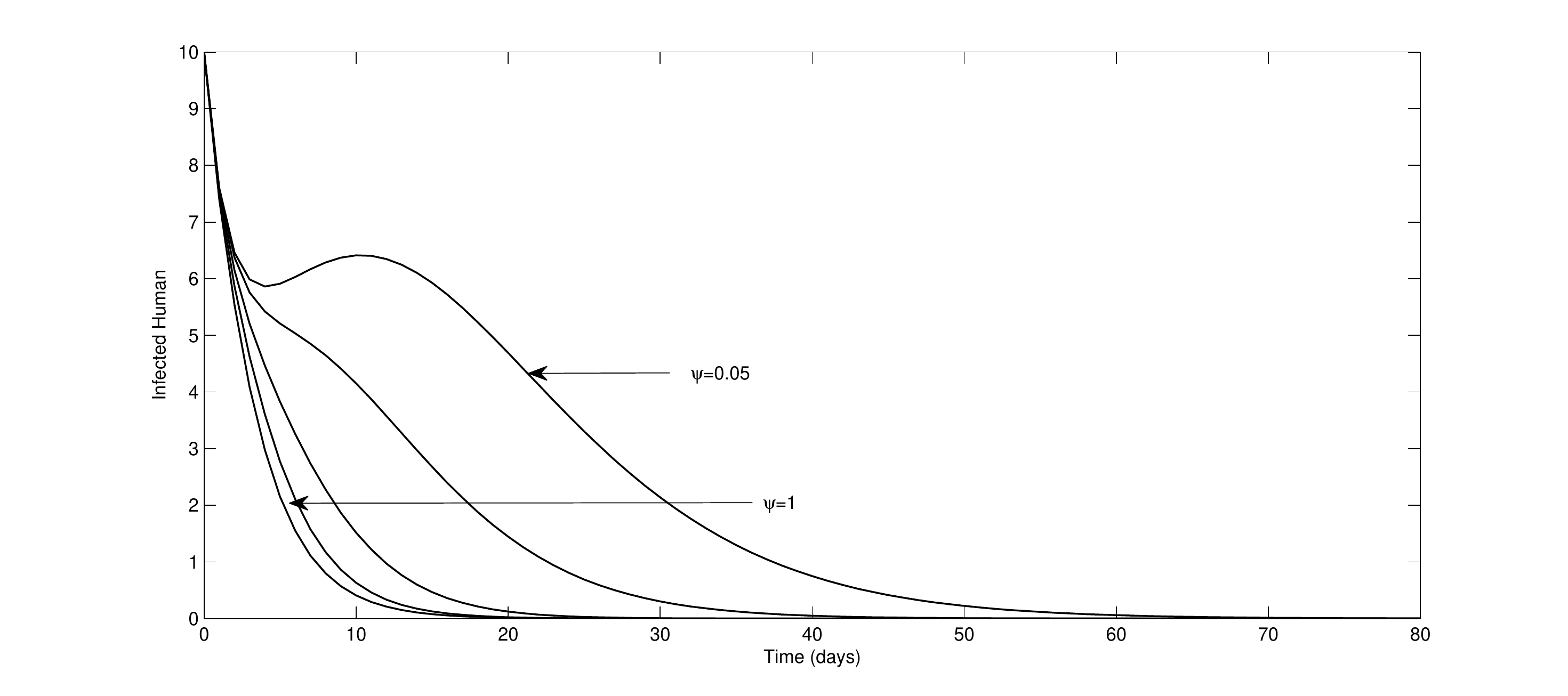}}
\caption{Infected human in an outbreak, varying the proportion
of susceptible population vaccinated ($\psi=0.05, 0.10, 0.25, 0.50, 1$)}
\label{cap7_psi_variation}
\end{figure}

Until here, we have considered a perfect vaccine, which means that
every vaccinated individual remains resistant to the disease.
However, a majority of the available vaccines for the human population
does not produce 100\% success in the disease battle. Usually,
the vaccines are imperfect, which means that a minor percentage of cases,
in spite of vaccination, are infected.


\subsection{Imperfect random mass vaccination}
\label{sec:7:2:3}

Most of the theory on the disease evolution is based on the assumption
that the host population is homogeneous. Individual hosts, however,
may differ and they may constitute very different habitats. In particular,
some habitats may provide more resources or be more vulnerable
to virus exploitation \cite{Gandon2003}. The use of models with
imperfect vaccines can describe better this type of human heterogeneity.
Another explanation for the use of imperfect vaccines is that until now
we had considered models that assumed that as soon as individuals
begin the vaccination process, they become immediately immune to the disease.
However, the time it takes for individuals to obtain immunity by completing
a vaccination process cannot be ignored, because meanwhile an individual can be infected.

In this section a continuous vaccination strategy is considered,
where a fraction $\psi$ of the susceptible class was vaccinated.
The vaccination may reduce but not completely eliminate susceptibility to infection.
For this reason we consider a factor $\sigma$ as the infection rate of
vaccinated members. When $\sigma=0$ the vaccine is perfectly
effective and when $\sigma=1$ the vaccine has no effect at all.
The value $1-\sigma$ can be understood as the efficacy level of the vaccine.

The new model for the human population is represented
in Figure~\ref{cap7_modeloSVIR_imperfect}.

\begin{figure}[ptbh]
\center
\includegraphics[scale=0.6]{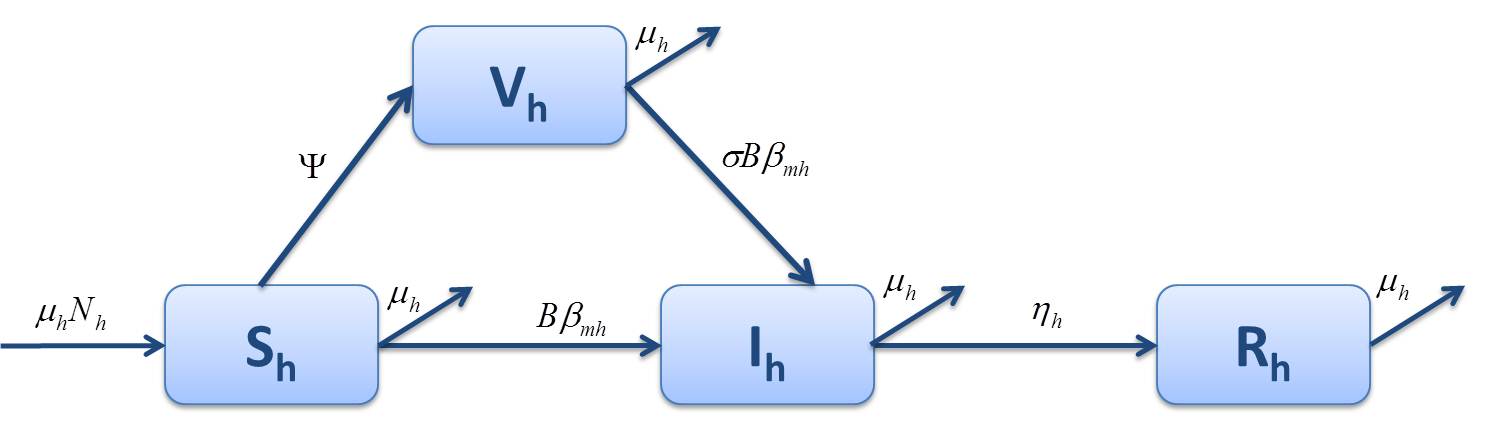}
\caption{\label{cap7_modeloSVIR_imperfect} Epidemiological $SVIR$ model
for human population with an imperfect vaccine}
\end{figure}

Therefore, the differential system is as follows:

\begin{equation}
\label{cap7_ode_vaccination_imperfect}
\begin{cases}
\frac{dS_h}{dt} = \mu_h N_h  - \left(B\beta_{mh}\frac{I_m}{N_h}+\psi +\mu_h\right)S_h\\
\frac{dV_h}{dt} = \psi S_h-\left(\sigma B \beta_{mh}\frac{I_m}{N_h}+\mu_h\right)V_h\\
\frac{dI_h}{dt} = B\beta_{mh}\frac{I_m}{N_h}(S_h+\sigma V_h) -(\eta_h+\mu_h) I_h\\
\frac{dR_h}{dt} = \eta_h I_h - \mu_h R_h.
\end{cases}
\end{equation}

For this system of differential equations, we have
a new basic reproduction number.

\bigskip

\begin{definition}[Basic reproduction number with an imperfect vaccine \cite{Liu2008}]
\label{chap7_thm:r0:imperfect}\index{Basic reproduction number}
The basic reproduction number with an imperfect vaccine, $\mathcal{R}_0^{\sigma}$, associated to the
differential system \eqref{cap7_ode_vaccination_imperfect} is
\begin{equation*}
\label{eq:R0_imperfect}
\mathcal{R}_0^{\sigma} = \left(1+\sigma \psi\right)\frac{\mu_h}{\mu_h
+\psi}\mathcal{R}_{0}=\left(1+\sigma \psi\right)\mathcal{R}_0^{\psi},
\end{equation*}
\noindent where $\mathcal{R}_{0}$ is defined in (\ref{cap7:eq:R0}).
\end{definition}

Notice that $\mathcal{R}_0^{\psi}\leq\mathcal{R}_0^{\sigma}$ and when
the vaccine is perfect ($\sigma=0$), $\mathcal{R}_0^{\sigma}$ degenerates
into $\mathcal{R}_0^{\psi}$. In other words, a high efficacy vaccine
leads to a lower vaccination coverage to eradicate the disease.
However, it is noted in \cite{Liu2008} that it is much more difficult to increase the
efficacy level of the vaccine when compared to controlling the vaccination rate $\psi$.

\begin{figure}
\centering
\subfloat[Epidemic scenario]{\label{epid_psi_vary}
\includegraphics[width=0.49\textwidth]{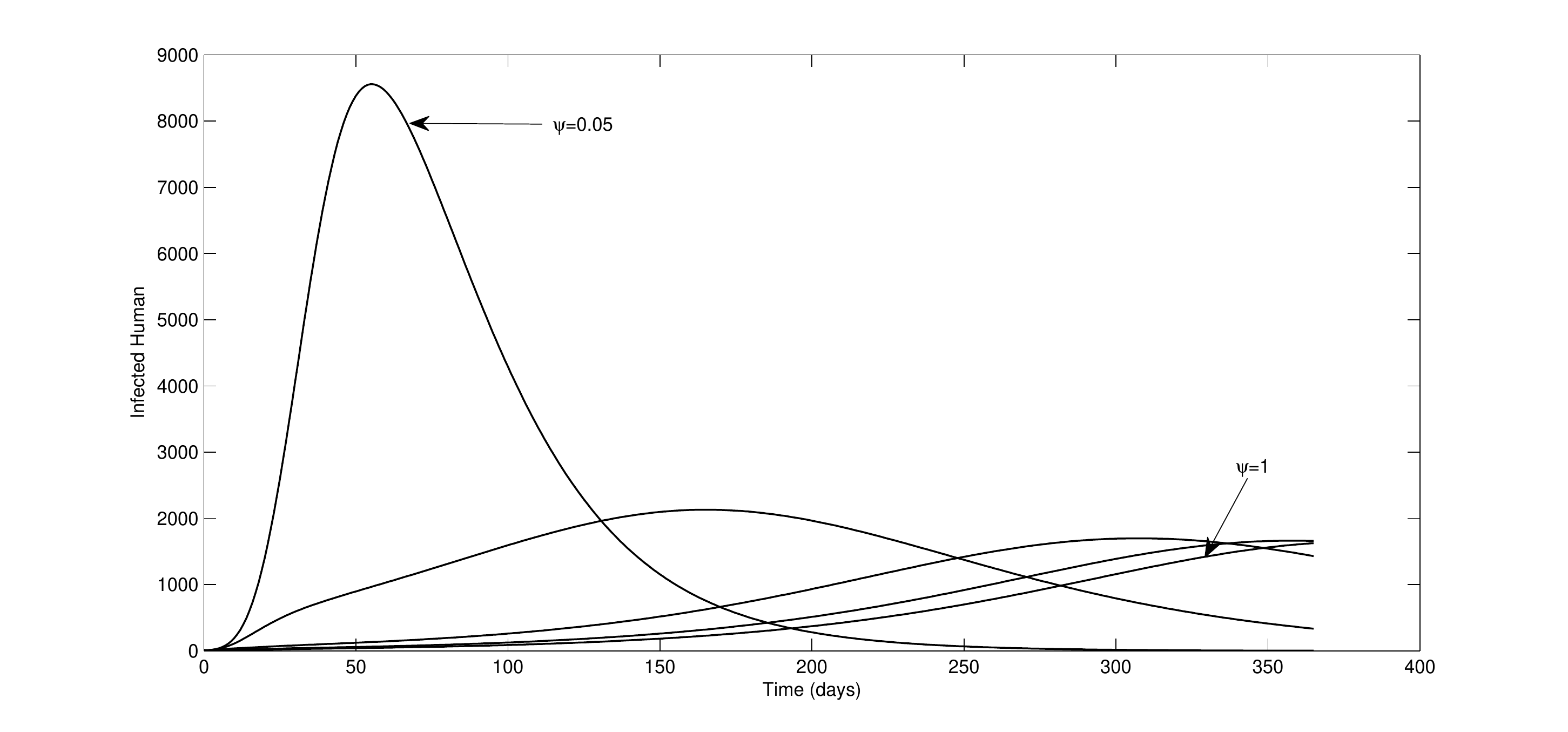}}
\subfloat[Endemic scenario]{\label{end_psi_vary}
\includegraphics[width=0.49\textwidth]{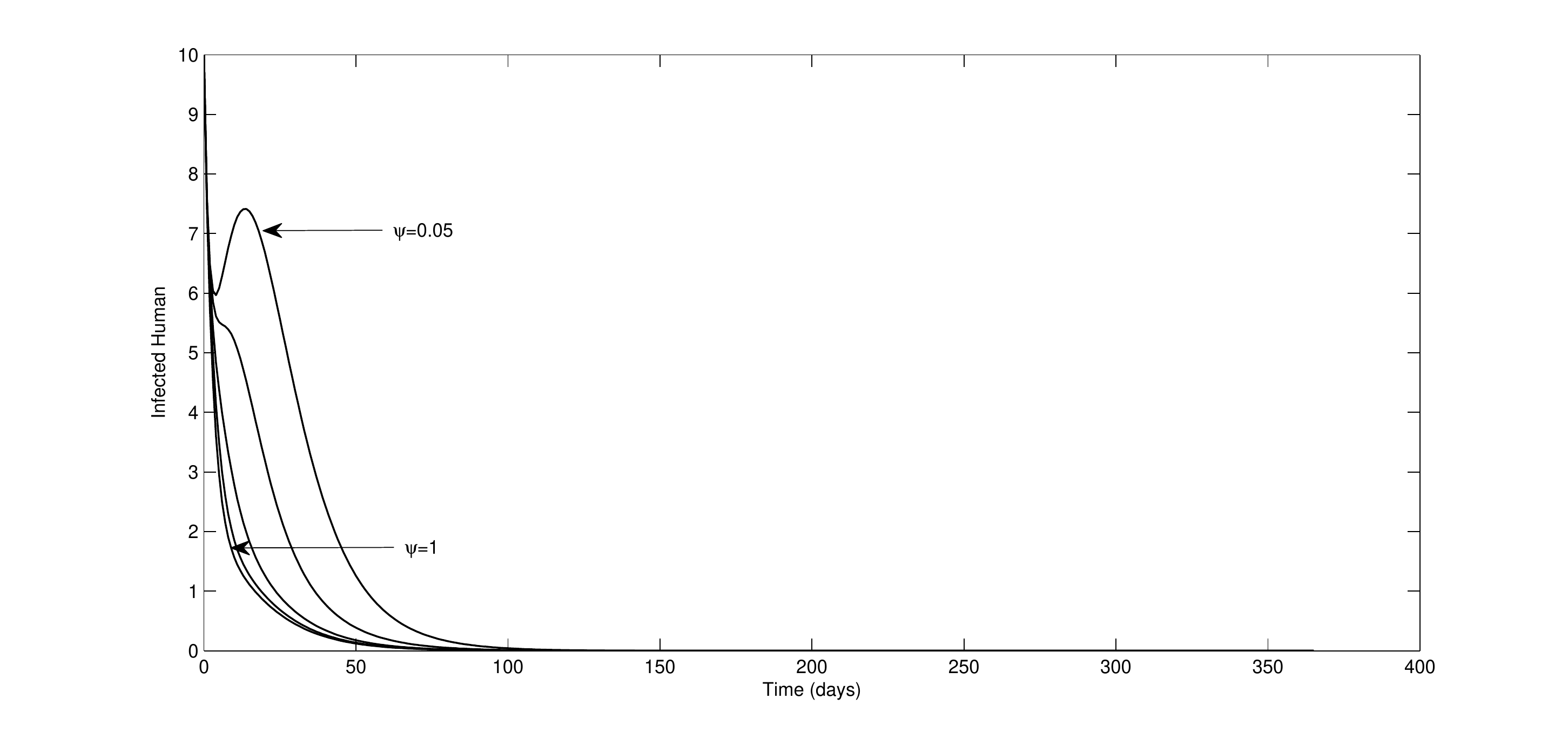}}\\
\subfloat[Epidemic scenario]{\label{epid_sigma_vary}
\includegraphics[width=0.49\textwidth]{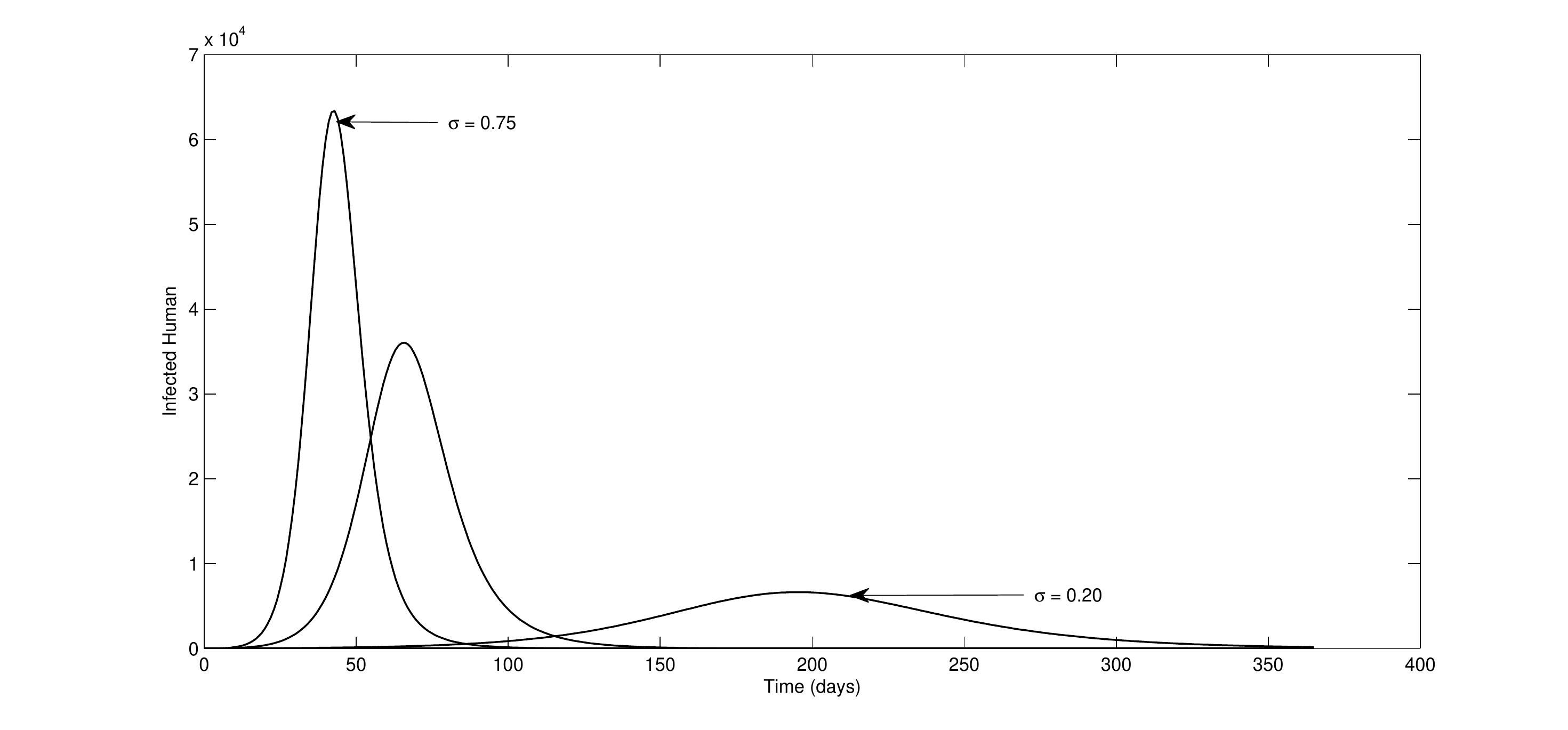}}
\subfloat[Endemic scenario]{\label{end_sigma_vary}
\includegraphics[width=0.49\textwidth]{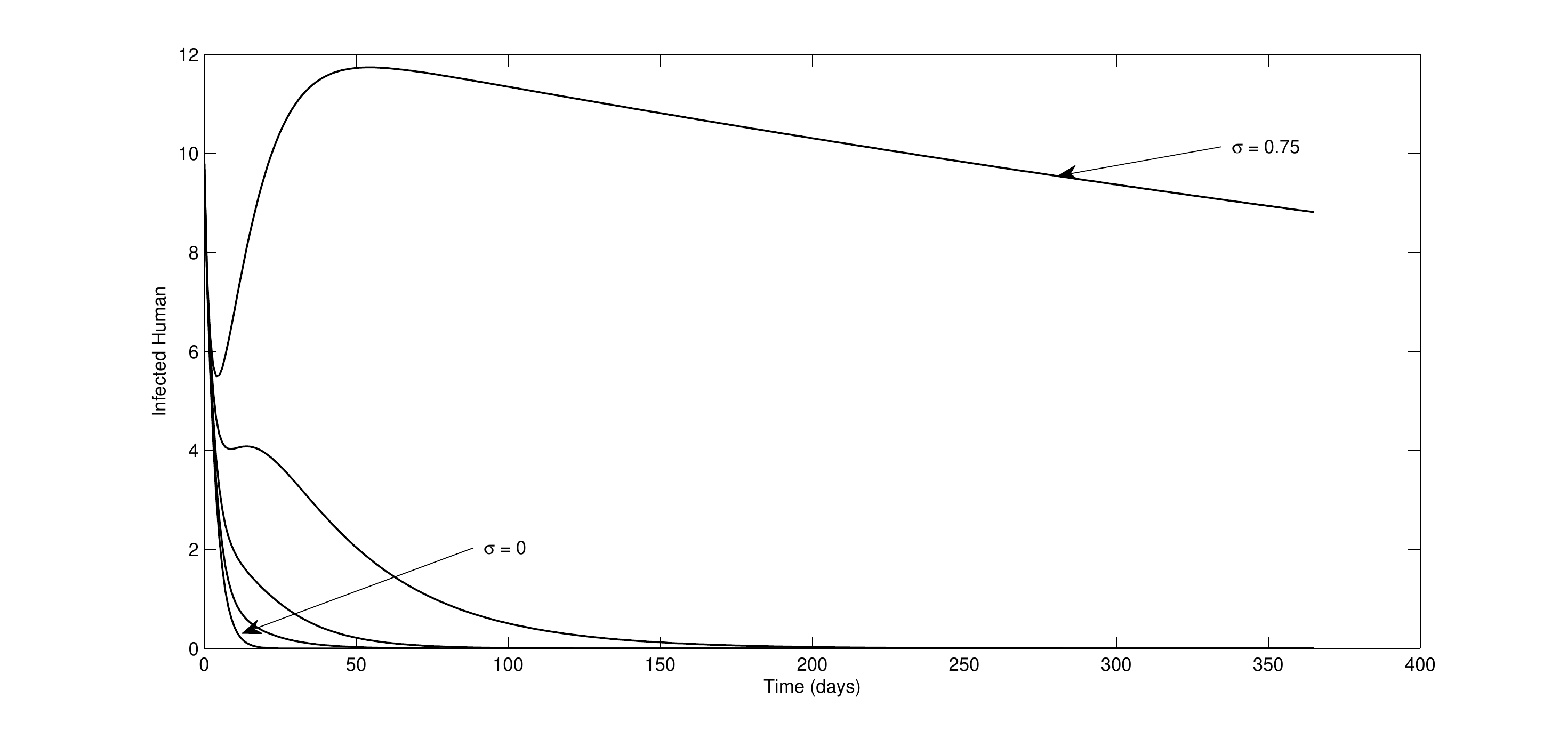}}
\caption{Infected humans in an outbreak: in cases (a) and (b) varying
the proportion of susceptible population vaccinated ($\psi=0.05, 0.10, 0.25, 0.50, 1$)
with a vaccine simulating 80\% of effectiveness ($\sigma=0.2$); in cases (c) and (d)
varying the efficacy level of the vaccine ($\sigma=0, 0.10, 0.20, 0.50, 0.75$)
and considering that 85\% of the human population is vaccinated ($\psi=0.85$)}
\label{cap7_sigma_variation}
\end{figure}

Figure~\ref{cap7_sigma_variation} shows several simulations,
by varying the vaccine efficacy and the percentage of population that is vaccinated.
Comparing with Figure~\ref{cap7_epid_psi_variation}, in the epidemic scenario
with a perfect vaccine, the number of human infected has reached to a maximum peak
of 1200 cases per day, in the worst scenario ($\psi=0.05$). Using an imperfect vaccine,
with a level of efficacy of 80\% (Figure~\ref{epid_psi_vary}),
with the same values for $\psi$, the maximum peak increases until 9000 cases.

We conclude that production of a vaccine with a high level
of efficacy has a preponderant role in the reduction of the disease spread.
Figures~\ref{epid_sigma_vary} and \ref{end_sigma_vary} reinforce the previous sentence.
Assuming that 85\% of the population is vaccinated, the numbers of infected cases
decreases sharply with the increasing of the  effectiveness level of the vaccine.

According to \cite{DeRoeck2003}, an acceptable level of efficacy is at
least 80\% against all four serotypes, and 3 to 5
years for the length of protection. These are commonly considered across countries, as the
minimum acceptable levels.

In the next subsection we study another type of imperfect vaccine:
a vaccine that confers a limited life-long protection.


\subsection{Random mass vaccination with waning immunity}
\label{sec:7:2:4}

Until the 1990s, this was an universal assumption
of mathematical models of vaccination: there is no
waning of vaccine-induced immunity. This assumption was
routinely made because, for most of the major vaccines against childhood
infectious diseases, it is approximately correct \cite{Scherer2002}.

Suppose that the immunity, obtained by the vaccination process, is temporary.
Assume that immunity has the waning rate $\theta$. Then the model for humans is given by
\begin{equation}
\label{cap7_ode_vaccination_waning}
\begin{cases}
\frac{dS_h}{dt} = \mu_h N_h  +\theta V_h
- \left(B\beta_{mh}\frac{I_m}{N_h}+\psi +\mu_h\right)S_h\\
\frac{dV_h}{dt} = \psi S_h-\left(\theta +\mu_h\right)V_h\\
\frac{dI_h}{dt} = B\beta_{mh}\frac{I_m}{N_h}S_h -(\eta_h+\mu_h) I_h\\
\frac{dR_h}{dt} = \eta_h I_h - \mu_h R_h.
\end{cases}
\end{equation}

This model can be represented by the following epidemiological scheme
(Figure~\ref{cap7_modeloSVIR_waning}).

\begin{figure}[ptbh]
\center
\includegraphics[scale=0.6]{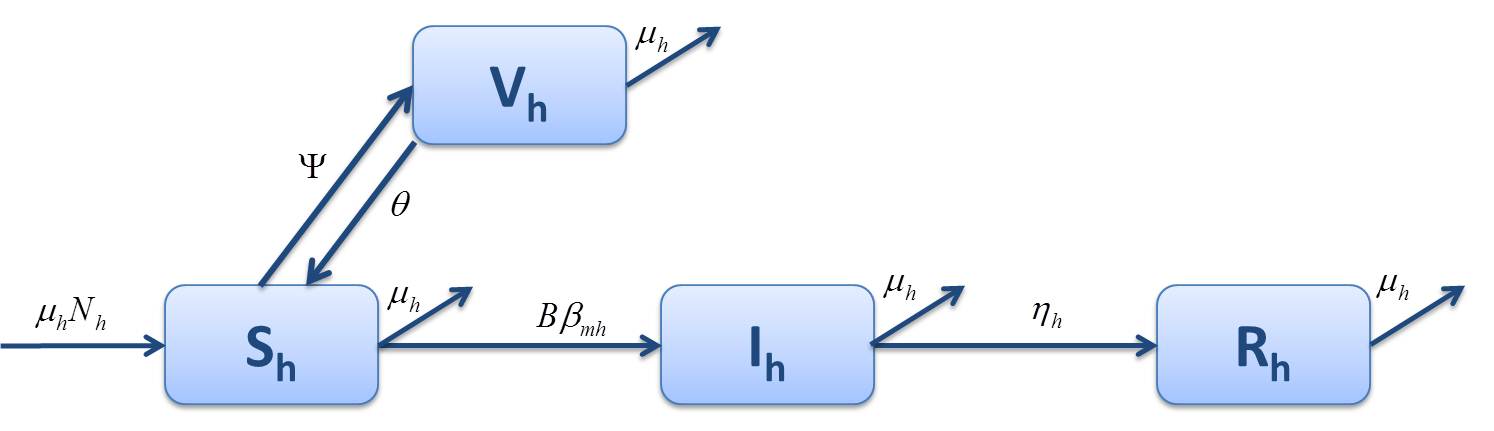}
\caption{\label{cap7_modeloSVIR_waning} Epidemiological $SVIR$ model
for human population with a waning immunity vaccine }
\end{figure}

This leads naturally to a new basic reproduction number.

\begin{definition}[Basic reproduction number with an imperfect vaccine]
\label{chap7_thm:r0:waning}\index{Basic reproduction number}
The basic reproduction number with an imperfect vaccine, $\mathcal{R}_0^{\theta}$, associated to the
differential system \eqref{cap7_ode_vaccination_waning} is
\begin{equation*}
\label{eq:R0:waning}
\mathcal{R}_0^{\theta} = \mathcal{R}_0^{\psi}
\end{equation*}
\noindent where $\mathcal{R}_{0}^{\psi}$ is defined in (\ref{eq:R0_mass}).
\end{definition}

According to \cite{Stechlinski2009}, the basic reproduction numbers are the same,
$\mathcal{R}_0^{\theta}$ and $\mathcal{R}_0^{\psi}$,
because the disease will still spread at the same rate with or without temporary immunity.
However, we should expect that the convergence rate will be different
between the random mass vaccination and random mass vaccination with waning immunity,
since the disease will be eradicated faster in the constant treatment model without
waning immunity compared to the other with waning immunity.

Figure~\ref{cap7_theta_variation} illustrates this statement.
Considering that 85\% of human population is vaccinated, the number of infected
is increasing as the value of the waning immunity is growing
($\theta=0, 0.05, 0.10, 0.15, 0.20$).

\begin{figure}
\centering
\subfloat[Epidemic scenario]{\label{cap7_epid_theta_variation}
\includegraphics[width=0.49\textwidth]{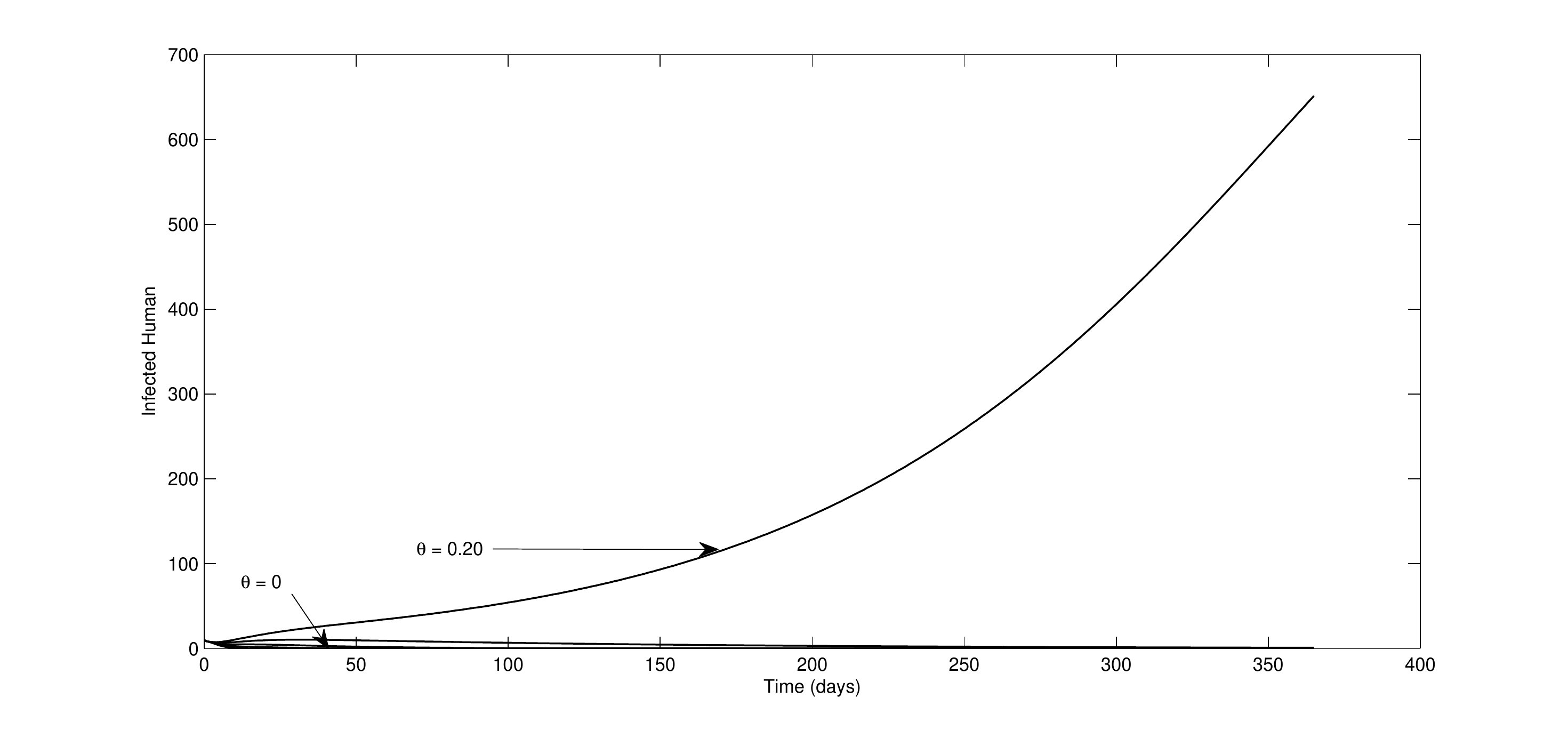}}
\subfloat[Endemic scenario]{\label{cap7_end_theta_variation}
\includegraphics[width=0.49\textwidth]{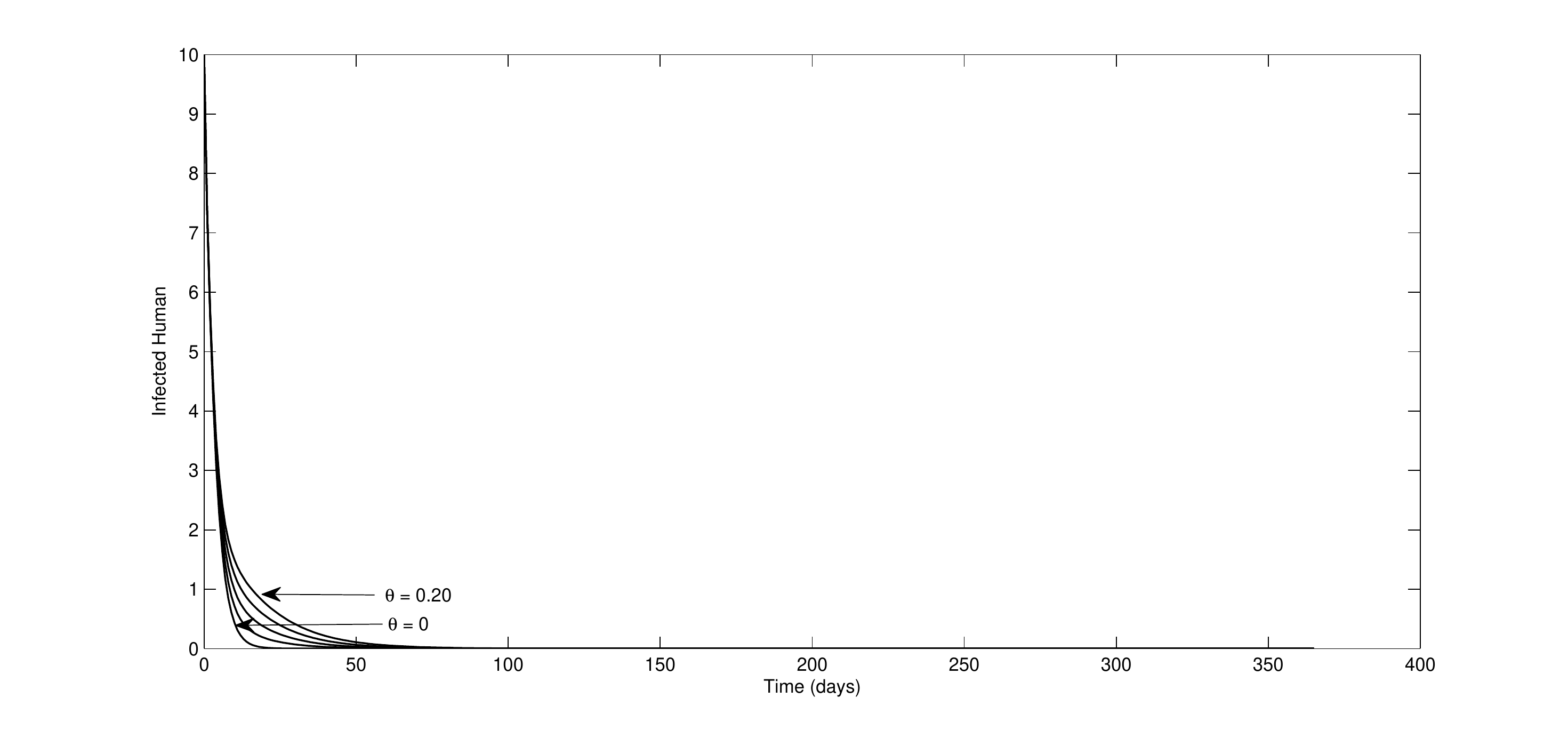}}
\caption{Infected humans in an outbreak considering 85\% of human
population vaccinated ($\psi=0.85$) and varying waning immunity
$(\theta=0, 0.05, 0.10, 0.15, 0.20$)}
\label{cap7_theta_variation}
\end{figure}

Depending on the vaccine that will be available on the market,
it will be possible to choose or even combine features.
In the next section we will define the vaccination process
as a control system.


\section{Vaccine as a control}
\label{sec:7:3}

In this section we consider a SIR model for humans and an ASI model for mosquitoes.
The parameters remain the same as in the previous chapter.

The vaccination is seen as a control variable to reduce
or even eradicate the disease. Let $u$ be the control variable
related to the proportion of susceptible humans that are vaccinated.
A random mass vaccination with waning immunity is selected. In this way,
a parameter $\theta$ associated to the control $u$ represents
the waning immunity process. Figure~\ref{cap7_modeloSIR_control_vaccine}
shows the epidemiological scheme for the human population.

\begin{figure}[ptbh]
\center
\includegraphics[scale=0.6]{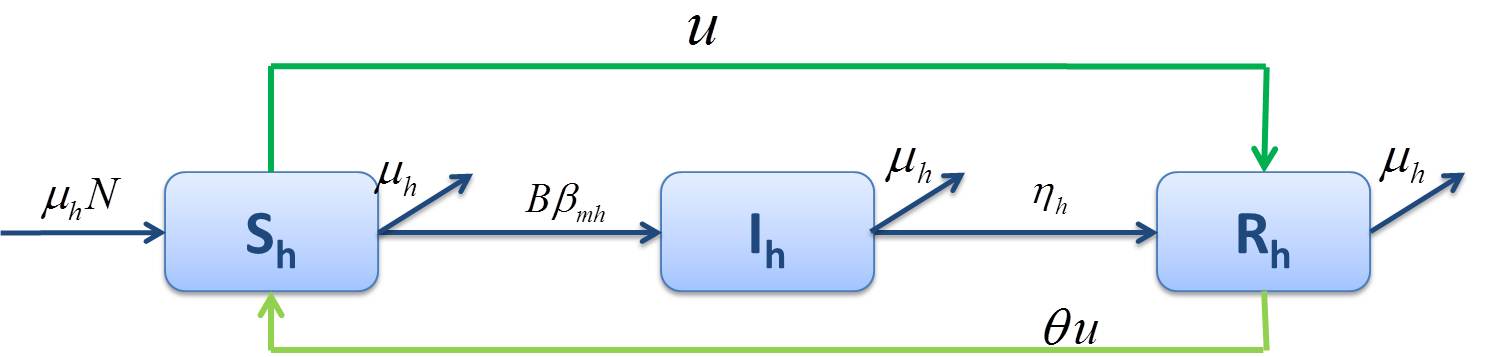}
\caption{\label{cap7_modeloSIR_control_vaccine} Epidemiological $SIR$ model
for the human population using the vaccine as a control}
\end{figure}

The model is described by an initial value problem with a system
of six differential equations:
\begin{equation}
\label{cap7_ode}
\begin{cases}
\frac{dS_h}{dt} = \mu_h N_h- \left(B\beta_{mh}\frac{I_m}{N_h}+\mu_h+u\right)S_h+\theta u R_h\\
\frac{dI_h}{dt} = B\beta_{mh} \frac{I_m}{N_h} S_h -(\eta_h+\mu_h) I_h\\
\frac{dR_h}{dt} = \eta_h I_h + u S_h - \left(\theta u+\mu_h\right) R_h\\
\frac{dA_m}{dt} = \varphi \left(1-\frac{A_m}{k N_h}\right) (S_m+I_m) - \left(\eta_A+\mu_A\right) A_m\\
\frac{dS_m}{dt} = \eta_A A_m - \left(B \beta_{hm}\frac{I_h}{N_h}+\mu_m\right) S_m\\
\frac{dI_m}{dt} = B \beta_{hm}\frac{I_h}{N_h}S_m -\mu_m I_m
\end{cases}
\end{equation}

The main aim is to study
the optimal vaccination strategy, considering both
the costs of treatment of infected individuals and the costs
of vaccination. The objective is to
\begin{equation}
\label{cap7_functional}
\text{ minimize } J[u]=\int_{0}^{t_f}\left[\gamma_D I_h(t)^2
+\gamma_V u(t)^2\right]dt,
\end{equation}
\noindent where $\gamma_D$ and $\gamma_V$ are positive constants
representing the weights of the costs
of treatment of infected and vaccination, respectively.

Using OC theory is possible to solve the problem.


\subsection{Pontryagin's Maximum Principle}

Let us consider the following set of admissible control functions:
$$
\Delta=\{u(\cdot)\in(L^{\infty}(0,t_f))|0\leq u(t)\leq 1, \forall t \in [0,t_f]\}.
$$

\begin{theorem}
Problem (\ref{cap7_ode})-(\ref{cap7_functional})
with initial conditions in Table (\ref{chap7_initial_conditions}), admits
a unique optimal solution $\left(S_h^{*}(\cdot),I_h^{*}(\cdot),
R_h^{*}(\cdot),A_m^{*}(\cdot),S_m^{*}(\cdot),I_m^{*}(\cdot)\right)$
associated with an optimal control $u^{*}(\cdot)$ on $[0,t_f]$,
with a fixed final time $t_f$. Moreover, there exists adjoint functions,
$\lambda_i^{*}(\cdot)$, $i=1...6$ such that
\begin{equation}
\label{cap7_lambda_res}
\begin{tabular}{l}
$\left\{
\begin{array}{l}
\dot{\lambda}_{1}^{*} = (\lambda_1-\lambda_2) \left(B \beta_{mh} \frac{I_m}{N_h}\right)
+\lambda_1 \mu_h+(\lambda_1-\lambda_3)u\\
\dot{\lambda}_{2}^{*} = -2\gamma_D I_h+\lambda_2(\eta_h+\mu_h)
-\lambda_3\eta_h+(\lambda_5-\lambda_6)\left(B\beta_{hm}\frac{S_m}{N_h}\right)\\
\dot{\lambda}_{3}^{*} = -\lambda_1\theta u+\lambda_3(\mu_h+\theta u)\\
\dot{\lambda}_{4}^{*} = \lambda_4 \varphi \frac{S_m+I_m}{k N_h}
+\lambda_4(\eta_A+\mu_A)-\lambda_5\eta_A\\
\dot{\lambda}_{5}^{*} = -\lambda_4 \varphi \left(1-\frac{A_m}{k N_h}\right)
+(\lambda_5-\lambda_6) B\beta_{hm} \frac{I_h}{N_h}+\lambda_5\mu_m\\
\dot{\lambda}_{6}^{*} = (\lambda_1-\lambda_2) \left(B\beta_{mh}\frac{S_h}{N_h}\right)
-\lambda_4 \varphi \left(1-\frac{A_m}{k N_h}\right) + \lambda_6 \mu_m\\
\end{array}
\right. $\\
\end{tabular}
\end{equation}
with the transversality conditions\index{Transversality condition}
$\dot{\lambda}_{i}(t_f)=0$, i=1, \ldots 6. Furthermore,
\begin{equation}
\label{cap7_control}
u^{*}=\min\left\{1,\max\left\{0,\frac{\left(\lambda_1
-\lambda_3\right)\left(S_h-\theta R_h\right)}{2\gamma_V}\right\}\right\}.
\end{equation}
\end{theorem}

\medskip

\begin{proof}
The existence of optimal solutions
$\left(S_h^{*}(\cdot),I_h^{*}(\cdot),R_h^{*}(\cdot),A_m^{*}(\cdot),S_m^{*}(\cdot),I_m^{*}(\cdot)\right)$
associated to the optimal control $u^{*}(\cdot)$ comes from the convexity of the integrand
of the cost function (\ref{cap7_functional}) with respect to the control $u$ and the Lipschitz property
of the state system with respect to state variables $\left(S_h,I_h,R_h,A_m,S_m,I_m\right)$
(for more details see \cite{Cesari1983}). According to the Pontryagin Maximum Principle
\cite{Pontryagin1962}\index{Pontryagin's Maximum Principle}, if $u^{*}(\cdot) \in \Delta$
is optimal for the problem considered, then there exists a nontrivial absolutely continuous mapping
$\lambda:[0,t_f]\rightarrow \mathbb{R}$, $\lambda(t)=\left(\lambda_1(t),\lambda_2(t),\lambda_3(t),
\lambda_4(t),\lambda_5(t),\lambda_6(t)\right)$, called the adjoint vector, such that
\begin{equation}
\label{cap7_frac}
\begin{cases}
\dot{S}_h=\frac{\partial H}{\partial \lambda_1}\\
\dot{I}_h=\frac{\partial H}{\partial \lambda_2}\\
\dot{R}_h=\frac{\partial H}{\partial \lambda_3}\\
\dot{A}_m=\frac{\partial H}{\partial \lambda_4}\\
\dot{S}_m=\frac{\partial H}{\partial \lambda_5}\\
\dot{I}_m=\frac{\partial H}{\partial \lambda_6}\\
\end{cases}
\end{equation}
and
\begin{equation}
\label{cap7_lambda_frac}
\begin{cases}
\dot{\lambda}_1=-\frac{\partial H}{\partial S_h}\\
\dot{\lambda}_2=-\frac{\partial H}{\partial I_h}\\
\dot{\lambda}_3=-\frac{\partial H}{\partial R_h}\\
\dot{\lambda}_4=-\frac{\partial H}{\partial A_m}\\
\dot{\lambda}_5=-\frac{\partial H}{\partial S_m}\\
\dot{\lambda}_6=-\frac{\partial H}{\partial I_m}\\
\end{cases}
\end{equation}
where the Hamiltonian $H$\index{Hamiltonian} is defined by
\begin{equation*}
\begin{tabular}{ll}
$H$&$=H(S_h, I_h,R_h,A_m,S_m,I_m,\lambda, u)$\\
& $=\gamma_D I_h(t)^2+\gamma_V u(t)^2$\\
&$+\lambda_1(\mu_h N_h- \left(B\beta_{mh}\frac{I_m}{N_h}+\mu_h+u\right)S_h+\theta u R_h)$\\
&$+\lambda_2(B\beta_{mh} \frac{I_m}{N_h} S_h -(\eta_h+\mu_h) I_h)$\\
&$+\lambda_3(\eta_h I_h + u S_h - \left(\theta u+\mu_h\right) R_h)$\\
&$+\lambda_4\left(\varphi \left(1-\frac{A_m}{k N_h}\right) (S_m+I_m) - \left(\eta_A+\mu_A\right) A_m\right)$\\
&$+\lambda_5(\eta_A A_m - \left(B \beta_{hm}\frac{I_h}{N_h}+\mu_m\right) S_m)$\\
&$+\lambda_6(B \beta_{hm}\frac{I_h}{N_h}S_m -\mu_m I_m)$\\
\end{tabular}
\end{equation*}
\noindent and the minimality condition
\begin{equation}
\label{cap7_Hamil_min}
\begin{tabular}{l}
$H\left(S_h^{*}(t),I_h^{*}(t),R_h^{*}(t),A_m^{*}(t),S_m^{*}(t),I_m^{*}(t),\lambda^{*}(t),u^{*}(t)\right)$\\
$=\underset{u}{min}H\left(S_h^{*}(t),I_h^{*}(t),R_h^{*}(t),A_m^{*}(t),S_m^{*}(t),I_m^{*}(t),\lambda^{*}(t),u\right)$\\
\end{tabular}
\end{equation}
\noindent holds almost everywhere on $[0,t_f]$. Moreover, the transversality
conditions\index{Transversality condition} $\lambda_i(t_f)=0$, $i=1,\ldots6,$ hold.
System (\ref{cap7_lambda_res}) is derived from (\ref{cap7_lambda_frac}),
and the optimal control (\ref{cap7_control}) comes from
the minimality condition (\ref{cap7_Hamil_min}).
\end{proof}


\subsection{Numerical simulation and discussion}
\label{numerical}

The simulations were carried out using the values of the previous section.
The system was normalized, using the same strategy as in Chapter~\ref{chp5}.
It was considered that the waning immunity was at a rate of $\theta=0.05$.
The OC problem was solved using two methods: direct \cite{Betts2001,Trelat2005}\index{Direct method}
and indirect \cite{Lenhart2007}\index{Indirect methods}.
The direct method uses the cost functional \eqref{cap7_functional}
and the state system \eqref{cap7_ode} and was solved by
\texttt{DOTcvp}\index{DOTcvp} \cite{Dotcvp}. The indirect method
used is an iterative method with a Runge--Kutta scheme\index{Runge Kutta scheme},
solved through \texttt{ode45}\index{ode45 routine} of \texttt{MatLab}\index{Matlab}.

Figure~\ref{cap7_optimal_control} shows the optimal control obtained by both methods.
Notice that \texttt{DOTcvp} only gives the optimal control as a constant piecewise function.
Table~\ref{cap7_resultados} shows the costs obtained by the two
methods in both scenarios. The indirect method\index{Indirect methods} gives a lower cost.
This method uses more mathematical theory about the problem, such as the adjoint system (\ref{cap7_frac})
and optimal control expression (\ref{cap7_control}).
Therefore it makes sense that the indirect method produces a better solution.

\begin{figure}[ptbh]
\begin{center}
\subfloat[Epidemic scenario]{\label{cap7_optimal_control_epidemic}
\includegraphics[width=0.49\textwidth]{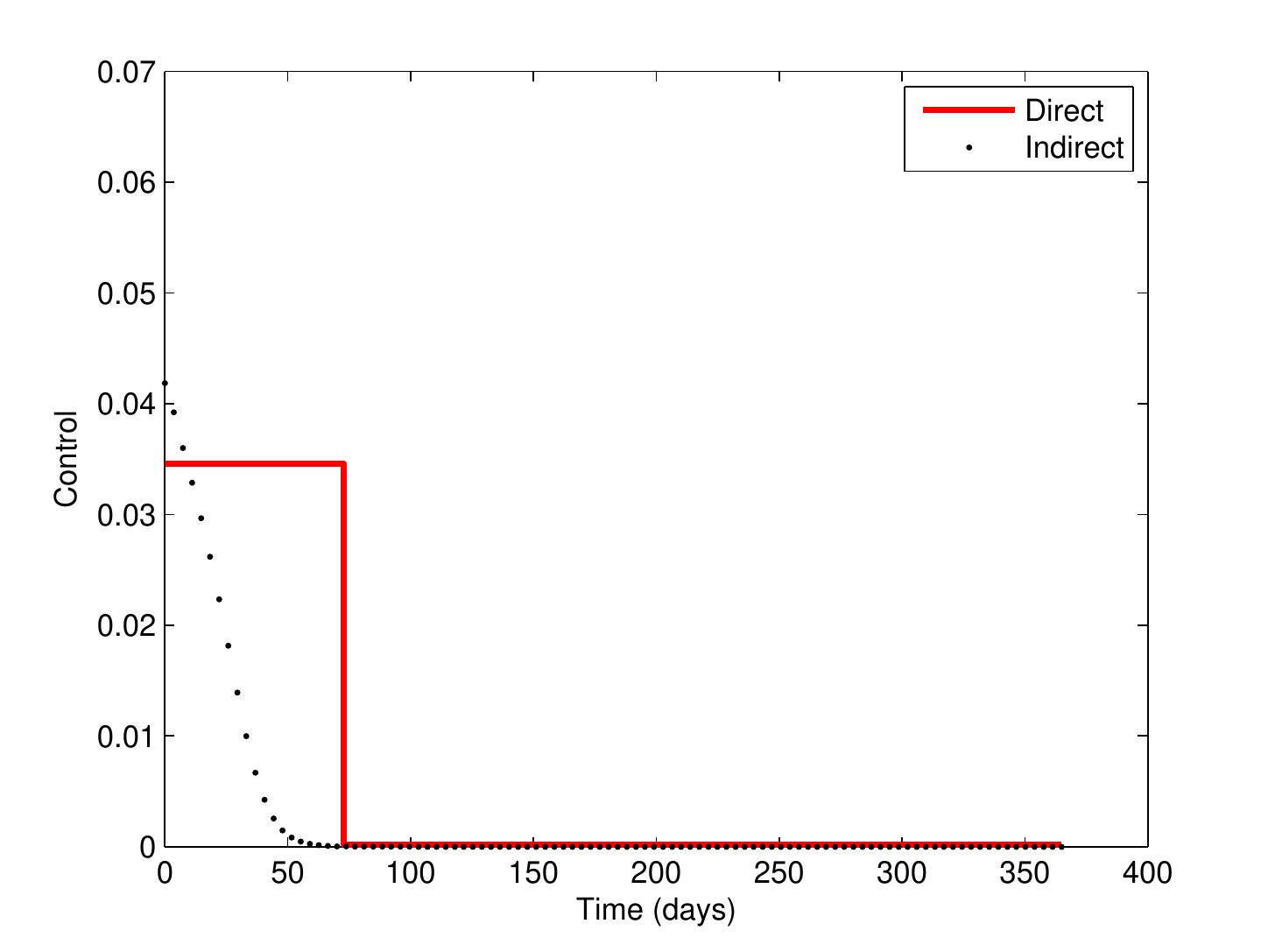}}
\subfloat[Endemic scenario]{\label{cap7_optimal_control_endemic}
\includegraphics[width=0.49\textwidth]{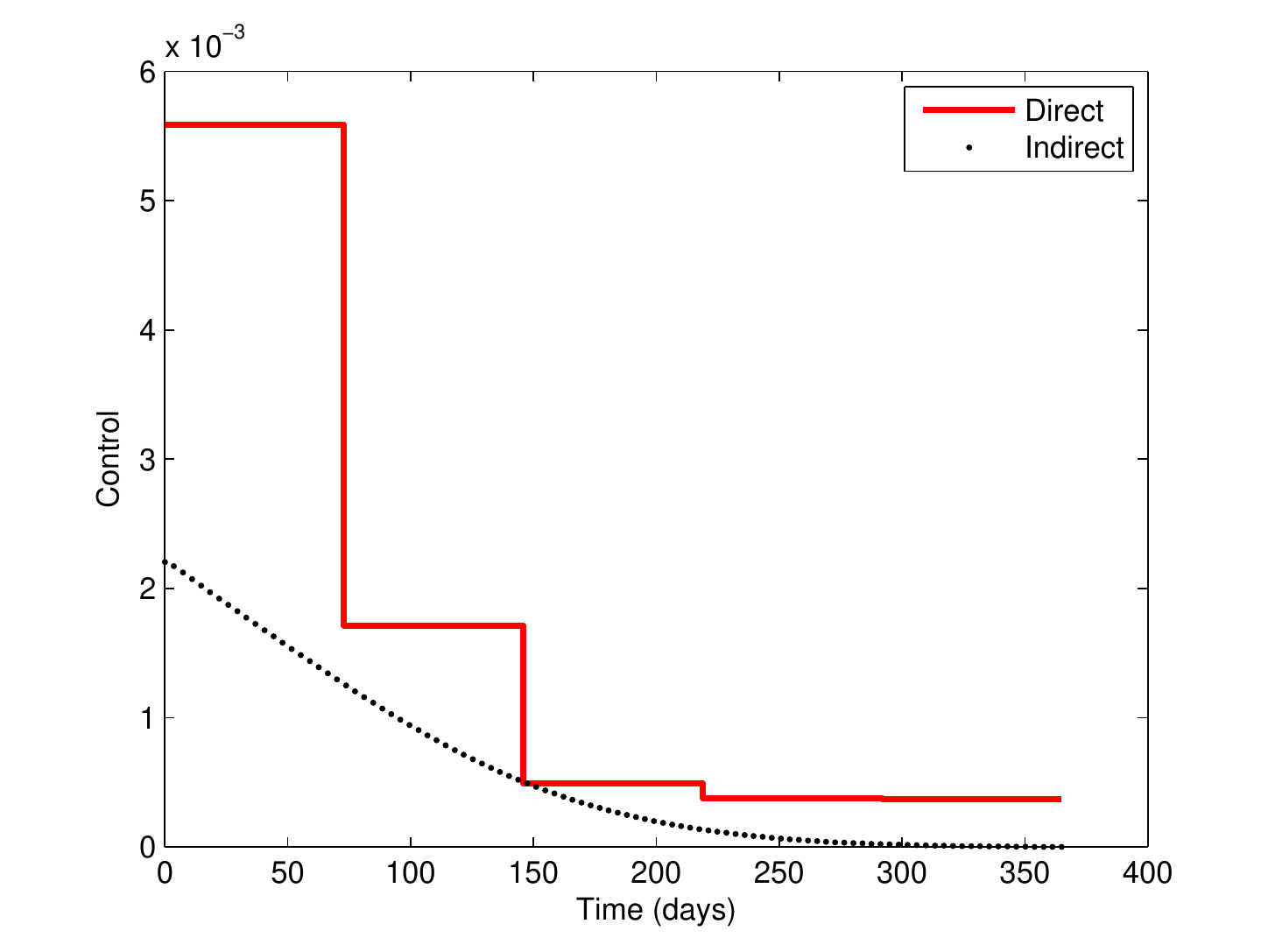}}
\caption{Optimal control with direct and indirect approach in both scenarios}
\label{cap7_optimal_control}
\end{center}
\end{figure}

\begin{table}
\begin{center}
\begin{tabular}{lccc}
\hline
Method & Epidemic scenario & Endemic scenario\\
\hline
Direct (DOTcvp) & 0.07505791 & 0.00189056\\
Indirect (backward-forward) & 0.06070556 & 0.00080618\\
\hline
\end{tabular}
\caption{Optimal values of the cost functional (\ref{cap7_functional})} \label{cap7_resultados}
\end{center}
\end{table}

Using the optimal solution as reference, some tests were performed,
regarding infected indivi\-duals and costs, when no control ($u\equiv0$)
or upper control ($u\equiv1$) is applied.

Table~\ref{cap7_results_no_upper_control} shows the results for \texttt{DOTcvp}\index{DOTcvp}
in the three situations. In both scenarios, using the optimal strategy
of vaccination produces better costs with the disease, when compared to not doing anything.
Once there is no control, the number of infected humans is higher and produces a more expensive cost functional.

\begin{table}
\begin{center}
\begin{tabular}{lccc}
\hline
 & Epidemic scenario & Endemic scenario\\
\hline
optimal control & 0.07505791 & 0.00189056\\
no control & 0.32326592 & 0.01045990\\
upper control & 147.82500296 & 116.800000275\\ \hline
\end{tabular}
\caption{Values of the cost functional with optimal control,
no control ($u\equiv 0$) and upper control ($u\equiv 1$)} \label{cap7_results_no_upper_control}
\end{center}
\end{table}

Figure~\ref{cap7_infected_no_upper_control} shows the number of infected
humans when different controls are considered. It is possible to see that
using the upper control, which means that everyone is vaccinated, implies
that just a few individuals were infected, allowing eradication of the
disease. Although the optimal control, in the sense of objective
(\ref{cap7_functional}), allows the occurrence of an outbreak,
the number of infected individuals is much lower when
compared with a situation where no one is vaccinated.

\begin{figure}[ptbh]
\begin{center}
\subfloat[Epidemic scenario]{\label{cap7_epidemic}
\includegraphics[width=0.49\textwidth]{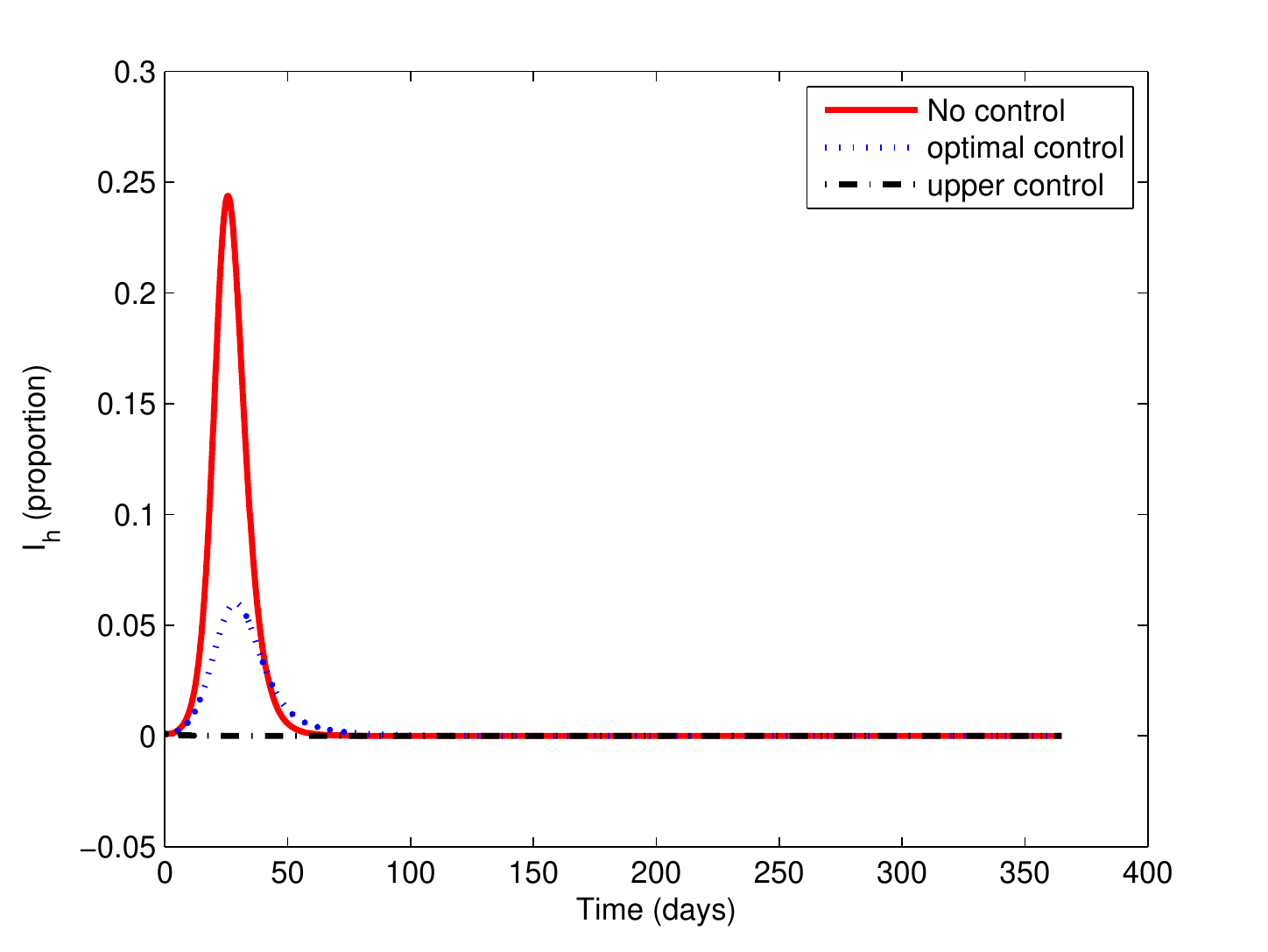}}
\subfloat[Endemic scenario]{\label{cap7_endemic}
\includegraphics[width=0.49\textwidth]{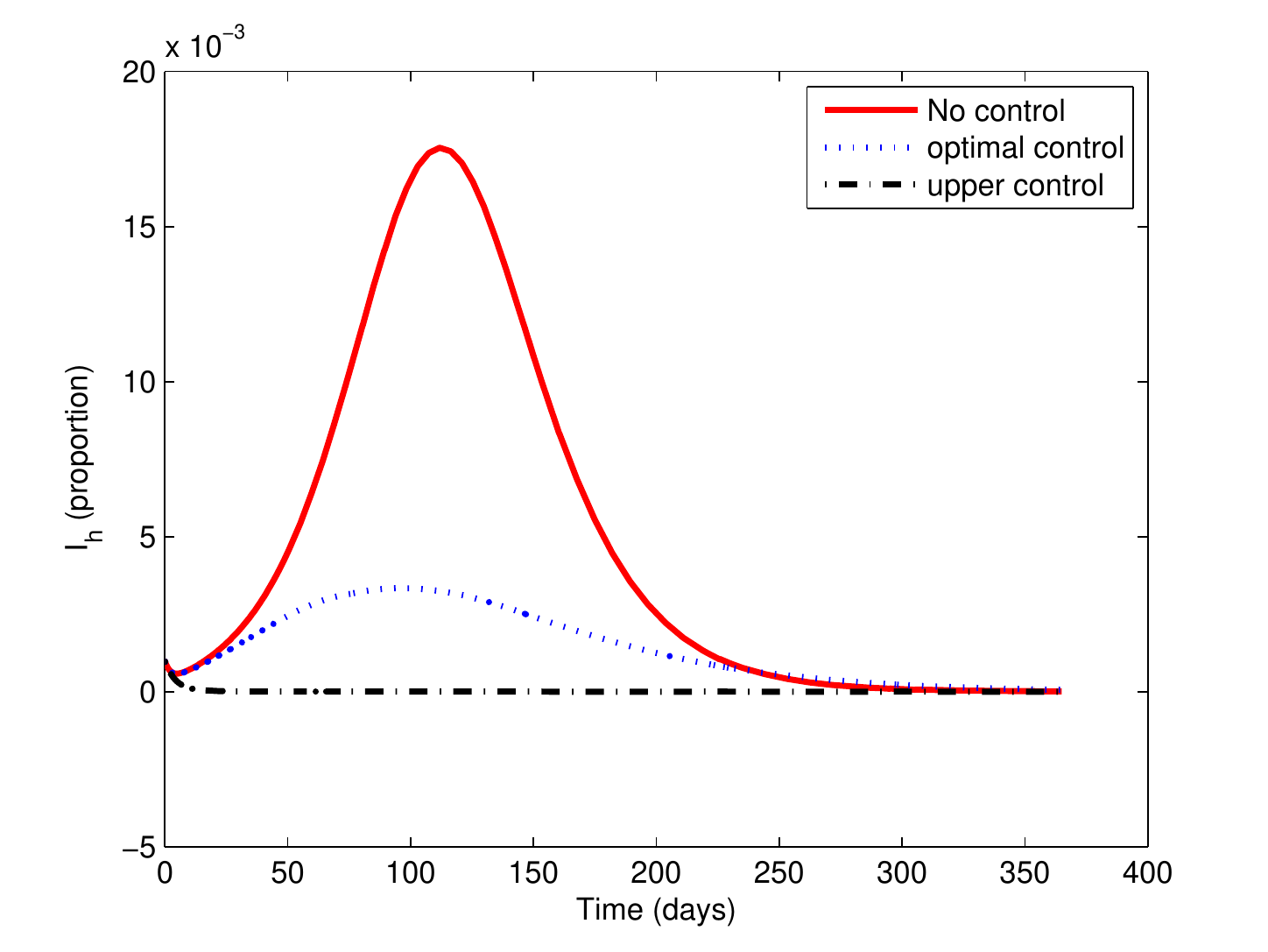}}
\caption{Optimal control obtained from a direct and indirect approach in both scenarios}
\label{cap7_infected_no_upper_control}
\end{center}
\end{figure}

We conclude that a vaccination campaign in the susceptible population,
and assuming a considerable efficacy level of the vaccine,
can quickly decrease the number of infected people.


\section{Conclusions}

The worldwide expansion of the Dengue fever is a growing health problem.
Dengue vaccine is an urgent challenge that needs to be overcome.
This may be commercially available within a few years, when the researchers
find a formula that protect against all four Dengue viruses.
A vaccination program is seen as an important measure used
in infectious disease control and immunization and
eradication programs.

In the first part of the chapter, different types of vaccine,
as well as their features and some of coverage thresholds,
were introduced. The main idea was to study some types of vaccines
in order to cover most of the future vaccine features.
The main goal of a vaccination program is to reduce the prevalence
of an infectious disease and ultimately to eradicate it.
It was shown that eradication success depends on the type of vaccine
as well as on the vaccination coverage. The imperfect vaccines may not
completely prevent infection but could reduce the probability of being infected,
thereby reducing the disease burden. In this study all the simulations were done
using epidemic and endemic scenarios to illustrate distinct realities.

A second analysis was made, using an OC approach. The vaccine behaved
as a new disease control variable and, when available, can be a promising
strategy to fight the disease.

Dengue is an infectious tropical disease difficult to prevent and manage.
Researchers agree that the development of a vaccine
for Dengue is a question of high priority.
In the present study we have shown how a vaccine results
in saving lives and at the same time in a
reduction of the budget related with the disease.
As future work we intend to study the interaction of a Dengue
vaccine with other kinds of control already investigated in the literature,
such as insecticide and educational campaigns
\cite{Sofia2010c,Sofia2012}.

This chapter was based on work accepted
in the peer reviewed proceedings \cite{Sofia2011}.

\clearpage{\thispagestyle{empty}\cleardoublepage}


\phantomsection\addcontentsline{toc}{chapter}{Conclusion and future directions}
\chapter*{Conclusions and future perspectives}
\label{concl}

Mathematical models can be a powerful tool to understand epidemiological phenomena.
These models can be used to compare, plan, implement and evaluate several programs
related to detection, prevention and control of infectious diseases.
Indeed, one of the most important issues in epidemiology is to improve control strategies
with the final goal to reduce or even eradicate the disease. In this thesis were constructed
and analyzed some models for the spreading of diseases, particularly for Dengue Fever.

The links between health and sustainable development are illustrated for this disease.
Any attempt in predicting and preventing the disasters caused by the disease will imply
a global strategy that takes into account environmental conditions, levels of poverty
and illiteracy and, eventually, degree of coverage by vaccination programs \cite{Derouich2006}.

Although vector control strategies were already available before the Second World War,
Dengue pandemic was underestimated. It became a global public health problem
in the past 60 years and a major concern for WHO.

The main contributions of this thesis can be classified into three main categories,
namely, model formulation, mathematical and computational analysis, and contributions
to public health/intervention design.

\bigskip


\noindent \textbf{Model formulation}

Deterministic models for assessing the combined impact
of several control measures in Dengue disease were considered.

In Chapter~\ref{chp4} an old OC problem for Dengue was revisited.
It was shown that new and robust tools bring refreshing solutions
to the problem. Some analysis with discretization schemes were carried out
in order to understand the best way to implement the direct methods\index{Direct method}
in future approaches. For this problem, is was better to use robust solvers
than to implement higher order discretization methods,
due to the increasing of problem's dimension.

In Chapter~\ref{chp5}, a SEIR+ASEI model was studied.
Threshold criteria was established ensuring the disease eradication
and hence convergence to the so-called disease free solution.
Using real data from Cape Verde, the application of insecticide
in the country during the outbreak was simulated and its repercussions were analyzed.
The study of this outbreak was performed in several phases: firstly, using a previously
calculated constant control for the whole period, with the aim of having
a basic reproduction number\index{Basic reproduction number} below unity; then,
several periodic strategies were studied in order to find the best logistics approaches
to implement that keep $\mathcal{R}_{0}$ less than one; finally, an OC approach to compare
with the previous suboptimal approaches was used. For this final phase, several perspectives
of the problem were analyzed, including bioeconomic, medical and economic approaches. 
Depending of the main target to achieve, the results for the control and infected individuals
vary.  

In Chapter~\ref{chp6}, a SIR+ASI model for Dengue was presented. The lost of two differential equations,
was compensated by the introduction of two more controls: in addition to adulticide,
larvicide and mechanical control were introduced. Similarly to Chapter~\ref{chp5}, a threshold criteria
for the eradication of the disease was established. The influence of the controls in this threshold
was analyzed. Then, varying the controls, separately and simultaneously, an analysis of the
importance/consequence of each control in the development of the disease was made.
Finally, an OC approach using different weights for the variables in the functional was studied,
in order to establish the best optimal curve for each control and the respective effects
in the development/erradication of the disease.

In Chapter~\ref{chp7}, some simulations with different types of vaccines were made.
As a Dengue vaccine is not yet available, distinct hypothetical models to introduce
the vaccine were studied. In Section~\ref{sec:7:2}, was considered a new compartment
$V_h$ producing a new model SVIR+ASI. This research comprised SVIR models with perfect vaccines
--- constant vaccination of newborns and constant vaccination of susceptibles ---
and imperfect vaccines --- using a level of efficacy below 100\% and also with waning immunity.
In Section~~\ref{sec:7:3}, the vaccination process was studied as a control of the epidemic model.
For this, as in the previous chapter, an OC approach was presented. All the simulations
were done in epidemic and endemic scenarios, in order to understand what type
of repercussions could bring each kind of vaccines.

\bigskip


\noindent\textbf{Mathematical and computational analysis}

In this work, there was a concern in producing a mathematical analysis
for all new models presented. Epidemiological concepts, such as the basic
reproduction number\index{Basic reproduction number} and equilibrium
points\index{Equilibrium point}, were calculated. The equilibrium points
were classified and their local stability analyzed. The OC theory was used
in order to provide the best strategies for each model,
involving direct and indirect approaches.

Throughout the thesis, a set of software packages were used, showing
the importance of the these tools in the development of some mathematical fields.
To solve ODE systems, codes in \texttt{Matlab}\index{Matlab} and \texttt{Scilab}\index{Scilab}
were implemented. To calculate the equilibrium points and the basic reproduction number,
\texttt{Mathematica}\index{Mathematica} and \texttt{Maple}\index{Maple} were used.
For the OC approach, several packages were also selected: from programmed codes in
\texttt{AMPL}\index{AMPL} and run in NEOS Server\index{NEOS platform}
(\texttt{Ipopt}\index{Ipopt}, \texttt{Snopt}\index{Snopt}, \texttt{Knitro}\index{Knitro},
\texttt{Muscod-II}\index{Scilab}) to \texttt{Fortran} codes in the Linux environment
(\texttt{OC-ODE})\index{OC-ODE} or even programs coded in \texttt{Matlab}
(\texttt{DOTcvp}\index{DOTcvp} and indirect methods\index{Indirect methods}).
The software choice was varying during the research process, due to availability,
robustness and solving speed. The main codes developed in this thesis are available at:
\url{https://sites.google.com/site/hsofiarodrigues/home/phd-codes-1} \cite{SofiaSITE}.
Using direct and indirect methods, the solutions obtained were similar, 
reinforcing the confidence in the results. 
\bigskip


\noindent\textbf{Public health/intervention design}

The study provides some important epidemiological insights into the impact
of vector control measures. Dengue burden decreases with the increasing
of vector control measures (adulticide, larvicide and mechanical control).
Furthermore, the adulticide should be the first/main measure to apply when
an outbreak occurs, whereas the other two measures should be considered
as a long time prevention.

The last control measure to be studied was the vaccination.
It was shown that the vaccine, when available, could bring advantages
not only in the reduction of infected individuals,
but also decreasing the disease costs.

\bigskip
\bigskip
\bigskip


A PhD thesis is always an unfinished process,
but some day it is necessary to stop.
Therefore, some topics were not explored
and can be understood as future directions of the work.

A first suggestion is to use heterogeneity for both populations,
dividing each one in more compartments \cite{Liu2006}.
Human population is not immunologically homogeneous,
presenting groups with distinct level of risk, related with age,
sex or the presence of a disease or immunosuppressive drugs
which would be the case with transplant patients or cancer suffers.
For example, the transmission probability in young children is higher
and generally with more severe symptoms when compared to adult transmission.
Moreover, the mosquito population does not have the same behavior during the whole year,
depending specially on weather conditions. Temperature and humidity
are key variables on vector population dynamics. It will be interesting
to add some seasonality factors into the last models \cite{Duque2006,Oki2011}.

Another open question is the introduction of immigration and tourism issues.
Along the work it was considered a constant population,
but the addition of new individuals could induce new outbreaks.

Other aspect is related to the disease development in the presence of several serotypes.
While in Cape Verde only one serotype was found, the interaction of several serotypes
in Asia is already a reality. This will induce changes in the model, not only increasing
of the number of variables but also causing a more expressive number of DHF cases \cite{Pessanha2011}.

Currently, Portuguese researchers are developing a new repellent for mosquito
that could be considered as a new control for the disease. This product does
not kill the mosquito, but prevents the bite by deviating it from the target
and consequently decreasing the chain of the disease transmission.

With the viability of the vaccine, it will be possible to fit a better model
according to the vaccine used. One of the possibilities is using pulse vaccination,
where children at a certain age cohorts are periodically immunized \cite{Onofrio2002}.
The theoretical challenge of pulse vaccination is the \emph{a priori} determination
of the pulse interval for specific values of $\mathcal{R}_{0}$, the proportion
of vaccinated $p$ and the population birth rate $\mu$.

Most of the analysis and models presented in this thesis can be adapted for
other vector-borne diseases --- such as malaria, yellow fever, West Nile virus,
chikungunya, Japanese encephalitis --- by just fitting some variables/parameters
and some initial assumptions intrinsic to the disease \cite{Lenhart1997,ElGohary2009,Joshi2002}.

When attempting to model epidemics and control for public health applications,
there is the compelling urge to make models as sophisticated as possible,
including many details about the host and the vector. Although this strategy may be useful
when such details are known or exist suitable data, it may lead to a false sense
of accuracy when reliable information is not available. Another approach is to keep
the model simple and, instead of using the conventional differential calculus,
to apply fractional calculus to fit to the disease reality \cite{SofiaShakoor}
or to use the general theory of time scales \cite{SofiaSPM}.


\bigskip
\bigskip
\bigskip

\begin{flushright}
\begin{minipage}[r]{9cm}
\bigskip
\small {\emph{``If people do not believe that mathematics is simple,
it is only because they do not realize how complicated life is.''}}

\begin{flushright}
\tiny{--- John Louis von Neumann, 1947}
\end{flushright}

\bigskip

\end{minipage}
\end{flushright}

\clearpage{\thispagestyle{empty}\cleardoublepage}


\renewcommand{\bibname}{References}
\phantomsection

\clearpage{\thispagestyle{empty}\cleardoublepage}


\phantomsection\addcontentsline{toc}{chapter}{Index}
\printindex
\clearpage{\thispagestyle{empty}\cleardoublepage}


\end{document}